\documentclass[
10pt, 
a4paper, 
oneside, 
headinclude,footinclude, 
]{scrartcl}

\usepackage[
nochapters, 
pdfspacing, 
dottedtoc 
]{classicthesis} 
\usepackage{amsopn}
\usepackage[T1]{fontenc} 
\usepackage{longtable}
\usepackage{multirow}
\usepackage[utf8]{inputenc} 
\usepackage{hhline}
\usepackage{graphicx} 
\graphicspath{{Figures/}} 
\usepackage{indentfirst}
\usepackage{enumitem} 
\usepackage{savesym}
\usepackage{mathrsfs}
\usepackage{xcolor}
\usepackage{geometry}
\usepackage[
    backend=biber,
    style=numeric,
    natbib=false,
    url=false, 
    doi=false,
    eprint=false,
    maxnames=50
]{biblatex}

\usepackage{esint}
\usepackage{subfig} 

\usepackage{amsmath,amssymb,amsthm,amsfonts} 

\usepackage{varioref} 

\usepackage{thmtools}
\usepackage{thm-restate}

\usepackage{hyperref}
\newcommand{\vertiii}[1]{{\left\vert\kern-0.25ex\left\vert\kern-0.25ex\left\vert #1 
    \right\vert\kern-0.25ex\right\vert\kern-0.25ex\right\vert}}

\theoremstyle{plain}
\newtheorem{teorema}{Theorem}[section]
\newtheorem{proposizione}[teorema]{Proposition}
\newtheorem{lemma}[teorema]{Lemma}
\newtheorem{corollario}[teorema]{Corollary}
\newtheorem*{theorem*}{Theorem}

\theoremstyle{definition}
\newtheorem{definizione}{Definition}[section]

\theoremstyle{remark}
\newtheorem{osservazione}{Remark}[section]

\newcommand{\N}{\mathbb{N}}

\newcommand{\R}{\mathbb{R}}

\newcommand{\HH}{\mathbb{H}}

\newcommand{\Tan}{\text{Tan}}
\newcommand{\res}

\DeclareMathOperator*{\supp}{supp}

\DeclareMathOperator*{\diam}{diam}

\DeclareMathOperator*{\dist}{dist_{eu}}

\hypersetup{
colorlinks=true, breaklinks=true,bookmarksnumbered,
urlcolor=webbrown, linkcolor=RoyalBlue, citecolor=webgreen, 
pdftitle={}, 
pdfauthor={\textcopyright}, 
pdfsubject={}, 
pdfkeywords={}, 
pdfcreator={pdfLaTeX}, 
pdfproducer={LaTeX with hyperref and ClassicThesis} 
} 

\title{\normalfont\spacedallcaps{Geometry of $1$-codimensional measures in Heisenberg groups}} 

\author{\spacedlowsmallcaps{Andrea Merlo\textsuperscript{*}}} 

\date{} 

\begin{document}


\renewcommand{\sectionmark}[1]{\markright{\spacedlowsmallcaps{#1}}} 
\lehead{\mbox{\llap{\small\thepage\kern1em\color{halfgray} \vline}\color{halfgray}\hspace{0.5em}\rightmark\hfil}} 

\pagestyle{scrheadings} 


\maketitle 

\setcounter{tocdepth}{2} 




\paragraph*{Abstract} 
This paper is devoted to the study of tangential properties of measures with density in the Heisenberg groups $\HH^n$. Among other results we prove that measures with $(2n+1)$-density have only flat tangents and conclude the classification of uniform measures in $\HH^1$.



{\let\thefootnote\relax\footnotetext{* \textit{Université Paris-Saclay, 307 Rue Michel Magat Bâtiment, 91400 Orsay, France.\\
A. M. is supported by the Simons
Foundation grant 601941, GD.}}}

\section{Introduction}

In Euclidean spaces the notion of rectifiability of a measure is linked to the metric by the celebrated:

\begin{teorema}[Preiss, \cite{Preiss1987GeometryDensities}]\label{Preiss}
Suppose $0\leq m\leq n$ are integers, $\phi$ is a Radon measure on $\R^n$ and:
\begin{equation}
    0<\Theta^m(\phi,x):=\lim_{r\to 0}\frac{\phi(U_r(x))}{r^m}<\infty\qquad \text{at }\phi\text{-almost every }x,
    \label{eq_3}
\end{equation}
where $U_r(x)$ is the Euclidean ball of center $x$ and radius $r$. Then $\phi$ is $m$-rectifiable, i.e.,  $\phi$ is absolutely continuous with respect to the $m$-dimensional Hausdorff measure $\mathcal{H}^m$ and $\phi$-almost all of $\R^n$ can be covered by countably many $m$-dimensional Lipschitz submanifolds of $\R^n$.
\end{teorema}

The strategy of the proof of Theorem \ref{Preiss} could be ideally divided into three steps. First of all, one asks if the existence of the limit \eqref{eq_3} implies that $m$ is an integer. This is actually the case and it was proven by J. Marstrand in \cite{Marstrandoriginal}. Secondly, one faces the question of whether the regularity yielded by \eqref{eq_3} is sufficient to give some \emph{geometric} information on the local structure of the measure $\phi$. This difficult task was carried out by D. Preiss in \cite{Preiss1987GeometryDensities} where he proved that if \eqref{eq_3} holds, then $\phi$ has flat tangents, i.e.:
\begin{equation}
\Tan(\phi,x)\subseteq \Theta^m(\phi,x)\{\mathcal{H}^m\llcorner V:V\text{ is an }m\text{-plane}\}\text{ at } \phi\text{-almost every }x\in\R^n.
\label{flatter}
\end{equation}
Finally, the third step is to patch together the point-wise information given by \eqref{flatter} to infer that the measure $\phi$ is supported on an $m$-dimensional rectifiable set. This was obtained by successive steps by J. Marstrand and P. Mattila in \cite{Marstrand61} and \cite{mattila75}. For a modern account on these type of results, now known as \emph{Marstrand-Mattila rectifiability criteria}, we refer to \cite[Chapter 5]{DeLellis2008RectifiableMeasures}.

The most difficult part of the above argument is by far the second step and it is worth mentioning that all the techniques employed by D. Preiss invoke in a substantial way the algebraic structure of the Euclidean distance. So much so that if we replace the Euclidean distance defining $U_r(x)$ in \eqref{eq_3} with some other norm on $\R^n$, at the moment we are writing we do not even know if Marstrand's theorem holds.
The only progress in this direction, to our knowledge, was done by A. Lorent, who proved that $2$-locally uniform measures in $\ell^3_\infty$ are rectifiable, see \cite[Theorem 5]{LORENT2003RECTIFIABILITYDENSITY}. As one should expect, although the assumption of local uniformity is far stronger than the mere existence of the density, already in this strengthened hypothesis the proof is really intricate and even in this case the shape of the balls plays a fundamental role in the computations.

In this paper we investigate to what extent the local structure of $1$-codimensional measures in $\HH^n$ is affected by their regular behaviour on Koranyi balls. Although the Heisenberg group shares many similarities with the underlying Euclidean space $\R^{2n+1}$, it has Hausdorff dimension $2n+2$ and it is a $k$-purely unrectifiable metric space for any $k\in\{n+1,\ldots,2n+2\}$, i.e. for any compact set $K\subseteq \R^k$ and any Lipschitz function $f:K\to(\HH^n,\lVert\cdot\rVert)$ we have:
$$\mathcal{H}^k_{\lVert\cdot\rVert}(f(K))=0,$$
where $\mathcal{H}^k_{\lVert\cdot\rVert}$ is Hausdorff measure associated to the Koranyi norm, see for instance  \cite[Theorem 7.2]{Ambrosio2000RectifiableSpaces} or \cite[Theorem 1.1]{MR2105335}. 
This degeneracy of the structure of $\HH^n$ poses a big obstacle to extending many Euclidean results and definitions to the context of the Heisenberg groups, or in bigger generality to Carnot groups. In particular it is not a priori clear what the correct notion of rectifiability should be, or even if there is one. In the paper \cite{Serapioni2001RectifiabilityGroup} B. Franchi, R. Serapioni and F. Serra Cassano, introduced an intrinsic notion of rectifiability in $\HH^n$. In their definition a set $E\subseteq \HH^n$ is said to be $\mathcal{C}^1_\HH$-\emph{rectifiable} if for $\mathcal{H}^{2n+1}_{\lVert\cdot\rVert}$-almost every $x\in E$ there exists a $(2n+1)$-dimensional homogeneous subgroup $V_x$ such that:
$$\Tan_{2n+1}(\mathcal{H}^{2n+1}_{\lVert\cdot\rVert}\llcorner E,x)=\{\mathcal{H}^{2n+1}_{\lVert\cdot\rVert}\llcorner V_x\},$$
and $\Theta^{2n+1}_*(\mathcal{H}^{2n+1}_{\lVert\cdot\rVert}\llcorner E,x)>0$ almost everywhere. 
This definition has the considerable feature of making the recovery of De Giorgi's rectifiability theorem of boundaries of finite perimeter sets possible in $\HH^n$, see \cite{Serapioni2001RectifiabilityGroup}.
In the recent preprint \cite{antonelli2020rectifiable}, the author of this work in collaboration with G. Antonelli, inspired by the above characterisation of $\mathcal{C}^1_\HH$-rectifiable sets prove that in arbitrary Carnot groups the following holds: if a measure $\phi$ has a (almost everywhere) unique tangent of dimension $h$ and it has positive $h$-lower density, then the density exists almost everywhere. This suggests that at every codimension and in any Carnot group, measures with unique tangents and positive lower density may be characterized by the coincidence of the lower and upper density. For further details on this we refer to \cite[Section 3]{merloMM} and to \cite{antonelli2020rectifiable}.

This paper addresses the question of whether the definition of $\mathcal{C}^1_\HH$-rectifiability given by Franchi, Serapioni and Serra Cassano can be characterised by the metric in a similar way as rectifiability is in the Euclidean spaces. In other words we are interested in determining if a result in the spirit of Theorem \ref{Preiss} is available in $\HH^n$, where \emph{rectifiability} is replaced with $\mathcal{C}^1_\HH$-\emph{rectifiability}.
In particular, the main goal of this paper is to prove the analogue of the inclusion \eqref{flatter} in the Heisenberg groups:

\begin{teorema}\label{main}
Suppose $\phi$ is a Radon measure in $\HH^n$ such that for $\phi$-almost every $x\in\HH^n$ we have:
\begin{equation}
    0<\Theta^{2n+1}(\phi,x):=\lim_{r\to 0}\frac{\phi(B_r(x))}{r^{2n+1}}<\infty,
    \label{eq_2}
\end{equation}
where $B_r(x)$ is the Koranyi ball. Then for $\phi$-almost every $x\in\HH^n$:
$$\Tan_{2n+1}(\phi,x)\subseteq \Theta^{2n+1}(\phi,x)\mathfrak{M}(2n+1),$$
where $\mathfrak{M}(2n+1)$ is the family of the Haar measures of $(2n+1)$-homogeneous subgroups of $\HH^n$ which assign measure $1$ to the unit ball.
\end{teorema}

Theorem \ref{main} in conjunction with the Marstrand-Mattila rectifiability criterion for $\mathcal{C}^1_\HH$-rectifiable measures
\cite[Theorem 3]{merloMM}, yield the first extension beyond the Euclidean spaces of Preiss's rectifiability theorem, positively answering the question posed above:

\begin{teorema}\label{main:preiss}
Suppose $\phi$ is a Radon measure on the Heisenberg group $\HH^n$ such that for $\phi$-almost every $x\in\HH^n$, we have:
$$0<\Theta^{2n+1}(\phi,x):=\lim_{r\to 0}\frac{\phi(B_r(x))}{r^{2n+1}}<\infty,$$
where $B_r(x)$ are the metric balls relative to the Koranyi metric.
Then $\phi$ is absolutely continuous with respect to $\mathcal{H}^{2n+1}$ and $\HH^n$ can be covered $\phi$-almost all with countably many $C^1_{\mathbb{\HH}}$-regular surfaces.
\end{teorema}

For completeness, we remark that if some further hypothesis on the measure $\phi$ are allowed, there are a couple of results in literature in the spirit of Theorem \ref{Preiss} in non-Euclidean metric spaces. Besides the already mentioned rectifiability result by A. Lorent for locally $2$-uniform measure in $\ell^3_\infty$, it is worth mentioning the paper \cite{preisstiserBesicovitch} by D. Preiss and J. Ti\v{s}er where the authors tackle the metric analogue of the $1/2$-Besicovitch's problem.

The study of the density problem in the Heisenberg groups was started in 2015 by V. Chousionis and J. Tyson in \cite{Chousionis2015MarstrandsGroup}, where it has been proved that if $\phi$ is a Radon measure on $\HH^n$ such that:
\begin{equation}
    0<\Theta^{\alpha}(\phi,x):=\lim_{r\to 0}\frac{\phi(B_r(x))}{r^\alpha}<\infty \qquad \text{for }\phi\text{-a.e. }x\text{,}
    \label{besi}
\end{equation}
where $B_r(x)$ is the Koranyi ball, then $\alpha\in\{0,\ldots,2n+2\}$. If a Radon measure satisfies the condition \eqref{besi} for a certain $\alpha$, in the following we will say that $\phi$ has $\alpha$-density. This was done, very much as in the Euclidean spaces, by proving that \eqref{besi} implies that $\phi$-almost everywhere tangent measures to $\phi$ are $\alpha$-uniform measures, see Definition \ref{uniform}, and that  the support of such $\alpha$-uniform measures are analytic manifolds. The same strategy with minor modifications works in general Carnot groups when endowed with a left invariant polynomial norm. 

The development of those ideas allowed Chousionis, Tyson and Magnani in \cite{ChousionisONGROUP} to characterise $1$ and $2$ uniform measures in $\HH^1$ and to prove that vertically ruled $3$-uniform measures are the Haar measure of $3$-dimensional homogeneous subgroups of $\HH^1$. As a byproduct of our analysis we complete the characterisation of uniform measures in $\HH^1$, see Section \ref{conclusioni}.

We present here a survey of the strategy of the proof of Theorem \ref{main}, giving for each key step a brief discussion of the ideas involved. In Section \ref{sezione1}, we prove that the support of a uniform measure $\mu$ is contained in a quadratic surface. This is a result in the spirit of \cite[Theorem 17.3 ]{Mattila1995GeometrySpaces} (cfr. also with \cite[Theorem 4.1]{Kowalski1986Besicovitch-typeSubmanifolds}):

\begin{teorema}\label{T:1}
Let $\alpha\in\{1,\ldots,2n+1\}\setminus\{2\}$ and suppose that $\mu$ is an $\alpha$-uniform measure. Then, there are $b\in\R^{2n}$, $\mathcal{T}\in\R$ and $\mathcal{Q}\in\mathrm{Sym}(2n)$ with $\text{Tr}(\mathcal{Q})\neq 0$ such that:
$$\supp(\mu)\subseteq\mathbb{K}(b,\mathcal{Q},\mathcal{T}):=\{(x,t)\in\R^{2n+1}:\langle b,x\rangle+\langle x,\mathcal{Q} x\rangle+\mathcal{T}t=0\}.$$
\end{teorema}

Despite the fact that we already know (thanks to \cite[Proposition 3.2]{Chousionis2015MarstrandsGroup}, that the supports of a uniform measure is an analytic variety, the algebraic simplicity of the quadrics containing the support in the $1$-codimensional case will be a fundamental simplification in our computations.

The proof of Theorem \ref{T:1} is based on an adaptation of the arguments of Section 3 of \cite{Preiss1987GeometryDensities}. In particular we have extended Preiss's moments to this non-Euclidean context, the Heisenberg moments $b_{k,s}^\mu$ are introduced in Definition \ref{defimom}, in such a way that it is possible to prove for any $s>0$ and any $u$ is the support of a given uniform measure $\mu$, that:
\begin{equation}
\Big\lvert \sum_{k=1}^{4}b_{k,s}^\mu(u)-s\lVert u\rVert^{4}\Big\rvert\leq s^{\frac{5}{4}}\lVert u\rVert^5(2+(s\lVert u\rVert^4)^{2}).
\label{numeroo1}
\end{equation} 
For more details on the above inequality and its proof we refer to Proposition \ref{expanzione}.
The left-hand side in the above expression is a polynomial of fourth degree in the coordinates of $u$, but with some work one can reduce \eqref{numeroo1}, see Proposition \ref{prop10}, to:
\begin{equation}
\Big\lvert\langle b(s), u_H\rangle+\langle \mathcal{Q}(s)[ u_H], u_H\rangle+\mathcal{T}(s)  u_T\Big\rvert\leq s^\frac{1}{4}\lVert u\rVert^3,
\nonumber
\end{equation}
where $u_H$ is the vector of the first $2n$ coordinates of $u$, $u_T$ is the last coordinate of $u$ and $b(s),\mathcal{Q}(s)$ and $\mathcal{T}(s)$ are introduced in Definition \ref{definizionecurve}.
From the above expression, sending $s$ to $0$ one gets the quadric containing $\supp(\mu)$. The most tricky part of Theorem \ref{T:1} is to show that $\text{Tr}(\mathcal{Q})\neq 0$ and to the proof of this fact is devoted the whole Subsection \ref{nondegge}. 

When $\mu$ is a $(2n+1)$-uniform measure, one expects that the fact that the support is contained in a quadric represents a strong information on the structure of $\supp(\mu)$. This idea is exploited in Section \ref{buchi} where we prove:

\begin{teorema}\label{T:2}
The support of a $(2n+1)$-uniform measure $\mu$ is the closure of a union of connected components of $\mathbb{K}(b,\mathcal{Q},\mathcal{T})\setminus\Sigma$, where $\Sigma$ is the set of those points where the tangent group to the surface $\mathbb{K}(b,\mathcal{Q},\mathcal{T})$ does not exists.
\end{teorema}

The idea behind the proof of Theorem \ref{T:2} is the following. Suppose $y\in\mathbb{K}(b,\mathcal{Q},\mathcal{T})\setminus\supp(\mu)$ and let $z$ be a point with minimal Euclidean distance of $y$ from $\supp(\mu)$. If $z\not\in \Sigma$, thanks to  Proposition \ref{verticalsamoa2}, we know that $\Tan_{2n+1}(\mu,x)=\{\mathcal{C}^{2n+1}\llcorner V\}$ where $V$ is the tangent group to $\mathbb{K}(b,\mathcal{Q},\mathcal{T})$ at $z$ and $\mathcal{C}^{2n+1}\llcorner V$ is one of its Haar measures. However, by means of careful computations, see Propositions \ref{spt1} and \ref{spt2}, we show that the blowup of ``the hole in the support'' $B_{\lvert y-z\rvert}(y)\cap \mathbb{K}(b,\mathcal{Q},\mathcal{T})$ is a non-empty open subset of $V$, which is in contradiction with the fact that the support the blowup of $\mu$ at $z$ coincides with the whole $V$. This implies that the boundaries of holes of $\supp(\mu)$ inside $\mathbb{K}(b,\mathcal{Q},\mathcal{T})$ must be contained in $\Sigma$ and thus a standard connection argument proves Theorem \ref{T:2}.

Theorem \ref{T:2} allows us get a better understanding of the behaviour of $(2n+1)$-uniform measures at infinity. In particular using it we are able to prove that:
\begin{itemize}
    \item[(i)]if $\Tan_{2n+1}(\mu,\infty)\cap \mathfrak{M}(2n+1)\neq \emptyset$ then $\mu\in\mathfrak{M}(2n+1)$,
    \item[(ii)]the set $\Tan_{2n+1}(\mu,\infty)$ is a singleton.
\end{itemize} 
In the Euclidean space these properties arise from a careful analysis of the algebraic properties of moments. In our framework the structure of moments is much more complicated and therefore (i) and (ii) are proved by means of an explicit geometric construction which relies on Theorem \ref{T:2}. Thanks to these two properties of $(2n+1)$-uniform measures, in Section \ref{buchi}, we prove the following:

\begin{teorema}\label{T:3}
Suppose there exists a functional $\mathscr{F}:\mathcal{M}\to \R$, continuous in the weak-$*$ convergence of measures, and a constant $\hbar=\hbar(\HH^n)>0$ such that:
\begin{itemize}
\item[(i)] if $\mu\in\mathfrak{M}(2n+1)$ then $\mathscr{F}(\mu)\leq \hbar/2$,
\item[(ii)] if $\mu$ is a $(2n+1)$-uniform cone, see Definition \ref{conelli}, and $\mathscr{F}(\mu)\leq\hbar$, then $\mu\in\mathfrak{M}(2n+1)$.
\end{itemize}  
Then, for any $\phi$ Radon measure with $(2n+1)$-density and for $\phi$-almost every $x$:
$$\Tan_{2n+1}(\phi,x)\subseteq \Theta^{2n+1}(\phi,x)\mathfrak{M}(2n+1).$$
\end{teorema}

The proof of Theorem \ref{T:3} follows closely its Euclidean counterpart, and it is a standard application of the very general principle that ``a tangent to a tangent is a tangent'', see Proposition \ref{preiss}. 

We are left to construct the functional $\mathscr{F}$ satisfying all the hypothesis of Theorem \ref{T:3}.
Suppose $\varphi$ is a smooth function with support contained in $B_2(0)$ such that $\varphi=1$ on $B_1(0)$. We claim that the functional:
$$\mathscr{F}(\mu):=\min_{\mathfrak{m}\in\mathbb{S}^{2n-1}}\int \varphi(z)\langle    \mathfrak{m},z_H\rangle^2 d\mu(z),$$
satisfies all the hypothesis of Theorem \ref{T:3} and therefore Theorem \ref{main} follows. 
The fact that $\mathscr{F}$ is a continuous operator on Radon measures is easy to prove, see Proposition \ref{conti}, and it is immediate to see that $\mathscr{F}$ is identically null on flat measures. The most challenging hypothesis to check, as in the Euclidean case, is the existence of $\hbar$.

Thanks to Theorem \ref{T:1} there are two kinds of $(2n+1)$-uniform measures. The ones which are contained in a quadric for with $\mathcal{T}=0$, that in the following are called \emph{vertical}, and the ones with $\mathcal{T}\neq 0$, that we will call \emph{horizontal}. 
The first step towards the verification of hypothesis (ii) of Theorem \ref{T:3} is the following:

\begin{teorema}\label{T:4}
There exists a constant $\mathfrak{C}_3(n)>0$ such that for any $\mathfrak{m}\in\mathbb{S}^{2n-1}$ and any horizontal $(2n+1)$-uniform cone $\mu$ we have:
$$\int_{B_1(0)} \langle \mathfrak{m},   z_H\rangle^2 d\mu(z)\geq\mathfrak{C}_3(n).$$
\end{teorema}

The proof of Theorem \ref{T:4} requires the entire the entire Section \ref{HORRI}, but the arguments therein contained all rely on Theorem \ref{appendicefinale}, which is the main result of Appendix \ref{TYLR}. Since Theorem \ref{T:4} requires so much work, we wish to discuss its proof more carefully, in order to help the reader keep in mind what the final goal of the Section \ref{HORRI} and Appendix \ref{TYLR} is.

If $\mu$ is a horizontal $(2n+1)$-uniform cone, we can find $\mathcal{D}\in\mathrm{Sym}(2n)\setminus\{0\}$ such that $\supp(\mu)\subseteq \mathbb{K}(0,\mathcal{D},-1)$. In Theorem \ref{appendicefinale}, we prove that such $\mathcal{D}$ must satisfy the algebraic constraint \eqref{eq16} which implies that the operator norm $\vertiii{\mathcal{D}}$ of $\mathcal{D}$ is bounded from above and below by universal positive constants $\mathfrak{C}_1(n)$ and $\mathfrak{C}_2(n)$, respectively, see Propositions \ref{boundi} and \ref{staccato}, and thus Theorem \ref{T:4} follows. We refer to the proof of Theorem \ref{bellalei} for further details. 

While the bound from below easily follows from Theorem \ref{appendicefinale}, obtaining the bound from above is quite complicated. Suppose $\{\mu_i\}$ is a sequence of $(2n+1)$-uniform measures invariant under dilations and assume that $\supp(\mu_i)\subseteq\mathbb{K}(0,\mathcal{D}_i,-1)$. If the sequence $\vertiii{\mathcal{D}_i}$ diverges, then the limit points of the sequence $\{\mu_i\}$ can only be vertical $(2n+1)$-uniform cones. 
Defined $\mathcal{Q}$ to be one of the limit points of the sequence $\mathcal{D}_i/\vertiii{\mathcal{D}_i}$, one can show that the algebraic constraints given by Theorem \ref{appendicefinale} on $\mathcal{D}_i$ imply that for any $h\not\in\text{Ker}(\mathcal{Q})$ we have:
\begin{equation}
2(\text{Tr}(\mathcal{Q}^2)-2\langle \mathfrak{n},\mathcal{Q}^2\mathfrak{n}\rangle+\langle \mathfrak{n},\mathcal{Q}\mathfrak{n}\rangle^2)-(\text{Tr}(\mathcal{Q})-\langle \mathfrak{n},\mathcal{Q}\mathfrak{n}\rangle)^2=0,
\nonumber
\end{equation}
where $\mathfrak{n}:=\mathcal{Q}h/\lvert\mathcal{Q}h\rvert$. We refer to Proposition \ref{equazione} for further details. By this key observation, via Proposition \ref{omgflat} we prove that the sequence $\{\mu_i\}$ can only have a flat measure as limit points.
The fact that the limit must be flat together with the fact that all the eigenvalues of the $\mathcal{D}_i$ except one, see Proposition \ref{bddss}, which is again a consequence of Theorem \ref{appendicefinale}, must be bounded, implies that the assumption that such a sequence $\{\mu_i\}$ exists was absurd. Indeed the boundedness of all eigenvalues except one would prevent the limit of the $\mu_i$'s from being flat. See the proof of Proposition \ref{boundi} for further details.

The above argument shows that the functional $\mathscr{F}$ disconnects horizontal $(2n+1)$-uniform cones and flat measures. The last piece of information we need to apply Theorem \ref{T:3} is that $\mathscr{F}$ disconnects vertical non-flat $(2n+1)$-uniform cones from flat measures:

\begin{teorema}\label{T:6}
There exists a constant $\mathfrak{C}_{10}(n)>0$ such that if $\mu$ is a vertical $(2n+1)$-uniform cone for which:
$$\min_{\mathfrak{m}\in\mathbb{S}^{2n-1}}\int_{B_1(0)}\langle    \mathfrak{m},z_H\rangle^2 d\mu(z)\leq \mathfrak{C}_{10}(n),$$
then $\mu$ is flat.
\end{teorema}

The proof of the above theorem relies on Theorems \ref{T:1}, \ref{T:2} and the representation formulas of Appendix \ref{appeA1} to get a very explicit and simple expression for the quadric containing $\supp(\mu)$, see Proposition \ref{SUPPORTO}. Thanks to the structural similarities of these quadrics to their Euclidean counterparts we were able to rearrange Preiss's original disconnection argument to conclude the proof of Theorem \ref{T:6}, see Theorem \ref{fine} and cfr. with the proof of  \cite[Theorem 3.14]{Preiss1987GeometryDensities}.

\section*{Acknowledgments} First of all, I would like to thank Roberto Monti: without his guidance and help this paper would have never seen the light of day. Not only he suggested me to study Preiss's theorem during my undergraduate years and sollecited my interest for sub-Riemmanian geometry afterwards, but he also helped me correct the (very badly written) early versions of this paper. 
Finally, I would like to thank the anonymous referee, who checked this work with unparalleled care, pointing out an endless number of inaccuracies and helping me to improve the exposition significantly.

\section*{Notation}
We add below a list of frequently used notations, together with the page of their first appearance:
\medskip

\begin{longtable}{c p{0.7\textwidth} p{\textwidth}}
$\lvert\cdot\rvert$\label{euclide} & Euclidean norm, & \pageref{euclide}\\

$\lVert\cdot\rVert$ & Koranyi norm, & \pageref{Koranyinorm}\\

$\vertiii{\cdot}$ & operator norm of matrices, & \pageref{operatornorm}\\

$\langle \cdot,\cdot\rangle$ & scalar product in $\R^{2n}$, & \pageref{preliminaries}\\

$V(\cdot,\cdot)$& polarisation function of the Koranyi norm, & \pageref{polariz}\\

$\pi_H(\cdot)$ & projection of $\R^{2n+1}$ onto the first $2n$ coordinates,& \pageref{preliminaries}\\

$\pi_T(\cdot)$ & projection of $\R^{2n+1}$ onto the last coordinate,& \pageref{preliminaries}\\

$x_H$ & shorthand for $\pi_H(x)$, & \pageref{preliminaries}\\

$x_T$ & shorthand for $\pi_T(x)$, & \pageref{preliminaries}\\

$\tau_x$ & left translation by $x$, & \pageref{preliminaries}\\

$D_\lambda$ & anisotropic dilations, & \pageref{dilatan}\\

$\mathcal{V}$ &  center of $\HH^n$, & \pageref{centerlli}\\

$U_r(x)$ &  open Euclidean ball of radius $r>0$ and center $x$, & \pageref{euclidean}\\

$B_r(x)$ &  open Koranyi ball of radius $r>0$ and center x, & \pageref{preliminaries}\\

$\Theta_\alpha(\phi,x)$ & $\alpha$-dimensional density of the Radon measure $\phi$ at the point $x$, &  \pageref{drens}\\

$\phi_{x,r}$ & dilated of a factor $r>0$ of the measure $\phi$ at the point $x\in\HH^n$, & \pageref{trasldil}\\

$\Tan_{\alpha}(\phi,x)$ & set of $\alpha$-dimensional tangent measures to the measure $\phi$ at $x$, & \pageref{trangents}\\

$\mathcal{U}(\alpha)$ & set of $\alpha$-uniform measures, & \pageref{uniform}\\

  $\mathcal{M}$ &  set of Radon measures on $\HH^n$, &\pageref{emme}\\

  $\supp(\mu)$ & support of the measure $\mu$, &\pageref{uniform}\\

  $\rightharpoonup$ & weak convergence of measures, &\pageref{convdeb}
  \\

  $V(\mathfrak{n})$ & the vertical hyperplane orthogonal to $\mathfrak{n}\in\R^{2n}$, &\pageref{plano}\\

$\mathbb{K}(b,\mathcal{Q},\mathcal{T})$ & the quadric $\langle b+\mathcal{Q}x,x\rangle+\mathcal{T}t=0$ where $(x,t)\in\R^{2n+1}$, & \pageref{simmi}\\
 
 $\Sigma(f)$ & characteristic set of the function $f:\R^{2n}\to\R$ & \pageref{sigmino}\\
 
 $\Sigma(F)$ & set where the horizontal gradient of the function $F:\HH^n\to\R$ is null, & \pageref{numbero2}\\
 
  $J$ & standard symplectic matrix,& \pageref{Koranyinorm}\\

  $\mathrm{Sym}(2n)$ &  set of symmetric matrices on $\R^{2n}$, &\pageref{simmi}\\
  
    $S(2n)$ &  subset of orthogonal matrices on $\R^{2n}$ inducing a linear isometry on $\HH^n$, & \pageref{inex}\\

  $\mathcal{C}_c$ & set of continuous functions with compact support,  &\pageref{compi}\\

$\mathcal{S}^\alpha$ & $\alpha$-dimensional spherical Hausdorff measure, & \pageref{Hausdro}\\

$\mathcal{H}^\alpha$ & $\alpha$-dimensional Hausdorff measure, & \pageref{Hausdro}\\

$\mathcal{C}^\alpha$ & $\alpha$-dimensional centered spherical Hausdorff measure, & \pageref{Hausdro}\\
  $\mathcal{H}^{k}_{eu}$ &  Euclidean $k$-dimensional Hausdorff measure, &\pageref{euclidean}\\

  $b_{k,s}^\mu$  & $k$-th moment of the measure $\mu$, & \pageref{defimom}
\end{longtable}

\paragraph{Other conventions:}
Throughout the paper the symbol $\phi$ will always denote a measure with density and the symbols $\mu,\nu$ uniform measures. Furthermore, with $\Gamma(t)$ we will always denote Euler's $\Gamma$ function:
$$\Gamma(t):=\int_0^\infty s^{t-1}e^{-s}ds.$$
Recall that $\Gamma$ enjoys the property that $\Gamma(t+1)=t\Gamma(t)$ for any $t>0$.

\section{Preliminaries}
\label{preliminaries}

In this preliminary section we recall many well known facts and introduce some notations. In case the proof of a Proposition is not present in literature, but the Euclidean argument applies verbatim, we will reduce ourselves to cite a reference where the Euclidean proof can be found.

\subsection{The Heisenberg group \texorpdfstring{$\HH^n$}{Lg}}

In this subsection we briefly recall some notations and very well known facts on the Heisenberg groups $\HH^n$. Let $\pi_H:\R^{2n+1}\to\R^{2n}$ be the projection onto the first $2n$ coordinates and $\pi_T:\R^{2n+1}\to \R$ be the projection onto the last one. The Lie groups $\HH^n$ are the smooth manifolds $\R^{2n+1}$ endowed with the product:
\begin{equation}
x*y:=\left(x_H+y_H, x_T+y_T+2\langle x_H,Jy_H\rangle\right),\nonumber
\end{equation}
where  $x_H$ and $x_T$ are shorthands for $\pi_H(x)$ and $\pi_T(x)$, while $J$ is the standard sympletic matrix on $\R^{2n}$:
$$J:=\begin{pmatrix}
    0      & \mathrm{id}_n   \\
    -\mathrm{id}_n     & 0  
\end{pmatrix},$$
We metrize the group $(\HH^n,*)$ with the \emph{Koranyi distance} $d(\cdot,\cdot):\HH^n\times\HH^n\to\R$ defined as:
\begin{equation}
d(x,y):=\big(\lvert y_H-x_H\rvert^4+\lvert y_T- x_T-2\langle x_H,Jy_H\rangle\rvert^2\big)^{1/4}\nonumber.
\end{equation}
Moreover, we let $\lVert x\rVert:=d(x,0)$\label{Koranyinorm} be the so called \emph{Koranyi norm} and $B_r(x):=\{z\in\HH^n:d(z,x)\leq r\}$ be the Koranyi ball. 
The geometry of $\HH^n$ is quite rich and it is well known that the metric $d(\cdot,\cdot)$ is left invariant, i.e., for any $z\in\HH^n$ one has:
$$d(z*x,z*y)=d(x,y).$$

As a consequence, left translations $\tau_x(y):=x*y$ are isometries and an elementary computations shows that $d(x,y)=\lVert x^{-1}*y\rVert=\lVert y^{-1}*x\rVert$.
Moreover, defined the anisotropic dilations $D_\lambda:\HH^n\to\HH^n$ as $D_\lambda(x):=(\lambda x_H,\lambda^2  x_T)$,\label{dilatan}
we also have that $d(\cdot,\cdot)$ is homogeneous with respect to $D_\lambda$, i.e.:
$$d(D_\lambda(x),D_\lambda(y))=\lambda d(x,y).$$
Besides left translations, we can find isometries of $(\HH^n,d)$ that behave like Euclidean orthogonal transformations on the first $2n$ coordinates. Define:
\begin{equation}
    \text{S}(2n):=\{U\in O(2n): U^T J U=J\}\cup\{U\in O(2n): U^T J U=-J\},
    \label{inex}
\end{equation}
where $O(2n)$ is the group of orthogonal transformations of $\R^{2n}$
and let $s$ be the function $s:\text{S}(2n)\to\{-1,1\}$ which satisfies $U^T J U=s(U) J$
.
It is easy to check that $\text{S}(2n)$ is a group under multiplication and that $s(\cdot)$ is a homomorphism between $(\text{S}(2n),\cdot)$ and $(\{-1,1\},\cdot)$. The following proposition gives us the aformentioned isometries of $\HH^n$:

\begin{proposizione}\label{isometrie2}
Let $U\in \text{S}(2n)$. The map $\Xi_U:\R^{2n}\times \R\to\R^{2n}\times \R$ defined as:
\begin{equation}
\label{numbero10}
\Xi_U:(x,t)\mapsto (Ux,s(U)t),
\end{equation}
is a surjective isometry of $\HH^n$ and in particular it is a continuous map.
\end{proposizione}

We conclude this short subsection introducing the various types of Hausdorff measures we will deal with throughout the paper:

\begin{definizione}\label{Hausdro}
For any $\alpha\in[0,\infty)$ we define the  $\alpha$-dimensional {\em Hausdorff measure}\label{hausmeas} and the $\alpha$-dimensional {\em spherical Hausdorff measure} relative to the left invariant homogeneous metric $d$ those measures that acts as:\label{sphericaldhausmeas}
\begin{equation*}
    \begin{split}
        \mathcal{H}^\alpha(A):=\sup_{\delta>0}\inf \bigg\{\sum_{j=1}^{\infty} 2^{-h}(\diam E_j)^\alpha:A \subseteq \bigcup_{j=1}^{\infty} E_j,\, \diam E\leq \delta\bigg\},\\
        \mathcal{S}^{\alpha}(A):=\sup_{\delta>0}\inf\bigg\{\sum_{j=1}^\infty  r_j^\alpha:A\subseteq \bigcup_{j=1}^\infty B_{r_j}(x_j),~r_j\leq\delta\bigg\},
    \end{split}
\end{equation*}
on every Borel set $A\subseteq \HH^n$.
In addition, we define the $\alpha$-dimensional {\em centered Hausdorff measure} relative to $d$ as the measure that acts as\label{centredhausmeas}
$$
\mathcal{C}^{\alpha}(A):=\underset{E\subseteq A}{\sup}\,\,\mathcal{C}_0(E),\qquad \text{on every Borel set $A\subseteq \HH^n$,}
$$
where:
$$
\mathcal{C}^{\alpha}_0(E):=\sup_{\delta>0}\inf\bigg\{\sum_{j=1}^\infty  r_j^m:E\subseteq \bigcup_{j=1}^\infty B_{r_j}(x_j),~ x_j\in E,~r_j\leq\delta\bigg\}.
$$
We stress that $\mathcal{C}^\alpha$ is an outer Borel regular measure, see for instance \cite[Proposition 4.1]{EdgarCentered}.
\end{definizione}

\begin{osservazione}
\label{remarkone}
Although the Hausdorff measures defined above do not coincide, they are all equivalent and more precisely for any $\alpha>0$ we have:
$$\mathcal{H}^\alpha\leq \mathcal{S}^\alpha\leq 2^\alpha\mathcal{H}^\alpha\qquad\text{and}\qquad2^{-\alpha}\mathcal{S}\leq\mathcal{C}^\alpha\leq 2^{\alpha}\mathcal{S}^\alpha.$$
This can be obtained for instance putting together \cite[\S 2.10.2]{Federer1996GeometricTheory} and \cite[Remark  2.3, Lemma 2.5]{FSSCArea}.
\end{osservazione}

For further references and a much more comprehensive account on the Heisenberg groups we refer to the monographs \cite{Capogna2007} and \cite{Monti2014IsoperimetricGroup}.

\subsection{Measures with density and their blowups}
We recall in this subsection some very well known facts about measures with density and their blowups.
\label{dns}

\begin{definizione}\label{drens}
Fix an $\alpha>0$. A Radon measure $\phi$ on $\HH^n$ is said to have $\alpha$\emph{-density} if the limit:
$$\Theta^\alpha(\phi,x):=\lim_{r\to 0}\frac{\phi(B_r(x))}{r^\alpha},$$
exists finite and non-zero for $\phi$-almost every $x\in \HH^n$. 
\end{definizione}

Assume $\{\mu_k\}$ is a sequence of measures in $\mathcal{M}$\label{emme}, the set of Radon measures on $\HH^n$. We say that $\{\mu_k\}$ converges to $\mu$ and write $\mu_k\rightharpoonup \mu$, if:\label{convdeb}
$$\lim_{k\to\infty}\int f(x)d\mu_k(x)=\int f(x)d\mu(x)\qquad\text{for any }f\in\mathcal{C}_c(\R^n).$$\label{compi}
Since the paper is concerned with the study of the tangents to measures with $\alpha$-density, we need a meaningful concept of tangent for a measure. Let $\phi$ be a Radon measure on $\HH^n$ with $\alpha$-density and denote by $\phi_{x,r}$ the measure that satisfies the identity:
\begin{equation}
    \phi_{x,r}(A)=\phi(x*D_r(A)),
    \label{trasldil}
\end{equation}
for any Borel set $A\subseteq \HH^n$. The set of tangent measures to $\phi$ at $x$ is denoted by $\Tan_\alpha(\phi,x)$ and it consists of the Radon measures $\mu$ for which there exists a sequence $r_i\to 0$ such that:
$$r_i^{-\alpha}\phi_{x,r_i}\rightharpoonup\mu.$$
The set $\Tan_\alpha(\phi,x)$ \label{trangents} is non-empty for $\phi$-almost  every $x\in\HH^n$. Indeed, let $x\in\HH^n$ be a point such that $\Theta^\alpha(\phi, x)<\infty$. Then, for every $\rho>0$ we have:
$$r^{-\alpha} \phi_{x,r}(B_\rho(0))=r^{-\alpha}\phi(B_{\rho  r}(x))\leq2\Theta^\alpha(\phi,x)\rho^\alpha,$$
for any sufficiently small $r$. Therefore, the family of measures $r^{-\alpha}\phi_{x,r}$ is uniformly bounded on compact sets and the compactness of measures, see for instance \cite[Proposition 1.12]{Preiss1987GeometryDensities}, yields the existence of a limit for a suitable subsequence.

\begin{definizione}\label{uniform}
We say that a Radon measure $\mu$ is an $\alpha$-uniform measure if:
\begin{itemize}
    \item [(i)] $0\in\supp(\mu)$,
    \item[(ii)] $\mu(B_r(x))= r^\alpha$ for any $r>0$ and any $x\in\supp(\mu)$.
\end{itemize}
We will denote the set of $\alpha$-uniform measures with the symbol $\mathcal{U}(\alpha)$.
\end{definizione}

The following two propositions are of capital importance to us since on the one hand, Proposition \ref{propup} insures that the tangents to a measure $\phi$ with $\alpha$-density are $\phi$-almost everywhere $\alpha$-uniform and on the other Proposition \ref{preiss} tells us that for $\phi$-almost every $x\in\HH^n$, the tangent measures to any element of $\Tan(\phi,x)$ are still in $\Tan(\phi,x)$. The latter stability property is usually summarized in the effective but imprecise expression \emph{tangent to tangents are tangents}.

\begin{proposizione}\label{propup}
Assume $\phi$ is a measure with $\alpha$-density on $\HH^n$. Then, for $\phi$-almost every $x\in\HH^n$ we have:
$$\Tan_\alpha(\phi,x)\subseteq \Theta^\alpha(\phi,x)\mathcal{U}(\alpha).$$
\end{proposizione}

\begin{proof}
The proof of this proposition follows almost without modifications the one given in the Euclidean case in \cite[Proposition 3.4]{DeLellis2008RectifiableMeasures}.
\end{proof}

\begin{proposizione}\label{preiss}
Let $\phi$ be a Borel measure having $\alpha$-density in $\HH^n$. Then for $\phi$-almost every $x\in\HH^n$ and any $\mu\in\Tan_\alpha(\phi,x)$ we have:
$$r^{-\alpha}\mu_{y,r}\in\Tan_\alpha(\phi,x),\qquad\text{for every }y\in\supp(\mu)\text{ and }r>0.$$
\end{proposizione}

\begin{proof}
For a proof in generic metric groups see for instance \cite[Proposition 3.1]{Mattila2005MeasuresGroups}.
\end{proof}

\begin{proposizione}\label{propspt1}
Let $\phi$ be a Radon measure with $\alpha$-density and assume $\mu\in\Tan_\alpha(\phi,x)$ is such that $r_i^{-\alpha}\phi_{x,r_i}\rightharpoonup \mu$ for some $r_i\to 0$. Then, for any $y\in\supp(\mu)$ there exists a sequence $\{z_i\}_{i\in\N}\subseteq\supp(\phi)$ such that $D_{1/r_i}(x^{-1}*z_i)\to y$.
\end{proposizione}

\begin{proof}
A simple argument by contradiction yields the claim, the proof follows verbatim its Euclidean analogue,  \cite[Proposition 3.4]{DeLellis2008RectifiableMeasures}.
\end{proof}

If $\mu$ is an $\alpha$-uniform measure, we can also define its blowups at infinity, or blowdowns. Such tangents at infinity are Radon measures $\nu$ for which there exists a sequence $\{R_i\}\to\infty$ such that:
$$R_i^{-\alpha}\phi_{0,R_i}\rightharpoonup \nu.$$
We will denote with $\Tan_\alpha(\mu,\infty)$ the set of tangent measures at infinity of $\mu$. The following proposition is a strengthened version of Proposition \ref{propup} for uniform measures. 

\begin{proposizione}\label{uniformup}
Assume $\mu$ is an $\alpha$-uniform measure. Then for any $z\in\supp(\mu)\cup\{\infty\}$  we have:
$$\emptyset\neq\Tan_\alpha(\mu, z)\subseteq \mathcal{U}(\alpha).$$
\end{proposizione}

\begin{proof}
A straightforward adaptation of the proof of \cite[Lemma 3.6]{DeLellis2008RectifiableMeasures} yields the desired conclusion.
\end{proof}

The following is a compactness result for uniform measures and for their supports.

\begin{lemma}\label{replica}
If $\{\mu_i\}_{i\in\N}$ is a sequence of $\alpha$-uniform measures converging in the weak topology to some $\nu$, then:
\begin{itemize}
\item[(i)] $\nu$ is an $\alpha$-uniform measure,
\item[(ii)] if $y\in\supp(\nu)$ then there exists a sequence $\{y_i\}\subseteq \HH^n$ such that $y_i\in\supp(\mu_i)$ and $y_i\to y$,
\item[(iii)] if there exists a sequence $\{y_i\}\subseteq \HH^n$ such that $y_i\in\supp(\mu_i)$ and $y_i\to y$, then $y\in\supp(\nu)$.
\end{itemize}
\end{lemma}

\begin{proof}
The proof of this lemma is achieved by adapting to the context of Heisenberg groups and to the particular case of uniform measures the proof of \cite[Proposition 3.4]{DeLellis2008RectifiableMeasures}.
\end{proof}

The following Theorem was proved by V. Chousionis, J.Tyson in \cite{Chousionis2015MarstrandsGroup}. They proved that if $\HH^n$ is endowed with the Koranyi metric, the density problem reduces to integer exponents only, in complete analogy with the Euclidean case:

\begin{teorema}\label{Mastrand}
The  set $\mathcal{U}(\alpha)$ is non-empty if and only if $\alpha\in\{0,1,\ldots, 2n+2\}$. In particular if $\phi$ is a measure with $\alpha$-density in $\HH^n$, then $\alpha\in\{0,\ldots,2n+2\}$.
\end{teorema}

\begin{osservazione}\label{oss12}
Note that $\mathcal{U}(0)=\{\delta_0\}$. Moreover, arguing as in \cite[Proposition 3.14]{DeLellis2008RectifiableMeasures}, we can also deduce that $\mathcal{U}(2n+2)=\{\mathcal{L}^{2n+1}(B_1(0))^{-1}\mathcal{L}^{2n+1}\}$, where $\mathcal{L}^{2n+1}$ denotes the usual Lebesgue measures in $\R^{2n+1}$.
For these reasons, \textbf{from now on we will always assume} $\alpha\in\{1,\ldots,2n+1\}$ and instead of $\alpha$ we will always use $m$ to denote the dimension of the measure to remark that it is an integer.
\end{osservazione}

\begin{osservazione}
The stratification of zeros of holomorphic functions is the tool used by B. Kirchheim and D. Preiss in \cite{Kirchheim2002UniformilySpaces} and later by V. Chousionis and J.Tyson  in \cite{Chousionis2015MarstrandsGroup} to prove Marstrand's theorem in the Euclidean spaces and in the Heisenberg groups, respectively. It is easy to see that the same  argument yields Marstrand's theorem in any homogeneous group endowed with a metric induced by a polynomial norm. 
\end{osservazione}

\begin{proposizione}\label{UComp}
For any $m\in\{0,\ldots,2n+2\}$ the set $\mathcal{U}(m)$ is compact with respect to the convergence of measures.
\end{proposizione}

\begin{proof}
Indeed, let $\{\mu_i\}_{i\in\N}$ be a sequence of measures in $\mathcal{U}(m)$. Then, by definition we have $\mu_i(B_r(x))\leq r^m$. Proposition \cite[Proposition 1.12]{Preiss1987GeometryDensities} immediately proves concludes the existence of a Radon measure $\mu$ such that $\mu_i\rightharpoonup\mu$. Proposition \ref{replica}(i) concludes the proof.
\end{proof}

\subsection{Basic properties of uniform measures in \texorpdfstring{$\HH^n$}{Lg}}

In this subsection we present some elementary results on the structure of uniform measures and we show that if a function has spherical symmetry, then its integral with respect to a uniform measure is easily computed.

\begin{proposizione}\label{isometrie}
Let $\Sigma:(\HH^n,\lVert\cdot\rVert)\to(\HH^n,\lVert\cdot\rVert)$ be a surjective isometry of $(\HH^n,\lVert\cdot\rVert)$ into itself. If $\mu\in \mathcal{U}(m)$ and there exist a $u\in\supp(\mu)$ such that $\Sigma(u)=0$, then $\Sigma_\#(\mu)\in\mathcal{U}(m)$ and
$\supp(\Sigma_\#(\mu))=\Sigma(\supp(\mu))$.
\end{proposizione}

\begin{proof}
Since $\Sigma^{-1}(B_r(g))=B_r(\Sigma^{-1}(g))$, for any $g\in \HH^n$ and any $r>0$ we have that:
$$\Sigma_\#\mu(B_r(g))=\mu(\Sigma^{-1}(B_r(g)))=\mu(B_r(\Sigma^{-1}(g))).$$
If $g\in\Sigma(\supp(\mu))$, then there exists $h\in\supp(\mu)$ such that $g=\Sigma(h)$ and thus $\Sigma_\#\mu(B_r(g))=\mu(B_r(h))=r^m$.
In particular we conclude that $\Sigma(\supp(\mu))\subseteq\supp(\Sigma_\#(\mu))$. The other inclusion can be obtained similarly.
\end{proof}

The following characterization of uniform measures will prove to be very useful throughout this paper.

\begin{proposizione}\label{supportoK}
If $\mu$ is a $m$-uniform measure on $\HH^n$, then $\mu=\mathcal{C}^{m}\llcorner{\supp(\mu)}$.
\end{proposizione}

\begin{proof}
The claim follows immediately from \cite[Theorem 3.1]{FSSCArea}.
\end{proof}

\begin{osservazione}
Note that the above proposition implicitly says that $m$-uniform measures are uniquely determined by their support.
\end{osservazione}

\begin{osservazione}
The above proposition could also be proved using the fact that we know thanks to \cite{Chousionis2015MarstrandsGroup} that supports of uniform measures are analytic varieties together with the area formulas of \cite{Magnani2017AMeasure}.
\end{osservazione}

\begin{definizione}[Radially symmetric functions]
We say that a function $\varphi:\HH^n\rightarrow\R$ is radially symmetric if there exists a profile function $g:[0,\infty)\rightarrow \R$ such that $\varphi(z)=g(\lVert z\rVert)$.
\end{definizione}

Integrals of radially symmetric functions with respect to uniform measures are easy to compute and we have the following change of variable formula, which will be extensively used throughout the paper:

\begin{proposizione}\label{prop5}
Let $\mu\in\mathcal{U}(m)$ and suppose $\varphi:\HH^n\rightarrow\R$ is a radially symmetric non-negative function. Then, for any $u\in\supp(\mu)$ we have:
\begin{equation}
\int  \varphi(u^{-1}*z)d\mu(z)=m\int_0^\infty r^{m-1}g(r)dr\nonumber,
\end{equation}
where $g$ is the profile function associated to $\varphi$.
\end{proposizione}

\begin{proof}
First one proves the formula for simple functions of the form
$\varphi(z):=\sum_{i=1}^k a_i\chi_{B_{r_i}(0)}$,
where $a_i,r_i\geq 0$ for any $i=1,\ldots,k$. Then, one proves the result for a general $\varphi$ by Beppo Levi's convergence theorem.
\end{proof}

An immediate application of the previous proposition is the following:

\begin{corollario}
\label{prop1}
For any $p>0$, any $\mu\in\mathcal{U}(m)$ and any $u\in\supp(\mu)$, we have:
\begin{equation}
\int  \lVert u^{-1}*z\rVert^p e^{-s\lVert   u^{-1}*z\rVert^4}d\mu(z)=\frac{m}{4s^\frac{m+p}{4}}\Gamma\left(\frac{m+p}{4}\right).
\nonumber
\end{equation}
\end{corollario}

\begin{proof}
Since the profile function associated to the radial function $\lVert z\rVert^p e^{-s\lVert  z\rVert^4}$ is $r^p e^{-sr^4}$, Proposition \ref{prop5} implies that:
$$\int  \lVert u^{-1}*z\rVert^p e^{-s\lVert   u^{-1}*z\rVert^4}d\mu(z)=m\int_0^\infty r^{m-1}\cdot r^p e^{-sr^4} dr.$$
In order to conclude the proof of the corollary, we are left to compute the integral on right-hand side of the above identity:
\begin{equation}
\begin{split}
m\int_0^\infty r^{m-1}r^p e^{-sr^4}dr=&\frac{m}{s^\frac{m+p}{4}}\int_0^\infty t^{m+p-1}e^{-t^4}dt\\
=&\frac{m}{4s^\frac{m+p}{4}}\int_0^\infty x^{\frac{m+p}{4}-1}e^{-x}dx=\frac{m}{4s^\frac{m+p}{4}}\Gamma\left(\frac{m+p}{4}\right).
\nonumber
\end{split}
\end{equation}
\end{proof}

\section{The supports of uniform measures are contained in quadratic surfaces}
\label{sezione1}
In order to efficiently discuss the content of this section, we need to first introduce some notation.

\begin{definizione}\label{simmi}
Let $b\in\R^{2n}$, $\mathcal{Q}\in\mathrm{Sym}(2n)$ be a symmetric matrix $2n\times 2n$ and $\mathcal{T}\in\R$. For any such triple, we define $\mathbb{K}(b,\mathcal{Q},\mathcal{T})$ to be the set of those $(x,t)\in\R^{2n}\times\R$ for which:
$$\langle b,x\rangle+\langle x,\mathcal{Q} x\rangle+\mathcal{T}t=0.$$
\end{definizione}

The main goal of this section is to prove that the supports of uniform measures are contained in quadratic surfaces, or more precisely:

\begin{teorema}\label{MOK}
Let $m\in\{1,\ldots,2n+1\}\setminus\{2\}$. For any $\mu\in\mathcal{U}(m)$ there exist $b\in\R^{2n}$, $\mathcal{T}\in\R$ and $\mathcal{Q}\in\mathrm{Sym}(2n)$ with $\text{Tr}(\mathcal{Q})\neq 0$ such that $\supp(\mu)\subseteq\mathbb{K}(b,\mathcal{Q},\mathcal{T})$.
\end{teorema}

The proof of the above theorem is divided into two main steps. First we construct $b\in\R^{2n}$, $\mathcal{Q}\in\mathrm{Sym}(2n)$, $\mathcal{T}\in\R$ for which:
    $$\langle  b,u_H\rangle+\langle u_H,\mathcal{Q} u_H\rangle+\mathcal{T}u_T=0,$$
for any $u\in\supp(\mu)$ and secondly we show that the constructed matrix $Q$ is non-trivial by proving that $\text{Tr}(\mathcal{Q})\neq 0$.

The first part of the above program follows closely  \cite[Chapter 7]{DeLellis2008RectifiableMeasures} and for a more detailed explanation we refer to the beginning of Subsection \ref{momentsss} below. However, the heuristic behind all these computations is that the identity $\mu(B_r(u))=\mu(B_r(0))$, that is true for any $u\in\supp(\mu)$ and any $r>0$, implies that:
$$ \Big\lvert \fint_{B_r(0)} (r^4-\lVert u^{-1}*y\rVert^4)^2d\mu(y)-\fint_{B_r(0)}(r^4-\lVert y\rVert^4)^2d\mu(y)\Big\rvert\leq 450\cdot 2^mm r^{5} \lVert u\rVert^3,$$
for every $u\in\supp(\mu)$ and any $r>2\lVert u\rVert$. With some careful algebraic manipulations of the above inequality, it is not hard to build a quadric containing $\supp(\mu)$ with a similar procedure to the one used in \cite[Theorem 17.3]{Mattila1995GeometrySpaces} to produce the quadric containing the support of uniform measures in Euclidean spaces. For more details on how we proceed, we refer to Subsection \ref{CDDQQ}.

The second part of the argument, where we prove that the constructed quadric is non degenerate, is contained in Subsection \ref{nondegge}. In the Euclidean case this non-degeneracy is almost free, however in the sub-Riemannian context it requires some effort. In particular we are able to prove that if the support of $\mu$ is far away from the vertical axis $\mathcal{V}:=\{x_H=0\}$,\label{centerlli} then the quadric is non-degenerate. The reason for which we have to avoid the case $m=2$ in Theorem \ref{MOK} is the following. On the one hand, if $m\neq 2$ we are able to prove that there is a lot of measure far away the vertical axis. On the other, it is straightforward to see that the measure $\mathcal{C}^2\llcorner \mathcal{V}$, where $\mathcal{V}:=\{x_H=0\}$, is $2$-uniform measure and that in this case our construction yields $b=0$, $\mathcal{Q}=0$ and $\mathcal{T}=0$.

\subsection{Moments in the Heisenberg group and their algebraic structure}
\label{momentsss}

One of the fundamental tools introduced by Preiss in \cite{Preiss1987GeometryDensities} are moments associated to uniform measures. If $\mu$ is an $m$-uniform measure in $\R^n$, for any $k\in\N$ and $s>0$ the $k$-th moment of $\mu$ is defined as:
$$b_{k,s}^\mu(u_1,\ldots,u_k):=\frac{(2s)^{k+\frac{m}{2}}}{I(m)k!}\int_{\R^n} \prod_{i=1}^k \langle z,u_i\rangle e^{-s\lvert z\rvert^2} d\mu(z),$$
where $I(m):=\int_{\R^n} e^{-\lvert z\rvert^2} d\mu(z)$.
Using these functions Preiss was able to prove the following expansion formula:
\begin{equation}
    \begin{split}
   \Big\lvert\sum_{k=1}^{2q} b_{k,s}^\mu(u)-\sum_{k=1}^q \frac{s^k\lVert u\rVert^k}{k!}\Big\rvert\leq 5^{n+9}(s\lVert u\rVert^2), \qquad \text{for any $u\in\supp(\mu)$,}
        \label{eqeqPReiz}
    \end{split}
\end{equation}
which allowed him to find algebraic equations that must be satisfied by the points of $\supp(\mu)$.

The problem we tackle in this subsection is to prove an analogue of the inequality \eqref{eqeqPReiz} in a context where there is no scalar product inducing the metric. The strategy of our choice is to use a suitable polarization $V(\cdot,\cdot)$ of the Koranyi norm, which is a $4$-th degree polynomial, see Proposition \ref{prop6}, for which a weak form of the Cauchy-Schwarz inequality holds, see Proposition \ref{prop4}. This will allow us to prove in the next subsection an inequality in the spirit of \eqref{eqeqPReiz} with our modified moments. For further details we refer to Proposition \ref{expanzione}.

It is possible to prove an inequality silmilar to \eqref{eqeqPReiz} even in Banach spaces with the suitable polarisation of the norm and with the proper definition of moments. The problem is that such an expansion would not yield much information on the structure of the support of measures: one really needs an explicit algebraic expression for the substitute of the scalar product in order to be able to push further the argument and make the computations. In Carnot groups however one always have a smooth polynomial norm which can be used as in Heisenberg case: computations would be just much more complicated.

\begin{definizione} [Substitute for the scalar product]\label{polariz}
We let $V:\HH^n\times\HH^n\to \R$, the \emph{polarisation} of the Koranyi norm, be the function defined as:
\begin{equation}
V(u,z):=\frac{\lVert u\rVert^4+\lVert z\rVert^4-\lVert u^{-1}*z\rVert^4}{2},\qquad\text{for any $u,z\in\HH^n$.}
\nonumber
\end{equation}
\end{definizione}

The function $V$ is a rather complicated $4$-homogenous polynomial of degree $4$ in the components of $u$ and $z$. The next proposition aims to unveil some of its structure in order to deal with computations more in a more effective way.

\begin{proposizione}\label{prop6}
The function $V(z,u)$ can be decomposed as $2V(u,z)=L(u,z)+Q(u,z)+T(u,z)$,  where:
\begin{itemize}
\item[(i)] $L(u,z):=\langle u_H,4\lvert   z_H\rvert^2   z_H+4  z_T J   z_H\rangle$,
\item[(ii)] $Q(u,z):=-4\langle    z_H,u_H\rangle^2-2\lvert   z_H\rvert^2\lvert u_H\rvert^2-4\langle J   z_H, u_H\rangle^2+2  z_T  u_T$,
\item[(iii)] $T(u,z):=\langle    z_H,4\lvert u_H\rvert^2 u_H+4  u_T J u_H\rangle$.
\end{itemize}
\end{proposizione}

\begin{proof}
Thanks to the definition of $V$ and of $\lVert\cdot\rVert$, we have:
\begin{equation}
\begin{split}
&\qquad\qquad\qquad\qquad2V(u,z)=\lVert u\rVert^4+\lVert z\rVert^4-\lVert u^{-1}*z\rVert^4\\
=&\lvert u_H\rvert^4+\lvert u_T\rvert^2+\lvert z_H\rvert^4+\lvert z_T\rvert^2-\lvert z_H-u_H\rvert^4-\lvert z_T-u_T-2\langle u_H,Jz_H\rangle\rvert^2\\
=&-4\langle u_H,   z_H\rangle^2-2\lvert u_H\rvert^2\lvert    z_H\rvert^2+4\lvert u_H\rvert^2\langle u_H,   z_H\rangle+4\lvert   z_H\rvert^2\langle u_H,   z_H\rangle\\
&-4\langle  u_H,J   z_H\rangle^2+4  z_T\langle  u_H,J   z_H\rangle+2  u_T  z_T-4  u_T\langle  u_H,J   z_H\rangle.
\nonumber
\end{split}
\end{equation}
Recognising $L(u,z),Q(u,z)\text{ and }T(u,z)$ in the computation above proves the claim.
\end{proof}

\begin{osservazione}\label{rk1}
The polynomials $L(u,z),Q(u,z),T(u,z)$ are respectively $1,2,3$-$D_\lambda$-homogeneous in $u$, i.e.:
$$L(D_\lambda(u),z)=\lambda L(u,z),\,\,\, Q(D_\lambda(u),z)=\lambda^2Q(u,z),\,\,\, T(D_\lambda(u),z)=\lambda^3T(u,z),$$
and are respectively $3,2,1$-$D_\lambda$-homogeneous in $z$, i.e.:
$$L(u,D_\lambda(z))=\lambda^3 L(u,z),\,\,\, Q(u,D_\lambda(z))=\lambda^2Q(u,z),\,\,\, T(u,D_\lambda(z))=\lambda T(u,z).$$
In addition, thanks to the definition of $L$ and $T$ it is immediate to see that $L(z,u)=T(u,z)$.
\end{osservazione}

The explicit expressions for $L,Q$ and $T$ allow us to obtain the following: 

\begin{proposizione}\label{prop7}
For any $z,u\in\HH^n$ the following estimates hold:
\begin{itemize}
\item[(i)] $\lvert L(u,z)\rvert\leq4\lVert u\rVert \lVert z\rVert^3$,
\item[(ii)] $\lvert Q(u,z)\rvert\leq 12\lVert z\rVert^2\lVert u\rVert^2$,
\item[(iii)] $\lvert T(u,z)\rvert\leq4\lVert z\rVert \lVert u\rVert^3$.
\end{itemize}
\end{proposizione}

\begin{proof}
We start with the estimate for $L(u,z)$:
\begin{equation}
\lvert L(u,z)\rvert=\lvert\langle  u_H,4\lvert   z_H\rvert^2   z_H+4  z_T J   z_H\rangle\rvert\leq 4\lVert u\rVert \big\lvert \lvert   z_H\rvert^2   z_H+  z_T J   z_H\big\rvert.
\label{numeroo2}
\end{equation}
Since $z$ and $Jz$ are orthogonal, we have that $\lvert \lvert   z_H\rvert^2   z_H+  z_T J   z_H\rvert^2=\lvert   z_H\rvert^6+  z_T^2 \lvert    z_H\rvert^2=\lvert   z_H\rvert^2 \lVert z\rVert^4$ and thus \eqref{numeroo2} can be rewritten as:
\begin{equation}
\lvert L(u,z)\rvert\leq 4\lVert u\rVert \lVert z\rVert^2 \lvert   z_H\rvert\nonumber.
\end{equation}
Item (i) immediately follows from the above inequality. Furthermore, since Remark \ref{rk1} shows that $L(z,u)=T(u,z)$, item (iii)
follows as well from the above discussion. 

We are thus left to prove item (ii), which follows from the following computation:
\begin{equation}
\begin{split}
\lvert Q(u,z)\rvert\leq &4\langle    z_H, u_H\rangle^2+2\lvert   z_H\rvert^2\lvert u_H\rvert^2+4\langle J   z_H, u_H\rangle^2+2\lvert  z_T\rvert\lvert  u_T\rvert\\
\leq& 10\lvert   z_H\rvert^2\lvert u_H\rvert^2+2\lvert  z_T\rvert\lvert  u_T\rvert\leq 12\lVert z\rVert^2\lVert u\rVert^2.
\end{split}
\nonumber
\end{equation}
This concludes the proof of the proposition.
\end{proof}

The following proposition shows that the polarisation function $V$ enjoys a Cauchy-Schwartz-type inequality, that will turn out to be fundamental for our computations.

\begin{proposizione}[Cauchy-Schwarz inequality for $V$]\label{prop4}
For any $u,z\in\HH^n$ we have:
\begin{equation}
\lvert V(u,z)\rvert\leq 2\lVert u\rVert \lVert z\rVert(\lVert u\rVert+\lVert z\rVert)^2\nonumber.
\end{equation}
\end{proposizione}

\begin{proof}
Thanks to the triangle inequality, we have that $\lVert u^{-1}*z\rVert\geq\lvert\lVert u\rVert-\lVert z\rVert\rvert$. In particular:
\begin{equation}
\begin{split}
\lVert u^{-1}*z\rVert^4\geq&\lvert\lVert z\rVert-\lVert u\rVert\rvert^4=\lVert u\rVert^4-4\lVert u\rVert^3\lVert z\rVert+6\lVert u\rVert^2\lVert z\rVert^2-4\lVert u\rVert\lVert z\rVert^3+\lVert z\rVert^4.
\nonumber
\end{split}
\end{equation}
With few algebraic manipulations of the above inequality, we can make the terms in the definition of $V$ appear:
\begin{equation}
\begin{split}
4\lVert u\rVert^3\lVert z\rVert-6\lVert u\rVert^2\lVert z\rVert^2+4\lVert u\rVert\lVert z\rVert^3\geq&\lVert u\rVert^4+\lVert z\rVert^4-\lVert u^{-1}*z\rVert^4=2V(u,z).
\nonumber
\end{split}
\end{equation}
Collecting terms, we finally have that $2\lVert u\rVert\lVert z\rVert(\lVert u\rVert+\lVert z\rVert)^2\geq V(u,z)$. 
The other inequality follows with a similar argument.
\end{proof}

The following definition extends from the Euclidean spaces to the Heisenberg groups the notion of moments of a uniform measure given by Preiss in \cite{Preiss1987GeometryDensities}:

\begin{definizione} [Preiss' moments]\label{defimom}
Suppose $\mu$ is an $m$-uniform measure. Then, for any $k\in\N$ any $s>0$ and any $u_1,\ldots, u_k \in\HH^n$, we define:
\begin{equation}
b_{k,s}^\mu(u_1,\ldots,u_k):=\frac{s^{k+\frac{m}{4}}}{k!C(m)}\int \prod_{i=1}^k 2V(u_i,z)e^{-s\lVert z\rVert^4}d\mu(z)\nonumber,
\end{equation}
where $C(m):=\Gamma\left(\frac{m}{4}+1\right)$ and where we let $b_{0,s}^\mu:=1$. Moreover, if $u_1=\ldots=u_k=u$, we will always simplify the notation to:
$$b_{k,s}^\mu(u):=b_{k,s}^\mu(u,\ldots,u).$$
\end{definizione}

The Cauchy-Schwartz inequality for $V$ allows us to obtain the following estimates:

\begin{proposizione}\label{prop2}
For any $\mu\in\mathcal{U}(m)$, any $k\in\N$, any $s>0$ and any $u\in\HH^n$, we have:
\begin{equation}
\lvert b_{k,s}^\mu(u)\rvert\leq 16^{k}\frac{(\lVert u\rVert s^\frac{1}{4})^k}{k!}\frac{\Gamma(\frac{m+3k}{4})}{\Gamma\left(\frac{m}{4}\right)}((\lVert u\rVert s^\frac{1}{4})^{2k}+1).
\nonumber
\end{equation}
\end{proposizione}

\begin{proof}
Thanks to Proposition \ref{prop4}, we infer the following preliminary estimate:
\begin{equation}
\begin{split}
\lvert b_{k,s}^\mu(u)\rvert \leq& s^\frac{m}{4}\frac{(2s)^k}{k!C(m)}\int \lvert V(u,z)\rvert^k e^{-s\lVert z\rVert^4}d\mu(z)\\
\leq& s^\frac{m}{4}\frac{(2s)^k}{k!C(m)}\int 2^k\lVert u\rVert^k \lVert z\rVert^k(\lVert u\rVert+\lVert z\rVert)^{2k} e^{-s\lVert z\rVert^4}d\mu(z).
\nonumber
\end{split}
\end{equation}
Moreover, Jensen inequality, used in the first line, and Corollary \ref{prop1}, used in the third line to explicitly compute the integrals, imply that:
\begin{equation}
\begin{split}
&\lvert b_{k,s}^\mu(u)\rvert
\leq 2^{3k}s^\frac{m}{4}\frac{(2s)^k}{k!C(m)}\int \lVert u\rVert^k \lVert z\rVert^k (\lVert u\rVert^{2k}+\lVert z\rVert^{2k}) e^{-s\lVert z\rVert^4}d\mu(z)\\
\leq&2^{3k}s^\frac{m}{4}\frac{(2s)^k}{k!C(m)}\left(\int \lVert u\rVert^{3k} \lVert z\rVert^k e^{-s\lVert z\rVert^4}d\mu(z)+\int \lVert u\rVert^k \lVert z\rVert^{3k} e^{-s\lVert z\rVert^4}d\mu(z)\right)\\
&\qquad=2^{4k}\frac{m}{4}\frac{\lVert u\rVert^k s^\frac{k}{4}}{k!C(m)}\left(\lVert u\rVert^{2k}s^\frac{k}{2}\Gamma\left(\frac{m+k}{4}\right)+\Gamma\left(\frac{m+3k}{4}\right)\right)\\
&\qquad\qquad\qquad\leq16^{k}\frac{(\lVert u\rVert s^\frac{1}{4})^k}{k!}\frac{\Gamma\left(\frac{m+3k}{4}\right)}{\Gamma\left(\frac{m}{4}\right)}((\lVert u\rVert s^\frac{1}{4})^{2k}+1),
\nonumber
\end{split}
\end{equation}
where in the last inequality we used the definition of $C(m)$, see Definition \ref{defimom}.
\end{proof}

\begin{definizione}\label{calpha}
For any $\mu\in\mathcal{U}(m)$, any $\alpha\in\N^3\setminus\{(0,0,0)\}$ and any $s>0$ we define the functions $c_{\alpha,s}^\mu :\bigotimes_{i=0}^{\lvert \alpha\rvert}\HH^n\to \R$ as:
\begin{equation} 
c_{\alpha,s}^\mu (u):=\frac{1}{\alpha_1!\alpha_2!\alpha_3!}\frac{s^{\lvert \alpha\rvert+\frac{m}{4}}}{C(m)}\int  L(u,z)^{\alpha_1} Q(u,z)^{\alpha_2} T(u,z)^{\alpha_3} e^{-s\lVert z\rVert^4}d\mu(z),
\nonumber
\end{equation}
where $\lvert \alpha\rvert:=\alpha_1+\alpha_2+\alpha_3$. Moreover, for any $l\in\N$, we let:
$$A(l):=\{\alpha\in\N^3\setminus \{(0,0,0)\}: \alpha_1+2\alpha_2+3\alpha_3\leq l\}.$$
\end{definizione}

The moments $b_{k,s}^\mu$ can be expressed by means of the functions $c_{\alpha,s}^\mu $ defined above:

\begin{equation}
\begin{split}
&\qquad\qquad\qquad b_{k,s}^\mu(u)=\frac{s^{k+\frac{m}{4}}}{k!C(m)}\int (2V(u,z))^k e^{-s\lVert z\rVert^4}d\mu(z)\\
&\qquad=\frac{s^{k+\frac{m}{4}}}{k!C(m)}\int (L(u,z)+Q(u,z)+T(u,z))^k e^{-s\lVert z\rVert^4}d\mu(z)\\
=&\frac{s^{k+\frac{m}{4}}}{k!C(m)}\int \sum_{\lvert\alpha\rvert=k}\frac{k!}{\alpha_1!\alpha_2!\alpha_3!}L(u,z)^{\alpha_1} Q(u,z)^{\alpha_2} T(u,z)^{\alpha_3} e^{-s\lVert z\rVert^4}d\mu(z)\\
=&\sum_{\lvert\alpha\rvert=k}\frac{1}{\alpha_1!\alpha_2!\alpha_3!}\frac{s^{k+\frac{m}{4}}}{C(m)}\int L(u,z)^{\alpha_1} Q(u,z)^{\alpha_2} T(u,z)^{\alpha_3} e^{-s\lVert z\rVert^4}d\mu(z)\\
&\qquad\qquad\qquad\qquad\qquad=\sum_{\lvert\alpha\rvert=k} c_{\alpha,s}^\mu (u).
\label{splitter}
\end{split}
\end{equation}

\begin{proposizione}\label{prop9}
For any $\mu\in\mathcal{U}(m)$ and any $\alpha\in\N^3\setminus\{(0,0,0)\}$, there exists a constant $D(\alpha)>0$ such that for any $s>0$ and any $u\in\HH^n$ we have:
\begin{equation}
\lvert c_{\alpha,s}^\mu (u)\rvert\leq D(\alpha)(s^{1/4}\lVert u\rVert)^{\alpha_1+2\alpha_2+3\alpha_3}.
\nonumber
\end{equation}
\end{proposizione}

\begin{proof}
Since Proposition \ref{prop7} gives bounds on $\lvert L(u,z)\rvert$, $\lvert Q(u,z)\rvert$ and $\lvert T(u,z)\rvert$, it allows us to estimate the integrand in the definition of $c_{\alpha,s}^\mu $ in the following way:
\begin{equation}
\begin{split}
\lvert c_{\alpha,s}^\mu (u)\rvert\leq&\frac{1}{\alpha_1!\alpha_2!\alpha_3!}\frac{s^{\lvert\alpha\rvert+\frac{m}{4}}}{C(m)}\int \lvert L(u,z)\rvert^{\alpha_1} \lvert Q(u,z)\rvert^{\alpha_2} \lvert T(u,z)\rvert^{\alpha_3} e^{-s\lVert z\rVert^4}d\mu(z)\\
\leq&\frac{4^{\alpha_1+\alpha_3}12^{\alpha_2}}{\alpha_1!\alpha_2!\alpha_3!}\frac{s^{\lvert \alpha\rvert+\frac{m}{4}}}{C(m)}\lVert u\rVert^{\alpha_1+2\alpha_2+3\alpha_3}\int \lVert z\rVert^{3\alpha_1+2\alpha_2+\alpha_3} e^{-s\lVert z\rVert^4}d\mu(z).
\label{numeroo3}
\end{split}
\end{equation}
Moreover, by Corollary \ref{prop1} we have:
\begin{equation}
    \int \lVert z\rVert^{3\alpha_1+2\alpha_2+\alpha_3} e^{-s\lVert z\rVert^4}d\mu(z)=\frac{m}{4s^\frac{m+3\alpha_1+2\alpha_2+\alpha_3}{4}}\Gamma\Big(\frac{m+3\alpha_1+2\alpha_2+\alpha_3}{4}\Big)
    \label{numeroo4}
\end{equation}
Therefore, plugging identity  \eqref{numeroo4} in \eqref{numeroo3}, we conclude with few elementary algebraic manipulations, that:
\begin{equation}
\begin{split}
\lvert c_{\alpha,s}^\mu (u)\rvert\leq\frac{4^{\alpha_1+\alpha_3}12^{\alpha_2}}{\alpha_1!\alpha_2!\alpha_3!}\Gamma\left(\frac{m}{4}\right)^{-1}\Gamma\left(\frac{m+3\alpha_1+2\alpha_2+\alpha_3}{4}\right)(\lVert u\rVert^4 s)^\frac{\alpha_1+2\alpha_2+3\alpha_3}{4}.
\end{split}
\nonumber
\end{equation}
With the choice
$D(\alpha):=\frac{4^{\alpha_1+\alpha_3}12^{\alpha_2}}{\alpha_1!\alpha_2!\alpha_3!}\Gamma\left(\frac{m}{4}\right)^{-1}\Gamma\left(\frac{m+3\alpha_1+2\alpha_2+\alpha_3}{4}\right)$, 
we get the desired conclusion.
\end{proof}

\begin{proposizione}\label{coni1}
Assume $\mu\in\mathcal{U}(m)$ is invariant under dilations, i.e., for any $\lambda>0$ we have $\mu_{0,\lambda}/\lambda^m=\mu$,
where $\mu_{0,\lambda}$ was defined in \eqref{trasldil}. Then, for any $\alpha\in\N^3\setminus\{(0,0,0)\}$ and any $s>0$ we have:
$$c_{\alpha,s}^\mu =s^{\frac{\alpha_1+2\alpha_2+3\alpha_3}{4}}c_{\alpha,1}.$$
\end{proposizione}

\begin{proof}
Thanks to Remark \ref{rk1}, for any $\lambda>0$ we have that:
\begin{equation}
    \begin{split}
    &L(u,D_\lambda (z))^{\alpha_1}Q(u,D_\lambda(z))^{\alpha_2} T(u, D_\lambda(z))^{\alpha_3}\\
    &\qquad\qquad \qquad =\lambda^{3\alpha_1+2\alpha_2+\alpha_3}L(u,z)^{\alpha_1} Q(u,z)^{\alpha_2} T(u,z)^{\alpha_3}.
        \label{numeroo5}
    \end{split}
\end{equation}
Therefore, defined $\lambda:=1/s^\frac{1}{4}$ we conclude that:
\begin{equation} 
\begin{split}
c_{\alpha,s}^\mu (u)=&\frac{1}{\alpha_1!\alpha_2!\alpha_3!}\frac{s^{\lvert \alpha\rvert+\frac{m}{4}}}{C(m)}\int L(u,z)^{\alpha_1} Q(u,z)^{\alpha_2} T(u,z)^{\alpha_3} e^{-s\lVert z\rVert^4}d\mu(z)\\
=&\frac{1}{\alpha_1!\alpha_2!\alpha_3!}\frac{s^\frac{\alpha_1+2\alpha_2+3\alpha_3}{4}}{C(m)}\int L(u,z)^{\alpha_1} Q(u,z)^{\alpha_2} T(u,z)^{\alpha_3} e^{-\lVert z\rVert^4}d\frac{\mu_{0,\lambda}(z)}{\lambda^m},
    \end{split}
\nonumber
\end{equation}
where in the last above identity we used \eqref{numeroo5} and the fact that $\mu_{0,\lambda}/\lambda^m=\mu$.
\end{proof}

\subsection{Expansion formulas for moments}
This subsection is devoted to the proof of \eqref{eq1}, the expasion formula for the moments of uniform measures. Moreover, in Proposition \ref{prop8}, we start to flesh out the complex algebra of the inequality \eqref{eq1}, in order to build the desired quadric containing $\supp(\mu)$. We start with a technical lemma which will be required in the proof of Proposition \ref{expanzione}:

\begin{lemma}\label{prop3}
For any $m,k\in\N$ we have
$\Gamma\left(\frac{3k+m}{4}\right)\leq 8^{m/4}(6k/7)^{3k/4}e^{-3k/4}\Gamma\left(\frac{m}{4}\right)$.
\end{lemma}

\begin{proof}
By definition of the $\Gamma$ function we have:
\begin{equation}
   \begin{split}
\Gamma\left(\frac{3k+m}{4}\right)=&\int_0^\infty t^{\frac{3k+m}{4}-1}e^{-t}dt\\
\leq& \lVert g\rVert_\infty\int_0^\infty t^{\frac{m}{4}-1}e^{-t/8}dt=8^\frac{m}{4}\lVert g\rVert_\infty\Gamma\left(\frac{m}{4}\right),
    \label{numeroo6}
\end{split} 
\end{equation}
where $g(t):=t^\frac{3k}{4} e^{-7t/8}$. An easy exercise shows that the function $g$ attains its maximum at $t_*:=6k/7$ and in particular:
$$\lVert g\rVert_\infty\leq g(t_*)= (6k/7)^{3k/4}e^{-3k/4}.$$
The above estimate of $\lVert g\rVert_\infty$ together with \eqref{numeroo6} concludes the proof.
\end{proof}

The following proposition is the technical core of this section. As we already remarked, \eqref{eq1} will allow us to construct the algebraic surfaces containing $\supp(\mu)$. The proof follows closely its Euclidean analogue which can be found in  \cite[Section 3.4]{Preiss1987GeometryDensities} or in \cite[Lemma 7.6]{DeLellis2008RectifiableMeasures}.

\begin{proposizione}[Expansion formula]\label{expanzione}
There exists a constant $G(m)>0$ such that for any $\mu\in\mathcal{U}(m)$, any $s>0$, any $q\in\N$ and any $u\in\supp(\mu)$ we have:
\begin{equation}
\Big\lvert \sum_{k=1}^{4q}b_{k,s}^\mu(u)-\sum_{k=1}^q\frac{s^k\lVert u\rVert^{4k}}{k!}\Big\rvert\leq G(m)(s\lVert u\rVert^4)^{q+\frac{1}{4}}(2+(s\lVert u\rVert^4)^{2q})\label{eq1}.
\end{equation} 
\end{proposizione}

\begin{proof}
Throughout the proof we will always assume that $s>0$ and that $u\in\supp(\mu)$.

First consider the case in which $s\lVert u\rVert^4\geq 1$. The triangle inequality and Proposition \ref{prop2}, used to get the bound in the second line, imply that:
\begin{equation}
\begin{split}
&\Big\lvert \sum_{k=1}^{4q}b_{k,s}^\mu(u)-\sum_{k=1}^q\frac{s^k\lVert u\rVert^{4k}}{k!}\Big\rvert\leq\Big\lvert \sum_{k=1}^{4q}b_{k,s}^\mu(u)\Big\rvert+\Big\lvert\sum_{k=1}^q\frac{s^k\lVert u\rVert^{4k}}{k!}\Big\rvert\\
\leq&\Gamma\left(\frac{m}{4}\right)^{-1}\sum_{k=1}^{4q}16^{k}\frac{(\lVert u\rVert s^\frac{1}{4})^k}{k!}\Gamma\left(\frac{m+3k}{4}\right)((\lVert u\rVert s^\frac{1}{4})^{2k}+1)+\sum_{k=1}^q\frac{s^k\lVert u\rVert^{4k}}{k!}\\
\leq&(\lVert u\rVert^4 s)^{q+\frac{1}{4}}((\lVert u\rVert^4 s)^{2q}+1)\Gamma\Big(\frac{m}{4}\Big)^{-1}\sum_{k=1}^{4q}\frac{16^{k}}{k!}\Gamma\Big(\frac{m+3k}{4}\Big)+(s\lVert u\rVert^{4})^q\sum_{k=1}^q\frac{1}{k!}\\
\leq& (\lVert u\rVert^4 s)^{q+\frac{1}{4}}\big(\big((\lVert u\rVert^4 s)^{2q}+1\big)E(m)+e\big),
\label{numeroo7}
\end{split}
\end{equation}
where $E(m):=\Gamma\left(\frac{m}{4}\right)^{-1}\sum_{k=1}^{\infty}\frac{16^{k}}{k!}\Gamma\left(\frac{m+3k}{4}\right)$. In order show that the last line of \eqref{numeroo7} is not a trivial bound, we need to show that $E(m)$ is finite. To this end, let us note that:
\begin{equation}
E(m)=\Gamma\left(\frac{m}{4}\right)^{-1}\sum_{k=1}^{\infty}\frac{16^{k}}{k!}\Gamma\left(\frac{m+3k}{4}\right)\leq8^{m/4}\sum_{k=1}^{\infty}\frac{16^{k}}{k!}(6k/7)^{3k/4}e^{-3k/4},
\label{numeroo8}
\end{equation}
where the inequality above comes from Lemma \ref{prop3}.
The series on the right-hand side of \eqref{numeroo8} converges by the ratio test and thus, by comparison, $E(m)$ is finite. Therefore, defined $G(m):=\max\{E(m),e\}$, 
inequality \eqref{eq1} immediately follows from \eqref{numeroo7} in the case $s\lVert u\rVert^4\geq 1$:
\begin{equation}
\begin{split}
\Big\lvert \sum_{k=1}^{4q}b_{k,s}^\mu(u)-\sum_{k=1}^q\frac{s^k\lVert u\rVert^{4k}}{k!}\Big\rvert
\leq G(m)(\lVert u\rVert^4 s)^{q+\frac{1}{4}}((\lVert u\rVert^4 s)^{2q}+2).
\end{split}
\nonumber
\end{equation}
We are left to prove \eqref{eq1} in the case $s\lVert u\rVert^4< 1$. First of all, thanks to the assumption $s\lVert u\rVert^4<1$ it is immediate to see that:
\begin{equation}
\begin{split}
\Big\lvert\sum_{k=q+1}^\infty \frac{s^k\lVert u\rVert^{4k}}{k!}\Big\rvert
\leq (s\lVert u\rVert^{4})^{q+1}\sum_{k=q+1}^\infty \frac{1}{k!}
\leq e(s\lVert u\rVert^{4})^{q+1}.\label{numeroo9}
\end{split}
\end{equation}
Secondly, since $s\lVert u\rVert^4<1$, the series $\sum_{k=1}^\infty b_{k,s}^\mu(u)$ converges absolutely:
\begin{equation}
\begin{split}
\sum_{k=1}^\infty \lvert b_{k,s}^\mu(u)\rvert\leq&\Gamma\left(\frac{m}{4}\right)^{-1}\sum_{k=1}^\infty 16^{k}\frac{(\lVert u\rVert s^\frac{1}{4})^k}{k!}\Gamma\left(\frac{m+3k}{4}\right)((\lVert u\rVert s^\frac{1}{4})^{2k}+1)\\
\leq &2\sum_{k=1}^\infty \frac{16^{k}}{k!}\frac{\Gamma\left(\frac{m+3k}{4}\right)}{\Gamma\left(\frac{m}{4}\right)}=2E(m)<\infty,
\nonumber
\end{split}
\end{equation}
where the first inequality comes from Proposition \ref{prop2} and the last one from the finiteness of $E(m)$ that was established above. Since it will come in handy later on, we estimate here the tail of $\sum_{k=1}^\infty  b_{k,s}^\mu(u)$:
\begin{equation}
\begin{split}
\Big\lvert \sum_{k=4q+1}^\infty& b_{k,s}^\mu(u)\Big\rvert
\leq\Gamma\Big(\frac{m}{4}\Big)^{-1}\sum_{k=4q+1}^\infty 16^{k}\frac{(\lVert u\rVert s^\frac{1}{4})^k}{k!}\Gamma\left(\frac{m+3k}{4}\right)((\lVert u\rVert s^\frac{1}{4})^{2k}+1)\\
\leq&(\lVert u\rVert^4 s)^{q+\frac{1}{4}}((\lVert u\rVert^4 s)^{2q+1}+1)\Gamma\left(\frac{m}{4}\right)^{-1}\sum_{k=4q+1}^\infty\frac{16^{k}}{k!}\Gamma\left(\frac{m+3k}{4}\right)\\
\leq&(\lVert u\rVert^4 s)^{q+\frac{1}{4}}((\lVert u\rVert^4 s)^{2q+1}+1)E(m),
\label{eqtail}
\end{split}
\end{equation}
where in the first inequality we used Proposition \ref{prop2} and in the second we used the hypothesis that $s\lVert u\rVert^4<1$.

The third step is to show that the following identity holds:
\begin{equation}
\sum_{k=0}^\infty b_{k,s}^\mu(u)=e^{s\lVert u\rVert^4}.
\label{eq7}
\end{equation}

To this end, let us note that the following holds by definition:
\begin{equation}
\sum_{k=0}^\infty b_{k,s}^\mu(u)=\lim_{q\to\infty}\frac{s^{\frac{m}{4}}}{C(m)}\int \sum_{k=0}^q\frac{(2sV(u,z))^k}{k!} e^{-s\lVert z\rVert^4}d\mu(z).
\label{numeroo10}
\end{equation}
We now prove that it is possible to exchange the limit and the integral in \eqref{numeroo10} using dominated convergence. To do so, we first remark that the following estimate holds for any $z\in \HH^n$:
\begin{equation}
\begin{split}
\Big\lvert\sum_{k=0}^q \frac{(2sV(u,z))^k}{k!} e^{-s\lVert z\rVert^4}\Big\rvert\leq &e^{-s\lVert z\rVert^4}\sum_{k=0}^q \frac{(4s\lVert u\rVert \lVert z\rVert(\lVert u\rVert+\lVert z\rVert)^2)^k}{k!}\\
\leq&
e^{-s\lVert z\rVert^4+4s\lVert u\rVert \lVert z\rVert(\lVert u\rVert+\lVert z\rVert)^2},
\nonumber
\end{split}
\end{equation}
where the first inequality comes from Proposition \ref{prop4}. It is an exercise with Proposition \ref{prop5} to prove that the function $f(z):=e^{-s\lVert z\rVert^4+4s\lVert u\rVert \lVert z\rVert(\lVert u\rVert+\lVert z\rVert)^2}$ is a function $L^1(\mu)$ with respect to the variable $z$. Therefore, applying the dominated convergence theorem (pointwise convergence is obvious), we get:
\begin{equation}
\begin{split}
\sum_{k=0}^\infty b_{k,s}^\mu(u)=&\frac{s^{\frac{m}{4}}}{C(m)}\int \Big(\sum_{k=0}^\infty \frac{(2sV(u,z))^k}{k!}\Big) e^{-s\lVert z\rVert^4}d\mu(z)\\
=&\frac{s^{\frac{m}{4}}}{C(m)}\int e^{2sV(u,z)} e^{-s\lVert z\rVert^4}d\mu(z)\\
=&\frac{s^{\frac{m}{4}}}{C(m)}e^{s\lVert u\rVert^4}\int e^{-s\lVert u\rVert^4+2sV(u,z)-s\lVert z\rVert^4}d\mu(z).
\label{numeroo11}
\end{split}
\end{equation}
By definition of $V(u,z)$, we have $-\lVert u\rVert^4+2V(u,z)-\lVert z\rVert^4=-\lVert u^{-1}*z\rVert^4$ and thus \eqref{numeroo11} becomes:
\begin{equation}
\begin{split}
\sum_{k=0}^\infty b_{k,s}^\mu(u)=&\frac{s^{\frac{m}{4}}}{C(m)}e^{s\lVert u\rVert^4}\int e^{-s\lVert u^{-1}*z\rVert^4}d\mu(z)\\
=&\frac{s^{\frac{m}{4}}}{C(m)}e^{s\lVert u\rVert^4}\frac{m}{4s^\frac{m}{4}}\Gamma\Big(\frac{m}{4}\Big)=e^{s\lVert u\rVert^4},
\nonumber
\end{split}
\end{equation}
where in the second last identity we used Corollary \ref{prop1} to compute the integral and in the last one we used the definition of $C(m)$, see Definition \ref{defimom}, and the fact that $\Gamma(t+1)=t\Gamma(t)$. We are now in position to conclude the proof of \eqref{eq1} in the case $s\lVert u\rVert^4<1$:
\begin{equation}
\begin{split}
\Big\lvert \sum_{k=1}^{4q}& b_{k,s}^\mu(u)-\sum_{k=1}^{q} \frac{s^k\lVert u\rVert^4}{k!}\Big\rvert\\
\leq &\Big\lvert \sum_{k=4q+1}^{\infty} b_{k,s}^\mu(u)\Big\rvert+\Big\lvert \sum_{k=1}^{\infty} b_{k,s}^\mu(u)-\sum_{k=1}^{\infty} \frac{s^k\lVert u\rVert^4}{k!}\Big\rvert+\Big\lvert\sum_{k=q+1}^{\infty} \frac{s^k\lVert u\rVert^4}{k!}\Big\rvert\\
=&\Big\lvert \sum_{k=4q+1}^{\infty} b_{k,s}^\mu(u)\Big\rvert+\Big\lvert \sum_{k=0}^{\infty} b_{k,s}^\mu(u)-e^{s\lVert u\rVert^4}\Big\rvert+\Big\lvert\sum_{k=q+1}^{\infty} \frac{s^k\lVert u\rVert^4}{k!}\Big\rvert,
\label{numeroo12}
\end{split}
\end{equation}
where in the last line we used the well known identity $e^t=\sum_{k=0}^\infty t^k/k!$ and the fact that by definition we have $b_{0,s}^\mu=1$. Finally, thanks to identity \eqref{eq7}, the bound \eqref{numeroo12} boils down to: 
\begin{equation}
    \begin{split}
        \Big\lvert \sum_{k=1}^{4q} b_{k,s}^\mu(u)-\sum_{k=1}^{q} \frac{s^k\lVert u\rVert^4}{k!}\Big\rvert
        \leq &\Big\lvert \sum_{k=4q+1}^{\infty} b_{k,s}^\mu(u)\Big\rvert+\Big\lvert\sum_{k=q+1}^{\infty} \frac{s^k\lVert u\rVert^4}{k!}\Big\rvert\\
        \leq&(\lVert u\rVert^4 s)^{q+\frac{1}{4}}((\lVert u\rVert^4 s)^{2q+1}+1)E(m)+e(s\lVert u\rVert^{4})^{q+1}\\
        \leq& G(m)(\lVert u\rVert^4 s)^{q+\frac{1}{4}}((\lVert u\rVert^4 s)^{2q+1}+2),
        \nonumber
    \end{split}
\end{equation}
where in the second inequality we used \eqref{numeroo9} and \eqref{eqtail} and in the last one the definition of $G(m)$. This concludes the proof of \eqref{eq1} in the case $s\lVert u\rVert^4<1$ and in turn the proposition.
\end{proof}

Recall that in \eqref{splitter} we showed how $b_{k,s}$ are a sum of functions $c_{\alpha,s}^\mu $:
\begin{equation}
    b_{k,s}^\mu(u)=\sum_{\lvert\alpha\rvert=k} c_{\alpha,s}^\mu (u).
    \label{numeroo13}
\end{equation}
Therefore, Proposition \ref{expanzione} when $q=1$ together with \eqref{numeroo13}, imply that:
\begin{equation}
    \Big\lvert \sum_{k=1}^{4}\sum_{\lvert\alpha\rvert=k} c_{\alpha,s}^\mu (u)-s\lVert u\rVert^{4}\Big\rvert\leq G(m)(s\lVert u\rVert^4)^{\frac{5}{4}}(2+(s\lVert u\rVert^4)^{2}).
    \label{numeroo14}
\end{equation}
On the left-hand side of the above inequality, one should expect that a lot of terms can be bounded from above by $(s\lVert u\rVert)^\frac{5}{4}$, and indeed this is the case. In the next proposition we get rid of those terms pushing them to the right-hand side, reducing in such a way the complexity of the algebraic expression we have to deal with.

\begin{proposizione}
For any $\mu\in \mathcal{U}(m)$, any $s>0$ and any $u\in\supp(\mu)$ we have:
\begin{equation}
\Big\lvert \sum_{\alpha\in A(4)}c_{\alpha,s}^\mu (u)-s\lVert u\rVert^{4}\Big\rvert\leq (s\lVert u\rVert^4)^{\frac{5}{4}} B(s^\frac{1}{4}\lVert u\rVert),
\label{eqquart}
\end{equation}
where $B(\cdot)$ is a suitable polynomial whereas $c_{\alpha,s}^\mu (\cdot)$ and $A(4)$ were defined in Definition \ref{calpha}.
\label{prop8}
\end{proposizione}

\begin{proof}
Since $\lvert \alpha\rvert\leq \alpha_1+2\alpha_2+3\alpha_3$ for any $\alpha\in\N^3$, we have that $A(4)$ is contained in the family of multi-indices:
\begin{equation}
    \mathcal{A}(4):=\{\alpha\in\N^3\setminus\{(0,0,0)\}:1\leq\lvert \alpha\rvert\leq 4\}.
    \label{numeroo15}
\end{equation}
This implies by the triangular inequality that:
\begin{equation}
\begin{split}
&\Big\lvert \sum_{\alpha\in A(4)}c_{\alpha,s}^\mu (u)-s\lVert u\rVert^{4}\Big\rvert\leq\Big\lvert \sum_{k=1}^{4}\sum_{\lvert\alpha\rvert=k}c_{\alpha,s}^\mu (u)-s\lVert u\rVert^{4}\Big\rvert+\Big\lvert\sum_{\alpha\in \mathcal{A}(4)\setminus A(4)}c_{\alpha,s}^\mu (u)\Big\rvert.
\label{eq:eq:1}
\end{split}
\end{equation}
Thanks to identity \eqref{numeroo14} in order to conclude the proof of the proposition, we just need to estimate the second factor in the right-hand side of \eqref{eq:eq:1}. We apply Proposition \ref{prop9} to each $\lvert c_{\alpha,s}^\mu (u)\rvert$, obtaining:
\begin{equation}
\begin{split}
    \Big\lvert\sum_{\alpha\in\mathcal{A}(4)\setminus A(4)}c_{\alpha,s}^\mu (u)\Big\rvert\leq&\sum_{\alpha\in\mathcal{A}(4)\setminus A(4)}\lvert c_{\alpha,s}^\mu (u)\rvert\\
    \leq& \sum_{\alpha\in\mathcal{A}(4)\setminus A(4)}D(\alpha)(s^\frac{1}{4}\lVert u\rVert)^{\alpha_1+2\alpha_2+3\alpha_3}.
    \end{split}
    \label{eq:eq:2}
\end{equation}
Note that since $\alpha\not\in A(4)$, see Definition \ref{calpha}, then $\alpha_1+2\alpha_2+3\alpha_3\geq 5$. Therefore, the exponents of $s^{1/4}\lVert u\rVert$ in the last term of the above chain of inequalities are all bigger than $5$. Putting together \eqref{numeroo14}, \eqref{eq:eq:1} and \eqref{eq:eq:2}, we infer that:
\begin{equation}
\begin{split}
      &\qquad\qquad\qquad\qquad\qquad\Big\lvert \sum_{\alpha\in A(4)}c_{\alpha,s}^\mu (u)-s\lVert u\rVert^{4}\Big\rvert\\
      \leq& G(m)(s\lVert u\rVert^4)^{\frac{5}{4}}(2+(s\lVert u\rVert^4)^{2})+\sum_{\alpha\in\mathcal{A}(4)\setminus A(4)}D(\alpha)(s^\frac{1}{4}\lVert u\rVert)^{\alpha_1+2\alpha_2+3\alpha_3}\\
      =&(s\lVert u\rVert^4)^{\frac{5}{4}}\Big(G(m)(2+(s^\frac{1}{4}\lVert u\rVert)^{8})+\sum_{\alpha\in\mathcal{A}(4)\setminus A(4)}D(\alpha)(s^\frac{1}{4}\lVert u\rVert)^{\alpha_1+2\alpha_2+3\alpha_3-5}\Big).
      \nonumber
\end{split}
\end{equation}
Therefore, defined:
$$B(t):=G(m)(2+t^{8})+\sum_{\alpha\in\mathcal{A}(4)\setminus A(4)}D(\alpha)t^{\alpha_1+2\alpha_2+3\alpha_3-5},$$
we have constructed the polynomial we were looking for and thus the proposition is proved.
\end{proof}

\subsection{Construction of the candidate quadric containing the support}
\label{CDDQQ}

Throughout this section, we let $\mu$ be a fixed $m$-uniform measure. Before proceeding with the description of the content of this subsection we give the following:

\begin{definizione} 
\label{definizionecurve}
For any $s\in(0,\infty)$ we let:
\begin{itemize}
\item [(i)] the \emph{horizontal barycenter} of the measure $\mu$ at time $s$ be the vector in $\R^{2n}$: $$b(s):=\frac{4s^{\frac{1}{2}+\frac{m}{4}}}{C(m)}\int(\lvert    z_H\rvert^2   z_H+  z_TJ   z_H)e^{-s\lVert z\rVert^4}d\mu(z),$$ 
\item[(ii)] the symmetric matrix $\mathcal{Q}(s)$ associated to the measure $\mu$ at time $s$ be the element of $\mathrm{Sym}(2n)$:
\begin{equation}
\begin{split}
\mathcal{Q}(s):=&-\frac{2s^{\frac{1}{2}+\frac{m}{4}}}{C(m)}\int \lvert   z_H\rvert^2 e^{-s\lVert z\rVert^4} d\mu(z)\mathrm{id}_{2n}\\
&-\frac{4s^{\frac{1}{2}+\frac{m}{4}}}{C(m)}\int(   z_H\otimes   z_H+J   z_H\otimes J   z_H)e^{-s\lVert z\rVert^4}d\mu(z)\\
&+\frac{8s^{\frac{3}{2}+\frac{m}{4}}}{C(m)}\int(\lvert   z_H\rvert^4   z_H\otimes   z_H+  z_T^2J   z_H\otimes J   z_H)e^{-s\lVert z\rVert^4}d\mu(z)\\
&+\frac{8s^{\frac{3}{2}+\frac{m}{4}}}{C(m)}\int \lvert   z_H\rvert^2  z_T(J   z_H\otimes   z_H+   z_H\otimes J   z_H)e^{-s\lVert z\rVert^4}d\mu(z),\nonumber
\end{split}
\end{equation}
\item [(iii)]
the \emph{vertical barycenter} of the measure $\mu$ at time $s$ be the real number:
$$\mathcal{T}(s):=\frac{2 s^{\frac{1}{2}+\frac{m}{4}}}{C(m)}\int z_T e^{-s\lVert z\rVert^4}d\mu(z).$$
\end{itemize}
\end{definizione}

The first half of this subsection is devoted to the proof of Proposition \ref{prop10}, where from Proposition \ref{prop8} we are able to further simplify the algebra of inequality \eqref{eq1} proving the existence of constant $C>0$ such that:
\begin{equation}
\Big\lvert\langle b(s), u_H\rangle+\langle \mathcal{Q}(s)u_H, u_H\rangle+\mathcal{T}(s)  u_T\Big\rvert\leq C s^\frac{1}{4}\lVert u\rVert.
\nonumber
\end{equation}

In the second half of this Subsection we prove that $b(\cdot)$, $\mathcal{Q}(\cdot)$ and $\mathcal{T}(\cdot)$ are bounded curves as $s$ goes to $0$ and therefore by compactness we can find $\overline{b}$, $\overline{\mathcal{Q}}$ and $\overline{\mathcal{T}}$ for which for any $u\in\supp(\mu)$ we have:
$$\langle\overline{b},u_H\rangle+\langle u_H,\overline{\mathcal{Q}} u_H\rangle+\overline{\mathcal{T}}u_T=0.$$
What is left to prove in Subsection \ref{nondegge} is that as $s\to 0$, we can find a limit $\overline{\mathcal{Q}}$ for which $\text{Tr}(\overline{\mathcal{Q}})\neq 0$.

\begin{proposizione}\label{prop10}
For any $0<s<1$ and any $u\in\supp(\mu)$ the following inequality holds:
\begin{equation}
\Big\lvert\langle b(s), u_H\rangle+\langle \mathcal{Q}(s) u_H, u_H\rangle+\mathcal{T}(s)  u_T\Big\rvert\leq s^\frac{1}{4}\lVert u\rVert^3B^\prime(s^\frac{1}{4}\lVert u\rVert),
\nonumber
\end{equation}
where  $B^\prime(\cdot)$ is a suitable polynomial and $b(\cdot)$, $\mathcal{Q}(\cdot)$ and $\mathcal{T}(\cdot)$ were introduced in Definition \ref{definizionecurve}.
\end{proposizione}

\begin{proof}
Since $\lvert \alpha\rvert\leq \alpha_1+2\alpha_2+3\alpha_3$ for any $\alpha\in\N^3$, then $A(2)$ is contained in the set $\mathcal{A}(4)$ that was introduced in \eqref{numeroo15}. Furthermore, let us remark that for any $\alpha\in \mathcal{A}(4)\setminus A(2)$, we have by definition that $\alpha_1+2\alpha_2+3\alpha_3\geq 3$. In particular, Proposition \ref{prop9} implies that:
\begin{equation}
\begin{split}
\sum_{\alpha\in \mathcal{A}(4)\setminus A(2)}\lvert c_{\alpha,s}^\mu (u)\rvert \leq &(s^\frac{1}{4}\lVert u\rVert)^3\sum_{\alpha\in \mathcal{A}(4)\setminus A(2)}D(\alpha)(s^\frac{1}{4}\lVert u\rVert)^{\alpha_1+2\alpha_2+3\alpha_3-3}\\
=&(s^\frac{1}{4}\lVert u\rVert)^3B^{\prime\prime}(s^\frac{1}{4}\lVert u\rVert).
\nonumber
\end{split}
\end{equation}
where $B^{\prime\prime}(t):=\sum_{\substack{\lvert\alpha\rvert\leq 4\\ \alpha\not\in A(2)}}D(\alpha)t^{\alpha_1+2\alpha_2+3\alpha_3-3}$. Hence, thanks to the triangular inequality and Proposition \ref{prop8}, we infer that:
\begin{equation}
\begin{split}
&\Big\lvert \sum_{\alpha\in A(2)}c_{\alpha,s}^\mu (u)\Big\rvert\leq \Big\lvert \sum_{\alpha\in A(4)}c_{\alpha,s}^\mu (u)-s\lVert u\rVert^4\Big\rvert+s\lVert u\rVert^4+\sum_{\alpha\in \mathcal{A}(4)\setminus A(2)}\lvert c_{\alpha,s}^\mu (u)\rvert\\
\leq&(s\lVert u\rVert^4)^{\frac{5}{4}} B(s^{\frac{1}{4}}\lVert u\rVert)+s\lVert u\rVert^4+(s^\frac{1}{4}\lVert u\rVert)^3B^{\prime\prime}(s^\frac{1}{4}\lVert u\rVert)
=s^\frac{3}{4}\lVert u\rVert^3B^\prime(s^\frac{1}{4}\lVert u\rVert).
\nonumber
\end{split}
\end{equation}
where $B^\prime(t):=t^2 B(t)+t+B^{\prime\prime}(t)$. A simple computation shows that $A(2)=\{(1,0,0),(2,0,0),(0,1,0)\}$  and thus to conclude the proof of the proposition we are left to prove the identity:
\begin{equation}
\begin{split}
    c_{(1,0,0),s}(u)+c_{(2,0,0),s}(u)+&c_{(0,1,0),s}(u)\\
    =&\sqrt{s}(\langle b(s),u_H\rangle+\langle u_H,\mathcal{Q}(s) u_H\rangle+\mathcal{T}(s)u_T).
    \label{numeroo16}
\end{split}
\end{equation}
In order to prove \eqref{numeroo16} we need to explicitly compute the expressions of $c_{(1,0,0),s}$, $c_{(2,0,0),s}$ and $c_{(0,1,0),s}$. As a first step, let us show that $c_{(1,0,0),s}(u)=s^{1/2}b(s)$:
\begin{equation}
\begin{split}
&\qquad\qquad\qquad c_{(1,0,0),s}(u)=\frac{s^{1+\frac{m}{4}}}{C(m)}\int L(u,z)e^{-s\lVert z\rVert^4}d\mu(z)\\
&=s^\frac{1}{2}\Big\langle  u_H, \frac{s^{\frac{1}{2}+\frac{m}{4}}}{C(m)}\int 4\lvert   z_H\rvert^2   z_H+4  z_T J   z_H e^{-s\lVert z\rVert^4}d\mu(z)\Big\rangle=s^\frac{1}{2}\langle  u_H, b(s)\rangle,
\nonumber
\end{split}
\end{equation}
Secondly, explicitly computing $c_{(2,0,0),s}(u)$ yields the first part of the quadric $\mathcal{Q}(s)$:
\begin{equation}
\begin{split}
 c_{(2,0,0),s}(u)=&\frac{1}{2}\frac{s^{2+\frac{m}{4}}}{C(m)}\int L(u,z)^2e^{-s\lVert z\rVert^4}d\mu(z)\\
=&\frac{16 s^{2+\frac{m}{4}}}{2C(m)}\int (\langle  u_H,\lvert   z_H\rvert^2   z_H\rangle^2
+\langle  u_H,  z_T J   z_H\rangle^2) e^{-s\lVert z\rVert^4}d\mu(z)\\
&\qquad+\frac{32 s^{2+\frac{m}{4}}}{2C(m)}\int \langle  u_H,\lvert   z_H\rvert^2   z_H\rangle\langle  u_H,  z_T J   z_H\rangle e^{-s\lVert z\rVert^4}d\mu(z)\\
=&s^\frac{1}{2}\langle  u_H,\mathcal{Q}_1(s) u_H \rangle+s^\frac{1}{2}\langle  u_H,\mathcal{Q}_2(s) u_H\rangle,
\nonumber
\end{split}
\end{equation}
where:
\begin{equation}
\begin{split}
    \mathcal{Q}_1(s):=&\frac{8 s^{\frac{3}{2}+\frac{m}{4}}}{C(m)}\int \lvert   z_H\rvert^4   z_H\otimes   z_H+  z_T^2 J   z_H\otimes J   z_H 
    e^{-s\lVert z\rVert^4}d\mu(z),\\
    \mathcal{Q}_2(s):=&\frac{8 s^{\frac{3}{2}+\frac{m}{4}}}{C(m)}\int \lvert   z_H\rvert^2  z_T (   z_H\otimes J   z_H+ J   z_H\otimes    z_H)e^{-s\lVert z\rVert^4}d\mu(z).
\end{split}
\label{eq:Q2}
\end{equation}
Finally, we see that the term $c_{(0,1,0),s}(u)$ contains the vertical barycenter and the second half of $\mathcal{Q}(s)$:
\begin{equation}
\begin{split}
 c_{(0,1,0),s}(u)=&\frac{s^{1+\frac{m}{4}}}{C(m)}\int Q(u,z) e^{-s\lVert z\rVert^4}d\mu(z)\\
=&-\frac{s^{1+\frac{m}{4}}}{C(m)}\int (4\langle    z_H, u_H\rangle^2+4\langle J   z_H, u_H\rangle^2) e^{-s\lVert z\rVert^4}d\mu(z)\\
&\qquad+\frac{s^{1+\frac{m}{4}}}{C(m)}\int (-2\lvert   z_H\rvert^2\lvert u_H\rvert^2+2  z_T  u_T) e^{-s\lVert z\rVert^4}d\mu(z)\\
=&-s^\frac{1}{2}\langle \mathcal{Q}_3(s) [u_H],u_H\rangle-s^\frac{1}{2}\langle\mathcal{Q}_4(s)[u_H], u_H\rangle+s^\frac{1}{2}\mathcal{T}(s)  u_T,
\nonumber
\end{split}
\end{equation}
where:
\begin{equation}
\begin{split}
    \mathcal{Q}_3(s):=&\frac{4s^{\frac{1}{2}+\frac{m}{4}}}{C(m)}\int (    z_H\otimes   z_H+ J  z_H\otimes J   z_H) e^{-s\lVert z\rVert^4}d\mu(z),\\
    \mathcal{Q}_4(s):=&\frac{2s^{\frac{1}{2}+\frac{m}{4}}}{C(m)}\int \lvert   z_H\rvert^2 e^{-s\lVert z\rVert^4}d\mu(z) \mathrm{id}_{2n}.
\end{split}
    \label{eq:Q4}
\end{equation}
Noticing that $\mathcal{Q}(s)=\mathcal{Q}_1(s)+\mathcal{Q}_2(s)-\mathcal{Q}_3(s)-\mathcal{Q}_4(s)\mathrm{id}_{2n}$, the claim is proven.
\end{proof}

\begin{osservazione}\label{oss1}
Let $\mathfrak{S}$ be the span of the horizontal projection of $\supp(\mu)$, i.e. $\mathfrak{S}:=\text{span}\{u_H:u\in\supp(\mu)\}$.
Then, with a small abuse of notation, we will always make the identification:
\begin{equation}
b(s)=\mathfrak{B}(s):=\frac{4s^{\frac{1}{2}+\frac{m}{4}}}{C(m)}\int(\lvert    z_H\rvert^2   z_H+  z_T \pi_\mathfrak{S}(J   z_H))e^{-s\lVert z\rVert^4}d\mu(z),\nonumber
\end{equation}
where the function $\pi_\mathfrak{S}:\R^{2n}\rightarrow \R^{2n}$ is the orthogonal projection onto the subspace $\mathfrak{S}$.
The reason for this identification is that:
$$\langle b(s), u_H\rangle=\langle \mathfrak{B}(s), u_H\rangle,\text{ for any }u\in\supp(\mu).$$
\end{osservazione}

\begin{proposizione}\label{bddcurve}
Both $\mathcal{Q}(s)$ and $\mathcal{T}(s)$ are bounded functions on $(0,\infty)$. To be precise:
\begin{itemize}
\item[(i)]
endowed $\text{Sym}(n)$ with some norm $\lvert\cdot\rvert$, there exists a constant $C_1>0$, such that $\sup_{s\in (0,\infty)}\lvert Q(s)\rvert\leq C_1$,
\item [(ii)]
there exists a constant $C_2>0$, such that $\sup_{s\in(0,\infty)}\lvert \mathcal{T}(s)\rvert\leq C_2$.
\end{itemize}
\end{proposizione}

\begin{osservazione}
Proposition \ref{bddcurve} implies in particular that the function $s\mapsto Tr(\mathcal{Q}(s))$ is bounded.
\end{osservazione}

\begin{proof}
Proposition \ref{prop9} implies that there exists a constant $\tilde{G}>0$ for which:
\begin{equation}
\begin{split}
\tilde{G}s^\frac{1}{2}\lVert u\rVert^2\geq&\lvert c_{(2,0,0),s}(u)+c_{(0,1,0),s}(u)\rvert=s^\frac{1}{2}\lvert \langle u_H,\mathcal{Q}(s)u_H \rangle+\mathcal{T}(s)  u_T\rvert\\
\geq &s^\frac{1}{2}\big\lvert \lvert\langle u_H,\mathcal{Q}(s) u_H \rangle\rvert-\lvert\mathcal{T}(s)\lvert\rvert  u_T\rvert\big\rvert.
\nonumber
\end{split}
\end{equation}
Thus, it suffices to give a bound for $\mathcal{T}(s)$ and the one for $\mathcal{Q}$ will follow, thanks to the inequality:
\begin{equation}
\lvert\langle u_H,\mathcal{Q}(s)[ u_H] \rangle\rvert\leq \Big(\tilde{G}+\sup_{s\in[0,\infty)}{\lvert\mathcal{T}(s)\rvert}\Big)\lVert u\rVert^2.
\label{numeroo17}
\end{equation}
Let us estimate the supremum norm of the curve $\mathcal{T}(s)$. As a first step, let us recall that by Corollary \ref{prop1} we have:
\begin{equation}
\begin{split}
        2\frac{s^{\frac{1}{2}+\frac{m}{4}}}{C(m)}\int\lVert z\rVert^2 e^{-s\lVert z\rVert^4}d\mu(z)=&2\frac{s^{\frac{1}{2}+\frac{m}{4}}}{C(m)}\cdot\frac{m}{4s^\frac{m+2}{4}}\Gamma\left(\frac{m+2}{4}\right)\\
    =&2\frac{m}{4C(m)}\Gamma\left(\frac{m+2}{4}\right)=2\frac{\Gamma\left(\frac{m+2}{4}\right)}{\Gamma\left(\frac{m}{4}\right)},
    \label{eq:eq:3}
\end{split}
\end{equation}
where the last identity comes from the definition of $C(m)$, see Definition \ref{defimom} and the fact that for any $t\in(0,\infty)$ we have $\Gamma(t+1)=t\Gamma(t)$.
Finally, the estimate on the supremum norm of $\mathcal{T}(s)$ follows by its definition and identity \eqref{eq:eq:3}:
\begin{equation}
\begin{split}
    \lvert\mathcal{T}(s)\rvert\leq&\frac{s^{\frac{1}{2}+\frac{m}{4}}}{C(m)}\int2\lVert  z_T\rVert e^{-s\lVert z\rVert^4}d\mu(z)\\
    \leq&\frac{s^{\frac{1}{2}+\frac{m}{4}}}{C(m)}\int2\lVert z\rVert^2 e^{-s\lVert z\rVert^4}d\mu(z)=2\Gamma\left(\frac{m}{4}\right)^{-1}\Gamma\left(\frac{m+2}{4}\right).
\end{split}
\label{numeroo18}
\end{equation}
Putting together \eqref{numeroo17} and \eqref{numeroo18} both items (i) and (ii) follow.
\end{proof}

From the above proposition we deduce that for any infinitesimal sequence $\{s_j\}_{j\in\N}$, by compactness we can extract a subsequence $\{s_{j_k}\}_{k\in\N}$, such that $\mathcal{Q}(s_{j_k})$ and $\mathcal{T}(s_{j_k})$ are converging to some $\tilde{\mathcal{Q}}$, $\tilde{\mathcal{T}}$. Therefore by Proposition \ref{prop10} we have:
\begin{equation}
\begin{split}
0\leq&\lim_{k\to\infty} \Big\lvert\langle b(s_{j_k}), u_H\rangle+\langle \mathcal{Q}(s_{j_k}) u_H, u_H\rangle+\mathcal{T}(s_{j_k}) u_T\Big\rvert\\
&\qquad\qquad\qquad\qquad\qquad\qquad\leq \lim_{k\to\infty} s_{j_k}^{1/4}\lVert u\rVert^3B^\prime(s_{j_k}^{1/4}\lVert u\rVert)=0.
\nonumber
\end{split}
\end{equation}
This in particular implies that for any $u\in\supp(\mu)$ we have:
\begin{equation}
 \lim_{k\to\infty}\langle b(s_{j_k}), u_H\rangle=-\langle \tilde{\mathcal{Q}} u_H, u_H\rangle-\tilde{\mathcal{T}}  u_T.
\label{eq8}
\end{equation}

\begin{proposizione}
There exists a $\overline{B}\in \mathfrak{S}$, which was introduced in Remark \ref{oss1}, such that $\lim_{k\to\infty}b(s_{j_k})=\overline{B}$.
\end{proposizione}

\begin{proof}
First of all we note that $\langle b(s),v\rangle=0$ for any $v\in \mathfrak{S}^\perp$, since $b(s_{j_k})\in \mathfrak{S}$ for any $k\in\N$, see Remark \ref{oss1}. Choose $\{u_1,\ldots,u_l\}\subseteq \supp(\mu)$, such that ${ (u_1)_H,\ldots, (u_l)_H}$ is a basis for $\mathfrak{S}$. 
Thus by equation \eqref{eq8}, we have:
\begin{equation}
\overline{B}=-\sum_{i=1}^l \left(\langle \tilde{\mathcal{Q}}(u_i)_H, (u_i)_H\rangle+\tilde{\mathcal{T}}  (u_i)_{T}\right)(u_i)_H.
\nonumber
\end{equation}
\end{proof}

If the measure $\mu$ is invariant under dilations, finding a candidate (non-degenerate) quadric containing $\supp(\mu)$ is quite easy:

\begin{proposizione}\label{CONO}
If $\mu$ is invariant under dilations, i.e. $\lambda^{-m}\mu_{0,\lambda}=\mu$ for any $\lambda>0$, then:
\begin{itemize}
    \item[(i)]$b(s)=0$ for any $s>0$,
    \item[(ii)]$\langle u_H,\mathcal{Q}(1) u_H\rangle+\mathcal{T}(1)  u_T=0 \text{ for any $u\in\supp(\mu)$.}$
\end{itemize}
\end{proposizione}

\begin{proof}
In the proof of Proposition \ref{prop10} we showed that: 
\begin{equation}
s^\frac{1}{2}\langle  u_H,b(s)\rangle=c_{(1,0,0),s}(u),
    \label{numeroo19}
\end{equation}
and that:
\begin{equation}
s^{1/2}(\langle u_H,\mathcal{Q}(s) u_H\rangle+\mathcal{T}(s)  u_T)=c_{(2,0,0),s}(u)+c_{(0,1,0),s}(u).
\label{numeroo20}
\end{equation}
Thanks to Proposition \ref{coni1}, identities \eqref{numeroo19} and \eqref{numeroo19} in this case become:
\begin{equation}
    s^\frac{1}{2}\langle  u_H,b(s)\rangle=s^\frac{1}{4}\langle b(1), u_H\rangle.
    \label{eq:eq:20}
\end{equation}
and:
\begin{equation}
s^{\frac{1}{2}}(\langle u_H,\mathcal{Q}(s) u_H\rangle+\mathcal{T}(s)  u_T)=s^\frac{1}{2}(\langle u_H,\mathcal{Q}(1) u_H\rangle+\mathcal{T}(1)  u_T),
    \label{eq:eq:21}
\end{equation}
Plugging identities \eqref{eq:eq:20} and \eqref{eq:eq:21} into the thesis of Proposition \ref{prop10} we conclude that:
$$\lvert s^{-\frac{1}{4}}\left\langle b(1), u_H\rangle+\right\langle u_H,\mathcal{Q}(1) u_H\rangle+\mathcal{T}(1)  u_T\rangle\rvert\leq s^\frac{1}{4}\lVert u\rVert^3B^\prime(s^\frac{1}{4}\lVert u\rVert),$$
for any $u\in\supp(\mu)$. Sending $s$ to $0$ we deduce that $b(1)=0$, as $b(1)\in \mathfrak{S}$, see Remark \ref{oss1}, and that:
$$\langle u_H,\mathcal{Q}(1) u_H\rangle+\mathcal{T}(1)  u_T=0.$$
\end{proof}

\begin{osservazione}
Note that as we already mentioned, if $\mu=\mathcal{C}^2\llcorner \mathcal{V}$ where $\mathcal{V}$ is the vertical axis $\{z\in\R^{2n+1}:z_H=0\}$, we have that $\mu$ is a dilation invariant $2$-uniform measure. However, an explicit computation shows that $\mathcal{Q}(1)=0$ and $\mathcal{T}(1)$. Therefore in this case our constructed quadric becomes trivial and thus not meaningful.
\end{osservazione}

\subsection{Non-degeneracy of the candidate quadric}
\label{nondegge}
The main result of this subsection can be stated as follows. Assume $\mu$ is a $m$-uniform measure in $\HH^n$ for which:
\begin{equation}
    \lim_{s\to 0} \text{Tr}(\mathcal{Q}(s))=0,
    \label{traccina}
\end{equation}
where $\mathcal{Q}(s)$ is the curve of symmetric matrices built in Proposition \ref{prop10}. Then:
\begin{equation}
    \Tan_m(\mu,\infty)=\{\mathcal{C}^2\llcorner \mathcal{V}\}.
    \nonumber
\end{equation}
It is not hard to show that the condition \eqref{traccina} is actually equivalent to:
$$
    b(s)\to 0,\qquad\mathcal{Q}(s)\to 0,\qquad
    \mathcal{T}(s)\to 0,
$$
as $s\to 0$ and thus we are really characterizing uniform measures for which our construction fails. This is where our argument significantly parts ways with its Euclidean analogue, see Section 4 of \cite{Kowalski1986Besicovitch-typeSubmanifolds}.

\begin{proposizione}\label{prop35}
Suppose $\mu\in\mathcal{U}(m)$ and let $\mathcal{Q}(s)$ be the curve of matrices introduced in Definition \ref{definizionecurve} relative to $\mu$. Then, for any $s>0$ the following equality holds:
\begin{equation}
Tr(\mathcal{Q}(s))=\frac{s^{\frac{m+2}{4}}}{C(m)}\int \lvert    z_H\rvert^2(8s\lVert z\rVert^4-(8+4n))e^{-s\lVert z\rVert^4}d\mu(z)\nonumber.
\end{equation}
\end{proposizione}

\begin{proof}
Let $\{e_1,\ldots, e_{2n}\}$ be an orthonormal basis of $\R^{2n}$. Since the trace is an invariant with respect to the change of basis, we have:
\begin{equation}
    \text{Tr}(\mathcal{Q}(s))=\sum_{i=1}^{n}\langle e_i,\mathcal{Q}(s)e_i\rangle+\langle e_{i+n},\mathcal{Q}(s)e_{i+n}\rangle.
    \label{eq:eq:25}
\end{equation}
Using the explicit expression for $\mathcal{Q}(s)$ we can compute $\langle e_i,\mathcal{Q}(s)e_i\rangle$ for any $i=1,\ldots, 2n$. If $1\leq i\leq n$ we have:
\begin{equation}
\begin{split}
\langle e_i,\mathcal{Q}(s)e_i\rangle= -&\frac{s^{\frac{1}{2}+\frac{m}{4}}}{C(m)}\int(4z_i^2+2\lvert   z_H\rvert^2+4z_{i+n}^2)e^{-s\lVert z\rVert^4}d\mu(z)\\
+&\frac{s^{\frac{3}{2}+\frac{m}{4}}}{C(m)}\int(8\lvert   z_H\rvert^4z_i^2+8  z_T^2z_{i+n}^2)e^{-s\lVert z\rVert^4}d\mu(z)\\
-&\frac{s^{\frac{3}{2}+\frac{m}{4}}}{C(m)}\int 8\lvert   z_H\rvert^2  z_T(2z_{i+n}z_i)e^{-s\lVert z\rVert^4}d\mu(z).
\end{split}
\label{eq:eq:23}
\end{equation}
On the other hand if $n+1\leq i\leq 2n$ we similarly have:
\begin{equation}
\begin{split}
\langle e_{i+n},\mathcal{Q}(s)e_{i+n}\rangle=-&\frac{s^{\frac{1}{2}+\frac{m}{4}}}{C(m)}\int(4z_{i+n}^2+2\lvert   z_H\rvert^2+4z_{i}^2)e^{-s\lVert z\rVert^4}d\mu(z)\\
+&\frac{s^{\frac{3}{2}+\frac{m}{4}}}{C(m)}\int(8\lvert   z_H\rvert^4z_{i+n}^2+8  z_T^2z_{i}^2)e^{-s\lVert z\rVert^4}d\mu(z)\\
+&\frac{s^{\frac{3}{2}+\frac{m}{4}}}{C(m)}\int 8\lvert   z_H\rvert^2  z_T(2z_{i+n}z_i)e^{-s\lVert z\rVert^4}d\mu(z).
\end{split}
\label{eq:eq:24}
\end{equation}
Putting together \eqref{eq:eq:25} with \eqref{eq:eq:23} and \eqref{eq:eq:24} we infer that:
\begin{equation}
\begin{split}
\text{Tr}(\mathcal{Q}(s))=&\sum_{i=1}^n-\frac{2s^{\frac{1}{2}+\frac{m}{4}}}{C(m)}\int(4z_{i+n}^2+2\lvert   z_H\rvert^2+4z_{i}^2)e^{-s\lVert z\rVert^4}d\mu(z)\\
&\qquad\qquad +\sum_{i=1}^{n}\frac{s^{\frac{3}{2}+\frac{m}{4}}}{C(m)}\int(8\lvert z_H\rvert^4(z_i^2+z_{i+n}^2)+8z_T^2(z_{i+n}^2+z_{i}^2))e^{-s\lVert z\rVert^4}d\mu(z)\\
=&\frac{s^{\frac{m+2}{4}}}{C(m)}\int \lvert    z_H\rvert^2(8s\lVert z\rVert^4-(8+4n))e^{-s\lVert z\rVert^4}d\mu(z).
\end{split}
\nonumber
\end{equation}
\end{proof}

\begin{proposizione}
Let $f:(0,\infty)\to[0,\infty)$ be the function:
\begin{equation}
f(s):=\int \lvert    z_H\rvert^2 e^{-s\lVert z\rVert^4}d\mu(z).
\label{effe}
\end{equation}
Then:
\begin{itemize}
\item[(i)] if $\supp(\mu)\not\subseteq\mathcal{V}$  one has $f(s)>0$ for any $s>0$,
\item[(ii)] $f$ is a smooth function and its derivatives are given by the formula:
\begin{equation}
f^{(i)}(s)=(-1)^i\int \lVert z\rVert^{4i}\lvert    z_H\rvert^2 e^{-s\lVert z\rVert^4}d\mu(z).
\label{cheeq}
\end{equation}
\item[(iii)] $s^{\frac{m+2}{4}+i}f^{(i)}(\cdot)$ is a bounded function on $(0,\infty)$ for any $i\in\N$.
\end{itemize}
\label{prop12}
\end{proposizione}

\begin{proof}
Assume there exists a $w\in\supp(\mu)$ such that $  w_H\neq 0$. Then:
\begin{equation}
\begin{split}
0<\frac{\lvert  w_H\rvert^2}{2} e^{-s\left(\frac{3\lVert w\rVert}{2}\right)^4}\mu(B_{\lvert  w_H\rvert/2}(w))\leq\int \lvert    z_H\rvert^2 e^{-s\lVert z\rVert^4}d\mu(z)= f(s),
\end{split}
\nonumber
\end{equation}
for any $s>0$. The fact that $f$ is smooth is proven by showing that \eqref{cheeq} holds and this is a standard application of the dominated convergence theorem.
The last point is a direct consequence of Corollary \ref{prop1} and the formula for $f^{(i)}$:
$$\lvert s^{\frac{m+2}{4}+i} f^{(i)}(s)\rvert\leq s^{\frac{m+2}{4}+i}\int \lVert z\rVert^{4i+2} e^{-s\lVert z\rVert^4}d\mu(z)\leq\frac{m}{4}\Gamma\left(\frac{m+2}{4}+i\right).$$
\end{proof}

\begin{osservazione}
Proposition \ref{prop35} and Proposition \ref{prop12} imply that:
\begin{itemize}
\item[(i)] if $\supp(\mu)\not\subseteq\mathcal{V}$, we have $(-1)^if^{(i)}(s)>0$ for any $i\in\N$,
\item[(ii)] the expression of the trace can be rewritten as follows:
\begin{equation}
\begin{split}
\text{Tr}(\mathcal{Q}(s))=&\frac{s^{\frac{m+2}{4}}}{C(m)}\int \lvert    z_H\rvert^2(8s\lVert z\rVert^4-(8+4n))e^{-s\lVert z\rVert^4}d\mu(z)\\
=&-8\frac{s^{\frac{m+6}{4}}}{C(m)}f^\prime(s)-(8+4n)\frac{s^{\frac{m+2}{4}}}{C(m)}f(s).
\label{eqeqeqeq}
\end{split}
\end{equation}
In particular this implies by Proposition \ref{prop12} that $\text{Tr}(\mathcal{Q}(s))$ is a smooth, bounded function on $(0,\infty)$.
\end{itemize}
\label{oss2}
\end{osservazione}

\begin{proposizione}\label{prop13}
The function $f$ defined in \eqref{effe} has the following representation by means of $\text{Tr}(\mathcal{Q}(\cdot))$:
\begin{equation}
f(s)=-\frac{C(m)}{8s^{\frac{n+2}{2}}}\int_0^s\lambda^{\frac{2n-2-m}{4}}\text{Tr}(\mathcal{Q}(\lambda))d\lambda,\qquad\text{for any $s>0$.}\nonumber
\end{equation}
\end{proposizione}

\begin{proof}
Since $f$ is smooth on $(0,\infty)$ by Proposition \ref{prop12}, the following equality holds true:
\begin{equation}
\frac{d}{ds}\left(s^{\frac{3}{2}+\frac{m}{4}}f(s)\right)=\left(\frac{3}{2}+\frac{m}{4}\right)s^{\frac{1}{2}+\frac{m}{4}}f(s)+s^{\frac{3}{2}+\frac{m}{4}}f^\prime(s),
\nonumber
\end{equation}
The expression for $\text{Tr}(\mathcal{Q}(s))$ in terms of $f$ and $f^\prime$ given in Remark \ref{oss2} (ii) together with the above identity imply:
\begin{equation}
\begin{split}
C(m)\text{Tr}(\mathcal{Q}(s))=&-8s^{\frac{m+6}{4}}f^\prime(s)-(8+4n)s^{\frac{m+2}{4}}f(s)\\
=&-8\frac{d}{ds}\left(s^{\frac{3}{2}+\frac{m}{4}}f(s)\right)+\left(4+2m-4n\right)s^{\frac{1}{2}+\frac{m}{4}}f(s).
\label{eq9}
\end{split}
\end{equation}
Define now $g(s):=s^{\frac{3}{2}+\frac{m}{4}}f(s)$, and note that equation \eqref{eq9} becomes:
\begin{equation}
C(m)\text{Tr}(\mathcal{Q}(s))=-8g^\prime(s)+\left(4+2m-4n\right)\frac{g(s)}{s}\nonumber.
\end{equation}
For any $\delta>0$ the function $g(s)$ solves the following Cauchy problem:
\begin{equation}
\begin{cases}
g^\prime(s)+\frac{\left(2n-m-2\right)}{4s}g(s)=-\frac{C(m)}{8}\text{Tr}(\mathcal{Q}(s)),\\
g(\delta)=\delta^{\frac{3}{2}+\frac{m}{4}}f(\delta).
\end{cases}
\nonumber
\end{equation}
Such Cauchy Problem has an explicit unique solution on $(\delta,\infty)$ since coefficients are smooth, Lipschitz and the vector field is sublinear in $g$, which is:
\begin{equation}
h_\delta(s)=-\frac{C(m)}{8s^\frac{2n-m-2}{4}}\left[\int_\delta^s\lambda^\frac{2n-m-2}{4} \text{Tr}(\mathcal{Q}(\lambda))d\lambda+\delta^\frac{2n+4}{4}f(\delta)\right]\nonumber,
\end{equation}
and coincides, by the uniqueness of solution, with $g$ on $(\delta,\infty)$. Point (iii) of Proposition \ref{prop12} implies that:
\begin{equation}
\begin{split}
\Big\lvert \delta^\frac{2n+4}{4}f(\delta)\Big\rvert\leq C\delta^\frac{2n+2-m}{4},
\nonumber
\end{split}
\end{equation}
for some $C>0$. Moreover, since $m\leq 2n+1$ we have that $\delta^\frac{2n+4}{4}f(\delta)\to 0$ as $\delta\to 0$. This implies that for any fixed $s>0$ we have:
\begin{equation}
\begin{split}
g(s)=\lim_{\delta\to 0} h_\delta(s)=&-\frac{C(m)}{8s^\frac{2n-m-2}{4}}\lim_{\delta\to 0}\left[\int_\delta^s\lambda^\frac{2n-m-2}{4} \text{Tr}(\mathcal{Q}(\lambda))d\lambda+\delta^\frac{2n+4}{4}f(\delta)\right]\\
=&-\frac{C(m)}{8s^\frac{2n-m-2}{4}}\int_0^s\lambda^\frac{2n-m-2}{4} \text{Tr}(\mathcal{Q}(\lambda))d\lambda,
\nonumber
\end{split}
\end{equation}
where the last equality comes from the fact that $\lvert\cdot\rvert^\frac{2n-m-2}{4} \text{Tr}(\mathcal{Q}(\cdot))\in L^1([0,1])$ as $m\leq 2n+1$.
\end{proof}

\begin{osservazione}
Since $\text{Tr}(\mathcal{Q}(s))$ is bounded, see Remark \ref{oss2}(ii), if $\lim_{s\to 0}Tr(\mathcal{Q}(s))$
does not exists or it exists and is non-zero, there is an infinitesimal sequence $\{s_j\}_{j\in\N}$ such that the trace of the matrices $\mathcal{Q}(s)$ converges to a non-zero value.
Up to passing to a subsequence (for which $b(\cdot)$, $\mathcal{Q}(\cdot)$ and $\mathcal{T}(\cdot)$ converge) we can build $\overline{b}$, $\tilde{\mathcal{Q}}$ and $\tilde{\mathcal{T}}$ as in  Proposition \ref{prop35} for which $\text{Tr}(\mathcal{Q})\neq 0$. This would imply that the quadric $\mathbb{K}(\overline{b},\tilde{\mathcal{Q}},\tilde{\mathcal{T}})$ contains $\supp(\mu)$ and would be non-degenerate. Therefore, without loss of generality, \textbf{in what follows we should always assume} $\lim_{s\to 0}Tr(\mathcal{Q}(s))=0$.
\end{osservazione}

\begin{proposizione}
Suppose that $\lim_{s\to 0} Tr(\mathcal{Q}(s))=0$. Then $\lim_{s\to 0}s^{\frac{m+2}{4}}f(s)=0$.
\end{proposizione}

\begin{proof}
For any $\epsilon$ there exists a $\delta>0$ such that for any $s\in(0,\delta)$ we have
$\lvert Tr(\mathcal{Q}(s))\rvert\leq\epsilon$.
In particular $Tr(\mathcal{Q}(s))>-\epsilon$, and Proposition \ref{prop13} implies for $s\in(0,\delta)$ that:
\begin{equation}
\begin{split}
f(s)=&-\frac{C(m)}{8s^{\frac{n+2}{2}}}\int_0^s\lambda^{\frac{2n-2-m}{4}}Tr(\mathcal{Q}(\lambda))d\lambda<\frac{\epsilon C(m)}{8s^{\frac{n+2}{2}}}\int_0^s\lambda^{\frac{2n-2-m}{4}}d\lambda\\
&\qquad\qquad<\frac{\epsilon C(m)}{8s^{\frac{n+2}{2}}}\frac{s^{\frac{2n+2-m}{4}}}{\frac{2n+2-m}{4}}=\frac{\epsilon C(m)}{2(2n+2-m)s^{\frac{m+2}{4}}}.
\end{split}
\nonumber
\end{equation}
Summing up, we proved that for any $s\in(0,\delta)$ we have $0<s^{\frac{m+2}{4}}f(s)<\frac{\epsilon C(m)}{2(2n+2-m)}$ and this concludes the proof.
\end{proof}

\begin{proposizione}
The following are equivalent:
\begin{itemize}
\item[(i)]$\lim_{s\to 0}s^{\frac{m+2}{4}}f(s)=0$,
\item[(ii)]
for any $\alpha>0$ there exists an $R(\alpha)>0$ such that if $R>R(\alpha)$, then $\supp(\mu)\setminus B_{R}(0)\subseteq\{z:\lvert    z_H\rvert\leq \alpha\lVert z\rVert\}$.
\end{itemize}
\label{propette}
\end{proposizione}

\begin{proof}
Suppose (ii) fails. Then, there exists an $\alpha^\prime>0$ such that for any $j\in \N$ there exists a $y_j\in\supp(\mu)\setminus B_{j}(0)$ for which $\lvert (y_j)_H\rvert\geq \alpha^\prime\lVert y_j\rVert$. We prove that along the sequence $s_j:=\lvert (y_j)_H\rvert^{-4}$, the function $s^{\frac{m+2}{4}}f(s)$ is bounded away from $0$, which contradicts (i):
\begin{equation}
\begin{split}
s_j^\frac{m+2}{4}\int \lvert   z_H\rvert^2&e^{-s_j\lVert z\rVert^4}d\mu(z)\geq s_j^\frac{m+2}{4}\frac{\lvert(y_j)_H\rvert^2}{4}e^{-s_j\left(3\alpha^\prime\lvert (y_j)_H\rvert/2\right)^4}\mu\left(B_{\lvert(y_j)_H\rvert/2}(y_j)\right)\\
=&\frac{ s_j^\frac{m+2}{4}}{2^{m+2}} e^{-s_j(3\alpha^\prime\lvert (y_j)_H\rvert/2)^4}\lvert(y_j)_H\rvert^{m+2}\geq \frac{ e^{-\frac{81(\alpha^\prime)^4}{16}}}{2^{m+2}},
\nonumber
\end{split}
\end{equation}
where we used the fact that for any $z\in B_{\lvert (y_j)_H\rvert/2}(y_j)$ one has that
$\lvert (y_j)_H\rvert/2\leq\lvert   z_H\rvert$ and $\lVert z\rVert\leq 3\alpha^\prime\lvert (y_j)_H\rvert/2$. Viceversa, suppose (ii) holds. This implies that for any $\alpha>0$ that:
\begin{equation}
\begin{split}
s^\frac{m+2}{4}\int \lvert   z_H\rvert^2e^{-s\lVert z\rVert^4}d\mu(z)
\leq s^\frac{m+2}{4}\int_{B_{R(\alpha)}(0)} &\lVert z\rVert^2e^{-s\lVert z\rVert^4}d\mu(z)\\
+&s^\frac{m+2}{4}\alpha^2\int_{B^c_{R(\alpha)}(0)} \lVert z\rVert^2e^{-s\lVert z\rVert^4}d\mu(z),
\nonumber
\end{split}
\end{equation}
The above computation, Proposition \ref{prop5} and Corollary \ref{prop1} imply that:
\begin{equation}
\begin{split}
&s^\frac{m+2}{4}\int \lvert   z_H\rvert^2e^{-s\lVert z\rVert^4}d\mu(z)
\leq m\int_0^{s^\frac{1}{4}R(\alpha)} t^{m+1}e^{-t^4}dr+\alpha^2\frac{m}{4}\Gamma\left(\frac{m+2}{4}\right).
\end{split}
\nonumber
\end{equation}
Therefore:
\begin{equation}
\begin{split}
0\leq\limsup_{s\to 0} s^\frac{m+2}{4}f(s)\leq&\limsup_{s\to 0} m \int_0^{s^\frac{1}{4}R(\alpha)} t^{m+1}e^{-t^4}dr+\alpha^2\frac{m}{4}\Gamma\left(\frac{m+2}{4}\right)\\
\leq&\alpha^2\frac{m}{4}\Gamma\left(\frac{m+2}{4}\right).
\end{split}
\nonumber
\end{equation} 
The arbitrariness of $\alpha$ concludes the proof.
\end{proof}

\begin{proposizione}\label{propvert}
If $\supp(\mu)\subseteq \mathcal{V}$, then $\mu=\mathcal{C}^2\llcorner\mathcal{V}$.
\end{proposizione}

\begin{proof}
Since $\supp(\mu)\subseteq \mathcal{V}$, then $\mu(B_r(z))=\mu(B_r(z)\cap \mathcal{V})=r^m$ for any $z\in\supp(\mu)$ and any $r>0$. Note that:
$$B_r(z)\cap \mathcal{V}=\{(0,s)\in\R^{2n}\times\R:\lvert s-  z_T\rvert<r^2\}=U_{r^2}(z)\cap \mathcal{V},$$
where \label{euclidean}as usual $U_{r^2}(z)$ denotes the Euclidean ball of radius $r^2$ and center $z$. This implies that
$\mu(U_r(z))=r^{\frac{m}{2}}$ and hence $\mu$ is a $m/2$-uniform measure with respect to Euclidean balls and whose support is contained in the line $\mathcal{V}$. Marstrand's theorem implies that $m/2$ must be an integer and since $\mathcal{V}$ is $1$-dimensional, we deduce by differentiation that $m/2$ is either $0$ or $1$. As we excluded by hypothesis $m=0$, thanks to the classification of $1$-uniform measures in $\R^n$ provided by \cite{Preiss1987GeometryDensities}, we deduce that $\mu=\frac{1}{2}\mathcal{H}_{eu}^1\llcorner \mathcal{V}$. Since $m=2$ and $\supp(\mu)=\mathcal{V}$, the result follows by Proposition \ref{supportoK}.
\end{proof} 

\begin{proposizione}
Suppose that for any $\alpha>0$ there exists an $R(\alpha)>0$ such that if $R>R(\alpha)$, then:
\begin{equation}
\supp(\mu)\setminus B_{R}(0)\subseteq\{z:\lvert    z_H\rvert\leq \alpha\lVert z\rVert\}.
\nonumber
\end{equation}
Then $m=2$ and $\Tan_2(\mu,\infty)=\{\mathcal{C}^2\llcorner\mathcal{V}\}$.
\end{proposizione}

\begin{proof}
Let $\eta\in\mathcal{C}^\infty_c(\R^{2n+1})$ be such that $0\leq\eta\leq1$, $\eta(B_1(0))=1$ and $\eta(B_2(0)^c)=0$. For $r>0$ define:
\begin{equation}
\eta_R(z):=\eta\left(D_{1/R}(z)\right).\nonumber
\end{equation}
Let $\nu\in\Tan(\mu,\infty)$ and $\{\lambda_l\}_{l\in\N}$ be such that $\lambda_l\to\infty$ and $\lambda_l^{-m}\mu_{0,\lambda_l}\rightharpoonup \nu$ and note that:
\begin{equation}
\begin{split}
\int \eta_R(z)\lvert   z_H\rvert^2e^{-\lVert z\rVert^2}d\nu(z)=&\lim_{l\to\infty}\int \eta_R(z)\lvert   z_H\rvert^2e^{-\lVert z\rVert^4}\frac{d\mu_{0,\lambda_l}(z)}{\lambda_l^m}\\
\leq&\lim_{l\to\infty}\int \big\lvert( D_{1/\lambda_l}(z))_H\big\rvert^2e^{-\lvert D_{1/\lambda_l}(z)\rvert^4}\frac{d\mu(z)}{\lambda_l^m}\\
=&\lim_{l\to\infty}\lambda_l^{-(m+2)}\int \lvert   z_H\rvert^2e^{-\lVert z\rVert^4/\lambda_l^4}d\mu(z)=0.
\nonumber
\end{split}
\end{equation}
The last equality in the above computation is provided by Proposition \ref{propette} and our hypothesis on $\mu$. The arbitrariness of $R>0$ and dominated convergence theorem imply:
\begin{equation}
\int \lvert   z_H\rvert^2e^{-\lVert z\rVert^2}d\nu(z)=0.\nonumber
\end{equation}
Proposition \ref{prop12} implies that $\supp(\nu)\subseteq \mathcal{V}$ and Proposition \ref{propvert} implies that $\nu=\mathcal{C}^2\llcorner\mathcal{V}$. Therefore by Proposition \ref{propup} $\mu$ is a $2$-uniform measure.
\end{proof}

As an immediate concequence we get:

\begin{corollario}\label{mainprimo}
Let $\mu\in\mathcal{U}(m)$. If $\lim_{s\to 0}\text{Tr}(\mathcal{Q}(s))=0$, then $m=2$ and:
$$\Tan_2(\mu,\infty)=\{\mathcal{C}^2\llcorner\mathcal{V}\}.$$
In particular for any $m\in\{1,\ldots,2n+1\}\setminus\{2\}$, there exist $b\in\R^{2n}$, $\mathcal{T}\in\R$ and $\mathcal{Q}\in\text{Sym}(n)$ with $\text{Tr}(\mathcal{Q})\neq 0$ such that:
$$\supp(\mu)\subseteq\mathbb{K}(b,\mathcal{Q},\mathcal{T}).$$
\end{corollario}

\section{\texorpdfstring{$(2n+1)$}{Lg}-uniform measures have no holes}
\label{buchi}
Let $\mu$ be a $(2n+1)$-uniform measure (which should be considered fixed throughout the section) and $\mathbb{K}(b,\mathcal{Q},\mathcal{T})$ be the non-degenerate quadric in which $\supp(\mu)$ is contained. The existence of such a quadric has been shown in Section \ref{sezione1}. What is left to understand is whether $\supp(\mu)$ is some kind of very irregular set \emph{inside} $\mathbb{K}(b,\mathcal{Q},\mathcal{T})$ or if it has a better structure. To state our main result we need some notation. Let $F:\R^{2n+1}\to \R$ be the quadratic polynomial:
\begin{equation}
    F(z):=\langle b,z_H\rangle+\langle z_H, \mathcal{Q} z_H\rangle+\mathcal{T}z_T,
    \label{numbero1}
\end{equation}
whose zero-set is the quadric $\mathbb{K}(b,\mathcal{Q},\mathcal{T})$.
We define the set of singular points of $\mathbb{K}(b,\mathcal{Q},\mathcal{T})$ as:
\begin{equation}
    \Sigma(F):=\{x\in\mathbb{K}(b,\mathcal{Q},\mathcal{T}): b+2(\mathcal{Q}-\mathcal{T}J)x_H= 0\}.
    \label{numbero2}
\end{equation}
Usually $\Sigma(F)$ is called \emph{characteristic set} if $\mathcal{T}\neq 0$ and \emph{singular set} if $\mathcal{T}=0$. We should not bother ourselves with such distinctions, and regard $\Sigma(F)$ as the set of points where $\mathbb{K}(b,\mathcal{Q},\mathcal{T})$ behaves like a cone tip.
The following theorem is the main result of this section and it should be regarded as an analogue of \cite[Proposition 4.3]{Kowalski1986Besicovitch-typeSubmanifolds}. We show not only that $\supp(\mu)$ is not a fractal inside $\mathbb{K}(b,\mathcal{Q},\mathcal{T})$ but also that it can be viewed as a quadratic surface with no holes:

\begin{teorema}\label{duppy}
Either the measure $\mu$ is flat or its support is (the closure of) a union of connected components of $\mathbb{K}(b,\mathcal{Q},\mathcal{T})\setminus\Sigma(F)$.
\end{teorema}

We give now a short account of the content for each subsection. Subsection \ref{reg} is split in two parts. The main result of the first part is Proposition \ref{verticalsamoa2}, where we show that everywhere outside $\Sigma(F)$, $(2n+1)$-uniform measures have flat blowups. Therefore, even though in principle $\mu$ may have holes, they are not inherited by tangents and therefore locally $\supp(\mu)$ behaves exactly as the whole surface $\mathbb{K}(b,\mathcal{Q},\mathcal{T})$.
The second part of Subsection \ref{reg} is devoted to the proof of Proposition \ref{verticale}, where we show that $\mathcal{T}$ is an invariant for $\mu$: if $\mathcal{T}=0$ then for any other quadric $\mathbb{K}(b^\prime,\mathcal{Q}^\prime,\mathcal{T}^\prime)$ containing $\supp(\mu)$ we have $\mathcal{T}^\prime=0$. This implies that there are two types of $(2n+1)$-uniform measures which are qualitatively different. If $\mathcal{T}=0$, Theorem \ref{duppy} implies that $\supp(\mu)$ is vertically ruled, and hence by translations with elements of the center, while if $\mathcal{T}\neq 0$ then $\supp(\mu)$ coincides with $\mathbb{K}(b,\mathcal{Q},\mathcal{T})$, which is a $t$-graph,. For a precise statement, see Proposition \ref{spt1}.

Subsection \ref{nozero} and Subsection \ref{zero} are completely devoted to the proof of Theorem \ref{duppy} in the cases $\mathcal{T}\neq 0$ and $\mathcal{T}=0$, respectively.
The idea behind the proof is the same in both situations: pick a point $x\in\supp(\mu)\setminus \Sigma(F)$ for which for any $r>0$ we have:
$$B_r(x)\cap\supp(\mu)^c\cap\mathbb{K}(b,\mathcal{Q},\mathcal{T})\neq \emptyset,$$
and show that the holes in the support pass to the blowup, contradicting the mentioned fact that at points of $\supp(\mu)$ outside $\Sigma(F)$ the tangent measure are planes.

\subsection{Regularity of \texorpdfstring{$(2n+1)$}{Lg}-uniform measures}
\label{reg}
The first part of this subsection is devoted to the study of the local properties of $\mu$. First of all, we introduce the definition of vertical hyperplane and flat measure:

\begin{definizione}\label{plano}
For any $\mathfrak{n}\in\R^{2n}$ we define:
$$V(\mathfrak{n}):=\{z\in\HH^n: \langle \mathfrak{n},   z_H\rangle=0\},$$
and we say that $V(\mathfrak{n})$ is the \emph{vertical hyperplane} orthogonal to $\mathfrak{n}$.
A $(2n+1)$-uniform measure $\nu$ is said to be a \emph{flat measure}, or simply \emph{flat}, if:
$$\nu=\mathcal{C}^{2n+1}\llcorner{V(\mathfrak{n})}\text{ for some $\mathfrak{n}\in\R^{2n}$.}$$
Throughout the paper the family of all flat measures will always be denoted by $\mathfrak{M}(2n+1)$.
\end{definizione}

\begin{osservazione}\label{uniformityflat}
The measure $\mathcal{C}^{2n+1}\llcorner{V(\mathfrak{n})}$ is a Haar measure of $V(\mathfrak{n})$ and it is $(2n+1)$-uniform.
\end{osservazione}

The following proposition shows that $(2n+1)$-uniform measure with support contained in a vertical hyperplane are flat. It is an adaptation of \cite[Remark 3.14]{DeLellis2008RectifiableMeasures}.

\begin{proposizione}\label{verticalsamoa}
Suppose that $\mu$ is a $(2n+1)$-uniform measure for which there exists an $\mathfrak{n}\in\R^{2n}$ for which $\supp(\mu)\subseteq V(\mathfrak{n})$. Then
$\mu=\mathcal{C}^{2n+1}\llcorner {V(\mathfrak{n})}$.
\end{proposizione}

\begin{proof}
By Proposition \ref{supportoK}, for any $x\in\supp(\mu)$ and any $r>0$ we have:
$$\mathcal{C}^{2n+1}\llcorner {\supp(\mu)} (B_r(x))=\mu(B_r(x))=r^{2n+1}.$$
Therefore, thanks to \cite[Proposition 2.11]{antonelli2020rectifiable} for any $r>0$ we have:
$$\mathcal{C}^{2n+1}\llcorner \supp(\mu) (B_r(0))=r^{2n+1}=\mathcal{C}^{2n+1}\llcorner V(\mathfrak{n}) (B_r(0)).$$
Since $\supp(\mu)$ is closed in $V(\mathfrak{n})$ if by contradiction $\text{supp}(\mu)\neq V(\mathfrak{n})$, there would exists a $p\in V(\mathfrak{n})$ and an $r_0>0$ such that $B_{r_0}(p)\cap \text{supp} (\mu)=\emptyset$. This however is impossible since it would imply:
\begin{equation}
    \begin{split}
      \mathcal{C}^{2n+1}\llcorner V(\mathfrak{n})&(B_{2(\lVert p\rVert +r_0)}(0))\\
        \geq& \mathcal{C}^{2n+1}(B_{2(\lVert p\rVert +r_0)}(0)\cap \text{supp}(\mu))+\mathcal{C}^{2n+1}(B_{r_0}(p)\cap V(\mathfrak{n}))\\
        >&\mu(B_{2(\lVert p\rVert +r_0)}(0)),
        \nonumber
    \end{split}
\end{equation}
where the last inequality above comes from the fact that $\mathcal{C}^{2n+1}(B_{r_0}(p)\cap V(\mathfrak{n}))>0$ since $\mathcal{C}^{2n+1}\llcorner V(\mathfrak{n})$ is a Haar measure of $V(\mathfrak{n})$ by Remark  \ref{uniformityflat}.
\end{proof}

\begin{proposizione}\label{verticalsamoa2}
For any $x\in\supp(\mu)\setminus \Sigma(F)$ we have:
$$\Tan_{2n+1}(\mu,x)=\{\mathcal{C}^{2n+1}\llcorner { V(\mathfrak{n}(x))}\},$$
where:
$$\mathfrak{n}(x):=b+2(\mathcal{Q}-\mathcal{T}J)x_H.$$
\end{proposizione}

\begin{proof}
By Proposition \ref{uniformup}, for any $x\in\supp(\mu)$ the set $\Tan_{2n+1}(\mu,x)$ is non-empty and it is contained in $\mathcal{U}(2n+1)$. Pick any $\nu\in\Tan_{2n+1}(\mu,x)$ and recall that by definition of tangent, there exist an infinitesimal sequence $\{r_i\}_{i\in\N}$ such that $\mu_{x,r_i}/r_i^{2n+1}\rightharpoonup \nu$. Therefore, Proposition \ref{propspt1} implies that for any $y\in\supp(\nu)$ there exists a sequence $\{x_i\}\subseteq \supp(\mu)$, for which $D_{1/r_i}(x^{-1}x_i)\to y$.
Defined $y_i:=D_{1/r_i}(x^{-1}x_i)$, we have that $x_i=xD_{r_i}(y_i)$ and thus for any $i\in\N$:
\begin{equation}
\begin{split}
0=&\langle b,(x_i)_H\rangle+\langle (x_i)_H,\mathcal{Q}(x_i)_H\rangle+\mathcal{T}  (x_i)_T\\
=&r_i\langle b,(y_i)_H\rangle+r_i\langle 2(\mathcal{Q}-\mathcal{T}J) x_H,(y_i)_H\rangle+ r_i^2\langle (y_i)_H,\mathcal{Q}(y_i)_H\rangle+r_i^2\mathcal{T}(y_i)_T.
\nonumber
\end{split}
\end{equation}
Since $y_i\to y$, dividing by $r_i$ and taking the limit as $i\to \infty$ we deduce:
$$0=\langle b+ 2(\mathcal{Q}-\mathcal{T}J) x_H,y_H\rangle.$$
This implies that for any $\nu\in\Tan_{2n+1}(\mu,x)$, we have $\supp(\nu)\subseteq V(\mathfrak{n}(x))$ and thus the claim follows by Proposition \ref{verticalsamoa}.
\end{proof}

In the upcoming proposition we show that $\mathcal{U}(2n+1)$ is splitted in two families that are characterized by the coefficient $\mathcal{T}$.

\begin{proposizione}
Suppose that there are $\overline{b}\in\R^{2n}$, $\overline{\mathcal{Q}}\in\mathrm{Sym}(2n)\setminus\{0\}$ and $\overline{\mathcal{T}}\in\R$ such that:
$$\supp(\mu)\subseteq \mathbb{K}(\overline{b},\overline{\mathcal{Q}},\overline{\mathcal{T}}).$$
Then $\mathcal{T}=0$ if and only if $\overline{\mathcal{T}}=0$.
\label{verticale}
\end{proposizione}

\begin{proof}
Assume by contradiction that $\mathcal{T}\neq 0$ and $\overline{\mathcal{T}}= 0$. Since:
$$\supp(\mu)\subseteq S:=\mathbb{K}(b,\mathcal{Q},\mathcal{T})\cap \mathbb{K}(\overline{b},\overline{\mathcal{Q}},0),$$
thanks to Proposition \ref{supportoK} we infer that $\mu(B_r(0))=\mathcal{C}^{2n+1}\llcorner {\supp(\mu)}(B_r(0))\leq \mathcal{C}^{2n+1}\llcorner S(B_r(0))$.
In order to conclude the proof of the proposition, let us note that:
\begin{equation}
\begin{split}
    \pi_H(S)\subseteq\big\{x\in \R^{2n}:\text{ there is }&\text{a $t\in \R$ such that}\\
    \text{the point} p:=&(x,t)\text{ satisfies }\langle p_H,\overline{\mathcal{Q}} p_H\rangle+\langle \overline{b},p_H\rangle+0\cdot t=0\big\}\\
    = \{x\in \R^{2n}: \langle x,\overline{\mathcal{Q}} x\rangle+&\langle \overline{b},x\rangle=0\}.
    \nonumber
\end{split}
\end{equation}
Since the polynomial $\langle x,\overline{\mathcal{Q}}x\rangle+\langle \overline{b},x\rangle$ is non-constant on $\R^{2n}$ by hypothesis, its level sets are analytical subvarieties and thus they have (Euclidean) Hausdorff dimension smaller than $2n-1$, see for instance \cite[3.4.8]{Federer1996GeometricTheory}. This in particular implies that the set $\pi_H(S)$ has null $\mathcal{L}^{2n}$-measure in $\R^{2n}$.
Finally, Proposition \ref{rapperhor} implies that $\mathcal{C}^{2n+1}\llcorner{\mathbb{K}(b,\mathcal{Q},\mathcal{T})}(S\cap B_r(0))=0$ for any $r>0$, which contradicts the fact that $0\in\supp(\mu)$.
\end{proof}

\begin{definizione}\label{terminologyhor}
If there exist $b\in\R^{2n}$ and $\mathcal{Q}\in\mathrm{Sym}(2n)\setminus\{0\}$ such that:
$$\supp(\mu)\subseteq \mathbb{K}(b,\mathcal{Q},0),$$
then $\mu$ is said to be a \emph{vertical uniform measure}.
If such $b$ and $\mathcal{Q}$ do not exist $\mu$ is said to be a \emph{horizontal uniform measure}.
\end{definizione}

\subsection{Structure of the support of horizontal uniform measures}
\label{nozero}
In this subsection we prove Theorem \ref{duppy} in case $\mu$ is a horizontal uniform measure and therefore \textbf{throughout this subsection we assume }$\mathcal{T}\neq0$. Let $f:\R^{2n}\to \R$ be the smooth function defined as:
\begin{equation}
f(h):=-\frac{\langle h,\mathcal{Q}h\rangle+\langle b,h\rangle}{\mathcal{T}}.\label{numbero14}
\end{equation}
Since $\supp(\mu)\subseteq \text{gr}(f)$, it is immediate to see that $\Sigma(F)=\{(h,f(h))\in\R^{2n+1}:h\in\Sigma(f)\}$,\label{sigmino}
where $\Sigma(f)$ is the set of \emph{characteristic points} of $f$:
\begin{equation}
    \Sigma(f):=\{h\in\R^{2n}:b+2(\mathcal{Q}-\mathcal{T}J)h=0\}.
    \label{char}
\end{equation}

Let us take a step back and explain why the set $\Sigma(f)$ turns up in our computations. Since in this subsection the measure $\mu$ is supposed to be horizontal, see Definition \ref{terminologyhor}, any quadric containing $\mu$ must be representable as the graph of $2$nd degree polynomial, taking values in the vertical axis. These quadratic graphs are obviously (Euclidean) smooth submanifolds of $\R^{2n+1}$ and the (Euclidean) normal to their graphs is well defined \emph{everywhere}. However, these embedded submanifolds from the perspective of $\HH^n$ may not be so smooth.

To see why this is the case, it is sufficient to note that $\Sigma(F)$ is the set 
of points of $\mathbb{K}(b,\mathcal{Q},\mathcal{T})$ where the Euclidean normal to $\mathbb{K}(b,\mathcal{Q},\mathcal{T})$ is orthogonal to the \emph{horizontal distribution}, i.e.: the distribution of $2n$-dimensional planes spanned by the vector fields $X_1,\ldots,X_n,Y_1,\ldots,Y_n$ that are defined at the beginning of Appendix \ref{appeA1}. At these points it is not hard to see that the intrinsic blowups of the measure are not flat. An immediate example of this phenomenon can be seen with the function $f(x)=\lvert x\rvert^2$ at the point $x=0$, where the Euclidean blowup is the horizontal plane, while the Heisenberg blowup of the graph of $f$ coincides with $\text{gr}(f)$ itself.

Thanks to Proposition \ref{quaddica}, if $n>1$ then $\Sigma(f)$ cannot disconnect $\R^{2n}$, as $\Sigma(f)$ is an affine space of dimension less than or equal to $n$. However if $n=1$ there might be the case in which $\R^2\setminus \Sigma(f)$ is split in two connected components, which we will denote in the following by $C_i$, with $i=1,2$. The following proposition is Theorem \ref{duppy} in case $\mu$ is a horizontal measure. The idea behind the proof is to tuck cylinders inside holes of $\supp(\mu)$ in such a way that they are also tangent to $\supp(\mu)$ is some point outside $\Sigma(F)$ and to show with some careful computations that the tangents to $\mu$ at the point of tangency cannot be flat.

\begin{proposizione}\label{spt1} If $n>1$ then $\supp(\mu)= \mathbb{K}(b,\mathcal{Q},\mathcal{T})$.
If $n=1$ we have two cases:
\begin{itemize}
    \item[(i)] if $\text{dim}(\Sigma(f))=0$, then $\supp(\mu)=\mathbb{K}(b,\mathcal{Q},\mathcal{T})$,
    \item[(ii)] if $\text{dim}(\Sigma(f))=1$, then either $\supp(\mu)=\mathbb{K}(b,\mathcal{Q},\mathcal{T})$ or it coincides with the closure of the graph of $f\rvert_{C_1}$ or $f\rvert_{C_2}$.
\end{itemize}
\end{proposizione}

\begin{proof}
Since $\supp(\mu)$ is a closed set and it is contained in $\mathbb{K}(b,\mathcal{Q},\mathcal{T})$, then the set $S:=\{p\in\R^{2n}:(p,f(p))\in\supp(\mu)\}$ is closed in $\R^{2n}$. 
This implies that for any $y\in\mathbb{K}(b,\mathcal{Q},\mathcal{T})$, there exists a $z(y)\in S$ (possibly coinciding with $y_H$) such that:
$$\lvert z(y)-y_H\rvert=\dist(y_H,S).$$
As a consequence the (possibly empty) cylinder:
$$c_{\lvert z(y)-y_H\rvert}(y):=\{(x,t)\in\R^{2n}\times\R:\lvert x-y_H\rvert< \lvert z(y)-y_H\rvert \},$$
does not intersect $\supp(\mu)$.

We claim that for any $y\in\mathbb{K}(b,\mathcal{Q},\mathcal{T})\setminus (\Sigma(F)\cup\supp(\mu))$ the tangency point $z=z(y)$ is contained in $\Sigma(f)$.

Let us prove the proposition assuming that the above claim holds. If $\Sigma(f)=\emptyset$, then $\supp(\mu)=\mathbb{K}(b,\mathcal{Q},\mathcal{T})$ (as $z(y)$ cannot exist). Hence, without loss of generality we can assume $\Sigma(f)\neq \emptyset$. 

By Proposition \ref{quaddica}, if $n>1$ then $\Omega:=\R^{2n}\setminus \Sigma(f)$ is a connected open set and $S$ is relatively closed inside it. Therefore, if $S$ is also relatively open in $\Omega$, by connectedness we infer that  $S\cap\Omega=\Omega$.
Therefore, assume by contradiction that there exists a $p\in S\cap\Omega$ such that for any $r>0$ there exists $q_r\in S^c\cap U_r(p)$, where $U_r(x)$ is the Euclidean ball in $\R^{2n}$. This would imply for a sufficiently small $r>0$ that:
$$\dist(S,q_r)<r<\dist(\Sigma(f),q_r),$$
contradicting the fact that if $z\in S$ satisfies $\lvert z-q_r\rvert=\dist(S,q_r)$ then $z\in\Sigma(f)$.
Note that the same argument works in the remaining cases in which $n=1$. If $\text{dim}(\Sigma(f))=1$, where we just apply the reasoning above to the connected components $C_1$ and $C_2$.

Let us now prove the claim. In order to ease the notation in the following, we define $\zeta:=(z,f(z))\in\supp(\mu)$. If by contradiction $z\not\in\Sigma(f)$, then $\zeta\not\in\Sigma(F)$ and thus Proposition \ref{verticalsamoa2} implies that:
$$\Tan_{2n+1}(\mu,\zeta)=\{\mathcal{C}^{2n+1}\llcorner V(\mathfrak{n}(\zeta))\},$$
where $\mathfrak{n}(\zeta)=-(b+2(\mathcal{Q}-\mathcal{T}J)z)/\mathcal{T}$. Moreover, Proposition \ref{propspt1} implies that if $w\in V(\mathfrak{n}(\zeta))$ and $r_i\to 0$, there exists a sequence $\{v_i\}_{i\in\N}\subseteq \supp(\mu)$ such that:
\begin{equation}
    w_i:=D_{1/r_i}(\zeta^{-1}v_i)\to w.
  \label{equaz}
\end{equation}
Since $v_i\not\in c_{\lvert z-y_H\rvert}(y)$ by construction, we deduce that:
$$\lvert z-y_H\rvert\leq\lvert (\zeta D_{r_i}(w_i))_H-y_H\rvert=\lvert z+r_i w_H-y_H\rvert,$$
which reduces, squaring both sides of the inequality, to:
\begin{equation}
\begin{split}
0\leq 2\langle z-y_H,  (w_i)_H\rangle+r_i\lvert   (w_i)_H\rvert^2.
\end{split}
\label{eq:eq:34}
\end{equation}
Taking the limit as $i\to\infty$ we deduce that for any $\nu\in \Tan_{2n+1}(\mu,\zeta)$ we have:
$$\supp(\nu)\subseteq V(\mathfrak{n}(\zeta))\cap \{w\in\R^{2n+1}:\langle z-y_H,  w_H\rangle\geq 0\}.$$
If $z-y_H$ is not parallel to $\mathfrak{n}(\zeta)$ this would contradict Proposition \ref{verticalsamoa2} as $\supp(\nu)$ would be contained in a proper subset of $V(\mathfrak{n}(\zeta))$.
This implies that there exists $\lambda\neq 0$ for which $z-y_H=\lambda\mathfrak{n}(\zeta)$.

Furthermore, from the convergence \eqref{equaz}
, we deduce that:
\begin{itemize}
    \item[($\alpha$)]$(v_i)_H=z+r_i  w_H+R_i,$ with $\lvert R_i\rvert=o(r_i)$,
    \item[($\beta$)]$r_i^2w_T+o(r_i^2)=(v_i)_T-\zeta_T -2\langle z,J(v_i)_H\rangle.$
\end{itemize}
Rearranging ($\beta$), we deduce:
\begin{equation}
\begin{split}
(v_i)_T=&\zeta_T+2\langle z,J(v_i)_H\rangle+r_i^2w_T+o(r_i^2)\\
=&  \zeta_T+2r_i\langle z,Jw_H\rangle+\langle z,JR_i\rangle+r_i^2w_T+o(r_i^2),
\label{eq:eq:32}
\end{split}
\end{equation}
where in the second identity we used ($\alpha$) to substitute $(v_i)_H$ and we also exploited the fact that $\langle z,Jz\rangle=0$.
On the other hand, since by assumption $\supp(\mu)$ is contained in $\text{gr}(f)$, we have that:
$$\mathcal{T}(v_i)_T=-\langle (v_i)_H,\mathcal{Q} (v_i)_H\rangle-\langle b,(v_i)_H\rangle.$$
Therefore, using ($\alpha$) in the above formula to substitute $(v_i)_H$, we deduce:
\begin{equation}
    \begin{split}
        \mathcal{T}(v_i)_T
        =&-\langle z+r_i  w_H+R_i,\mathcal{Q} [z+r_i  w_H+R_i]\rangle-\langle b,z+r_i  w_H+R_i\rangle\\
        =&(-\langle z,\mathcal{Q}z\rangle-\langle b,z \rangle)-r_i\langle b+2\mathcal{Q}z,w_H\rangle\\
        &\qquad\qquad\qquad\qquad-\langle b+2\mathcal{Q}z,R_i\rangle-r_i^2\langle   w_H,Q  w_H\rangle+o(r_i^2).
    \end{split}
    \label{eq:eq:30}
\end{equation}
Furthermore, since by definition we have $\zeta_T=f(z)$, identity \eqref{eq:eq:30} simplifies to:
\begin{equation}
    \mathcal{T}(v_i)_T=\mathcal{T}\zeta_T-r_i\langle b+2\mathcal{Q}z,w_H\rangle-\langle b+2\mathcal{Q}z,R_i\rangle-r_i^2\langle   w_H,Q  w_H\rangle+o(r_i^2),
    \label{eq:eq:31}
\end{equation}
Since the left hand sides of \eqref{eq:eq:32} and \eqref{eq:eq:31} are equal up to the multiplication by $\mathcal{T}$, we can put the two formulas together to infer:
\begin{equation}
\begin{split}
        &\qquad\qquad\zeta_T+2r_i\langle z,Jw_H\rangle+\langle z,JR_i\rangle+r_i^2w_T+o(r_i^2)\\
    =&\zeta_T-r_i\frac{\langle b+2\mathcal{Q}z,w_H\rangle}{\mathcal{T}}-\frac{\langle b+2\mathcal{Q}z,R_i\rangle}{\mathcal{T}}-r_i^2\frac{\langle   w_H,Q  w_H\rangle}{\mathcal{T}}+o(r_i^2).
    \nonumber
\end{split}
\end{equation}
Simplifying the above equation and pushing some of the terms on the left-hand side to the right hand side, we have:
\begin{equation}
\begin{split}
    r_i^2w_T=&-r_i\frac{\langle b+2(\mathcal{Q}-\mathcal{T} J)z,w_H\rangle}{\mathcal{T}}-\frac{\langle b+2(\mathcal{Q}-\mathcal{T} J)z,R_i\rangle}{\mathcal{T}}\\&\quad\,\,\,\,\,\qquad\qquad\qquad\qquad\qquad-r_i^2\frac{\langle   w_H,Q  w_H\rangle}{\mathcal{T}}+o(r_i^2)\\
    =&\langle\mathfrak{n}(\zeta),R_i\rangle-r_i^2\frac{\langle   w_H,Q  w_H\rangle}{\mathcal{T}}+o(r_i^2),
\end{split}
\nonumber
\end{equation}
where the last line comes from the fact that we defined $\mathfrak{n}(\zeta):=-(b+2(\mathcal{Q}-\mathcal{T} J)z)/\mathcal{T}$ and that thanks to our choice of $\zeta$, as argued at the beginning of this proof, we have $w\in V(\mathfrak{n}(\zeta))$ and thus $\langle b+2(\mathcal{Q}-\mathcal{T} J)z,w_H\rangle=0$.
Therefore, recollecting terms in the above equality, we deduce that:
\begin{equation}
    \langle\mathfrak{n}(\zeta),R_i\rangle=r_i^2\left(\frac{\langle   w_H,\mathcal{Q}w_H\rangle}{\mathcal{T}}+w_T\right)+o(r_i^2).
    \label{eq11}
\end{equation}
Thanks to the definition of $w_i$ and identity ($\alpha$) we have $(w_i)_H=w_H+R_i/r_i$ and this, in  conjunction with \eqref{eq:eq:34} implies that:
\begin{equation}
\begin{split}
    0\leq &2\langle z-y_H,(w_i)_H\rangle+r_i\lvert (w_i)_H\rvert^2=2\lambda\langle  \mathfrak{n}(\zeta),w_H+R_i/r_i\rangle+r_i\Big\lvert   w_H+\frac{R_i}{r_i}\Big\rvert^2,
    \nonumber
\end{split}
\end{equation}
where in the last identity we used the fact that $z-y_H=\lambda \mathfrak{n}(\zeta)$.
Furthermore, since $w\in V(\mathfrak{n}(\zeta))$, the above inequality can be further simplified to:
\begin{equation}
\begin{split}
    0\leq 2\lambda\langle  \mathfrak{n}(\zeta),R_i/r_i\rangle+r_i\Big\lvert   w_H+\frac{R_i}{r_i}\Big\rvert^2.
    \label{eq:eq:36}
\end{split}
\end{equation}
Dividing both sides of \eqref{eq11} by $r_i$ and plugging the obtained expression into \eqref{eq:eq:36} we infer:
\begin{equation}
    \begin{split}
        0\leq& 2\lambda\langle  \mathfrak{n}(\zeta),R_i/r_i\rangle+r_i\Big\lvert  w_H+\frac{R_i}{r_i}\Big\rvert^2\\
        =&2\lambda r_i\left(\frac{\langle   w_H,\mathcal{Q}w_H\rangle}{\mathcal{T}}+w_T\right)+o(r_i)+r_i\Big\lvert   w_H+\frac{R_i}{r_i}\Big\rvert^2.
        \label{eq:numiero1}
    \end{split}
\end{equation}
Finally, dividing  \eqref{eq:numiero1} by $r_i$ and sending $i\to\infty$, we get:
\begin{equation}
\begin{split}
    -\lambda w_T\leq\lambda\frac{\langle   w_H,\mathcal{Q}w_H\rangle}{\mathcal{T}}+\frac{\lvert   w_H\rvert^2}{2},
    \nonumber
\end{split}
\end{equation}
which constitutes a non-trivial bound on $w_T$ and this contradicts Proposition \ref{verticalsamoa2}. This finally shows that $z\in\Sigma(f)$.
\end{proof}

\subsection{Structure of the support of vertical uniform measures}
\label{zero}
In this subsection we prove Theorem \ref{duppy} in case $\mu$ is a vertical uniform measure and by Definition \ref{terminologyhor}, \textbf{ there are $b\in\R^{2n}$ and $\mathcal{Q}\in\mathrm{Sym}(2n)\setminus\{0\}$ such that
$\supp(\mu)\subseteq \mathbb{K}(b,\mathcal{Q},0)$, that from now on we shall consider fixed}.

The first step towards our goal is the following technical proposition that establishes the relation between the center of an Euclidean ball tangent to $\supp(\mu)$ and the point of tangency.

\begin{proposizione}
Let $y\in\mathbb{K}(b,\mathcal{Q},0)\setminus \supp(\mu)$ and $\zeta\in\supp(\mu)\setminus \Sigma(F)$ be such that $\lvert \zeta-y\rvert=\dist(y,\supp(\mu))$.
Then:
\begin{itemize}
\item[(i)] $\zeta_T=  y_T$,
\item[(ii)] there exists a $\lambda\neq 0$ such that $\lambda(\zeta_H-y_H)=2\mathcal{Q}\zeta_H + b$.
\end{itemize}
\label{verticalfun}
\end{proposizione}

\begin{proof}
Since $\zeta\not\in \Sigma(F)$, Proposition \ref{verticalsamoa2} implies that:
\begin{equation}
    \Tan_{2n+1}(\mu,\zeta)=\{\mathcal{C}^{2n+1}\llcorner V(2\mathcal{Q}\zeta_H+b)\}.
    \label{eq:eq:100}
\end{equation}
By Proposition \ref{propspt1}, for any $w\in V(2\mathcal{Q}\zeta_H+b)$ and any $r_i\to 0$, there exists a sequence $\{v_i\}_{i\in\N}\subseteq \supp(\mu)$ such that $w_i:=D_{1/r_i}(\zeta^{-1}v_i)\to w$. Therefore writing this convergence componentwise, we have:
\begin{itemize}
    \item[($\alpha$)]$(v_i)_H=\zeta_H+r_i w_H+R_i,$ with $\lvert R_i\rvert=o(r_i)$,
    \item[($\beta$)]$r_i^2w_T+o(r_i^2)=(v_i)_T-\zeta_T -2\langle \zeta_H,J(v_i)_H\rangle.$
\end{itemize}
As in \eqref{eq:eq:32}, rearranging ($\beta$) we deduce that:
\begin{equation}
\begin{split}
    (v_i)_T=&\zeta_T+2\langle \zeta_H,J(v_i)_H\rangle+r_i^2w_T+o(r_i^2)\\
    =&  \zeta_T+2r_i\langle \zeta_H,J  w_H\rangle+2\langle \zeta_H,JR_i\rangle +r_i^2w_T+o(r_i^2),
\label{eq:eq:40}
\end{split}
\end{equation}
where in the second identity we used ($\alpha$) to substitute $(v_i)_H$ and the fact that $\zeta_H$ and $J\zeta_H$ are orthogonal.
In addition, the fact that by assumption we have $\lvert \zeta - y\rvert\leq \lvert v_i-y\rvert$ together with ($\alpha$) and \eqref{eq:eq:40} imply that:
\begin{equation}
\begin{split}
0\leq& 2r_i\langle\zeta_H-y_H-2(\zeta_T-  y_T)J\zeta_H,  w_H\rangle +2\langle \zeta_H-y_H-2(\zeta_T-  y_T)J\zeta_H,R_i\rangle\\
&\qquad+\lvert r_i  w_H+R_i\rvert^2+2(\zeta_T-  y_T)(r_i^2w_T+o(r_i^2))\\
&\qquad\qquad+(2r_i\langle \zeta_H,J   w_H\rangle+2\langle \zeta_H,J R_i\rangle+r_i^2w_T+o(r_i^2))^2,
\label{eq13}
\end{split}
\end{equation}
for any $i\in\N$. Define $N:=\zeta_H-y_H-2(\zeta_T-  y_T)J\zeta_H$ and note that if $N=0$, dividing the above inequality by $r_i^2$ and sending $i\to\infty$, we get:
\begin{equation}
\begin{split}
0\leq&\lvert  w_H\rvert^2+2(\zeta_T-  y_T)w_T+4\langle \zeta_H,J  w_H\rangle^2.
\label{eq14}
\end{split}
\end{equation}
Let us further observe that if $N=0$ then $\zeta_T\neq   y_T$,
otherwise we would have that $\zeta=y$. Indeed, if $\zeta_T=   y_T$, by definition of $N$ and the fact that we are here assuming that $N=0$, would imply:
$$0=N=\zeta_H-y_H-2(\zeta_T-  y_T)J\zeta_H=\zeta_H-y_H,$$
and this would be in contradiction with the fact that we assumed $\zeta$ and $y$ to be different.
Therefore, since $N=0$ and $\zeta_T\neq   y_T$,  we infer that \eqref{eq14} constitutes a non-trivial bound on $w_T$. This however is in contradiction with \eqref{eq:eq:100} and the fact that $w$ was chosen arbitrarily in $V(2\mathcal{Q}\zeta_H+b)$.

On the other hand, if $N\neq 0$, dividing by $r_i$ inequality \eqref{eq13} and sending $i\to\infty$, we deduce that $0\leq \langle N,  w_H\rangle$. This implies that:
$$V(2\mathcal{Q}\zeta_H + b)\subseteq \{(x,t):\langle N, x \rangle\geq 0\},$$
and therefore there exists $\lambda\neq 0$ such that $\lambda N=2\mathcal{Q}\zeta_H + b$. Furthermore, since $w$ is an element of $V(2\mathcal{Q}\zeta_H + b)$, we infer that $0=\langle 2\mathcal{Q}\zeta_H + b,w_H\rangle=\lambda\langle N,w_H\rangle$. Since $\lambda\neq 0$, this in particular implies:
\begin{equation}
    0=\langle N,w_H\rangle.
    \label{eq:eq:101}
\end{equation}
Using the definition of $N$ and \eqref{eq:eq:101}, inequality \eqref{eq13} becomes:
\begin{equation}
\begin{split}
0\leq&2\langle N,R_i\rangle+\lvert r_i  w_H+R_i\rvert^2+2(\zeta_T-  y_T)(r_i^2w_T+o(r_i^2))\\
&\qquad\qquad\qquad\qquad+(2r_i\langle \zeta_H,J   w_H\rangle+2\langle \zeta_H,J R_i\rangle+r_i^2w_T+o(r_i^2))^2.
\label{eq132}
\end{split}
\end{equation}
On the other hand, the fact that $v_i$ is contained in $\mathbb{K}(b,\mathcal{Q},0)$ implies:
\begin{equation}
    \begin{split}
        0
        =&\langle \zeta_H, \mathcal{Q}\zeta_H\rangle+2r_i\langle \zeta_H, \mathcal{Q} w_H \rangle+2\langle \zeta_H, \mathcal{Q} R_i\rangle+r_i^2\langle w_H, \mathcal{Q} w_H\rangle\\
        &+2r_i\langle w_H, \mathcal{Q} R_i\rangle +\langle R_i, \mathcal{Q} R_i\rangle+\langle b,\zeta_H\rangle+r_i\langle  b,w_H\rangle+\langle b,R_i\rangle
        \label{eq:eq:102}
    \end{split}
\end{equation}
Furthermore, since $\zeta$ is supposed to be contained in $\mathbb{K}(b,\mathcal{Q},0)$, we have $\langle \zeta_H, \mathcal{Q}\zeta_H+b\rangle=0$ and thus, \eqref{eq:eq:102} simplifies to:
\begin{equation}
    \begin{split}
        0=&2r_i\langle \zeta_H, \mathcal{Q} w_H \rangle+2\langle \zeta_H, \mathcal{Q} R_i\rangle+r_i^2\langle w_H, \mathcal{Q} w_H\rangle+2r_i\langle w_H, \mathcal{Q} R_i\rangle\\ &\qquad\qquad\qquad\qquad\qquad\qquad+\langle R_i, \mathcal{Q} R_i\rangle+r_i\langle  b,w_H\rangle+\langle b,R_i\rangle\\
        =&\langle2\mathcal{Q}\zeta_H+b,R_i\rangle+r_i^2\langle   w_H,\mathcal{Q} w_H\rangle+2r_i\langle   w_H,Q R_i\rangle+\langle R_i,\mathcal{Q} R_i\rangle,
        \label{eq:eq:103}
    \end{split}
\end{equation}
where the second identity comes from the already observed fact that $\langle 2\mathcal{Q}\zeta_H+b,  w_H \rangle=0$, as $w\in V(2\mathcal{Q}\zeta_H+b)$.
Identity \eqref{eq:eq:103} together with the fact that $\lambda N=2\mathcal{Q}\zeta_H+b$, yields:
$$-\lambda\langle N,R_i\rangle=r_i^2\langle   w_H,\mathcal{Q} w_H\rangle+2r_i\langle   w_H,Q R_i\rangle+\langle R_i,\mathcal{Q} R_i\rangle.$$
Using the above identity, \eqref{eq132} becomes:
\begin{equation}
\begin{split}
0\leq-&\frac{2}{\lambda}\big(r_i^2\langle   w_H,\mathcal{Q} w_H\rangle+2r_i\langle   w_H,Q R_i\rangle+\langle R_i,\mathcal{Q} R_i\rangle\big)\\
&\qquad+\lvert r_i  w_H+R_i\rvert^2+2(\zeta_T-  y_T)(r_i^2w_T+o(r_i^2))\\
&\qquad\qquad+(2r_i\langle \zeta_H,J   w_H\rangle+2\langle \zeta_H,J R_i\rangle+r_i^2w_T+o(r_i^2))^2.\nonumber
\end{split}
\end{equation}
Dividing the above inequality by $r_i^2$ and sending $i\to \infty$, we deduce that:
$$0\leq-\frac{2}{\lambda} \langle   w_H,\mathcal{Q}w_H\rangle+\lvert  w_H\rvert^2+2(\zeta_T-  y_T)w_T+4\langle \zeta_H,J   w_H\rangle^2,$$
which if $\zeta_T\neq  y_T$ is a non-trivial bound on $w_T$. This contradicts Proposition \ref{verticalsamoa2}.
Therefore $\zeta_T=  y_T$ and thus:$$\lambda(  \zeta_H-y_H)=\lambda N=2\mathcal{Q}\zeta_H+b,$$
concluding the proof of the proposition.
\end{proof}

In the following proposition we prove Theorem \ref{duppy} in case $\mu$ is a vertical measure. The idea behind the proof is the following. Let $\zeta$ and $y$ be as in the statement of Proposition \ref{verticalfun}. If $\lvert\zeta_H-y_H\rvert$ is small, the vector $\zeta_H-y_H$ roughly lies in the tangent space of $\mathbb{K}(b,\mathcal{Q},0)$ at $\zeta$. Therefore,  Proposition \ref{verticalfun}(ii) implies that such a vector in the tangent space to $\mathbb{K}(b,\mathcal{Q},0)$ at $\zeta$ should be parallel to the normal to $\mathbb{K}(b,\mathcal{Q},0)$ at $\zeta$, which is clearly not possible.

\begin{proposizione}\label{spt2}
Either the measure $\mu$ is flat, or fixed a connected component $C$ of $\mathbb{K}(b,\mathcal{Q},0)\setminus \Sigma(F)$ we have:
\begin{itemize}
\item[(i)] either $\supp(\mu)\cap C=\emptyset$,
\item[(ii)] or $C\subseteq \supp(\mu)$.
\end{itemize}
\end{proposizione}

\begin{proof}
Thanks to Proposition \ref{esclusion} we can assume without loss of generality that:
$$\mathcal{S}^{2n+1}(\Sigma(F)\cap \mathbb{K}(b,\mathcal{Q},0))=0,$$
since otherwise $\mu$ is flat.
The set $C\cap \supp(\mu)$ is relatively closed in $C$, thus if it is also relatively open in $C$, by connectedness either $C\cap\supp(\mu)=\emptyset$ or $C\cap\supp(\mu)=C$. By contradiction suppose that this is not the case, and thus there exist an $x\in\supp(\mu)\cap C$ and a $r_0>0$ such that:
\begin{itemize}
\item[($\alpha$)] for any $0<r<r_0$ there exists $y_r\in\supp(\mu)^c\cap C$,
\item[($\beta$)] $\text{cl}(U_{r_0}(x))\cap \mathbb{K}(b,\mathcal{Q},0)=\text{cl}(U_{r_0}(x))\cap C$.
\end{itemize}
Note that since $x\not\in \Sigma(F)$ by construction, then $2\mathcal{Q}x_H+b\neq 0$, see \eqref{numbero2}. Thus, for any $0<r<r_0$ there exists $\zeta_r\in\supp(\mu)\cap B_r(x)$ such that $\lvert \zeta_r-y_r\rvert=\dist(y_r,\supp(\mu))$ and $\lvert \zeta_r-y_r\rvert\to 0$ as $r$ goes to $0$.
By Proposition \ref{verticalfun}, we deduce that for any $0<r<r_0$:
\begin{itemize}
\item[($\gamma$)] $(\zeta_r)_T=  (y_r)_T$,
\item[($\delta$)] there exists $\lambda_r\neq 0$ such that $\lambda_r((\zeta_r)_H-(y_r)_H)=2\mathcal{Q}(\zeta_r)_H+ b$.
\end{itemize}
As $\zeta_r, y_r\in\mathbb{K}(b,\mathcal{Q},0)$ we have that:
\begin{equation}
\begin{split}
0=\langle (y_r)_H-(\zeta_r)_H, \mathcal{Q}[(y_r)_H-(\zeta_r)_H]\rangle+\langle2\mathcal{Q} [(\zeta_r)_H]+b,(y_r)_H-(\zeta_r)_H\rangle.
\end{split}
\nonumber
\end{equation}
Therefore, for a sufficiently small $r>0$, equation ($\delta$) implies:
\begin{equation}
\begin{split}
1=&\Big\lvert\left\langle \frac{(\zeta_r)_H-(y_r)_H}{\lvert(\zeta_r)_H-(y_r)_H\rvert}, \frac{2\mathcal{Q}(\zeta_r)_H + b}{\lvert2\mathcal{Q}(\zeta_r)_H + b\rvert}\right\rangle\Big\rvert\\
=&\frac{\lvert\langle (y_r)_H-(\zeta_r)_H, \mathcal{Q}[(y_r)_H-(\zeta_r)_H]\rangle\rvert}{\lvert (y_r)_H-(\zeta_r)_H\rvert\lvert 2\mathcal{Q}(\zeta_r)_H + b\rvert}
\leq \frac{\lVert\mathcal{Q}\rVert\lvert (y_r)_H-(\zeta_r)_H\rvert}{\lvert 2\mathcal{Q}(\zeta_r)_H + b\rvert}.
\nonumber
\end{split}
\end{equation}
However, since $\lvert y_r-(\zeta_r)_H\rvert$ converges to $0$ and $2\mathcal{Q}(\zeta_r)_H + b$ converges to $2\mathcal{Q}x_H+b\neq 0$ as $r$ tends to zero, we have a contradiction.
\end{proof}

\section{Disconnectedness of \texorpdfstring{$(2n+1)$}{Lg}-uniform cones implies rigidity of tangents}
\label{section:dis}

In this section we reduce the problem of establishing the flatness of blowups of measures with $(2n+1)$-density to the study of some properties of $(2n+1)$-uniform cones, that we introduce in the following:

\begin{definizione}\label{conelli}
An $m$-uniform measure $\mu$ on $\HH^n$ is said to be an $m$-\emph{uniform cone} if
$\lambda^{-m}\mu_{0,\lambda}=\mu$, for any $\lambda>0$. We denote by $\mathcal{C}_{\HH^n}(m)$ the set of $m$-uniform cones.
\end{definizione}

Such reduction consists in constructing a continuous functional on Radon measures which ''disconnects" $(2n+1)$-flat measures to the non-flat $(2n+1)$-uniform cones in the following way:

\begin{teorema}\label{disco}
Suppose that there exists a functional $\mathscr{F}:\mathcal{M}\to \R$, continuous in the weak-$*$ convergence of measures, and a constant $\hbar=\hbar(\HH^n)>0$ such that:
\begin{itemize}
\item[(i)] if $\mu\in\mathfrak{M}(2n+1)$ then $\mathscr{F}(\mu)\leq \hbar/2$,
\item[(ii)] if $\mu\in\mathcal{C}_{\HH^n}(2n+1)$ and $\mathscr{F}(\mu)\leq\hbar$, then $\mu\in\mathfrak{M}(2n+1)$.
\end{itemize}  
Then, for any $\phi$ Radon measure with $(2n+1)$-density and for $\phi$-almost every $x$:
$$\Tan_{2n+1}(\phi,x)\subseteq \Theta^{2n+1}(\phi,x)\mathfrak{M}(2n+1).$$
\end{teorema}

The proof of Theorem \ref{disco} relies on the following two properties of $(2n+1)$-uniform measures:
\begin{enumerate}[label=\textbf{P.\arabic*}]
    \item \label{(uno)}if $\Tan_{2n+1}(\mu,\infty)\cap \mathfrak{M}(2n+1)\neq \emptyset$ then $\mu\in\mathfrak{M}(2n+1)$,
    \item\label{(due)} the set $\Tan_{2n+1}(\mu,\infty)$ is a singleton.
\end{enumerate} 

In the Euclidean case, these properties are algebraic consequences of the expansion of moments. For instance the proof of \ref{(due)} in $\R^n$ is quite immediate, see  \cite[Theorem 3.6(2)]{Preiss1987GeometryDensities}, but it really relies on the fact that moments are symmetric multilinear functions. In $\HH^n$ the structure of moments is much more complicated because they are not multilinear. This is the reason why we could prove these properties only in the codimension $1$ case, where fairly strong structure results for $\supp(\mu)$ are available, see Section \ref{buchi}. 

In Subsections \ref{flatty} and \ref{flattibus} we will establish properties \ref{(uno)} and \ref{(due)}, respectively, while in Subsection \ref{flattitris} we will prove Theorem \ref{disco}.

\subsection{Flatness at infinity implies flatness}
\label{flatty}
In this section we prove \ref{(uno)}. As a first step, we show that if $\mu$ is a uniform measure whose support is contained in $\mathbb{K}(b,\mathcal{Q},\mathcal{T})$, then any $\nu\in\Tan_{2n+1}(\mu,\infty)$ has support contained in $\mathbb{K}(0,\mathcal{Q},\mathcal{T})$. This implies that if $\mu$ has a flat tangent at infinity then $\mathbb{K}(0,\mathcal{Q},\mathcal{T})$ must contain a hyperplane. This is  only possible when $\text{rk}(\mathcal{Q})=1,2$, see the proof of Theorem \ref{flat}, and $\mathcal{T}=0$. In Proposition \ref{rank1} we prove that if $\text{rk}(\mathcal{Q})=1$, then $\mu$ must be flat, while in Proposition \ref{rank2}, we show that if $\text{rk}(\mathcal{Q})=2$, then either $\mu$ is flat or it has a unique non-flat tangent at infitity.

\begin{proposizione}\label{Kinfty1}
Let $\mu$ be a $(2n+1)$-uniform measure for which $\supp(\mu)\subseteq \mathbb{K}(b,\mathcal{Q},\mathcal{T})$. Then, for any $\nu\in\Tan_{2n+1}(\mu,\infty)$ we have $\supp(\nu)\subseteq \mathbb{K}(0,\mathcal{Q},\mathcal{T})$.
\end{proposizione}

\begin{proof}
Proposition \ref{replica} implies that for any $w\in\supp(\nu)$ there is are sequences $\{v_i\}_{i\in\N}\subseteq\supp(\mu)$ and $R_i\to\infty$ for which $w_i:=D_{1/R_i}(v_i)\to w$. The condition $v_i=D_{R_i}(w_i)\in\mathbb{K}(b,\mathcal{Q},\mathcal{T})$ reads:
$$R_i^2\langle (w_i)_H,\mathcal{Q}[(w_i)_H]\rangle+R_i\langle b,(w_i)_H\rangle+R_i^2\mathcal{T}(w_i)_T=0.$$
Dividing the above identity by $R_i^2$ and sending $i$ to infinity, we get that:
$$\langle  w_H,\mathcal{Q}[  w_H]\rangle+\mathcal{T}w_T=0,$$
which implies that $\supp(\nu)\subseteq \mathbb{K}(0,\mathcal{Q},\mathcal{T})$.
\end{proof}

\begin{proposizione}\label{escl}
Let $E,F$ be closed sets in $\HH^n$ and suppose that $\mathcal{C}^{2n+1}\llcorner E$ is a $(2n+1)$-uniform measure and that $\mathcal{C}^{2n+1}(E\cap F)=0$. Then, the measure $\mathcal{C}^{2n+1}\llcorner{(E\cup F)}$ is $(2n+1)$-uniform measure if and only if $\mathcal{C}^{2n+1}(F)=0$.
\end{proposizione}

\begin{proof}
If $F$ is $\mathcal{C}^{2n+1}$-null there is nothing to prove and $\mathcal{C}^{2n+1}\llcorner (E\cup F)$ is uniform.

Viceversa, let us assume that $F$ is \emph{not} $\mathcal{C}^{2n+1}$-null.
Let $x\in E$ be such that $\rho:=\text{dist}(x,F)>0$. It is immediate to see that for any $0<r<\rho/2$ we have:
\begin{equation}
\begin{split}
        \mathcal{C}^{2n+1}\llcorner (E\cup F)(B_r(x))=&\mathcal{C}^{2n+1}(B_r(x)\cap (E\cup F))\\
        =&\mathcal{C}^{2n+1}(B_r(x)\cap E)=r^{2n+1}.
    \label{eq:eq:110}
\end{split}
\end{equation}
On the other hand, if $r>\rho$ is such that $\mathcal{C}^{2n+1}(B_r(x)\cap F)>0$, we infer that:
\begin{equation}
\begin{split}
    &\qquad\qquad\mathcal{C}^{2n+1}\llcorner (E\cup F)(B_r(x))=\mathcal{C}^{2n+1}(B_r(x)\cap (E\cup F))\\
    =&\mathcal{C}^{2n+1}(B_r(x)\cap E)+\mathcal{C}^{2n+1}(B_r(x)\cap F)-\mathcal{C}^{2n+1}(B_r(x)\cap E\cap F)\\
    &\qquad\qquad\qquad\qquad=r^{2n+1}+\mathcal{C}^{2n+1}(B_r(x)\cap F)>r^{2n+1}.
    \label{eq:eq:111}
\end{split}    
\end{equation}
Thus, \eqref{eq:eq:110} and \eqref{eq:eq:111} together show that under the hypothesis $\mathcal{C}^{2n+1}(F)>0$, the measure $\mathcal{C}^{2n+1}\llcorner (E\cup F)$ cannot be uniform.
\end{proof}

\begin{lemma}\label{conn}
Suppose the quadric $\mathbb{K}(b,\mathcal{Q},\mathcal{T})$ is connected and that it supports a $(2n+1)$-uniform measure $\mu$. If $\Sigma(F)=\emptyset$, then $\mu=\mathcal{C}^{2n+1}\llcorner{\mathbb{K}(b,\mathcal{Q},\mathcal{T})}$.
\end{lemma}

\begin{proof}
The Lemma is an immediate consequence of Propositions \ref{supportoK} and \ref{spt1}(i) for the case $\mathcal{T}\neq 0$ and Propositions \ref{supportoK} and \ref{spt2} for the case $\mathcal{T}=0$.
\end{proof}

We study here the case $\text{rk}(\mathcal{Q})=1$, i.e., there exists a non-zero vector $\mathfrak{n}$ such that $\mathcal{Q}=\mathfrak{n}\otimes \mathfrak{n}$.

\begin{proposizione}\label{rank1}
Let $\mu$ be a $(2n+1)$-uniform measure supported on $\mathbb{K}(b,\mathfrak{n}\otimes \mathfrak{n},0)$. Then $\mu$ is flat.
\end{proposizione}

\begin{proof}
By scaling we can assume that $\mathfrak{n}$ is unitary.
If $b=0$ there is nothing to prove since Proposition \ref{verticalsamoa} directly implies that $\mu\in\mathfrak{M}(2n+1)$. Therefore we can assume without loss of generality that $b\neq 0$. A consequence of Proposition \ref{Kinfty1} and the discussion of the case in which $b=0$, is that for any $\mu$ satisfying the hypothesis of the proposition, we have $\Tan(\mu,\infty)=\big\{\mathcal{C}^{2n+1}\llcorner{V(\mathfrak{n})}\big\}.$

There are two possibilities for $b$: either it is parallel to $\mathfrak{n}$ or it is not. We begin with the simpler case in which  $b$ is parallel to $\mathfrak{n}$. In such a case, there exists $\lambda\in\R\setminus\{0\}$ for which $b=\lambda \mathfrak{n}$ and:
\begin{equation}
    \mathbb{K}(\lambda \mathfrak{n},\mathfrak{n}\otimes \mathfrak{n},0)=V(\mathfrak{n})\cup\bigg(-\lambda\frac{\mathfrak{n}}{\lvert \mathfrak{n}\rvert}+V(\mathfrak{n})\bigg).
    \label{eq:eq:eq1}
\end{equation}
The points $w$ of the singular set $\Sigma(F)$, see \eqref{numbero2}, must satisfy the equation:
$$(\lambda +2\langle w_H,\mathfrak{n}\rangle)\mathfrak{n}=0,$$
which implies together with \eqref{eq:eq:eq1} that
$\Sigma(F)=\emptyset$.
Since $\mathcal{C}^{2n+1}\llcorner{V(\mathfrak{n})}$ is $(2n+1)$-uniform, Proposition \ref{escl} together with Proposition \ref{spt2}, imply that $\mu=\mathcal{C}^{2n+1}\llcorner {V(\mathfrak{n})}$.

We are left to discuss the case in which $b$ is not parallel to $\mathfrak{n}$. Since $\Tan_{2n+1}(\mu,\infty)=\{\mathcal{C}^{2n+1}\llcorner{V(\mathfrak{n})}\}$, Proposition \ref{propspt1} implies that for any $w\in V(\mathfrak{n})$ there exists a sequence $\{v_i\}_{i\in\N}\subseteq \supp(\mu)$ such that $D_{1/i}(v_i)\to w$. 
Let $u\in\R^{2n}$ be a unitary vector, orthogonal to $\mathfrak{n}$ and such that $\langle b,u\rangle>0$. Moreover, let $W$ be the orthogonal in $\R^{2n}$ of the span of the vectors $u$ and $\mathfrak{n}$ and denote by $P_W$ the orthogonal projection on $W$. Recall that for every $i$ the $v_i$'s must satisfy the equation: 
$$\langle b,(v_i)_H\rangle+\langle\mathfrak{n},(v_i)_H\rangle^2=0,$$
which, decomposing $v_i$ along $u$, $\mathfrak{n}$ and $W$, becomes:
\begin{equation}
\begin{split}
    0=&\langle b,\mathfrak{n}\rangle\langle (v_i)_H,\mathfrak{n}\rangle+\langle b,u\rangle\langle (v_i)_H,u\rangle+\langle b, P_{W}[(v_i)_H]\rangle+\langle \mathfrak{n},(v_i)_H\rangle^2\\
    \geq&\langle b,\mathfrak{n}\rangle\langle (v_i)_H,\mathfrak{n}\rangle+\langle b,u\rangle\langle (v_i)_H,u\rangle+\langle b, P_{W}[(v_i)_H]\rangle,
\end{split}
    \nonumber
\end{equation}
for any $i\in\N$. If we divide by $i$ the above inequality and let $i\to\infty$, we get:
\begin{equation}
    \langle b,u\rangle\langle w_H,u\rangle\leq-\langle b, P_{W}[w_H]\rangle,
    \label{numbero7}
\end{equation}
since $w_H$ is orthogonal to $\mathfrak{n}$. By the arbitrariness of $w\in V(\mathfrak{n})$, inequality \eqref{numbero7} must be satisfied for any $w_H$ orthogonal to $\mathfrak{n}$. Therefore, since \eqref{numbero7} holds for both $w_H$ and $-w_H$, then: 
$$\langle b,u\rangle\langle w_H,u\rangle=-\langle b, P_{W}[w_H]\rangle.$$
However, the above identity cannot be satisfied for any $w_H$ orthogonal to $\mathfrak{n}$, proving that $\mathbb{K}(b,\mathcal{Q},0)$ in this case cannot support a uniform measure.
\end{proof}

\begin{proposizione}\label{rango}
Let $\mu$ be a $(2n+1)$-uniform measure supported on $\mathbb{K}(b,\mathcal{Q},0)$. If $\mathcal{Q}$ is semidefinite, then $\text{rk}(\mathcal{Q})=1$.
\end{proposizione}

\begin{proof} For any $\nu\in\Tan_{2n+1}(\mu,\infty)$,
Proposition \ref{Kinfty1} implies that $\supp(\nu)\subseteq\mathbb{K}(0,\mathcal{Q},0)$. Suppose by contradiction that $\text{rk}(\mathcal{Q})\geq 2$, then: 
$$\supp(\nu)\subseteq \bigcap_{i=1}^{\text{rk}(\mathcal{Q})} V(\mathfrak{n}_i).$$
where $\mathfrak{n}_i$ are the eigenvectors relative to non-zero eigenvalues of $\mathcal{Q}$. This would imply by Corollary \ref{dege} that
$\mathcal{C}^{2n+1}(\supp(\nu))=0$,
which is a contradiction.
\end{proof}

The following proposition will be useful in the rest of the section as it provides an efficient way to describe the structure of the support of tangent measures at infinity to those $(2n+1)$-uniform measures which are supported on graphs.

\begin{proposizione}\label{grappico}
Let $\mathfrak{n}\in\mathbb{S}^{2n-1}$ and suppose that  $g:\R^{2n-1}\to\R$ is a continuous function such that:
\begin{equation}
    \lim_{\lambda\to\infty}\frac{g(\lambda z)}{\lambda}=g^\infty\bigg(\frac{z}{\lvert z\rvert}\bigg)\lvert z\rvert,
    \label{equi}
\end{equation}
for any $z\in\R^{2n-1}$ where $g^\infty:\mathbb{S}^{2n-2}\to\R$ is a continuous function.
Moreover we define:
$$\Gamma:=\{z+g(z)\mathfrak{n}:z\in\mathfrak{n}^\perp\}\qquad\text{and}\qquad\Gamma_\infty:=\{z+g^\infty(z/\lvert z\rvert)\lvert z\rvert\mathfrak{n}:z\in\mathfrak{n}^\perp\}.$$
Let $\mu$ be a $(2n+1)$-uniform measure in $\mathbb{H}^n$ for which $\Gamma\times \R e_{2n+1}\subseteq \supp(\mu)$. Then, for any $\nu\in\Tan_{2n+1}(\mu,\infty)$ we have:
$$\Gamma_\infty\times \R e_{2n+1}\subseteq \supp(\nu).$$
\end{proposizione}

\begin{proof}
Let $(w,\tau)\in\Gamma_\infty\times\R$ and define the curve $\gamma:[0,\infty)\to\Gamma\times \R$ as:
$$\gamma(t):= (t P_{\mathfrak{n}}(w)+g(t P_{\mathfrak{n}}(w))\mathfrak{n},t^2 \tau),$$
where $P_\mathfrak{n}:\R^{2n}\to\mathfrak{n}^\perp$ is the orthogonal projection on $\mathfrak{n}^\perp$. The curve $\gamma$ is contained in $\supp(\mu)$ and
$$\lim_{t\to\infty} D_{1/t}(\gamma(t))=\bigg(P_{\mathfrak{n}}(w)+g^\infty\bigg(\frac{ P_{\mathfrak{n}}(w)}{\lvert P_{\mathfrak{n}}(w)\rvert}\bigg)\lvert P_{\mathfrak{n}}(w)\rvert\mathfrak{n},\tau\bigg)=(w,\tau),$$
where the last equality comes from the fact that $w\in\Gamma_\infty$. 
Let $\nu\in\Tan_{2n+1}(\mu,\infty)$ and $R_i\to \infty$ be the sequence for which  $R_i^{-(2n+1)}\mu_{0,R_i}\rightharpoonup\nu$. Then, we have that:
$$D_{1/R_i}(\gamma(R_i))\in\supp(R_i^{-(2n+1)}\mu_{0,R_i}).$$
Therefore \eqref{equi} implies by Proposition \ref{replica} that $(w,\tau)\in\supp(\nu)$. By the arbitrariness of $(w,\tau)$ and of $\nu$, we have that $\Gamma_\infty\times \R e_{2n+1}\subseteq \supp(\nu)$ for any $\nu\in\Tan_{2n+1}(\mu,\infty)$.
\end{proof}

The following proposition establishes both properties \ref{(uno)} and \ref{(due)} in the case the quadric $\mathbb{K}(b,\mathcal{Q},0)$ with $\text{rk}(\mathcal{Q})=2$.

\begin{proposizione}
Suppose $\mu$ is a $(2n+1)$-uniform measure supported on $\mathbb{K}(b,\mathcal{Q},0)$. If $\text{rk}(\mathcal{Q})=2$, then one of the following two mutually exclusive conditions holds:
\begin{itemize}
\item[(i)] $\mu\in\mathfrak{M}(2n+1)$,
\item[(ii)] $\Tan_{2n+1}(\mu,\infty)=\{\nu\}$ and $\nu$ is not flat.
\end{itemize}
\label{rank2}
\end{proposizione}

\begin{proof}
Since $\mathcal{Q}$ is symmetric, has rank $2$ and cannot be semidefinite, otherwise by Proposition \ref{rango} it would have rank $1$, there are $e_1,e_2\in\R^{2n}$ orthonormal vectors and $\lambda_1,\lambda_2>0$ for which  $\mathcal{Q}=-\lambda_1^2 e_1\otimes e_1 +\lambda_2^2 e_2\otimes e_2$. We define $\mathfrak{n}:=-\lambda_1 e_1+\lambda_2 e_2$ and $\mathfrak{m}:=\lambda_1 e_1+\lambda_2 e_2$.

If $b=0$ it is readily seen that $\mathbb{K}(0,\mathcal{Q},0)=V(\mathfrak{n})\cup V(\mathfrak{m})$ and that the singular set $\Sigma(F)$ coincides with $V(\mathfrak{n})\cap V(\mathfrak{m})$. In particular $\mathbb{K}(0,\mathcal{Q},0)$ is disconnected by $\Sigma(F)$ in four half planes which we denote by $C_i$, with $i=1,\ldots,4$. We claim that $\supp(\mu)$ can coincide with the closure of the union of just two of the $C_i$'s. First of all, Proposition \ref{spt2} implies that $\supp(\mu)$ must coincide with the closure of the union of some of these half-planes and on the other hand Proposition \ref{verticalsamoa} implies that $\supp(\mu)$ cannot coincide with the closure of just one half-plane.

If on the other hand, all 4 half planes $C_1,\ldots,C_4$ are contained in $\supp(\mu)$, both $V(\mathfrak{n})$ and $V(\mathfrak{m})$ must be contained in $\supp(\mu)$. Recall that by Remark \ref{uniformityflat} the measures $\mathcal{C}^{2n+1}\llcorner V(\mathfrak{n})$ and $\mathcal{C}^{2n+1}\llcorner V(\mathfrak{m})$ are uniform measures and since by assumption $\mathfrak{m}$ and $\mathfrak{n}$ are linearly independent, as we supposed $\text{rk}(\mathcal{Q})=2$, by Corollary \ref{dege} we also deduce that:
\begin{equation}
\mathcal{C}^{2n+1}(V(\mathfrak{n})\cap V(\mathfrak{m}))=0.
    \label{eq/inter0}
\end{equation}
The above identity however shows that we fall in the hypothesis of Proposition \ref{escl} and thus $\mathcal{C}^{2n+1}\llcorner (V(\mathfrak{n})\cup V(\mathfrak{m}))$ cannot be a uniform measure. 

Finally, let us exclude the 3 half-planes case. Assume without loss of generality (as in the other 3 cases the argument is identical) that $\supp(\mu)=\text{cl}(C_1\cup C_2\cup C_3)$ and that $V(\mathfrak{n})=\text{cl}(C_1\cup C_3)$. As in the previous case Remark \ref{uniformityflat} implies that $\mathcal{C}^{2n+1}\llcorner V(\mathfrak{n})$
is uniform and \eqref{eq/inter0} shows that:
$$\mathcal{C}^{2n+1}(V(\mathfrak{n})\cap \text{cl}(C_2))\leq \mathcal{C}^{2n+1}(V(\mathfrak{n})\cap V(\mathfrak{m}))=0.$$
The above identity however shows that we fall in the hypothesis of Proposition \ref{escl} and thus $\mathcal{C}^{2n+1}\llcorner (V(\mathfrak{n})\cup \text{cl}(C_2))$ cannot be a uniform measure.

Therefore, the only remaining possibility is that there are only two half-planes contained in $\supp(\mu)$. If these half-planes are both contained either in $V(\mathfrak{n})$ or in $V(\mathfrak{m})$, then $\mu$ is flat by Proposition \ref{verticalsamoa} and this falls in case (i).

If on the other hand one is contained in $V(\mathfrak{n})$ and the other one in $V(\mathfrak{m})$, then $\mu$ cannot be flat, since as already remarked $V(\mathfrak{n})$ and $V(\mathfrak{m})$ are not coplanar. Furthermore, since the half planes $C_1,\ldots, C_4$ are invariant under intrinsic dilations, it is easy to see that $\mu_{0,\lambda}/\lambda^{2n+1}=\mu$ for any $\lambda>0$. In particular, being $\mu$ dilation-invariant we have $\Tan_{2n+1}(\mu,\infty)=\{\mu\}$ and thus the tangent at infinity of $\mu$ is unique and not flat. Note that this case falls in case (ii).

If $b\neq 0$ we have two subcases. Either $b$ is contained in the image of $\mathcal{Q}$ or it  is not. 
First we discuss the simpler case in which $b\not \in \text{span}(e_1,e_2)$ and note that this implies that $n>1$. In this case $\mathbb{K}(b,\mathcal{Q},0)$ is a graph of a quadratic polynomial. Indeed, complete $e_1,e_2$ to an orthonormal basis $\{e_1,\ldots,e_{2n}\}$ of $\R^{2n}$ and assume without loss of generality that $\langle b,e_3\rangle\neq 0$. Then $\mathbb{K}(b,\mathcal{Q},0)$ is an $e_3$-graph, indeed if  $h\in\R^{2n}$ satisfies $\langle h,b\rangle+\langle h,\mathcal{Q}h\rangle=0$, then:
$$\langle h,e_3\rangle=-\frac{-\lambda_1^2 \langle h,e_1\rangle^2+\lambda_2^2 \langle h,e_2\rangle^2+\sum_{i\neq 3} \langle b,e_i\rangle \langle h,e_i\rangle}{\langle b,e_3\rangle}.$$
This implies that the quadric $\mathbb{K}(b,\mathcal{Q},0)$ is an $e_3$-graph and thus it is a connected set. Moreover, since the equation $b+2\mathcal{Q}h=0$ does not have solutions $h\in\R^{2n}$, the singular set $\Sigma(F)$ is empty. Therefore Lemma \ref{conn} implies that $\supp(\mu)=\mathbb{K}(b,\mathcal{Q},0)$ and Proposition \ref{grappico} concludes that for any $\nu\in\Tan_{2n+1}(\mu,\infty)$ we have: 
\begin{equation}
    \mathbb{K}(0,\mathcal{Q},0)\subseteq\supp(\nu).
    \label{incl1}
\end{equation}
On the other hand, Proposition \ref{Kinfty1} implies that $\supp(\nu)\subseteq \mathbb{K}(0,\mathcal{Q},0)$. In conjunction with \eqref{incl1} this shows that for any $\nu\in \Tan_{2n+1}(\mu,\infty)$ we have:
$$\supp(\nu)=\mathbb{K}(0,\mathcal{Q},0),$$
and thanks to Proposition \ref{supportoK}, $\Tan_{2n+1}(\mu,\infty)=\{\mathcal{C}^{2n+1}\llcorner \mathbb{K}(0,\mathcal{Q},0)\}$. Now note that the closed set $\mathbb{K}(0,\mathcal{Q},0)$ coincides with two vertical, non-coinciding hyperplanes. This however, shows that $\nu$ cannot be a uniform measure thanks the $4$ half-planes subcase in the discussion for $b=0$. Since by Proposition \ref{replica}(i) the tangents at infinity of a uniform measure are uniform, we reach a contradiction with the fact that $\mu$ was uniform.

Thus, we are left to study the case where $b\neq 0$ and $b=b_1 e_1+b_2 e_2$ for some $b_1,b_2\in\R$. For any $x\in\mathbb{K}(b,\mathcal{Q},0)$, once completed the squares we have that:
\begin{equation}
\begin{split}
0=-\left(\lambda_1\langle x,e_1\rangle-\frac{b_1}{2\lambda_1}\right)^2+\left(\lambda_2\langle x,e_2\rangle+\frac{b_2}{2\lambda_2}\right)^2+\frac{b_1^2}{4\lambda_1^2}-\frac{b_2^2}{4\lambda_2^2},
\label{numbero3}
\end{split}
\end{equation}
For any $\overline{x}\in\Sigma(F)$, see \eqref{numbero2}, we have:
$$-2\lambda_1^2\langle\overline{x}_H,e_1\rangle e_1+2\lambda_2^2\langle\overline{x}_H,e_2\rangle e_2+ b_1 e_1+ b_2 e_2=0,$$
and in particular $b_1=2\lambda_1^2\langle\overline{x}_H,e_1\rangle$ and $b_2=-2\lambda_2^2\langle\overline{x}_H,e_2\rangle$. This in particular implies by \eqref{numbero3} that $\overline{x}$ cannot be contained in $\mathbb{K}(b,\mathcal{Q},0)$ if $b_1^2/4\lambda_1-b_2^2/4\lambda_2\neq0$.

If $b_1^2/4\lambda_1-b_2^2/4\lambda_2>0$ by the above discussion we deduce that $\Sigma(F)=\emptyset$. Thanks to the identity \eqref{numbero3}, the quadric
$\mathbb{K}(b,\mathcal{Q},0)$ is easily seen to be the disjoint union of two $e_1$-graphs $\Gamma_1$ (which we assume contains $0$) and $\Gamma_2$. The functions $g_1,g_2:e_1^\perp\to e_1$ which define $\Gamma_1$ and $\Gamma_2$ respectively, satisfy the hypothesis of Proposition \ref{grappico}. Indeed:
$$\lim_{t\to\infty} \frac{g_{1,2}(t h)}{t}
=\pm\frac{\lambda_2\lvert \langle h, e_2\rangle\rvert}{\lambda_1}e_1=g_{1,2}^\infty\bigg(\frac{h}{\lvert h\rvert}\bigg)\lvert h\rvert e_1.$$
Proposition \ref{spt2} implies that $\Gamma_1$ must be contained in $\supp(\mu)$ and therefore the graph of $g_1^\infty$ is contained in the support of any tangent measure at infinity to $\mu$ by Proposition \ref{grappico}. Suppose now by contradiction that $\supp(\mu)$ contains also $\Gamma_2$. Again by Proposition \ref{grappico} we would have that the graph of $g_2^\infty$ is contained in the support of any tangent measure at infinity to $\mu$. However, since the union the graphs of $g_1^\infty$ and $g_2^\infty$ coincides with $\mathbb{K}(0,\mathcal{Q},0)$, by the discussion of the case $b=0$, this is not possible. Therefore, the support of any tangent measure $\nu$ at infinity to $\mu$ coincides with the graph of $g_1^\infty$. Hence, by Proposition \ref{supportoK} $\Tan_{2n+1}(\mu,\infty)$ is a singleton and its only element cannot be flat (as the graph of $g_1^\infty$ is not a hyperplane). The case $b_1^2/4\lambda_1-b_2^2/4\lambda_2<0$ is treated in the same way with the roles of $e_1$ and $e_2$ reversed.

If $b_1^2/4\lambda_1-b_2^2/4\lambda_2=0$, the quadric $\mathbb{K}(b,\mathcal{Q},0)$ coincides with the solutions of the equation:
$$\left(\lambda_1x_1-\frac{b_1}{2\lambda_1}\right)^2=\left(\lambda_1x_2+\frac{b_2}{2\lambda_2}\right)^2.$$
Let $\tau:=(b_1/2\lambda_1^2) e_1+(b_2/2\lambda_2^2)e_2$ and note that $\mathbb{K}(b,\mathcal{Q},0)$ coincides with $\tau*(V(\mathfrak{n})\cup V(\mathfrak{m}))$. Since left translations are isometries of $\HH^n$, the discussion of this case reduces to the one in which $b=0$, concluding the proof of the proposition.
\end{proof}

Eventually, the following theorem concludes the proof of \ref{(uno)}.

\begin{teorema}
If $\mu$ is a $(2n+1)$-uniform measure for which there exists $\mathfrak{n}\in\mathbb{S}^{2n-1}$ such that $\mathcal{C}^{2n+1}\llcorner{V(\mathfrak{n})}\in\Tan_{2n+1}(\mu,\infty)$,
then $\mu=\mathcal{C}^{2n+1}\llcorner{ V(\mathfrak{n})}$.
\label{flat}
\end{teorema}

\begin{proof}
If $\supp(\mu)\subseteq \mathbb{K}(b,\mathcal{Q},\mathcal{T})$, then for any $\nu\in\Tan_{2n+1}(\mu,\infty)$ we have $\supp(\nu)\subseteq\mathbb{K}(0,\mathcal{Q},\mathcal{T})$ by Proposition \ref{Kinfty1}. In particular $V(\mathfrak{n})\subseteq \mathbb{K}(0,\mathcal{Q},\mathcal{T})$ and this implies that $\mathcal{T}=0$. Complete $\mathfrak{n}$ to an orthonormal basis $\{\mathfrak{n},e_2,\ldots,e_{2n}\}$ of $\R^{2n}$ and note that, since $V(\mathfrak{n})\subseteq\mathbb{K}(0,\mathcal{Q},0)$, we have:
$$\sum_{i=2}^{2n}\langle e_i,\mathcal{Q}e_i\rangle \langle w_H,e_i\rangle^2+2\sum_{2\leq i<j\leq 2n}\langle e_i,\mathcal{Q}e_j\rangle \langle w_H,e_i\rangle\langle w_H,e_j\rangle=0,$$
for any $w\in V(\mathfrak{n})$ and thus $\langle e_i,\mathcal{Q}e_j\rangle=0$ for any $2\leq i,j \leq 2n$. This implies that for any $x\in\HH^n$ we have:
$$\langle x_H,\mathcal{Q}x_H\rangle=\langle x_H,\mathfrak{n}\rangle\Big(\langle \mathfrak{n},\mathcal{Q}\mathfrak{n}\rangle\langle x_H,\mathfrak{n}\rangle+2\sum_{i=2}^{2n} \langle \mathfrak{n},\mathcal{Q}e_i\rangle \langle x_H,e_i\rangle\Big)=\langle x,\mathfrak{n}\rangle\langle \mathfrak{m},x\rangle,$$
where $\mathfrak{m}:=\langle \mathfrak{n},\mathcal{Q}\mathfrak{n}\rangle\mathfrak{n}+2\sum_{i=2}^{2n} \langle \mathfrak{n},\mathcal{Q}e_i\rangle e_i$. In particular $\text{rk}(\mathcal{Q})\leq 2$ and if $\mathfrak{m}$ is parallel to $\mathfrak{n}$, Proposition \ref{rank1} implies that $\mu$ is flat. On the other hand if $\mathfrak{m}$ is not parallel to $\mathfrak{n}$, since $\text{rk}(Q)=2$ and $\mu$ has a flat tangent at infinity, Proposition \ref{rank2} implies that $\mu$ is flat.
\end{proof}

\subsection{Uniqueness of the tangent at infinity}
\label{flattibus}

This subsection is devoted to the proof of \ref{(due)}.
The uniqueness of the tangents at infinity also implies that they are $(2n+1)$-uniform cones. Indeed, if for the sequence $R_i\to\infty$ we have $R_i^{-(2n+1)}\mu_{0,R_i}\rightharpoonup\nu$, for any $\lambda>0$ we also have:
$$(\lambda R_i)^{-(2n+1)}\mu_{0,\lambda R_i}\rightharpoonup\lambda^{-(2n+1)}\nu_{0,\lambda}.$$
Therefore, \ref{(due)} would imply that $\lambda^{-(2n+1)}\nu_{0,\lambda}=\nu$ for any $\lambda>0$, forcing $\nu$ to be a $(2n+1)$-uniform cone by Definition \ref{conelli}.

The idea behind the proof of the uniqueness of the tangent at infinity is the following. Let $\nu\in \Tan_{2n+1}(\mu,\infty)$ and fix a point $w\in \supp(\nu)$. If we can find a continuous curve $\gamma:[0,\infty)\to\HH^n$ contained in $\supp(\mu)$ for which:
$$\lim_{t\to\infty} D_{1/t}(\gamma(t))=w,$$
then $w\in\supp(\xi)$ for any $\xi\in\Tan_{2n+1}(\mu,\infty)$ by Proposition \ref{replica}(iii), since $D_{1/t}(\gamma(t))\in t^{-(2n+1)}\mu_{0,t}$.

In the various cases, the curve $\gamma$ will always be constructed inside the quadric $\mathbb{K}(b,\mathcal{Q},\mathcal{T})$ supporting $\mu$ and its initial point  $\gamma(0)$ will always be a point of $\supp(\mu)$. In order to make sure that the whole $\gamma$ is contained in $\supp(\mu)$, we force $\gamma$ to avoid the singular set $\Sigma(F)$, so that continuity implies that $\gamma$ contained in just one connected component of $\mathbb{K}(b,\mathcal{Q},0)\setminus \Sigma(F)$. Since the starting point was contained in $\supp(\mu)$ by hypothesis, this implies that $\gamma$ must be contained in $\supp(\mu)$ by Theorem \ref{duppy}.

\begin{teorema}\label{conol}
Suppose $\mu$ is a $(2n+1)$-uniform measure. Then $\Tan_{2n+1}(\mu,\infty)$ is a singleton.
\end{teorema}

\begin{proof}
Let $\supp(\mu)\subseteq \mathbb{K}(b,\mathcal{Q},\mathcal{T})$ for some non-zero $\mathcal{Q}\in \mathrm{Sym}(2n)$ and assume without loss of generality that $b\neq 0$. Indeed if $b=0$, the singular set $\Sigma(F)$ and the quadric $\mathbb{K}(0,\mathcal{Q},\mathcal{T})$ are dilation invariant. Therefore the measure $\mu$ by Theorem \ref{duppy} is dilation invariant too. This implies that the tangent at infinity to $\mu$ coincides with $\mu$ itself and in particular $\Tan_{2n+1}(\mu,\infty)$ is a singleton. 

In the following the measure $\nu$ will be always considered an arbitrary element of $\Tan_{2n+1}(\mu,\infty)$. We also let $\{R_i\}$ be the sequence for which $R_i^{-2n-1}\mu_{0,R_i}\rightharpoonup \nu$ with $R_i\to\infty$. 

First of all we study the simpler case in which $\mathcal{T}\neq0$. For the reader's convenience we recall that $f$ and $\Sigma(f)$ were introduced in the Section \ref{buchi} in \eqref{numbero14} and \eqref{char}, respectively.

Suppose that  $\mathcal{T}\neq 0$ and that $\supp(\mu)=\mathbb{K}(b,\mathcal{Q},\mathcal{T})$ and let $w\in\mathbb{K}(0,\mathcal{Q},\mathcal{T})$. By definition of $\mathbb{K}(0,\mathcal{Q},\mathcal{T})$ we infer that:
$$w_T=-\langle\mathcal{Q}w_H,w_H\rangle/\mathcal{T}.$$
The curve $\gamma_w(\cdot)$ defined for any $t\in[0,\infty)$ as: 
$$\gamma_w(t):=(tw_H,-\langle t b+t^2\mathcal{Q}w_H,w_H\rangle/\mathcal{T}),$$
is contained in $\mathbb{K}(b,\mathcal{Q},\mathcal{T})$, indeed it can be easily checked that:
\begin{equation}
    \begin{split}
        &\langle (\gamma_w(t))_H,\mathcal{Q}(\gamma_w(t))_H\rangle+\langle b, (\gamma_w(t))_H\rangle+\mathcal{T}\gamma_w(t)=0 \qquad\text{for any $t\in[0,\infty)$.}
        \nonumber
    \end{split}
\end{equation}
Thus, since by assumption $\supp(\mu)=\mathbb{K}(b,\mathcal{Q},\mathcal{T})$, then the image of $\gamma_w$ is also contained in $\supp(\mu)$. In addition, we also have that:
\begin{equation}
    \begin{split}
        \lim_{t\to\infty}D_{1/t}(\gamma_w(t))=&\lim_{t\to\infty}\big((\gamma_w(t))_H/t,(\gamma_w(t))_T/t^2\big)\\=&\lim_{t\to\infty}\big(w_H,-\langle b/t+\mathcal{Q}w_H,w_H\rangle\big)=w.
        \nonumber
    \end{split}
\end{equation}
By the arbitrariness of $w$ and Proposition \ref{replica} this implies that $\mathbb{K}(0,\mathcal{Q},\mathcal{T})\subseteq\supp(\nu)$. On the other hand Proposition \ref{Kinfty1} implies the previous inclusion holds as an equality and thus Proposition \ref{supportoK} concludes that $\Tan_{2n+1}(\mu,\infty)$ is a singleton.

If $\mathcal{T}\neq 0$ and $\supp(\mu)\varsubsetneq\mathbb{K}(b,\mathcal{Q},\mathcal{T})$,
 Proposition \ref{spt1} implies that $n=1$ and that $\supp(\mu)$ is contained in the image under $f$, see \eqref{numbero14}, of one of the two connected components in which $\R^2$ is splitted by $\Sigma(f)$, see \eqref{char}. Let $p,v\in\R^{2}$ be such that $p+\text{span}(v)=\Sigma(f)$ and $\mu=\mathcal{C}^{3}\llcorner {f(C)}$
where $C:=p+H^+(v)$ and $H^+(v):=\{z\in\R^{2n}:\langle v,z\rangle\geq 0\}$. For any $w\in H^+(v)$ the curve:
\begin{equation}
    t\mapsto (p+tw_H,-\langle p+tw_H,\mathcal{Q}(p+tw_H)+b\rangle/\mathcal{T})=:\gamma_w(t),
    \nonumber
\end{equation}
is contained in $\supp(\mu)$ and $\lim_{t\to\infty}D_{1/t}(\gamma_w(t))=w$. This implies by Proposition \ref{replica} that $\{(z,-\langle z,\mathcal{Q}z\rangle/\mathcal{T}):z\in H^+(v)\}=\supp(\nu)$.

\smallskip

We are left to prove the thesis when $\mathcal{T}=0$. When the quadrics are vertically invariant it is harder to build the curves we mentioned in the short paragraph introducing this proposition because the singular set $\Sigma(F)$ may have codimension $1$ inside $\mathbb{K}(b,\mathcal{Q},0)$. 

However, if there exists an $\mathfrak{n}\in\R^{2n}$ for which $\nu(V(\mathfrak{n}))>0$, Lemma \ref{RAPP1} implies that the intersection $V(\mathfrak{n})\cap \mathbb{K}(0,\mathcal{Q},0)$ must have positive $\mathcal{H}^{2n}_{eu}$-measures, which is only possible if $V(\mathfrak{n})\subseteq \mathbb{K}(0,\mathcal{Q},0)$ thanks to the very rigid algebra of the quadrics. Arguing as in the proof of Theorem \ref{flat}, we deduce that either $\text{rk}(\mathcal{Q})=1$ or $\text{rk}(\mathcal{Q})=2$ and thus Proposition \ref{rank1} and Proposition \ref{rank2} yield the uniqueness of the tangent at infinity.

Thanks to the arbitrariness of our choice of $\nu$, from now on we can therefore assume without loss of generality that for any $\nu\in \Tan_{2n+1}(\mu,\infty)$ and any $\mathfrak{n}\in\R^{2n}$ we have $\nu(V(\mathfrak{n}))=0$. This in particular implies that $\text{rk}(\mathcal{Q})\geq 3$ and by Proposition \ref{rango} we infer that $\mathcal{Q}$ cannot be semidefinite.

Throughout the rest of the proof, we reserve the symbol $w$ for an element of $\supp(\nu)$ that is going to be chosen later. Note that by Proposition \ref{replica} however we choose $w$, we can find a sequence $\{v_i\}\subseteq \supp(\mu)$ for which $D_{1/R_i}(v_i)\to w$. Without loss of generality we can assume that these $v_i$s are all contained in the same connected component of $\mathbb{K}(b,\mathcal{Q},0)\setminus\Sigma(F)$. We now distinguish two cases.

At first we assume that $\Sigma(F)\neq \emptyset$ and note that this also implies that $\Sigma(F)\setminus\{0\}\neq \emptyset$. Indeed, if $(x,0)\in \Sigma(F)$, then for any $s\in\R$ the point $(x,s)$ belongs to $\Sigma(F)$ since  both the quadric $\mathbb{K}(b,\mathcal{Q},0)$ and the singular set $\Sigma(F)$ are vertically invariant. Furthermore, for any $y\in \Sigma(F)$ we have:
\begin{equation}
    2\mathcal{Q}y_H+b=0\qquad\text{and}\qquad\langle y_H,\mathcal{Q}y_H+b\rangle=0.
    \label{numbero21}
\end{equation}
In particular, the identities in \eqref{numbero21} imply:
\begin{equation}
0=\langle y_H,2\mathcal{Q}y_H+b\rangle-\langle y_H,\mathcal{Q}y_H+b\rangle=\langle y_H,\mathcal{Q}y_H\rangle=-\langle b,y_H\rangle,
\label{numbero20}
\end{equation}
We now restrict $w$ to be an element of $\supp(\nu)$ for which $\langle b,  w_H\rangle\neq0$.
From now on we let $\tilde{x}$ to be a \emph{non-zero} fixed element in $\Sigma(F)$ and define:
$$\gamma(t)=\gamma_i(t):=\big((v_i)_H+t  w_H+t\theta(t)\tilde{x}_H,(v_i)_T+t^2 w_T\big)$$
where $\theta(t):=\langle 2\mathcal{Q}[(v_i)_H]+b,w_H\rangle/\langle b,(v_i)_H+tw_H\rangle$. If $i$ is big enough then $\langle b,(v_i)_H+tw_H\rangle\neq0$ (since we have that $D_{1/R_i}(v_i)\to w$) and thus $\theta(t)$ is well defined for $t\geq0$. First we check that $\lim_{t\to\infty} D_{1/t}(\gamma(t))=w$. Indeed:
\begin{equation}
\begin{split}
     \lim_{t\to\infty}D_{1/t}(\gamma(t))=&\lim_{t\to\infty}\bigg(\frac{(v_i)_H+t  w_H+t\theta(t)\tilde{x}_H}{t},\frac{(v_i)_T+t^2 w_T}{t^2}\bigg)\\
     =&\Big(w_H+\lim_{t\to\infty} \theta(t)\tilde{x}_H,w_T\Big)=w,
    \nonumber
\end{split}
\end{equation}
since $\lim_{t\to \infty} \theta(t)=0$. Secondly, we check that  $\gamma(t)\in\mathbb{K}(b,\mathcal{Q},0)$ for any $t>0$.
Indeed, recall that since $w\in\supp(\nu)$ then $w\in \mathbb{K}(0,\mathcal{Q},0)$ and in particular:
\begin{equation}
    \begin{split}
        \langle w_H,\mathcal{Q} w_H\rangle=0
        \label{eq:eq:quasi}
    \end{split}
\end{equation}
Putting together \eqref{numbero20}, \eqref{eq:eq:quasi} and the fact that $\langle (v_i)_H,\mathcal{Q} (v_i)_H+b\rangle=0$, which comes from the assumption $v_i\in \supp(\mu)$, we infer:
\begin{equation}
\begin{split}
    &\qquad\qquad\qquad\qquad\langle(\gamma(t))_H,b+\mathcal{Q}(\gamma(t))_H\rangle\\
    &\qquad=\langle (v_i)_H+t  w_H+t\theta(t)\tilde{x}_H, b+\mathcal{Q}[(v_i)_H+t  w_H+t\theta(t)\tilde{x}_H]\rangle\\
    =&2t\langle (v_i)_H,\mathcal{Q}w_H\rangle+2t\theta(t)\langle (v_i)_H,\mathcal{Q}\tilde{x}_H\rangle+2t^2\theta(t)\langle w_H,\mathcal{Q} \tilde{x}_H\rangle+t\langle b,w_H\rangle.
\end{split}
\end{equation}
Rearranging the terms in the righ-hand side of the above identity, we deduce:
\begin{equation}
\begin{split}
    \langle(\gamma(t))_H,b+\mathcal{Q}(\gamma(t))_H\rangle=&t\langle 2\mathcal{Q}(v_i)_H+b,w_H\rangle+t\theta(t)\langle (v_i)_H+tw_H,2\mathcal{Q}\tilde{x}_H\rangle\\
    =&t\langle 2\mathcal{Q}(v_i)_H+b,w_H\rangle-t\theta(t)\langle (v_i)_H+tw_H,b\rangle=0,\nonumber
\end{split}
\end{equation}
where in the first identity of the last line we used the first identity in \eqref{numbero21} and in the last one the specific choice of the function $\theta$. This concludes the proof of the fact that the image of the curve $\gamma=\gamma_i$ is contained in $\mathbb{K}(b,\mathcal{Q},0)$ for any $i\in\N$.

Now we prove that the curve $\gamma$ does not intersect $\Sigma(F)$, indeed:
\begin{equation}
    \begin{split}
        \langle 2\mathcal{Q}(\gamma(t))_H+b,\tilde{x}_H\rangle=&\langle (\gamma(t))_H,2\mathcal{Q}\tilde{x}_H\rangle+\langle b,\tilde{x}_H\rangle\\
        =&-\langle (\gamma(t))_H, b\rangle=-\langle (v_i)_H+t  w_H,b\rangle\neq 0,
        \label{eq:pll}
    \end{split}
\end{equation}
for any $t>0$ provided $i$ is big enough, where in the second identity we used the fact that $\langle b,\tilde{x}_H\rangle=0$ and that $2\mathcal{Q}\tilde{x}_H=-b$, see \eqref{eq:eq:quasi}. This implies that for any $t\geq 0$, the curve $\gamma$ is contained in the same connected component of $\mathbb{K}(0,\mathcal{Q},0)\setminus \Sigma(F)$ of $v_i$: if for some $t>0$ we had that $\gamma(t)$ crossed $\Sigma(F)$ we would have had  $2\mathcal{Q}(\gamma(t))_H+b=0$, contradicting \eqref{eq:pll}.

Since the initial point of $\gamma$ is contained in $\supp(\mu)$, the  whole curve $\gamma$ by continuity is contained in the support of $\mu$ by Theorem \ref{duppy}. Since $D_{1/t}(\gamma(t))$ is contained in the support of $t^{-(2n+1)}\mu_{0,t}$ and:
$$\lim_{t\to\infty} D_{1/t}(\gamma(t))=w,$$
we deduce that $w\in\supp(\xi)$ for any $\xi\in\Tan_{2n+1}(\mu,\infty)$ by Proposition \ref{replica}. Thanks to the fact that $w$ was an arbitrary element of $\supp(\nu)\setminus V(b)$ and to the above discussion we conclude that: 
\begin{equation}
    \supp(\nu)=\mathrm{cl}(\supp(\nu)\setminus V(b))\subseteq\supp(\xi)\qquad\text{for any $\xi\in\Tan_{2n+1}(\mu,\infty)$},
    \label{eq:eq:eq4}
\end{equation}
where the first identity above comes from the fact that $\nu(V(b))=0$.
The arbitrariness of $\nu$ and $\xi$ imply that the supports of all tangents at infinity to $\mu$ coincide and thus, by Proposition  \ref{supportoK} the tangent at infinity is unique.

Finally, suppose that $\Sigma(F)=\emptyset$. For any $x\in\mathbb{K}(b,\mathcal{Q},0)$ we have:
\begin{equation}
    0=\lambda_1\langle x_H,e_1\rangle^2+\langle b,e_1\rangle\langle x_H,e_1\rangle+\langle \mathcal{Q}[P_1(x_H)]+b,P_1(x_H)\rangle,
    \label{numbero15}
\end{equation}
where $e_1$ is a unitary eigenvector of $\mathcal{Q}$ relative to a positive eigenvalue $\lambda_1$ and $P_1$ is the orthogonal projection on $e_1^\perp$. As a first step let us suppose that our element $w\in\supp(\nu)$ satisfies the inequality $\langle w_H,e_1\rangle> 0$. Since $w\in\mathbb{K}(0,\mathcal{Q},0)$ we have:
\begin{equation}
    \langle \mathcal{Q}[P_1(  w_H)],P_1(  w_H)\rangle=-\lambda_1\langle w_H, e_1\rangle^2<0.
    \label{numbero16}
\end{equation}
In addition, since $D_{1/R_i}(v_i)\to w$, defined $s(t)=s_i(t):=P_1[(v_i)_H+t  w_H]$ and provided $i$ is sufficiently big, we have:
$$\langle \mathcal{Q}s(t)+b,s(t)\rangle<0\text{ for any }t\geq 0.$$
Therefore, the curve:
$$\gamma_w(t):=\bigg(\frac{\sqrt{\langle b,e_1\rangle^2-4\lambda_1\langle \mathcal{Q}s(t)+b,s(t)\rangle}-\langle b,e_1\rangle}{2\lambda_1}e_1+s(t),(v_i)_T+t^2w_T\bigg),$$
is well defined for any $t\geq 0$. The component of $\gamma_w$ along $e_1$ is by construction a solution to the equation:
$$\lambda_1\zeta^2+\langle b,e_1\rangle\zeta+\langle \mathcal{Q}s(t)+b,s(t)\rangle=0,$$
which by \eqref{numbero15} implies that $\gamma_w$ is contained in $\mathbb{K}(b,\mathcal{Q},0)$.
Since $\gamma_w$ is continuous, $\gamma_w(0)=v_i\in\supp(\mu)$ and $\Sigma(F)=\emptyset$, by Proposition \ref{spt2} we deduce that $\gamma_w(t)\in\supp(\mu)$ for any $t\geq0$.
We are left to compute the limit $\lim_{t\to\infty}D_{1/t}(\gamma(t))$. In the case of the vertical component the computation is immediate, indeed $\lim_{t\to\infty}(\gamma_w(t))_T/t^2=w_T$. The limit for the horizontal components is:
\begin{equation}
    \begin{split}
        \lim_{t\to\infty}\frac{(\gamma_w(t))_H}{t}&=\lim_{t\to\infty}\frac{\sqrt{\langle b,e_1\rangle^2-4\lambda_1\langle \mathcal{Q}s(t)+b,s(t)\rangle}-\langle b,e_1\rangle}{2\lambda_1t}e_1+P_1[w_H]\\
        &=\frac{\sqrt{-4\lambda_1\langle \mathcal{Q}[P_1(w_H)],P_1(w_H)\rangle}}{2\lambda_1}e_1+P_1[w_H]\\
        &\qquad\underset{\eqref{numbero16}}{=}\langle e_1,w_H\rangle e_1+P_1[w_H]=w_H,
        \nonumber
    \end{split}
\end{equation}
where in the second last identity we also used the fact that $\langle e_1,w_H\rangle>0$ so that no absolute value is introduced.
Thus $\lim_{t\to\infty}D_{1/t}(\gamma(t))=w$, which implies that: 
\begin{equation}
    \supp(\nu)\cap\{w:\langle w_H,e_1\rangle> 0\}\subseteq \supp(\xi)\qquad\text{for any $\xi\in\Tan_{2n+1}(\mu,\infty)$.}  
    \label{eq:eq:eq2}
\end{equation}
With very few changes, the above argument can be adapted to prove that the analogue inclusion to \eqref{eq:eq:eq2} holds for the points $\supp(\nu)\cap\{w:\langle w_H,e_1\rangle> 0\}$. These two inclusions together imply:
\begin{equation}
\supp(\nu)=\mathrm{cl}(\supp(\nu)\setminus V(e_1))\subseteq \supp(\xi)\quad\text{for any $\xi\in\Tan_{2n+1}(\mu,\infty)$,}
\label{eq:eq:eq3}
\end{equation}
where the first identity above comes from the fact that $\nu(V(e_1))=0$.
The same argument we used to infer the uniqueness of the tangent at infinity from \eqref{eq:eq:eq4} in the case $\Sigma(F)\neq \emptyset$ can be applied here verbatim, concluding the proof of the proposition.
\end{proof}

\subsection{Proof of Theorem \ref{disco}}
\label{flattitris}

The following lemma insures that if $\phi$ is a measure with $(2n+1)$-density, at $\phi$-almost every $x$ there is a flat measure contained in $\Tan_{2n+1}(\mu,x)$.

\begin{lemma}
If $\phi$ is a Radon measure with $(2n+1)$-density, then for $\phi$ almost every $x$ there exists a $\mathfrak{n}\in\R^{2n}$ for which:
$$\mathcal{C}^{2n+1}\llcorner{V(\mathfrak{n})}\in \Tan_{2n+1}(\phi,x).$$
\end{lemma}

\begin{proof}
Proposition \ref{preiss} and Proposition \ref{verticalsamoa2}  directly imply the claim.
\end{proof}

We are ready to finish the proof of Theorem \ref{disco}.
The argument we will use follows closely the proof of \cite[Proposition  6.10]{DeLellis2008RectifiableMeasures}. By contradiction, suppose there exists a point $x$ such that:
\begin{itemize}
\item[(i)] $\Tan_{2n+1}(\phi,x)\subseteq\Theta(\phi,x)\mathcal{U}(2n+1)$,
\item[(ii)] there are $\zeta,\nu\in\Tan_{2n+1}(\phi,x)$ such that $\nu$ is flat and $\zeta$ is not flat,
\item[(iii)] Proposition \ref{preiss} holds at $x$.
\end{itemize}
We can also assume without loss of generality that $\Theta(\phi,x)=1$. Since $\zeta$ is not flat, its tangent at infinity $\chi$ cannot be flat otherwise Theorem \ref{flat} would imply that $\zeta$ is flat. In particular, by the assumption on the functional $\mathscr{F}$ we have $\mathscr{F}(\chi)>\hbar$. Fix $r_k\to0$ and $s_k\to 0$ in such a way that:
$$\frac{\phi_{x,r_k}}{r_k^{2n+1}}\rightharpoonup \nu\qquad\text{and}\qquad\frac{\phi_{x,s_k}}{s_k^{2n+1}}\rightharpoonup \chi.$$
We can further suppose that $s_k<r_k$. Define for any $r>0$ the function
$f(r):=\mathscr{F}(r^{-(2n+1)}\phi_{x,r})$,
and note that since $\mathscr{F}$ is continuous with respect to the weak-$*$ convergence of measures, $f$ is continuous in $r$.
Since $\nu$ is flat, then:
$$\lim_{r_k\to 0}f(r_k)=\mathscr{F}(\nu)\leq \hbar/2,$$
Thus for sufficiently small $r_k$ we have $f(r_k)<\hbar$. On the other hand, since:
$$\lim_{s_k\to 0}f(s_k)=\mathscr{F}(\chi)> \hbar,$$
for suffciently small $s_k$ we have $f(s_k)>\hbar$. Fix $\sigma_k\in [s_k,r_k]$ such that $f(\sigma_k)=\hbar$ and $f(r)\leq\hbar$ for $r\in[\sigma_k,r_k]$. By compactness there exists a subsequence of $\{\sigma_k\}_{k\in\N}$, not relabeled, such that $\sigma_k^{-(2n+1)}\phi_{x,\sigma_k}$ converges weakly-$*$ to a measure $\xi\in \mathcal{U}(2n+1)$. Clearly by continuity:
$$\mathscr{F}(\xi)=\lim_{\sigma_k\to 0} f(\sigma_k)=\hbar.$$
Note that $r_k/\sigma_k\to \infty$
, otherwise if for some subsequence not relabeled, we had that $r_k/\sigma_k$ converged to a constant $C$ (larger than $1$) we would conclude that $\frac{\xi_{0,C}}{C^{2n+1}}=\nu$
since:
$$\nu=\lim_{k\to\infty}\frac{\phi_{x,r_k}}{r_k^{2n+1}}=\lim_{k\to \infty} \left(\frac{\sigma_k}{r_k}\right)^{2n+1}\left(\frac{\phi_{x,\sigma_k}}{\sigma_k^{2n+1}}\right)_{0,r_k/\sigma_k}.$$
In particular $\xi$ would be flat, which is not possible as $\mathscr{F}(\xi)=\hbar$. Note that for any given $R>0$ we have:
$$(R\sigma_k)^{-(2n+1)}\phi_{x,R\sigma_k}\rightharpoonup R^{-(2n+1)}\xi_{0,R}.$$
This by continuity of $\mathscr{F}$ implies that:
$$\mathscr{F}(R^{-(2n+1)}\xi_{0,R})=\lim_{k\to\infty}f(R\sigma_k).$$
Moreover, since $r_k/\sigma_k\to\infty$ we conclude that for any $R>1$ we have that $R\sigma_k\in[\sigma_k,r_k]$ whenever $k$ is large enough. This, by our choice of $\sigma_k$ and $r_k$, implies that:
\begin{equation}
    \mathscr{F}(R^{-(2n+1)}\xi_{0,R})=\lim_{k\to\infty}f(R\sigma_k)\leq \hbar,
    \label{eq1010}
\end{equation}
for every $R\geq 1$. Let $\psi$ be the tangent measure at infinity to $\xi$, which by Theorem \ref{conol} is unique and it is a cone. Thanks to \eqref{eq1010}, we infer on $\psi$ that:
$$\mathscr{F}(\psi)=\lim_{R\to\infty}\mathscr{F}(R^{-(2n+1)}\xi_{0,R})\leq \hbar,$$
and in particular thanks to the assumptions on the functional $\mathscr{F}$, we have that $\psi\in\mathfrak{M}(2n+1)$. This is in contradiction with Theorem \ref{flat}, which would imply that $\xi$ is flat.

\section{Disconnection of horizontal cones and flat measures}
\label{HORRI}
The main task of this section is to prove the existence of a constant $\mathfrak{C}_3(n)>0$ such that for any $\mathfrak{m}\in\mathbb{S}^{2n-1}$ and any horizontal $(2n+1)$-uniform cone $\mu$ we have:
\begin{equation}
\int_{B_1(0)} \langle \mathfrak{m},   z_H\rangle^2 d\mu(z)\geq\mathfrak{C}_3(n).
\label{numeroo30}
\end{equation}
The geometric meaning of \eqref{numeroo30} is that for every $\mathfrak{m}\in\R^{2n}$, the support of any horizontal $(2n+1)$-uniform cone $\mu$ must be quantitatively detached from $V(\mathfrak{m})$. Indeed, if we let $\mu$ as above and we suppose that there exists an $\mathfrak{n}\in \R^{2n}$ such that $\mathrm{dist}(B_1(0)\cap\supp(\mu),B_1(0)\cap V(\mathfrak{n}))< \sqrt{\mathfrak{C}_3(n)}$, then:
\begin{equation}
    \int_{B_1(0)} \langle \mathfrak{m},   z_H\rangle^2 d\mu(z)=\int_{B_1(0)} \mathrm{dist}(z,V(\mathfrak{n}))^2 d\mu(z)<\mathfrak{C}_3(n)\mu(B_1(0))=\mathfrak{C}_3(n),
    \nonumber
\end{equation}
where the first identity comes for instance from \cite[Proposition 1.5]{merloMM}. This however is in contradiction with \eqref{numeroo30}. This rigidity of horizontal $(2n+1)$-uniform cones together with Theorem \ref{disco} suggest that horizontal $(2n+1)$-uniform measures cannot be tangent measures to measures with $(2n+1)$-density. For a rigorous proof of this fact, we refer to Propositions \ref{conti} and \ref{consta}.

\subsection{Limits of sequences of horizontal \texorpdfstring{$(2n+1)$}{Lg}-uniform cones}\label{limhorcon}

For the sake of readability, we fix here some notation. \textbf{Throughout this subsection}, if not otherwise specified, \textbf{we will always suppose that}  $\{\mu_i\}_{i\in\N}$ \textbf{is a sequence of horizontal} $(2n+1)$-\textbf{uniform cones supported on the quadrics} $\mathbb{K}(0,\mathcal{D}_i,-1)$.
Since $\mathcal{U}(2n+1)$ is compact, \textbf{we can assume without loss of generality that the} $\mu_i$ \textbf{are weakly converging to some Radon measure} $\nu$\textbf{ which is a} $(2n+1)$\textbf{-uniform cone}, thanks to Proposition \ref{replica}. Finally, up to non-relabeled subsequences, \textbf{we can further suppose that the sequence} $-\mathcal{D}_i/\vertiii{\mathcal{D}_i}$, where here with $\vertiii{A}$ we denote the operator norm of the matrix $A$, \textbf{converges to some matrix }$\mathcal{Q}\in \mathrm{Sym}(2n)$ \textbf{with} $\vertiii{\mathcal{Q}}=1$.
\label{operatornorm}

\medskip

The plan of the subsection is the following. First we prove that the measure $\nu$ is vertical if and only if the sequence of the norms $\vertiii{\mathcal{D}_i}$ diverges. Secondly, we prove that the only possible vertical limits of sequences of horizontal $(2n+1)$-uniform cones are flat measures. This second step of the section is where the results of Appendix \ref{TYLR} come into play. Indeed, Theorem \ref{appendicefinale} applies to $\mu_i$ for any $i\in\N$ and it forces $\mathcal{D}_i$ to satisfy the following identity for any $z\in\R^{2n}$ for which $(\mathcal{D}_i+J)z\neq 0$:
\begin{equation}
   \begin{split}
 &\frac{\text{Tr}(\mathcal{D}_i^2)-2\langle \mathfrak{n}_i,\mathcal{D}^2_i\mathfrak{n}_i\rangle+\langle \mathfrak{n}_i,D_i\mathfrak{n}_i\rangle^2}{4(2n-1)}+\frac{n-1}{2n-1}-\frac{1}{4}\\
 &\qquad\qquad\qquad\qquad+\frac{\langle \mathcal{D}_i J\mathfrak{n}_i,\mathfrak{n}_i\rangle}{2n-1}-\frac{\left(\text{Tr}(\mathcal{D}_i)-\langle \mathfrak{n}_i,\mathcal{D}_i\mathfrak{n}_i\rangle\right)^2}{8(2n-1)}=0.
    \label{eq:1100}
\end{split} 
\end{equation}
where $\mathfrak{n}_i(z):=(\mathcal{D}_i+J)z/\lvert(\mathcal{D}_i+J)z\rvert$ is the so called \emph{horizontal normal} to the graph of $f(x)=\langle x,\mathcal{D}_i x\rangle$ at the point $(x,f(x))$.
Since identity \eqref{eq:1100} holds for any $i\in\N$, we can deduce some constraints, see Proposition \ref{equazione}, for the limit matrix  $\mathcal{Q}$. Using a similar argument we are able to prove in Proposition \ref{bddss} that the sequence of the second biggest (in modulus) eigenvalue of $\mathcal{D}_i$ is bounded. Putting together these pieces of information, we are able to prove in Proposition \ref{boundi} that there must exists a constant, depending only on $n$, which bounds the biggest eigenvalue of the matrices $\mathcal{D}_i$ whose associated quadrics $\mathbb{K}(0,\mathcal{D}_i,-1)$ support a horizontal $(2n+1)$-uniform measure.
\smallskip

The following proposition shows the equivalence between the geometric condition for $\nu$ to be a vertical $(2n+1)$-uniform cone and the algebraic condition on the divergence of the sequence $\big\{\vertiii{\mathcal{D}_i}\big\}_{i\in\N}$. This will be very useful in the forthcoming computations.

\begin{proposizione}\label{verticale2}
The following are equivalent:
\begin{itemize}
\item[(i)]$\nu$ is supported on $\mathbb{K}(0,\mathcal{Q},0)$,
\item[(ii)] $\lim_{i\to\infty}\vertiii{\mathcal{D}_i}=\infty$.
\end{itemize}
\end{proposizione}

\begin{proof} Proposition \ref{replica} implies that for any $y\in\supp(\nu)$ there exists a sequence $\{y_i\}_{i\in\N}$ such that $y_i\in\supp(\mu_i)$ and $y_i\to y$. Assume at first the sequence $\vertiii{\mathcal{D}_i}$ is bounded. This implies that we can find a (non-relabeled) subsequence $\{\mathcal{D}_i\}$ such that the matrices $\mathcal{D}_i$ converge to some $\mathcal{D}\in\mathrm{Sym}(2n)$. Thanks to our assumption on $y_i$ we know that $(y_i)_T=\langle (y_i)_H,\mathcal{D}_i [(y_i)_H]\rangle$. Thus, taking the limit as $i$ to infinity, we get:
$$  y_T=\langle y_H,\mathcal{D}y_H\rangle.$$
Therefore if $\vertiii{\mathcal{D}_i}$ is bounded, $\nu$ is supported on both the quadrics $\mathbb{K}(0,\mathcal{Q},0)$ and $\mathbb{K}(0,\mathcal{D},-1)$. This however is not possible thanks to Proposition \ref{verticale}, which implies that either $\mathcal{T}=0$ or $\mathcal{T}\neq 0$ for any quadric containing $\supp(\mu)$.

Viceversa, if we suppose that $\{\vertiii{\mathcal{D}_i}\}_{i\in\N}$ diverges, with some iterations of the triangle inequality, we deduce the following bound:
\begin{equation}
\begin{split}
\lvert \langle y_H, \mathcal{Q}y_H\rangle\rvert\leq& \vertiii{ \mathcal{Q}-\frac{D_i}{\vertiii{D_i}}}\lvert y_H\rvert^2+\lvert y_H-(y_i)_H\rvert(\lvert y_H\rvert+\lvert(y_i)_H\rvert)+\frac{\lvert(y_i)_T\rvert}{\vertiii{\mathcal{D}_i}}.
    \label{eq:3}
\end{split}
\end{equation}
The right-hand side of the above inequality goes to $0$ as $i\to \infty$ since we can assume without loss of generality that $\lVert y_i\rVert\leq 2\lVert y\rVert$. Therefore for any $y\in\supp(\nu)$ we have that $\langle y_H,\mathcal{Q}y_H\rangle=0$ and thus $\supp(\nu)\subseteq \mathbb{K}(0,\mathcal{Q},0)$.
\end{proof}

\begin{proposizione}\label{equivaut1}
Suppose that $\{\eta_i\}_{i\in\N}$ is a sequence of $(2n+1)$-uniform measures converging to some Radon measure $\xi$. Then:
\begin{itemize}
    \item[(i)]$\xi$ is $(2n+1)$-uniform,
    \item[(ii)] for any $U\in S(2n)$ the measures $(\Xi_U)_\#\eta_i$ are $(2n+1)$-uniform,
    \item[(iii)]the following are equivalent:
    \begin{itemize}
    \item[($\alpha$)] $\xi$ if flat.
    \item[($\beta$)] there exists a $U\in S(2n)$ and a flat measure $\xi_U$ such that $(\Xi_U)_\#\eta_i\rightharpoonup\xi_U$
    \item[($\gamma$)] for any $U\in S(2n)$ there exists a flat measure $\xi_U$ such that $(\Xi_U)_\#\eta_i\rightharpoonup\xi_U$,
\end{itemize}
\end{itemize}
where we recall that the maps $\Xi_U$ were introduced in \eqref{numbero10}.
\end{proposizione}

\begin{proof}
The fact that $\xi$ is a $(2n+1)$-uniform measure is an immediate consequence of Proposition \ref{UComp} and fact that $(\Xi_U)_\#\eta_i$ are $(2n+1)$-uniform measure for any $U\in S(2n)$ and any $i\in\N$ is a consequence of Proposition \ref{isometrie2}.

In order to prove that ($\alpha$) implies ($\gamma$) we let $\supp(\xi)=V(\mathfrak{n})$ for some $\mathfrak{n}\in\R^{2n}$. Since the maps $\Xi_U$ are a surjective continuous isometries of $\HH^n$ for any $U\in S(2n)$, see Proposition \ref{isometrie2}, an elementary computation shows that $(\Xi_U)_\#\eta_i\rightharpoonup (\Xi_U)_\#\xi$. Thanks to Proposition \ref{isometrie} and to the fact that $\Xi_U(0)=0$, we infer that:
$$\supp((\Xi_U)_\#\xi)=\Xi_U(V(\mathfrak{n}))=V(U^T\mathfrak{n}).$$
Finally, the above identity together with Proposition \ref{verticalsamoa} concludes the proof of the first implication.

The fact that ($\gamma$) implies ($\beta$) is obvious, so we are left to prove that  ($\beta$) implies ($\alpha$). Since $S(2n)$ is a group, we have that $U^{-1}=U^T\in S(2n)$, and thus we can apply the already proven implication ($\alpha$)$\Rightarrow$($\beta$) to the sequence of measures $(\Xi_U)_\#\eta_i$ to show that there exists a flat measure $\tilde{\xi}_{U^T}$ such that $(\Xi_{U^T})_\#((\Xi_U)_\#\eta_i)\rightharpoonup \tilde{\xi}_{U^T}$. An elementary computation shows that $(\Xi_{U^T})_\#((\Xi_U)_\#\eta_i)=\eta_i$ thus proving by the uniqueness of the weak limit that $\xi=\tilde{\xi}_{U^T}$ is flat.
\end{proof}

\textbf{From now on} and if not otherwise specified, \textbf{we shall always assume that} $\vertiii{\mathcal{D}_i}\to\infty$ or, rephrasing by means of Proposition \ref{verticale2}, that the limit measure $\nu$ is a vertical $(2n+1)$-uniform cone. In the following proposition we show how the quadric $\mathbb{K}(0, \mathcal{Q},0)$ ``remembers'' that the measure $\nu$ is the limit of a sequence of horizontal $(2n+1)$-uniform cones.

\begin{proposizione}\label{equazione}
For any $h\not\in\text{Ker}(\mathcal{Q})$ defined $\mathfrak{n}:=\mathcal{Q}h/\lvert\mathcal{Q}h\rvert$, we have:
\begin{equation}
\begin{split}
2(\text{Tr}(\mathcal{Q}^2)-2\langle &\mathfrak{n},\mathcal{Q}^2\mathfrak{n}\rangle+\langle \mathfrak{n},\mathcal{Q}\mathfrak{n}\rangle^2)-(\text{Tr}(\mathcal{Q})-\langle \mathfrak{n},\mathcal{Q}\mathfrak{n}\rangle)^2=0.
\label{eq:4}
\end{split}
\end{equation}
\end{proposizione}

\begin{proof}
First of all, since $h\not\in \text{ker}(\mathcal{Q})$, there exists $N\in\N$ for which for any $i\geq N$ we have:
\begin{equation}
  (\mathcal{D}_i+J)h\neq 0. 
  \label{numbero11}
\end{equation}
Indeed, if there was a (non-relabeled) subsequence of indeces for which $(\mathcal{D}_i+J)h= 0$,
dividing the above equality by $\vertiii{\mathcal{D}_i}$ and sending $i$ to infinity, we would deduce that $\mathcal{Q}h=0$. Secondly, defined
$\mathfrak{n}_i:=(\mathcal{D}_i+J)h/\lvert (\mathcal{D}_i+J)h\rvert$ we have:
$$\lim_{i\to\infty}\mathfrak{n}_i(h)= \mathcal{Q}h/\lvert\mathcal{Q}h\rvert=\mathfrak{n}.$$
Thanks to the above observations, Theorem \ref{appendicefinale} implies that for any $i\in\N$ for which \eqref{numbero11} holds, thus for any $i$ sufficiently big, we have:
\begin{equation}
   \begin{split}
 &\frac{\text{Tr}(\mathcal{D}_i^2)-2\langle \mathfrak{n}_i,\mathcal{D}^2_i\mathfrak{n}_i\rangle+\langle \mathfrak{n}_i,D_i\mathfrak{n}_i\rangle^2}{4(2n-1)}+\frac{n-1}{2n-1}-\frac{1}{4}\\&\qquad\qquad\qquad\qquad+\frac{\langle \mathcal{D}_i J\mathfrak{n}_i,\mathfrak{n}_i\rangle}{2n-1}-\frac{\left(\text{Tr}(\mathcal{D}_i)-\langle \mathfrak{n}_i,\mathcal{D}_i\mathfrak{n}_i\rangle\right)^2}{8(2n-1)}=0.
    \label{eq:10}
\end{split} 
\end{equation}
Dividing \eqref{eq:10} by $\vertiii{\mathcal{D}_i}^2$ and taking the limit as $i$ goes to infinity, we obtain \eqref{eq:4}.
\end{proof}

The following corollary is an easy application of Proposition \ref{equazione} in the case $h$ is a non-zero eigenvector of $\mathcal{Q}$.

\begin{corollario}\label{core}
Every non-zero eigenvalue $\lambda$ of $\mathcal{Q}$ satisfies the equation:
\begin{equation}
    -3\lambda^2+2\text{Tr}(\mathcal{Q})\lambda+(2\text{Tr}(\mathcal{Q}^2)-\text{Tr}(\mathcal{Q})^2)=0.
    \label{eq20}
\end{equation}
In particular $\mathcal{Q}$ has at most two distinct non-zero eigenvalues.
\end{corollario}

\begin{proof}
For any $h\in\R^{2n}$ unitary eigenvector relative to the eigenvalue $\lambda\neq 0$, we have $\mathfrak{n}(h)=\mathcal{Q}h/\lvert\mathcal{Q}h\rvert=\text{sgn}(\lambda)h$.
Thus, plugging the expression for $\mathfrak{n}(h)$ into \eqref{eq:4}, we infer:
\begin{equation}
\begin{split}
2(\text{Tr}(\mathcal{Q}^2)-2\lambda^2+\lambda^2)-(\text{Tr}(\mathcal{Q})-\lambda)^2=0,
\nonumber
\end{split}
\end{equation}
which collecting $\lambda$, proves the corollary.
\end{proof}






The following proposition shows that there is only one non-zero eigenvalue of $\mathcal{Q}$. In particular, this implies that $\mathcal{Q}$ is semidefinite and that $\nu$ is flat. 

\begin{proposizione}\label{omgflat} The matrix $\mathcal{Q}$ has rank $1$ and in particular the limit $\nu$ of the sequence $\{\mu_i\}$ is flat.
\end{proposizione}

\begin{proof}
Let us first assume that $1$ is an eigenvalue of $\mathcal{Q}$. Then,
Corollary \ref{core} implies that $1$ solves the equation \eqref{eq20} and in particular:
\begin{equation}
    -3+2\text{Tr}(\mathcal{Q})+2\text{Tr}(\mathcal{Q}^2)-\text{Tr}(\mathcal{Q})^2=0.
    \label{eq22}
\end{equation}
Using the above equality to express $\text{Tr}(\mathcal{Q}^2)$ in function of $\text{Tr}(\mathcal{Q})$ and substituting it back into \eqref{eq20}, we have:
\begin{equation}
    -3\lambda^2+2\text{Tr}(\mathcal{Q})\lambda+(3-2\text{Tr}(\mathcal{Q}))=0.
    \nonumber
\end{equation}
Solving the above quadratic equation, we find that its solutions are:
\begin{equation}
    \lambda_1=1\qquad\text{and}\qquad\lambda_2=2 \text{Tr}(\mathcal{Q})/3-1.
    \label{numbero12}
\end{equation}
If $\lambda_2\geq0$, the matrix $\mathcal{Q}$ would be semi-definite and since $\mathbb{K}(0,\mathcal{Q},0)$ supports a $(2n+1)$-uniform measure, Proposition \ref{rango} implies that $\text{rk}(\mathcal{Q})=1$ and Proposition \ref{rank1} concludes that $\mu$ must be flat. Therefore, we can assume without loss of generality that $\lambda_2<0$. Let now $k_1$ be the dimension of the eigenspace relative to the eigenvalue $\lambda_1=1$ and $k_2$ be the dimension of the eigenspace relative to the eigenvalue $\lambda_2$. By assumption $k_1\geq 1$ and  we can suppose without loss of generality that $k_2\geq 1$, otherwise $\mathcal{Q}$ would be semidefinite and Proposition \ref{rank1} and Proposition \ref{rango} again imply that $\mu$ is flat. Thanks to the second identity in \eqref{numbero12} we have:
$$\text{Tr}(\mathcal{Q})=k_1+\lambda_2 k_2=k_1+\frac{2 \text{Tr}(\mathcal{Q})k_2}{3}-k_2.$$
Solving for $\text{Tr}(\mathcal{Q})$ the above equation, we conclude that:
\begin{equation}
     \text{Tr}(\mathcal{Q})=\frac{3(k_1-k_2)}{3-2k_2}.
     \label{numero13.1}
\end{equation}
Thanks to the above identity it is possible to get an expression for $\lambda_2$ in function of $k_1$ and $k_2$, indeed:
\begin{equation}
    \lambda_2\underset{\eqref{numbero12}}{=}\frac{2 \text{Tr}(\mathcal{Q})}{3}-1\underset{\eqref{numero13.1}}{=}\frac{2(k_1-k_2)}{3-2k_2}-1 =\frac{2k_1-3}{3-2k_2}.
    \label{numeroo21}
\end{equation}
This allows us to obtain an expression also for $\mathrm{Tr}(\mathcal{Q}^2)$ depending only on $k_1$ and $k_2$:
    \begin{equation}
      \text{Tr}(\mathcal{Q}^2)=\lambda_1^2 k_1+\lambda_2^2 k_2=k_1+\left(\frac{2k_1-3}{3-2k_2}\right)^2k_2.
        \label{numbero13.2}
    \end{equation}
Substituting in \eqref{eq22} the above identities for $\mathrm{Tr}(\mathcal{Q})$ and $\mathrm{Tr}(\mathcal{Q}^2)$, we deduce that:
$$-3+2\frac{3(k_1-k_2)}{3-2k_2}+2k_1+2\left(\frac{2k_1-3}{3-2k_2}\right)^2k_2-\frac{9(k_1-k_2)^2}{(3-2k_2)^2}=0.$$
Recollecting terms, we can rearrange the above equality in the following fashion:
$$0=k_1^2(8k_2-9)+k_1(8k_2^2-42k_2+36)+(-9k_2^2+36k_2-27),$$
which allows us to express $k_1$ in terms of $k_2$. Indeed, the solutions of the above quadratic equation are:
$$k_1=\frac{9k_2-9}{8k_2-9}\qquad\text{ or }\qquad k_1=3-k_2.$$

There are two couples of natural numbers for which $k_1=\frac{9k_2-9}{8k_2-9}$ holds and they are $(1,0)$ and $(0,1)$. This can be proved showing by means of an elementary algebraic computation that $0<k_2/(8k_2-9)<1/2$ for any $k_2\geq2$, which in turn implies
$1<\frac{9k_2-9}{8k_2-9}=1+\frac{k_2}{8k_2-9}<3/2$. This shows that if $k_2\geq 2$, $k_1$ cannot be an integer.
However, neither $(1,0)$ and $(0,1)$ satisfy the constraint that both $k_1$ and $k_2$ must be bigger than $1$.

On the other hand, there are four couples of natural numbers satisfying the relation $k_1=3-k_2$ and such couples are $(3,0);(2,1);(2,1);(0,3)$. However, the only couples of natural numbers for which $\mathcal{Q}$ is not semi-defined are $(1,2)$ and $(2,1)$. In both these cases \eqref{numero13.1} implies that $\text{Tr}(\mathcal{Q})=3$ and therefore $\lambda_2=1$ by \eqref{numeroo21}, which is in contradiction with the assumption that $\lambda_2<0$.

Let us now suppose that $1$ is \emph{not} an eigenvalue of $\mathcal{Q}$. Thanks to Proposition \ref{equivaut1}, if we are able to exhibit an $U\in S(2n)$ such that the sequence $(\Xi_U)_\#\mu_i$ converges to some flat measure then $\nu$ itself must be flat. Therefore, thanks to the first part of the proposition is sufficient to prove that:
\begin{enumerate}
\item the measures $(\Xi_U)_\#\mu_i$ converge to some $(2n+1)$-uniform measure $\nu_U$,
\item $(\Xi_U)_\#\mu_i$ are supported on quadrics $\mathbb{K}(0,\mathcal{D}^\prime_i,-1)$ where $\vertiii{\mathcal{D}^\prime_i}\to \infty$,
    \item there exists a symmetric matrix $\mathcal{Q}^\prime$ such that
    $\lim_{i\to\infty}-\mathcal{D}^\prime_i/\vertiii{\mathcal{D}^\prime_i}=\mathcal{Q}^\prime$,
    \item $1$ is an eigenvalue of $\mathcal{Q}^\prime$.
\end{enumerate}

Let us show that any $U\in\text{S}(2n)$ for which $s(U)=-1$, where $s(\cdot)$ was defined after \eqref{inex}, satisfies the requirements above. By Proposition \ref{equivaut1} the measures $(\Xi_U)_\#\mu_i$ are $(2n+1)$-uniform and it is not hard to see that $(\Xi_U)_\#\mu_i\rightharpoonup (\Xi_U)_\#\nu$. Note that by Proposition \ref{isometrie2}, the measure $(\Xi_U)_\#\nu$ is $(2n+1)$-uniform. Furthermore, by Proposition \ref{isometrie2}, we have:
\begin{equation}
\begin{split}
       \supp((\Xi_U)_\#\mu_i)\subseteq&\Xi_U\big(\mathbb{K}(0,\mathcal{D}_i,-1)\big)=\mathbb{K}(0,U\mathcal{D}_iU^T,1)\\
       =&\mathbb{K}(0,-U\mathcal{D}_iU^T,-1)=\mathbb{K}(0,\mathcal{D}_i^\prime,-1),
        \label{eq:1}
\end{split}
    \end{equation}
where  $\mathcal{D}_i^\prime:=-U\mathcal{D}_iU^T$. Since the matrix $U$ is by construction an Euclidean isometry of $\R^{2n}$, the operator norm of $\mathcal{D}_i^\prime$ must coincide with the one of $\mathcal{D}_i$ and thus:
$$\lim_{i\to\infty}-\mathcal{D}_i^\prime/\vertiii{\mathcal{D}_i^\prime}=\lim_{i\to\infty}U\mathcal{D}_iU^T/\vertiii{\mathcal{D}_i}=-U\mathcal{Q}U^T=\mathcal{Q}^\prime.$$
Note that if $v$ is an eigenvector of $\mathcal{Q}$ relative to the eigenvalue $-1$, that must exists since $\vertiii{Q}=1$, then the vector $Uv$ is an eigenvalue of $\mathcal{Q}^\prime$ and:
$$\mathcal{Q}^\prime (Uv)=-U\mathcal{Q}U^T(Uv)=-U\mathcal{Q}v=Uv.$$
This concludes the proof that conditions 1. to 4. are satisfied and thus thanks to the first part of the proposition $\nu_U$ is flat. Hence, even if $1$ is not an eigenvalue of $\mathcal{Q}$, the measure $\nu$ must be flat.
\end{proof}

We introduce here further notation. For any $i\in\N$ we let $\lambda_i(1),\ldots,\lambda_{i}(2n)$ be the eigenvalues of $\mathcal{D}_i$. Such eigenvalues are ordered in the following way:
$$\vertiii{\mathcal{D}_i}=\lvert \lambda_i(1)\rvert\geq\lvert\lambda_i(2)\rvert\geq\ldots\geq\lvert \lambda_{i}(2n)\rvert.$$
We further let $\{e_i(1),\ldots, e_{i}(2n)\}$ be an orthonormal basis of $\R^{2n}$ for which $e_i(j)$ is an eigenvector relative to the eigenvalue $\lambda_i(j)$. 

The following result will allow us to show in Proposition \ref{boundi} that $\nu$ must be a horizontal $(2n+1)$-uniform cone. If this is not the case, we know by Proposition \ref{omgflat} that $\supp(\nu)$ is a vertical hyperplane $V$ and in particular that the limit of the quadrics $\mathbb{K}(0,\mathcal{D}_i,-1)$ must contain $V$, see Proposition \ref{replica}. Furthermore, Proposition \ref{bddss} shows that while $\lambda_1(j)$ is diverging, all the other eigenvalues remain bounded. However, as we shall see in Proposition \ref{boundi}, this implies that the limit of the $\mathbb{K}(0,\mathcal{D}_i,-1)$ cannot contain any vertical hyperplane.

\begin{proposizione}\label{bddss}
Let the sequence of matrices $\{\mathcal{D}_i\}$ be as above, i.e.:
\begin{itemize}
    \item[(i)] $\mathbb{K}(0,\mathcal{D}_i,-1)$ for any $i\in\N$ supports a $(2n+1)$-uniform cone $\mu_i$ which converges to some $\nu$,
    \item[(ii)]$\lim_{i\to\infty}\vertiii{\mathcal{D}_i}=\lim_{i\to\infty}\lvert\lambda_i(1)\rvert=\infty.$
\end{itemize}
Then, the sequence $\{\lambda_i(2)\}_{i\in\N}$ is bounded.
\end{proposizione}

\begin{proof}
By contradiction assume that $\lvert\lambda_i(2)\rvert\to\infty$ and let $e_i:=e_i(2)$ and $\lambda_i:=\lambda_i(2)$. We want to apply Theorem \ref{appendicefinale} to the points $h=e_i$. In order to do so we need to give an explicit expression to the quantities involved in \eqref{eq16} of Theorem \ref{appendicefinale}. First we compute the horizontal normal of $\mathbb{K}(0,\mathcal{D}_i,-1)$ at the point  $(e_i,\langle e_i,\mathcal{D}_ie_i\rangle)$:
\begin{equation}
    \mathfrak{n}_i:=\mathfrak{n}(e_i):=\frac{\mathcal{D}_ie_i+Je_i}{\lvert\mathcal{D}_ie_i+Je_i\rvert}=\frac{\lambda_ie_i+Je_i}{\lvert\lambda_ie_i+Je_i\rvert}=\frac{\lambda_ie_i+Je_i}{(1+\lambda_i^2)^\frac{1}{2}}
    \nonumber
\end{equation}
Secondly we compute the quantities $\langle\mathfrak{n}_i,\mathcal{D}_i\mathfrak{n}_i\rangle$, $\langle\mathfrak{n}_i,\mathcal{D}_i^2\mathfrak{n}_i\rangle$ and $\langle \mathfrak{n}_i,\mathcal{D}_iJ\mathfrak{n}_i\rangle$:
\begin{equation}
\begin{split}
    \langle \mathfrak{n}_i,\mathcal{D}_i^k\mathfrak{n}_i\rangle=&\left\langle\frac{\lambda_ie_i+Je_i}{(1+\lambda_i^2)^\frac{1}{2}}, D_i^k\left(\frac{\lambda_ie_i+Je_i}{(1+\lambda_i^2)^\frac{1}{2}}\right)\right\rangle=\lambda_i^k-\frac{\lambda_i^{k}+\langle e_i,J\mathcal{D}_i^kJe_i\rangle}{1+\lambda_i^2},\\
    \langle \mathfrak{n}_i,\mathcal{D}_iJ\mathfrak{n}_i\rangle=&\left\langle\frac{\lambda_ie_i+Je_i}{(1+\lambda_i^2)^\frac{1}{2}}, D_i J\left(\frac{\lambda_ie_i+Je_i}{(1+\lambda_i^2)^\frac{1}{2}}\right)\right\rangle=-\frac{\lambda_i^2+\lambda_i\langle e_i,J\mathcal{D}_iJ e_i\rangle}{1+\lambda_i^2}.
    \end{split}
    \label{eq1002}
\end{equation}
In order to further simplify the notation, we define:
\begin{equation}
    \begin{split}
        (I)_i:=&\text{Tr}(\mathcal{D}_i^2)-2\langle\mathfrak{n}_i,\mathcal{D}_i^2\mathfrak{n}_i\rangle+\langle\mathfrak{n}_i,\mathcal{D}_i\mathfrak{n}_i\rangle^2,\\
    (II)_i:=&\text{Tr}(\mathcal{D}_i)-\langle \mathfrak{n}_i,\mathcal{D}_i\mathfrak{n}_i\rangle.
        \nonumber
    \end{split}
\end{equation}
With these notations, Theorem \ref{appendicefinale} applied to $h=e_i$ turns into:
\begin{equation}
\begin{split}
\frac{(I)_i}{4(2n-1)}+\frac{n-1}{2n-1}
        -\frac{1}{4}-\frac{\lambda_i^2+\lambda_i\langle e_i,J\mathcal{D}_iJ e_i\rangle}{(2n+1)(1+\lambda_i^2)}-\frac{(II)_i^2}{8(2n-1)}=0.
        \label{numbero17}
\end{split}
\end{equation}
We give now a more explicit description of both $(I)_i$ and $(II)_i$ using the expressions for $\langle\mathfrak{n}_i,\mathcal{D}_i\mathfrak{n}_i\rangle$, $\langle\mathfrak{n}_i,\mathcal{D}_i^2\mathfrak{n}_i\rangle$ and $\langle \mathfrak{n}_i,\mathcal{D}_iJ\mathfrak{n}_i\rangle$ we found in \eqref{eq1002}:
\begin{equation}
\begin{split}
    (I)_i=&\sum_{j=1}^{2n}\lambda_i(j)^2-2\lambda_i^2+2\frac{\lambda_i^{2}+\langle e_i,J\mathcal{D}_i^2 Je_i\rangle}{1+\lambda_i^2}+\left(\lambda_i-\frac{\lambda_i+\langle e_i,J\mathcal{D}_i Je_i\rangle}{1+\lambda_i^2}\right)^2\\
    =&\sum_{j\neq 2}\lambda_i(j)^2+2\frac{\lambda_i^{2}+\langle e_i,J\mathcal{D}_i^2 Je_i\rangle}{1+\lambda_i^2}\\
    &\qquad\qquad\,\,-2\lambda_i\frac{\lambda_i+\langle e_i,J\mathcal{D}_i Je_i\rangle}{1+\lambda_i^2}+\frac{(\lambda_i
    +\langle e_i,J\mathcal{D}_i Je_i\rangle)^2}{(1+\lambda_i^2)^2},\\
(II)_i=&\sum_{j\neq 2}\lambda_i(j)+\frac{\lambda_i+\langle e_i,J\mathcal{D}_iJe_i\rangle}{1+\lambda_i^2}.
\label{eq1003}
\end{split}
\end{equation}
The absurd assumption that $\lvert\lambda_i\rvert\to\infty$ has the following consequences. First of all:
\begin{equation}
    \lim_{i\to\infty} \frac{\langle e_i,J(\mathcal{D}_i/\vertiii{\mathcal{D}_i})^kJ e_i\rangle}{(1+\lambda_i^2)}=0,\qquad\text{for any $k\in\N$.}
    \label{numbero168}
\end{equation}
Secondly, thanks to \eqref{eq1003} and \eqref{numbero168} we deduce that:
\begin{equation}
    \lim_{i\to\infty}\frac{(I)_i}{\vertiii{\mathcal{D}_i}^2}=1\qquad\text{and}\qquad\lim_{i\to\infty} \frac{(II)_i}{\vertiii{\mathcal{D}_i}}=1,
   \label{eq21}
\end{equation}
where we used the fact that by definition $\lvert\lambda_i(1)\rvert=\vertiii{\mathcal{D}_i}$ and that by Proposition \ref{omgflat} we have $\lim_{i\to \infty}\lambda_i(j)/\vertiii{\mathcal{D}_i}=0$ for any $j\in\{2,\ldots,2n\}$.
Dividing by $\vertiii{\mathcal{D}_i}^2$ the left-hand side of the identity \eqref{numbero17}, and sending $i\to\infty$, \eqref{eq21} yields:
\begin{equation}
    \begin{split}
        &\lim_{i\to\infty}\frac{(I)_i/\vertiii{\mathcal{D}_i}^2}{4(2n-1)}-\frac{((II)_i/\vertiii{\mathcal{D}_i})^2}{8(2n-1)}+ \frac{\left(\frac{n-1}{2n-1}
        -\frac{1}{4}\right)}{\vertiii{\mathcal{D}_i}^2}\\
        &\qquad\qquad-\lim_{i\to\infty}\frac{(\lambda_i/\lVert\mathcal{D}_i\Vert)^2+(\lambda_i/\vertiii{\mathcal{D}_i})\langle e_i,J(\mathcal{D}_i/\vertiii{\mathcal{D}_i})J e_i\rangle}{(2n+1)(1+\lambda_i^2)}=\frac{1}{8(2n-1)},
        \nonumber
    \end{split}
\end{equation}
which shows that the assumption $\lvert\lambda_i\rvert\to\infty$ is absurd thanks to fact that identity \eqref{numbero17} holds for every $i\in\N$.
\end{proof}

\subsection{Proof of the existence of the constant \texorpdfstring{$\mathfrak{C}_3(\mathfrak{n})$}{Lg}}

This subsection is devoted to the proof of Theorem \ref{bellalei} that is the main result of this section.
The first step towards this direction is the following proposition whose heuristic meaning is that horizontal $(2n+1)$-uniform cones must be quantitatively close to the horizontal plane $\{z:z_T=0\}$.

\begin{proposizione}\label{boundi}
There exists a constant $\mathfrak{C}_1(n)>0$ such that if $\mathbb{K}(0,\mathcal{D},-1)$ supports a $(2n+1)$-uniform cone, then $\vertiii{\mathcal{D}}\leq \mathfrak{C}_1(n)$.
\end{proposizione}

\begin{proof}
By contradiction assume that there exists a sequence $\{\eta_i\}_{i\in\N}$ of $(2n+1)$-uniform cones such that the quadric $\mathbb{K}(0,\mathcal{D}_i,-1)$ supports $\eta_i$ and $\vertiii{\mathcal{D}_i}\to\infty$. By Proposition \ref{UComp} we can assume, up to non-relabeled subsequences, that the measures $\eta_i$ converge to some $(2n+1)$-uniform measure $\xi$, that is also a cone thanks to Proposition \ref{equivaut1}. 
Furthermore, thanks to Propositions \ref{verticale2} and \ref{omgflat} we deduce that the limit $\xi$ is a flat measure and in order to fix notations we let $\mathfrak{n}\in\R^{2n}$ be the vector for which $\supp(\xi)=V(\mathfrak{n})$. 
Finally, Proposition \ref{bddss} implies that with the exception of $\lambda_1(i)$, the other eigenvalues are bounded in modulus by some constant $\mathfrak{C}>0$, that a priori depends on the sequence $\{\mu_i\}$. Let us see how this is in contradiction with $\xi$ being flat.

Up to passing to a non-relabeled subsequence, we have two possibilities: either $\lambda_i(1)\to\infty$ or $\lambda_i(1)\to-\infty$.
Let us suppose that $\lambda_i(1)\to\infty$ and note that for any $y\in\supp(\xi)$, Proposition \ref{replica} implies that we can find a sequence $\{y_i\}\subseteq \HH^n$ such that $y_i\in\supp(\mu_i)$ and $y_i\to y$. Therefore, for such a sequence $\{y_i\}$, we have:
\begin{equation}
    \begin{split}
           (y_i)_T=\langle (y_i)_H,\mathcal{D}_i (y_i)_H\rangle=&\sum_{j=1}^{2n} \lambda_i(j)\langle e_i(j),(y_i)_H\rangle^2\\
        \geq& \sum_{j=2}^{2n}\lambda_i(j)\langle e_i(j),(y_i)_H\rangle^2\geq -(2n-1)\mathfrak{C}\lvert(y_i)_H\rvert^2.
        \nonumber
    \end{split}
\end{equation}
Sending $i$ to infinity we conclude that for any $y\in\supp(\xi)$ we have:
\begin{equation}
y_T\geq -(2n-1)\mathfrak{C}\lvert y_H\rvert^2
\label{numeroo22}
\end{equation}
However, since the vector $(0,-1)\in\R^{2n}\times \R$ belongs to $V(\mathfrak{n})=\supp(\xi)$ but does not satisfy \eqref{numeroo22}, we get a contradiction, proving that the assumption $\lambda_i(1)\to\infty$ is absurd. The same argument excludes also the case in which $\lambda_i(1)\to-\infty$, proving that the assumption $\vertiii{\mathcal{D}_i}\to \infty$ is an absurd assumption.

Finally, the fact that the constant $\mathfrak{C}_1(\mathfrak{n})$ does not depend on the specific sequence $\{\eta_i\}_{i\in\N}$, can be shown with the usual diagonal argument.
\end{proof}

Not only $(2n+1)$-uniform cones cannot be too far from the horizontal plane  $\{z:z_T=0\}$ but also they cannot be too close to it, as the following proposition shows:

\begin{proposizione}\label{staccato}
There exists a constant $\mathfrak{C}_2(n)>0$ such that if $\mathbb{K}(0,\mathcal{D},-1)$ supports a $(2n+1)$-uniform cone, then $\vertiii{\mathcal{D}}\geq \mathfrak{C}_2(n)$.
\end{proposizione}

\begin{proof}
First of all, we list some inequalities that will be helpful in the following and whose elementary proof we omit:
 \begin{equation}
 \begin{split}
       \text{Tr}(\mathcal{D}^2)\leq& 2n\vertiii{\mathcal{D}}^2,\qquad\,\,\,\quad \lvert\text{Tr}(\mathcal{D})\rvert\leq 2n\vertiii{\mathcal{D}},\\
     \lvert\langle v,\mathcal{D}^kv\rangle\rvert\leq & \vertiii{\mathcal{D}}^k\lvert v\rvert^2,\qquad \lvert\langle v,DJv\rangle\rvert\leq \vertiii{\mathcal{D}}\lvert v\rvert^2, 
      \label{numeroo23}
 \end{split}
 \end{equation}
for any $v\in\mathbb{S}^{2n-1}$. Thanks to the estimates in \eqref{numeroo23} and to the triangle inequality, it is not hard to show that:
\begin{equation}
\begin{split}
&\Big\lvert\frac{\text{Tr}(\mathcal{D}^2)-2\langle v,\mathcal{D}^2v\rangle+\langle v,\mathcal{D}v\rangle^2}{4(2n-1)}+\frac{\langle \mathcal{D}Jv,v\rangle}{2n-1}-\frac{\left(\text{Tr}(\mathcal{D})-\langle v,\mathcal{D}v\rangle\right)^2}{8(2n-1)}\Big\rvert\\
&\qquad\qquad\qquad\qquad\qquad\qquad\qquad\qquad\leq\frac{(4n^2+8n+7)\vertiii{\mathcal{D}}^2+8\vertiii{\mathcal{D}}}{8(2n-1)}.
        \label{numeroo24}
\end{split}
\end{equation}
for any $v\in\mathbb{S}^{2n-1}$. Furthermore, if we let $e$ be an eigenvector of $\mathcal{D}$, since $\mathbb{K}(0,\mathcal{D},-1)$ supports a $(2n+1)$-uniform cone, identity \eqref{eq16} in Theorem \ref{appendicefinale} must be satisfied at $h=e$\footnote{Theorem \ref{appendicefinale} can be applied at $h=e$ since $(\mathcal{D}+J)e=\lambda e+Je\neq 0$}, and in particular with the choice $v=(\mathcal{D}+J)e/\lvert (\mathcal{D}+J)e\rvert$ we have:
\begin{equation}
   \Big\lvert \frac{n-1}{2n-1}-\frac{1}{4}\Big\rvert\leq\frac{(4n^2+8n+7)\vertiii{\mathcal{D}}^2+8\vertiii{\mathcal{D}}}{8(2n-1)}.
    \label{numbero22}
\end{equation}
With some algebraic computations, which we omit, we see that the positive solutions in $\vertiii{D}$ to \eqref{numbero22}
are contained in the interval $[1/\sqrt{4n^2+8n+7},\infty)$ and thus $\vertiii{D}\geq 1/\sqrt{4n^2+8n+7}=:\mathfrak{C}_2(n)$.
\end{proof}

Finally, putting together Propositions \ref{boundi} and \ref{staccato} we are in position to prove:

\begin{teorema}\label{bellalei}
There exists a constant $\mathfrak{C}_3(n)>0$ such that for any $\mathfrak{m}\in\mathbb{S}^{2n-1}$ and any horizontal $(2n+1)$-uniform cone $\mu$ we have:
$$\int_{B_1(0)} \langle \mathfrak{m},   z_H\rangle^2 d\mu(z)\geq\mathfrak{C}_3(n).$$
\end{teorema}

\begin{proof}
Since $\mu$ is a $(2n+1)$-horizontal cone, by Proposition \ref{CONO} we have that $\supp(\mu)\subseteq\mathbb{K}(0,\mathcal{Q},\mathcal{T})$. Furthermore, since $\mathcal{T}\neq 0$, as $\mu$ is supposed to be horizontal, defined $\mathcal{D}:=-\mathcal{Q}/\mathcal{T}$ we conclude that $\supp(\mu)\subseteq\mathbb{K}(0,\mathcal{D},-1)$. In addition, thanks to Propositions \ref{supportoK} and \ref{rapperhor} for any positive Borel function $h:\HH^n\to\R$ we have:
\begin{equation}
    \begin{split}
        \int h(z)d\mu(z)=&\int h(z) d\mathcal{C}^{2n+1}\llcorner{\supp(\mu)}(z)\\
        =&\frac{1}{\mathfrak{c}_n}\int_{\pi_H(\supp(\mu))} h\big((y,\langle y,\mathcal{D}y\rangle)\big)\lvert (\mathcal{D}+J)y \rvert dy.
        \nonumber
    \end{split}
\end{equation}
If $n>1$, Proposition \ref{spt1} implies that $\pi_H(\supp(\mu))=\R^{2n}$, and thus:
\begin{equation}
    \begin{split}
        &\mathfrak{c}_n\int_{B_1(0)}\langle \mathfrak{m},   z_H\rangle^2 d\mu(z)= \int_{\lvert y\rvert^4+\langle y,\mathcal{D}y\rangle^2\leq 1} \langle \mathfrak{m}, y\rangle^2\lvert (\mathcal{D}+J)y \rvert dy\\
        &\quad=\int_{\mathbb{S}^{2n-1}} \langle \mathfrak{m},v\rangle^2\lvert(\mathcal{D}+J)v\rvert \int_0^{(1+\langle v,\mathcal{D} v\rangle^2)^{-1/4}} r^{2n+2} dr d\sigma(v)\\
        &\qquad\qquad=\frac{1}{(2n+3)}\int_{\mathbb{S}^{2n-1}}\frac{\langle \mathfrak{m},v\rangle^2\lvert (\mathcal{D}+J)v\rvert^2}{(1+\langle v,\mathcal{D} v\rangle^2)^\frac{2n+3}{4}}d\sigma(v),
        \label{eq:105}
    \end{split}
\end{equation}
where $\sigma:=\mathcal{H}^{2n-1}_{eu}\llcorner \mathbb{S}^{2n-1}$ and in the second equality we performed the change of variables $y=rv$ in $\R^{2n}$. Furthermore, since $\mathbb{K}(0,\mathcal{D},-1)$ supports a uniform measure, Proposition \ref{boundi} implies that $\vertiii{\mathcal{D}}\leq\mathfrak{C}_1(n)$, and thus thanks to \eqref{eq:105}, we infer:
\begin{equation}
\int_{B_1(0)} \langle \mathfrak{m},   z_H\rangle^2 d\mu(z)\geq \frac{ \int_{\mathbb{S}^{2n-1}}\langle \mathfrak{m},v\rangle^2\lvert(\mathcal{D}+J)v\rvert d\sigma(v)}{(2n+3)\mathfrak{c}_n(1+\mathfrak{C}_1(n)^2)^\frac{2n+3}{4}}.
\label{numbero23}
\end{equation}
Suppose $e\in\mathbb{S}^{2n-1}$ is the vector at which $\mathcal{D}$ attains the operator norm, i.e. $\lvert \mathcal{D}e\rvert=\vertiii{\mathcal{D}}$. Then, $e$ is an eigenvector of $\mathcal{D}$ and thus:
\begin{equation}
    \lvert (\mathcal{D}+J)e\rvert^2=\lvert De\rvert^2+\langle De,Je\rangle+\lvert Je\rvert^2=\vertiii{D}^2+1,
    \label{numeroo31}
\end{equation}
where the last identity comes from the fact that $De$ and $e$ are parallel.
Therefore, for any $u\in\mathbb{S}^{2n-1}$ for which $\lvert u-e\rvert\leq 1/8$, we have:
\begin{equation}
    \begin{split}
        \lvert (\mathcal{D}+J)u\rvert^2\geq&\lvert(\mathcal{D}+J)e\rvert^2+2\langle(\mathcal{D}+J)e, (\mathcal{D}+J)(u-e)\rangle\\
        \geq&(\vertiii{\mathcal{D}}^2+1)-2\vertiii{\mathcal{D}+J}^2\lvert u-e\rvert\\
        \geq&(\vertiii{\mathcal{D}}^2+1)(1-4\lvert u-e\rvert)      \geq(\mathfrak{C}_2(n)^2+1)/2=:\mathfrak{C}_4(n),
        \label{eq:106}
    \end{split}
\end{equation}
where in the second last inequality we used the Jensen inequality $\vertiii{\mathcal{D}+J}^2\leq2(\vertiii{\mathcal{D}}^2+1)$ and in the last one the bound on $\vertiii{\mathcal{D}}$ yielded by Proposition \ref{staccato}. Putting together \eqref{eq:105} and \eqref{eq:106} we deduce that:
\begin{equation}
\begin{split}
     \int_{B_1(0)} \langle \mathfrak{m},   z_H\rangle^2 d\mu(z)\geq&\frac{\mathfrak{C}_4(n)\int_{\mathbb{S}^{2n-1}\cap U_{1/8}(e)}\langle \mathfrak{m},v\rangle^2 d\sigma(v)}{(2n+3)\mathfrak{c}_n(1+\mathfrak{C}_1(n)^2)^\frac{2n+3}{4}}\\
    \geq&\frac{\mathfrak{C}_4(n)\mathfrak{C}_5(n)}{(2n+3)\mathfrak{c}_n(1+\mathfrak{C}_1(n)^2)^\frac{2n+3}{4}}=:\mathfrak{C}_3(n),
    \label{numbero1004}
\end{split}
\end{equation}
where $\mathfrak{C}_5(n):=\min_{\mathfrak{m}\in\mathbb{S}^{2n-1}}\int_{\mathbb{S}^{2n-1}\cap U_{1/8}(e)}\langle \mathfrak{m},v\rangle^2 d\sigma(v)$ and as usual $U_{1/8}(e)$ denotes the Euclidean ball of radius $1/8$ and center $e$ in $\R^{2n}$.

If $n=1$, Proposition \ref{spt1} gives us two cases. Either the set $\Sigma(f)=\{x\in\R^2:(\mathcal{D}+J)x=0\}$ is $0$-dimensional, and in this case $\pi_H(\supp(\mu))=\R^2$, or it is $1$-dimensional and thus $\pi_H(\supp(\mu))$ may be just a one of the two half spaces with boundary $\Sigma(f)$.

In either cases, if $\pi_H(\supp(\mu))=\R^2$, then the above argument can be applied verbatim. If on the other hand $\pi_H(\supp(\mu))$ is just one half space, we must be careful to chose the eigenvector $e$ in the half space coinciding with $\pi_H(\supp(\mu))$. This can always been done since, as shown by \eqref{numeroo31}, we have $\pm e\not\in\Sigma(f)$ and thus $e$ must lie on one side of $\Sigma(f)$ and $-e$ on the other. Hence, at least one half of the ball $U_{1/8}$ is contained in the half plane $\pi(\supp(\mu))$ and as a consequence the computation in \eqref{numbero1004} can be performed similarly even in this case. This concludes the proof.
\end{proof}

\section{Disconnection of vertical non-flat cones and flat measures}

This section is devoted to prove that non-flat vertical $(2n+1)$-uniform cones are quantitatively disconnected from flat measures. To be precise, we prove that there is a universal constant $\mathfrak{C}_{10}(n)>0$ such that, if $\mu$ is a vertical $(2n+1)$-uniform cone and:
\begin{equation}
    \min_{\mathfrak{m}\in\mathbb{S}^{2n-1}}\int_{B_1(0)}\langle    \mathfrak{m},z_H\rangle^2 d\mu(z)\leq \mathfrak{C}_{10}(n),
    \label{eq:141}
\end{equation}
then $\mu$ is flat. The first step towards the proof is to use the study of the support of vertical $(2n+1)$-uniform cones carried on in Subsection \ref{zero} and the representation formulas of the perimeter given in Appendix \ref{appeA1} to obtain the following more explicit expression for the quadric containing $\mu$, see Proposition \ref{SUPPORTO}. To be precise, we will show that for any $w\in\supp(\mu)$, we have that:
\begin{equation}
    \lvert  w_H\rvert^2-(2n-1)\fint \langle   w_H,u\rangle^2 d\omega_\mu(u)=0,
    \label{eq:130}
\end{equation} 
where $\omega_\mu:=\mathcal{H}^{2n-2}_{eu}\llcorner \pi_H(\supp(\mu))\cap \mathbb{S}^{2n-1}$. The existence of $\mathfrak{C}_{10}(n)$, as showed in Theorem \ref{fine}, is a direct consequence of \eqref{eq:130} by means of few algebraic manipulations.

\smallskip

The following technical lemma is a consequence of the coarea formula and the representation formulas for the intrinsic perimeter we proved in Appendix \ref{vertic}.

\begin{lemma}\label{RAPP2}
Let $f:\R\to\R$ and $g:\R^{2n}\to\R$ be positive Borel functions and $\mu$ a vertical $(2n+1)$-uniform cone. Defined $\omega_\mu:=\mathcal{H}^{2n-2}_{eu}\llcorner \mathbb{S}^{2n-1}\cap \pi_H(\supp(\mu))$, we have that:
\begin{itemize}
    \item[(i)]$\int_{B_1(0)}g(   z_H)d\mu(z)=\frac{2}{\mathfrak{c}_n}\int_0^1 r^{2n-2}\sqrt{1-r^4}\int g(r v)d\omega_\mu(v)dr$,
    \item[(ii)]$\int f(z_T)g(   z_H)d\mu(z)=\frac{1}{\mathfrak{c}_n}\int f(t)dt\cdot\int_0^\infty r^{2n-2}\big(\int g(r v)d\omega_\mu(v)\big)dr$.
\end{itemize}
\end{lemma}

\begin{proof}
Since $\mu$ is supposed to be a cone, Proposition \ref{CONO} and Corollary \ref{mainprimo} imply that there exists a non-null symmetric matrix $\mathcal{D}$ such that $\supp(\mu)\subseteq \mathbb{K}(0,\mathcal{D},0)$. In addition, since the dilations $D_\lambda$ behave like the multiplication by $\lambda$ on the first $2n$ coordinates, the set $S:=\pi_H(\supp(\mu))$ is an Euclidean cone in $\R^{2n}$. Thanks to its definition, $S$ coincides with the set of isotropic vectors in $\R^{2n}$ of the matrix $\mathcal{D}$ and thus by \cite[Theorem 3.4.8]{Federer1996GeometricTheory} and Lemma \ref{RAPP1}, it must also be a $(2n-1)$-dimensional analytic manifold of $\R^{2n}$.

Thanks to the coarea formula in \cite[(10.6) Chapter 2]{Simon1983LecturesTheory}, that is well known to extend to general rectifiable sets, see for instance \cite[Theorem 3.2.22]{Federer1996GeometricTheory}, we infer that
\begin{equation}
    \int h(x)J^*_{S}u(x)d\mathcal{H}^{2n-1}_{eu}\llcorner S(x)=\int_0^\infty\int h(w)d\mathcal{H}^{2n-2}_{eu}\llcorner r\mathbb{S}^{2n-1}\cap S(w) \,dr,
    \label{eq:num:num:1}
\end{equation}
where $J^*_{S} u=\lvert \nabla_T u\rvert$ is the Jacobian of the tangential gradient of $u$ along $S$. Furthermore, since for any $x\in S$ and any $s\geq -\lvert x\rvert/2$ the point $x+sx/\lvert x\rvert$ is contained in $S$, the vector $\nabla u(x)=x/\lvert x\rvert$ is contained in the tangent plane to $S$ for $\mathcal{H}^{2n-1}_{eu}$-almost any $x\in S$. In particular $\nabla u(x)=\nabla_T u(x)$ and thus $J^*_{S}(x)=1$ for $\mathcal{H}^{2n-1}_{eu}$-almost every $x\in S$. This, in conjunction with \eqref{eq:num:num:1}, yields:
\begin{equation}
\begin{split}
        \int h(x)d\mathcal{H}^{2n-1}_{eu}(x)=&\int_0^\infty\int h(w)d\mathcal{H}^{2n-2}_{eu}\llcorner r\mathbb{S}^{2n-1}\cap S(w) \,dr\\
    =&\int_0^\infty r^{2n-2}\int h(rw)d\mathcal{H}^{2n-2}_{eu}\llcorner \mathbb{S}^{2n-1}\cap S(w) \,dr,
    \label{eq:108}
\end{split}
\end{equation}
where the last identity comes from the fact that, denoted by $T_{0,1/r}$ the operator that pushes forward measures in $\R^{2n}$ with respect to the multiplication by $1/r$,  $\mathcal{H}^{2n-2}_{eu}\llcorner r\mathbb{S}^{2n-1}\cap S=r^{2n-2}T_{0,1/r}\mathcal{H}^{2n-2}_{eu}\llcorner \mathbb{S}^{2n-1}\cap S$.

We are ready to prove the identities (i) and (ii) of the statement. Thanks to Lemma \ref{RAPP1} for any non-negative Borel function $G:\HH^n\to\R$ we have:
\begin{equation}
    \int G(z)d\mu(z)=\mathfrak{c}_n^{-1}\int G(z_H,z_T)d\mathcal{H}^{2n-1}_{eu}\llcorner S(z_H)\otimes d
    z_T.
    \label{eq:107}
\end{equation}
In order to prove identity (i), we choose $G(z):=\chi_{B_1(0)}(z)g(z_H)$ and note that combining \eqref{eq:108} and \eqref{eq:107}, we get:
\begin{equation}
\begin{split}
    \int_{B_1(0)}G(z)d\mu(z)=&\frac{1}{\mathfrak{c}_n}\int_{\lvert y\rvert^4+t^2\leq 1} g(y) d\mathcal{H}^{2n-1}_{eu}\llcorner S(y)\otimes\mathcal{L}^1(t)\\
    = &\frac{2}{\mathfrak{c}_n}\int_{\lvert y\rvert\leq 1}g(y)\sqrt{1-\lvert y\rvert^4} d\mathcal{H}^{2n-1}_{eu}\llcorner S(y)\\
    =&\frac{2}{\mathfrak{c}_n}\int_0^1 r^{2n-2}\sqrt{1-r^4}\int g(r v)d\omega_\mu(v)dr.
    \nonumber
\end{split}
\end{equation}
Moreover, (ii) follows immediately with the choice $G(z):=f(z_T)g(z_H)$:
\begin{equation}
\begin{split}
    \int_{B_1(0)}G(z)d\mu(z)=&\frac{1}{\mathfrak{c}_n}\int f(t)g(y) d\mathcal{H}^{2n-1}_{eu}\llcorner S(y)\otimes dt\\
   =&\frac{1}{\mathfrak{c}_n}\int f(t)dt\cdot\int g(y) d\mathcal{H}^{2n-1}_{eu}\llcorner S(y)\\
   =&\frac{1}{\mathfrak{c}_n}\int f(t)dt\cdot\int_0^\infty r^{2n-2}\big(\int g(rv) d\omega_\mu(v)\big) dr,
    \nonumber
\end{split}
\end{equation}
where in the first identity we applied \eqref{eq:107}, in the second Tonelli's theorem and in the last one \eqref{eq:108}.
\end{proof}

The following proposition is a refinement of Lemma \ref{CONO} in case $\mu$ is a vertical $(2n+1)$-uniform cone. In such a case $\mathcal{T}(1)=0$ and the integrals defining $\mathcal{Q}(1)$, recall that $\mathcal{T}$ and $\mathcal{Q}$ were introduced in Definition \ref{definizionecurve}, can be explicitly computed thanks to Lemma \ref{RAPP2}.

\begin{proposizione}\label{SUPPORTO}
Suppose $\mu$ is a vertical $(2n+1)$-uniform cone. Then, for any $w\in\supp(\mu)$ we have:
\begin{equation}
    \lvert  w_H\rvert^2=(2n-1)\fint \langle   w_H,u\rangle^2 d\omega_\mu(u),
    \label{eq:115}
\end{equation}
where $\omega_\mu=\mathcal{H}^{2n-2}_{eu}\llcorner \mathbb{S}^{2n-1}\cap \pi_H(\supp(\mu))$.
\end{proposizione}

\begin{proof}
Since $\mu$ is a cone, Proposition \ref{CONO} implies that for any $w\in\supp(\mu)$ we have:
\begin{equation}
\langle   w_H,\mathcal{Q}(1) w_H\rangle+\mathcal{T}(1)w_T=0.
\label{eq:111}
\end{equation}
Moreover, since by assumption $\mu$ is a vertical $(2n+1)$-uniform cone, thanks to Proposition \ref{verticale} we have $\mathcal{T}(1)=0$. In order to prove the proposition, we are left to give an explicit expression for $\mathcal{Q}(1)$. 

Let $\varphi:\R^{2n}\to\R$ be a (Euclidean) $k$-homogeneous function and $i\in\N$. Lemma \ref{RAPP2}(ii) and the fact that $ z_T^i\varphi(z)e^{-\lVert z\rVert^4}\in L^1(\mu)$ imply that:
\begin{equation}
    \begin{split}
    \int   z_T^i \varphi(z_H)e^{-\lVert z\rVert^4}d\mu(z)
        =&\frac{1}{\mathfrak{c}_n}\int t^ie^{-t^2}dt\cdot\int_0^\infty r^{2n-2}\int \varphi(r v)e^{-r^4}d\omega_\mu(v)dr\\
        =&\frac{1}{\mathfrak{c}_n}\int t^ie^{-t^2}dt\cdot\int_0^\infty r^{2n+k-2}e^{-r^4}dr\cdot\int \varphi(v)d\omega_\mu(v).
        \label{numeroo51}
    \end{split}
\end{equation}
Thanks to \eqref{numeroo51}, we easily see that if $i$ is odd, then:
\begin{equation}
    \int   z_T^i \varphi(   z_H)e^{-\lVert z\rVert^4}d\mu(z)=0,
    \label{numeroo60}
\end{equation}
while if $i$ is even:
\begin{equation}
\begin{split}
     &\int   z_T^i \varphi(   z_H)e^{-\lVert z\rVert^4}d\mu(z)=\int t^ie^{-t^2}dt\cdot\int_0^\infty r^{2n-2}\int \varphi(r v)e^{-r^4}d\omega_\mu(v)dr\\
        &\qquad\qquad=\mathfrak{c}_n^{-1}\int t^ie^{-t^2}dt\cdot\int_0^\infty r^{2n+k-2}e^{-r^4}dr\cdot\int \varphi(v)d\omega_\mu(v)\\
&\qquad\qquad=   \mathfrak{c}_n^{-1}\cdot\frac{1}{2}\Gamma\Big(\frac{i+1}{2}\Big)\cdot \frac{1}{4}\Gamma\Big(\frac{2n-1+k}{4}\Big)\int\varphi(v)d\omega_\mu(v)
\end{split}
    \label{eq:110}
\end{equation}

In order to compute the matrix $\mathcal{Q}(1)$ by means of identities \eqref{numeroo60} and \eqref{eq:110}, it will be convenient to make use of the notation introduced in the proof of Proposition \ref{prop10}, where we defined the matrices $\mathcal{Q}_1,\mathcal{Q}_2$ and $\mathcal{Q}_3,\mathcal{Q}_4$ in  \eqref{eq:Q2} and \eqref{eq:Q4} respectively. To simplify further the notation we let $\mathfrak{C}_6(n):=\Gamma\big(\frac{2n+1}{4}\big)\Gamma\big(\frac{1}{2}\big)/8\mathfrak{c}_nC(2n+1)$, where the constant $C(2n+1)$ was introduced in Definition \ref{defimom}. With these notations we have:
\begin{equation}
    \begin{split}
      & \mathcal{Q}_1(1)=\frac{8}{C(2n+1)}\int (\lvert   z_H\rvert^4   z_H\otimes   z_H+  z_T^2 J   z_H\otimes J   z_H) e^{-\lVert z\rVert^4}d\mu(z)\\
        &\qquad\underset{\eqref{eq:110}}{=}\;\frac{8}{C(2n+1)}\bigg( \frac{1}{8\mathfrak{c}_n}\Gamma\Big(\frac{1}{2}\Big) \Gamma\Big(\frac{2n+5}{4}\Big)\int v\otimes v\;d\omega_\mu(v)\\
        &\qquad\qquad\qquad\qquad\qquad+\frac{1}{8\mathfrak{c}_n}\Gamma\Big(\frac{3}{2}\Big) \Gamma\Big(\frac{2n+1}{4}\Big)\int Jv\otimes Jv\;d\omega_\mu(v)\bigg)\\
        &\qquad\,=\mathfrak{C}_6(n)\bigg( 2(2n+1)\int v\otimes v\;d\omega_\mu(v)+4\int Jv\otimes Jv\;d\omega_\mu(v)\bigg),
        \label{numeroQ1}
    \end{split}
\end{equation}
\begin{equation}
    \mathcal{Q}_2(1)=\frac{8}{C(2n+1)}\int \lvert   z_H\rvert^2  z_T (   z_H\otimes J   z_H+ J   z_H\otimes    z_H)e^{-\lVert z\rVert^4}d\mu(z)\underset{\eqref{numeroo60}}{=}0,
     \label{numeroQ2}
\end{equation}

\begin{equation}
    \begin{split}
&\mathcal{Q}_3(1)=\frac{4}{C(2n+1)}\bigg(\int (    z_H\otimes   z_H+ J  z_H\otimes J   z_H) e^{-\lVert z\rVert^4}d\mu(z)\bigg)\\ \underset{\eqref{eq:110}}{=}&\;\frac{4}{C(2n+1)}\frac{1}{8\mathfrak{c}_n}\Gamma\Big(\frac{1}{2}\Big) \Gamma\Big(\frac{2n+1}{4}\Big)\int (   v\otimes   v+ J  v\otimes J   v)\;d\omega_\mu(z)\\
&\qquad\,\,\,=4\mathfrak{C}_6(n)\int (   v\otimes   v+ J  v\otimes J   v)\;d\omega_\mu(z),
        \label{numeroQ3}
    \end{split}
\end{equation}

\begin{equation}
    \begin{split}
        \mathcal{Q}_4(&1)=\frac{2}{C(2n+1)}\int \lvert   z_H\rvert^2 e^{-\lVert z\rVert^4}d\mu(z)\;\mathrm{id}_{2n}\\
        &\,\,\underset{\eqref{eq:110}}{=}\;\frac{2}{C(2n+1)}\frac{1}{8\mathfrak{c}_n}\Gamma\Big(\frac{1}{2}\Big) \Gamma\Big(\frac{2n+1}{4}\Big)\omega_\mu(\mathbb{S}^{2n-1}) \;\mathrm{id}_{2n}\\
        &\,\,\,\,=2\mathfrak{C}_6(n)\omega_\mu(\mathbb{S}^{2n-1}) \;\mathrm{id}_{2n}.
        \label{numeroQ4}
    \end{split}
\end{equation}
Putting together \eqref{numeroQ1}, \eqref{numeroQ2}, \eqref{numeroQ3}, \eqref{numeroQ4} and recalling that $\mathcal{Q}(1)=\mathcal{Q}_1(1)+\mathcal{Q}_2(1)-\mathcal{Q}_3(1)-\mathcal{Q}_4(1)$, see the proof of Proposition \ref{prop10}, we deduce that:
\begin{equation}
    \begin{split}
    \frac{\mathcal{Q}(1)}{\mathfrak{C}_6(n)\omega_\mu(\mathbb{S}^{2n-1}) }=&2(2n+1)\fint v\otimes v d\omega_\mu(v)+4\fint Jv\otimes Jv d\omega_\mu(v)\\
    &\qquad\qquad\qquad-2\mathrm{id}_{2n}-4\fint(v\otimes v+Jv\otimes Jv)d\omega_\mu(v)\\
=&2(2n-1)\fint u\otimes u\, d\omega_\mu(u)-2\mathrm{id}_{2n}.
    \nonumber
    \end{split}
\end{equation}
Therefore, thanks to \eqref{eq:111}, we deduce that for any $w\in\supp(\mu)$, we have:
$$0=\langle   w_H,\mathcal{Q}(1) w_H\rangle=2\mathfrak{C}_6(n)\omega_\mu(\mathbb{S}^{2n-1})\big((2n-1)\fint\langle v,  w_H\rangle^2 d\omega_\mu(v)-\lvert  w_H\rvert^2\big).$$
This concludes the proof.
\end{proof}

\begin{definizione}\label{defM}
For any vertical $(2n+1)$-uniform cone $\mu$, define:
$$M_\mu:=\fint u\otimes u \;d\omega_\mu(u),$$
where $\omega_\mu:=\mathcal{H}^{2n-2}_{eu}\llcorner \pi_H(\supp(\mu))\cap \mathbb{S}^{2n-1}$. Moreover, let $\alpha_1\geq\ldots\geq \alpha_{2n}\geq0$ be the eigenvalues of $M_\mu$ and $\mathfrak{e}_1,\ldots,\mathfrak{e}_{2n}$ their relative eigenvectors.
\end{definizione} 

\begin{osservazione}
The trace of the matrix $M_\mu$ can be explicitly computed:
$$\text{Tr}(M_\mu)=\sum_{i=1}^{2n}\alpha_i=\sum_{i=1}^{2n}\langle \mathfrak{e}_i,M_\mu  \mathfrak{e}_i\rangle=\sum_{i=1}^{2n}\fint \langle u, \mathfrak{e}_i\rangle^2d\omega_\mu(u)=1.$$
\end{osservazione}

The following proposition links the matrix $M_\mu$ defined above to the functional $\min_{\mathfrak{m}\in\mathbb{S}^{2n-1}}\int_{B_1(0)}\langle \mathfrak{m},    z_H\rangle^2 d\mu(z)$, which will play a fundamental role in the proof of our main result. 

\begin{proposizione}\label{OXX}
There exists a constant $\mathfrak{C}_7(n)>0$ for which for any vertical $(2n+1)$-uniform cone $\mu$ and any $\mathfrak{m}\in\mathbb{S}^{2n-1}$ we have:
$$\int_{B_1(0)}\langle \mathfrak{m},    z_H\rangle^2 d\mu(z)=\mathfrak{C}_7(n)\langle \mathfrak{m}, M_\mu\mathfrak{m}\rangle.$$
\end{proposizione}

\begin{proof}
Thanks to Lemma \ref{RAPP2}(i) we have that:
\begin{equation}
\begin{split}
     \int_{B_1(0)}\langle \mathfrak{m},z_H\rangle^2 d\mu(z)=&\frac{2}{\mathfrak{c}_n}\int_0^1r^{2n}\sqrt{1-r^4}dr\int_{\mathbb{S}^{2n-1}}\langle \mathfrak{m},v \rangle^2 d\omega_\mu(v)\\
    =&\frac{2\mathfrak{C}_8(n)}{\mathfrak{c}_n}\int_{\mathbb{S}^{2n-1}}\langle \mathfrak{m},v \rangle^2 d\omega_\mu(v),
    \nonumber
\end{split}
\end{equation}
where $\mathfrak{C}_8(n):=\int_0^1r^{2n}\sqrt{1-r^4}dr$ and $\omega_\mu=\mathcal{H}^{2n-2}_{eu}\llcorner \mathbb{S}^{2n-1}\cap \pi_H(\supp(\mu))$. In order to prove the proposition, we are left to show that $\omega_\mu(\mathbb{S}^{2n-1})$ is a constant depending only on $n$. This can be seen by Lemma \ref{RAPP2}(i) that implies with the choice $g=1$:
$$1=\mu(B_1(0))=\frac{2\omega_\mu(\mathbb{S}^{2n+1})}{\mathfrak{c}_n}\int_0^1 r^{2n-2}\sqrt{1-r^4}dr.$$
Therefore, defined $\mathfrak{C_9}:=\int_0^1 r^{2n-2}\sqrt{1-r^4}dr$, thanks to the above identity  we have that: $\omega_\mu(\mathbb{S}^{2n-1})=\mathfrak{c}_n/2\mathfrak{C}_9(n)$. In particular:
$$\int_{B_1(0)}\langle \mathfrak{m},    z_H\rangle^2 d\mu(z)=\frac{\mathfrak{C}_8(n)}{\mathfrak{C}_9(n)}\fint \langle \mathfrak{m},v\rangle^2 d\omega_\mu(v)=\frac{\mathfrak{C}_8(n)}{\mathfrak{C}_9(n)}\langle \mathfrak{m}, M_\mu\mathfrak{m}\rangle.$$
The proof is complete with the choice $\mathfrak{C}_7(n)=\mathfrak{C}_8(n)/\mathfrak{C}_9(n)$.
\end{proof}

An immediate consequence of Proposition \ref{OXX} is that:
\begin{equation}
   \min_{\mathfrak{m}\in\mathbb{S}^{2n-1}}\int_{B_1(0)}\langle \mathfrak{m},    z_H\rangle^2 d\mu(z)=\mathfrak{C}_7(n)\alpha_{2n}=\int_{B_1(0)}\langle \mathfrak{e}_{2n},    z_H\rangle^2 d\mu(z). 
   \label{numbero1005}
\end{equation}
In the following proposition we link the number $\int_{B_1(0)}\langle \mathfrak{e}_{2n},   z_H\rangle^2 d\mu(z)$, which thanks to \eqref{numbero1005} is the smallest eigenvalue of the matrix $\int_{B_1(0)}z_H\otimes z_H d\mu(z)$, to the geometric structure of the measure $\mu$.

\begin{proposizione}\label{GEOM}
Suppose $\mu$ is a vertical $(2n+1)$-uniform cone and let $S:=\pi_H(\supp(\mu))$. For any $\delta>0$ there exists an $\epsilon(\delta,n)>0$ such that, if:
$$\int_{B_1(0)}\langle \mathfrak{e}_{2n},   z_H\rangle^2d\mu(z)\leq \epsilon(\delta,n),$$
then for any $x\in\R^{2n}$ such that $\lvert x\rvert\leq 1$ and $\langle x,\mathfrak{e}_{2n}\rangle=0$, there exists a $z\in S$ for which $\lvert    z-x\rvert\leq \delta$.
\end{proposizione}

\begin{proof}
Assume by contradiction that there exists a $\delta>0$, an infinitesimal sequence $\{\epsilon_i\}_{i\in\N}$ and a sequence of $(2n+1)$-uniform cones $\{\mu_i\}_{i\in\N}$ for which:
\begin{itemize}
    \item[($\alpha$)] defined $\mathfrak{e}_{2n}(i)$ be eigenvector relative to the minimum eigenvalue of the matrix $\int    z_H\otimes    z_H \;e^{-\lVert z\rVert^4}d\mu_i(z)$, we have:
    $$\int_{B_1(0)} \langle   z_H,\mathfrak{e}_{2n}(i)\rangle^2 d\mu_i(z)\leq \epsilon_i,$$
    \item[($\beta$)] there exists $\lvert x_i\rvert\leq 1$ and $\langle x_i,\mathfrak{e}_{2n}(i)\rangle=0$ for which $\lvert    z-x_i\rvert\geq \delta$ for any $z\in S_i:=\pi_H(\supp(\mu_i))$.
\end{itemize}
By compactness, up to non-relabeled subsequences, we can assume that $\mu_i\rightharpoonup\nu$, $\mathfrak{e}_{2n}(i)\to \mathfrak{e}$ and $x_i\to x$. Since $\langle\cdot,\mathfrak{e}_{2n}(i)\rangle^2\chi_{B_1(0)}(\cdot)$ is uniformly converging to $\langle\cdot,\mathfrak{e}\rangle^2\chi_{B_1(0)}(\cdot)$, we have:
\begin{equation}
    \int_{B_1(0)} \langle   z_H,\mathfrak{e}\rangle^2 d\nu(z)=\lim_{i\to\infty}\int_{B_1(0)} \langle   z_H,\mathfrak{e}_{2n}(i)\rangle^2 d\mu_i(z)=0.
    \label{numeroo70}
\end{equation}
Thanks to Lemma \ref{replica} we know that $\nu$ is a $(2n+1)$-uniform measure and by \eqref{numeroo70} we infer that its support is contained in $V(\mathfrak{e})$. Therefore, applying Proposition \ref{rank1} we deduce that $\nu$ is flat and $V(\mathfrak{e})=\supp(\nu)$. 

Thanks to ($\beta$), if $i$ is chosen sufficiently big, we have that
$\supp(\mu_i)\cap\{y\in\HH^n:\lvert y_H-x_H\rvert\leq \delta/2\}=\emptyset$,
and thus Lemma \ref{replica} implies on the one hand that $x\in\supp(\nu)$ and on the other:
\begin{equation}
V(\mathfrak{e})\cap\{y\in\HH^n:\lvert y_H-x_H\rvert< \delta/2\}=\supp(\nu)\cap\{y\in\HH^n:\lvert y_H-x_H\rvert< \delta/2\}=\emptyset,
\label{numeroo71}
\end{equation}
However, since $x\in V(\mathfrak{e})=\supp(\nu)$ by the continuity of the scalar product, \eqref{numeroo71} is a contradiction.
\end{proof}

The following proposition shows that non-flat vertical $(2n+1)$-uniform cones are quantitatively disconnected from flat measures. The proof of this Theorem follows closely its Euclidean counterpart, see for instance \cite[Proposition 8.5]{DeLellis2008RectifiableMeasures}. This is due to the following algebraic similarity.

Let $\mu$ be an $m$-uniform cone in $\R^n$. Then, for any $w\in\supp(\mu)$ we have:
$$ 2\pi^{-m/2}\int \langle w,z\rangle^2 e^{-\lvert  z\rvert^2} d\mu(z)=\lvert w\rvert^2,$$
see for instance identity (8.7) of \cite[Lemma 8.6]{DeLellis2008RectifiableMeasures}. 
The structure of the Euclidean quadric containing $\supp(\mu)$ is the same of the quadric in \eqref{eq:115} which contains the support of vertical $(2n+1)$-uniform cones. As a consequence, one should expect the same kind of algebraic computations to work.

\begin{teorema}\label{fine}
There exists a constant $\mathfrak{C}_{10}(n)>0$ such that if $\mu$ is a vertical $(2n+1)$-uniform cone for which:
$$\min_{\mathfrak{m}\in\mathbb{S}^{2n-1}}\int_{B_1(0)}\langle    \mathfrak{m},z_H\rangle^2 d\mu(z)\leq \mathfrak{C}_{10}(n),$$
then $\mu$ is flat.
\end{teorema}

\begin{proof}
Fix some $0<\delta<1/2$ and suppose that $\min_{\mathfrak{m}\in\mathbb{S}^{2n-1}}\int_{B_1(0)}\langle    \mathfrak{m},z_H\rangle^2 d\mu(z)\leq \epsilon(\delta,n)/\mathfrak{C}_7(n)$, where $\epsilon(\delta,n)$ is the constant yielded by Proposition \ref{GEOM}. In the following, sticking to the notations of Definition \ref{defM}, we will let $\alpha_1,\ldots,\alpha_{2n}$ and $\mathfrak{e}_1,\ldots,\mathfrak{e}_{2n}$ be the eigenvalues and the eigenvectors respectively of the matrix $M_\mu$.

Identity \eqref{numbero1005} implies that $\int_{B_1(0)}\langle    e_{2n},z_H\rangle^2 d\mu(z)\leq \epsilon(\delta,n)$ and thus, thanks to Proposition \ref{GEOM}, there exists a $z\in\supp(\mu)$ such that $\lvert z_H-\mathfrak{e}_{2n-1}\rvert\leq \delta$. 
Thanks to the order imposed on the $\alpha_i$'s, we have that: $$\alpha_i+(2n-2)\alpha_{2n-1}\leq \text{Tr}(M_\mu)=1\qquad\text{for every $i\leq 2n-2$.}$$ 
This in particular implies that:
\begin{equation}
   \alpha_i-\frac{1}{2n-1}\leq (2n-2)\left(\frac{1}{2n-1}-\alpha_{2n-1}\right)\qquad \text{for every  }i\leq 2n-2. 
   \label{eq:117}
\end{equation}
Since $z\in\supp(\mu)$, Proposition \ref{SUPPORTO} implies, once we write \eqref{eq:115} in the basis $\{\mathfrak{e}_1,\ldots,\mathfrak{e}_{2n}\}$, that:
\begin{equation}
    0=\sum_{i=1}^{2n}\left(\alpha_i-\frac{1}{2n-1}\right)\langle \mathfrak{e}_i,  z_H\rangle^2.
    \label{eq:116}
\end{equation}
We observe that since $M_\mu$ is positive semidefinite and $\text{Tr}(M_\mu)=1$, we have that $\alpha_{2n}\leq 1/2n$.
Therefore, putting together \eqref{eq:117}, \eqref{eq:116} and the fact that $\alpha_{2n}\leq1/2n$, we deduce that:
\begin{equation}
\begin{split}
0=&\sum_{i=1}^{2n}\left(\alpha_i-\frac{1}{2n-1}\right)\langle \mathfrak{e}_i,  z_H\rangle^2\\
\leq& \sum_{i=1}^{2n-2}\left(\alpha_i-\frac{1}{2n-1}\right)\langle \mathfrak{e}_i,  z_H\rangle^2+\left(\alpha_{2n-1}-\frac{1}{2n-1}\right)\langle \mathfrak{e}_{2n-1},  z_H\rangle^2.
\label{eq:118}
\end{split}
\end{equation}
Summing up, inequalities \eqref{eq:117} and \eqref{eq:118} together with some algebraic manipulations imply:
\begin{equation}
    \begin{split}
    0\leq&(2n-2)\Big(\frac{1}{2n-1}-\alpha_{2n-1}\Big)\sum_{i=1}^{2n-2}\langle \mathfrak{e}_i, z_H\rangle^2
-\Big(\frac{1}{2n-1}-\alpha_{2n-1}\Big)\langle \mathfrak{e}_{2n-1},  z_H\rangle^2\\
=&(2n-2)\Big(\frac{1}{2n-1}-\alpha_{2n-1}\Big)\sum_{i=1}^{2n-2}\langle \mathfrak{e}_i, z_H\rangle^2\\
&\qquad\qquad\qquad\qquad\qquad-\Big(\frac{1}{2n-1}-\alpha_{2n-1}\Big)(1+\langle \mathfrak{e}_{2n-1}, z_H-\mathfrak{e}_{2n-1}\rangle)^2\\
\leq&\Big(\frac{1}{2n-1}-\alpha_{2n-1}\Big)\Big((2n-2)\sum_{i=1}^{2n-2}\lvert z_H-\mathfrak{e}_{2n-1}\rvert^2-(1-\lvert z_H-\mathfrak{e}_{2n-1}\rvert)^2\Big)\\
\leq&\Big(\frac{1}{2n-1}-\alpha_{2n-1}\Big)((2n-2)^2\delta^2-(1-\delta)^2),
    \label{eq:120}
\end{split}
\end{equation}
where in order to pass from the second to the third line above one needs to note that $\lvert\langle e_i,z_H\rangle\rvert=\lvert\langle e_i,z_H-e_{2n-1}\rangle\rvert\leq \lvert z_H-e_{2n-1}\rvert$ for any $i=1,\ldots,2n-2$ and that
$0< 1-\lvert z_H-\mathfrak{e}_{2n-1}\rvert\leq 1+\langle \mathfrak{e}_{2n-1}, z_H-\mathfrak{e}_{2n-1}\rangle$ as $\lvert z_H-e_{2n-1}\rvert<\delta$.
Therefore, if $\delta$ is small enough, thanks to the inequality \eqref{eq:120}, we have that $\alpha_{2n-1}\geq1/(2n-1)$.
In this case, since $M_\mu$ is positive semidefinite and $\text{Tr}(M_\mu)=1$, we have that: $$\alpha_{1}=\ldots=\alpha_{2n-1}=\frac{1}{2n-1}\qquad\text{and}\qquad\alpha_{2n}=0.$$
This by the equation \eqref{numbero1005} implies  that $\supp(\mu)\subseteq V(\mathfrak{e}_{2n})$ and thus by Proposition \ref{rank1}, $\mu$ must be flat.
\end{proof}

\section{Conclusions}
 \label{conclusioni}
In this section we complete the proof of Theorem \ref{main}. In order to conclude the proof we need to construct the continuous functional $\mathscr{F}$ on Radon measures which fulfills the hypothesis of Theorem \ref{disco}.

\begin{proposizione}\label{conti}
Let $\eta$ be a non-negative smooth function such that $\eta=1$ on $B_1(0)$ and $\eta=0$ on $B_2^c(0)$. Then, the functional $\mathscr{F}:\mathcal{M}\to\R$ defined by:
\begin{equation}
    \mathscr{F}(\mu):=\min_{\mathfrak{m}\in\mathbb{S}^{2n-1}}\int \eta(z) \langle    z_H,\mathfrak{m} \rangle^2 d\mu(z),
    \label{eq:eq:eq5}
\end{equation}
is continuous with respect to the weak convergence of measures.
\end{proposizione}

\begin{proof}
Suppose that $\mu_i\rightharpoonup \mu$ and $\mathfrak{m}_i\in\mathbb{S}^{2n-1}$ are such that:
$$\int \eta(z)\langle    z_H, \mathfrak{m}_i\rangle^2d\mu_i(z)=\min_{\mathfrak{m}\in \mathcal{S}^{2n-1}}\int \eta(z)\langle    z_H,\mathfrak{m}\rangle^2 d\mu_i(z).$$
Up to passing to a (non-relabeled) subsequence we can also suppose that $\mathfrak{m}_i$ converges to some $\mathfrak{m}\in\mathbb{S}^{2n-1}$. Thus the function $\eta(\cdot)\langle \pi_H (\cdot),\mathfrak{m}_i\rangle^2$ is uniformly converging to  $\eta(\cdot)\langle \pi_H (\cdot),\mathfrak{m}\rangle^2$. This implies that:
$$\lim_{i\to\infty} \int \eta(z) \langle    z_H, \mathfrak{m}_i\rangle^2d\mu_i(z)= \int \eta(z) \langle    z_H, \mathfrak{m}\rangle^2d\mu(z),$$
from which we infer that $\liminf_{i\to \infty} \mathscr{F}(\mu_i)\geq \mathscr{F}(\mu)$. 
On the other hand, let $\mathscr{F}(\mu)=\int \eta(z)\langle z_H, \overline{\mathfrak{m}}\rangle^2 d\mu(z)$ for some $\overline{\mathfrak{m}}\in\mathbb{S}^{2n-1}$. Since $\mu_i\rightharpoonup \mu$, we deduce that:
$$\lim_{i\to\infty} \int \eta(z) \langle    z_H, \overline{\mathfrak{m}}\rangle^2d\mu_i(z)= \int \eta(z) \langle    z_H, \overline{\mathfrak{m}}\rangle^2d\mu(z).$$
This implies that $\limsup_{i\to\infty} \mathscr{F}(\mu_i)\leq \mathscr{F}(\mu)$ and this concludes the proof.
\end{proof}

The following proposition shows that $\mathscr{F}$ satisfies the hypothesis (ii) of Theorem \ref{disco}.

\begin{proposizione}\label{consta}
There exists a constant $\hbar(n)>0$ such that if $\mu$ is a $(2n+1)$-uniform cone and $\mathscr{F}(\mu)\leq \hbar(n)$, then $\mu$ is flat.
\end{proposizione}

\begin{proof}
Thanks to the definition the functional  $\mathscr{F}$ given in \eqref{eq:eq:eq5}, we have $\min_{\mathfrak{m}\in\mathbb{S}^{2n-1}} \int_{B_1(0)} \langle z_H,\mathfrak{m}\rangle^2 d\mu(z)\leq\mathscr{F}(\mu)$. Therefore, thanks to Theorems \ref{bellalei} and \ref{fine}, if $$\mathscr{F}(\mu)\leq \min\{\mathfrak{C}_3(n),\mathfrak{C}_{10}(n)\}/2=:\hbar(n),$$
then the measure $\mu$ is flat.
\end{proof}

We are ready to prove our main result Theorem \ref{main}.

\begin{teorema}
If $\phi$ is a measure with $(2n+1)$-density, then
$\Tan_{2n+1}(\phi,x)\subseteq \mathfrak{M}(2n+1)$
for $\phi$-almost all $x\in\HH^n$.
\end{teorema}

\begin{proof}
Since $\mathscr{F}(\mu)=0$, whenever $\mu$ is flat, Propositions \ref{conti} and \ref{consta} imply that we are in the hypothesis of Theorem \ref{disco}. which proves the claim.
\end{proof}

A byproduct of our analysis is the conclusion of the classification of uniform measures in $\HH^1$, which was systematically carried out in \cite{ChousionisONGROUP}. Our contribution is to prove that $3$-uniform measures in $\HH^1$ are flat. The final result reads:

\begin{teorema}
In $\HH^1$ we have the following complete classification of uniform measures:
\begin{itemize}
    \item[(i)] If $\mu\in\mathcal{U}_{\HH^1}(1)$, then $\mu=\mathcal{S}^1\llcorner L$, where $L$ is a horizontal line.
    \item[(ii)] If $\mu\in\mathcal{U}_{\HH^1}(2)$ then $\mu=\mathcal{S}^2\llcorner \mathcal{V}$, where $\mathcal{V}$ is the vertical axis.
    \item[(iii)] If $\mu\in\mathcal{U}_{\HH^1}(3)$ then $\mu=\mathcal{S}^3\llcorner W$, where $W$ is a $2$-dimensional vertical plane.
\end{itemize}
\end{teorema}

\begin{proof}
Points (i) and (ii) are \cite[Theorem 1.3, 1.4]{ChousionisONGROUP}. Corollary \ref{mainprimo} implies that if $\mu$ is supported on a quadric. Proposition 1.6 in \cite{ChousionisONGROUP} implies that such quadric cannot be a $t$-graph and therefore by Theorem 1.5 of the same paper we conclude the proof.
\end{proof}

\appendix

\section{Representations of sub-Riemannian centered Hausdorff measure}
\label{appeA1}

In the following we will adopt the notations introduced in Section \ref{buchi} and as usual we assume $b\in\R^{2n}$, $\mathcal{Q}\in\mathrm{Sym}(2n)$ and $\mathcal{T}\in\R$. The main goal of this section is to find a representation of the spherical Hausdorff measures $\mathcal{C}^{2n+1}\llcorner A$, where $A$ is a Borel subset of the quadric $\mathbb{K}(b,\mathcal{Q},\mathcal{T})$, in terms of the Euclidean Hausdorff measure $\mathcal{H}^{2n}_{eu}$. To do so, we will study separately the case where $\mathcal{T}\neq 0$ and the case $\mathcal{T}=0$.

Before proceeding we need to review the definition and some well known facts about the horizontal perimeter measure in $\HH^n$.
For any $x\in\HH^n$ and any $i\in\{1,\ldots,n\}$ we define the vector fields:
\begin{equation}
    X_i(x):=e_i-2x_{i+n}e_{2n+1},\qquad Y_i(x):=e_{i+n}+2x_ie_{2n+1},
    \label{campi}
\end{equation}
where $\{e_1,\ldots, e_{2n+1}\}$ is the standard othonormal basis of $\R^{2n+1}$ and the \emph{horizontal distribution}:
$$H\mathbb{H}^n(x):=\text{span}\{X_1(x),\ldots,X_n(x),Y_1(x),\ldots,Y_n(x)\}.$$
We recall here below the definitions of functions of bounded variations and of finite perimeter sets and we collect from various papers some results that will be useful in the following.

\begin{definizione}
We say that a function $f:\mathbb{H}^n\to \R$ is of \emph{local} $\HH$-\emph{bounded variation} if $f\in L^1_{loc}(\mathbb{H}^n)$ and:
$$\lVert \nabla_{\mathbb{H}} f\rVert(\Omega):=\sup\Big\{ \int_{\Omega} f(x)\text{div}_{\mathbb{H}} \varphi(x) dx:\varphi\in \mathcal{C}^1_0(\Omega,H\mathbb{H}^n),\lvert \varphi(x)\rvert\leq 1\Big\}<\infty,$$
for any bounded open set $\Omega\subseteq \mathbb{H}^n$, where $\text{div}_\mathbb{H}\varphi:=\sum_{i=1}^{n} X_i\varphi_i+\sum_{i=1}^{n} Y_i\varphi_i$ and where  $X_1,\ldots,Y_n$ are the vector fields introduced in \eqref{campi}. We denote by $\text{BV}_{\mathbb{H},loc}(\mathbb{H}^n)$ the set of all  functions of locally $\HH$-bounded variation. As usual a Borel set $E\subseteq \mathbb{H}^n$ is said to be of $\HH$-\emph{finite perimeter} if $\chi_E$ is of bounded variation. 
\end{definizione}

\begin{definizione}\label{def:norm}
If $E\subseteq \mathbb{H}^n$ is a Borel set of locally finite perimeter, we let $\lvert\partial E\rvert_\mathbb{H}:=\lVert \nabla_{\mathbb{H}} \chi_E\rVert$. Furthermore we call \emph{generalized horizontal inward} $\mathbb{H}$-\emph{normal to} $\partial E$ the horizontal vector $\mathfrak{n}_E(x):=\sigma_{\chi_E}(x)$. Finally, we define the \emph{reduced boundary} $\partial_{\mathbb{H}}^*E$ to be the set of those $x\in\mathbb{H}^n$ for which:
\begin{itemize}
    \item[(i)]$\lvert\partial E\rvert_\mathbb{H}(B_r(x))>0$ for any $r>0$,
    \item[(ii)] $\lim_{r\to 0}\fint_{B_r(x)}\mathfrak{n}_Ed\lvert \partial E\rvert_{\mathbb{H}}$ exists,
    \item[(iii)] $\lim_{r\to 0}\Big\lVert\fint_{B_r(x)}\mathfrak{n}_Ed\lvert \partial E\rvert_{\mathbb{H}}\Big\rVert =1$.
\end{itemize}
\end{definizione}

The following proposition is very useful in order to get a representation of the perimeter measure on smooth domains by means of the Euclidean Hausdorff measure:

\begin{proposizione}[Proposition 3.1, \cite{Monti2014IsoperimetricGroup}]\label{Monti}
Let $E\subseteq \HH^n$ be a set with Euclidean Lipschitz boundary and $\Omega$ a fixed open set in $\HH^n$. Then:
$$\lvert\partial E\rvert_\HH(A\cap \Omega)=\int \chi_A\lvert \mathfrak{n}_H\rvert d\mathcal{H}_{\mathrm{eu}}^{2n}\llcorner \partial E\cap\Omega,\qquad\text{for any open set $A\subseteq \HH^n$,}$$
and where denoted with $\mathfrak{n}_E^{\mathrm{eu}}(\cdot)$ the inward Euclidean unit normal  to $\partial E$, we have:
\begin{equation}
    \mathfrak{n}_E(x)=(\langle X_1(x),\mathfrak{n}^{\mathrm{eu}}_E(x)\rangle,\ldots, \langle X_n(x),\mathfrak{n}_E^{\mathrm{eu}}(x)\rangle,\langle Y_1(x),\mathfrak{n}_E^{\mathrm{eu}}(x)\rangle,\ldots, \langle Y_{2n}(x),\mathfrak{n}_E^{\mathrm{eu}}(x)\rangle).
    \label{num:norhor}
\end{equation}
\end{proposizione}

\begin{osservazione}\label{boggi}
The above proposition together with \cite[Proposition 1.9.4]{Bogachev2007MeasureTheory} implies that the measures $\lvert\partial E\rvert_{\HH}\llcorner \Omega$ and $\lvert \mathfrak{n}_E\rvert \mathcal{H}^{2n}_{\mathrm{eu}}\llcorner \partial E\cap \Omega$ coincide on Borel sets.
\end{osservazione}

\textbf{Throughout this entire section we assume that } $b\in\R^{2n}$\textbf{,} $\mathcal{Q}\in \mathrm{Sym}(2n)\setminus\{0\}$ \textbf{and} $\mathcal{T}\in \R$ \textbf{are fixed and define:}
$$E:=\{x\in\R^{2n+1}:\langle\mathcal{Q}x_H+b,x_H\rangle+\mathcal{T}x_T <0\}.$$
The open set $E$ is a set with Lipschitz boundary and the inward Euclidean unit normal of its boundary is:
$$\mathfrak{n}^{\mathrm{eu}}_E(x)=\frac{-(2\mathcal{Q}x_H+b,\mathcal{T})}{\lvert(2\mathcal{Q}x_H+b,\mathcal{T})\rvert}.$$
Therefore, thanks to identity \eqref{num:norhor} we have:
$$\mathfrak{n}_E(x)=\frac{-2\mathcal{Q}x_H-b+2\mathcal{T}Jx_H}{\lvert(2\mathcal{Q}x_H+b,\mathcal{T})\rvert}=-\frac{b+2(\mathcal{Q}-\mathcal{T}J)x_H}{\lvert (2\mathcal{Q}x_H+b,\mathcal{T})\rvert}.$$

\begin{osservazione}\label{rkA}
Note that if $\mathcal{T}=0$, then $\mathfrak{n}_E(x)$ coincides with the Euclidean normal $\mathfrak{n}^{\mathrm{eu}}_E(x)$.
\end{osservazione}

\subsection{Area on \texorpdfstring{$t$}{Lg}-quadratic graphs}

This subsection is devoted to find representation formulas of the centered Hausdorff measure $\mathcal{C}^{2n+1}$ when restricted to subsets of quadratic $t$-graphs. Here below we assume that $f:\R^{2n}\to\R$ is the quadratic polynomial:
 $$f(w):=-\frac{\langle b,w\rangle+\langle w,\mathcal{Q}w\rangle}{\mathcal{T}}.$$ 
and we let $\Sigma(f):=\{w\in\R^{2n}:2(\mathcal{Q}-\mathcal{T}J)w+b=0\}$ be the characteristic set of $f$, see \eqref{char}. Thanks to the simple algebraic simplicity of $f$, the following characterisation of $\Sigma(f)$ is available:

\begin{proposizione}\label{quaddica}
The set $\Sigma(f)$ is an affine plane in $\R^{2n}$ of dimension at most $n$.
\end{proposizione}

\begin{proof}
The proof can be established thanks to a reasonably short argument, that we choose to omit, that uses the fact that $J$ is an invertible skew-symmetric matrix and that $Q$ is symmetric.
\end{proof}


It is worth noting that Proposition \ref{quaddica} is a very simple case of the results obtained by Z. Balogh in \cite{Balogh2003SizeGradient}.

\begin{definizione}
For any $e\in\R^{2n}$, since the ball $B_1(0)$ is symmetric by rotation around the vertical axis, we have that  $\mathcal{H}^{2n}_{eu}(B_1(0)\cap V(e))$ does not depend on $e$. Therefore, throughout the paper we will always let $\mathfrak{c}_n$ be the constant:
$$\mathfrak{c}_n=\mathcal{H}^{2n}_{eu}(B_1(0)\cap V(e_1)).$$
\end{definizione}

\begin{proposizione}\label{rapperhor}
Let $\Omega$ be an open set in $\R^{2n}$ and define $\Omega^\prime:=\Omega\times \R$. For any positive Borel function $h:\HH^n\to \R$ we have:
$$\int h(z)d\mathcal{C}^{2n+1}\llcorner {\Omega^\prime\cap\mathbb{K}(b,\mathcal{Q},\mathcal{T}) }(z)=\frac{1}{\lvert\mathcal{T}\rvert\mathfrak{c}_n}\int_\Omega h(w,f(w))\lvert b+2(\mathcal{Q}-\mathcal{T}J)w\rvert dw,$$
where we recall that $\mathbb{K}(b,\mathcal{Q},\mathcal{T})$ was defined in Definition \ref{simmi}.
\end{proposizione}

\begin{proof}
Let $E:=\{u\in\R^{2n+1}:u_T-f(u_H)\geq 0\}$ and note that by \cite[Proposition 3.2]{Monti2014IsoperimetricGroup} we have that $E$ is a set of intrinsic finite perimeter and:
\begin{equation}
\lvert \partial E\rvert\llcorner{\Omega^\prime}(A)=\int_{\Omega\cap \pi_H(A)} \lvert\nabla f(z)+2Jz\rvert dz=\frac{1}{\lvert\mathcal{T}\rvert}\int_{\Omega\cap \pi_H(A)} \lvert b+2(\mathcal{Q}-\mathcal{T}J)z\rvert dz,
\label{numeroo40}
\end{equation}
for any open set $A\subseteq \HH^n$. With the same line of reasoning of Remark \ref{boggi}, one can show that the above identity also holds for $A$ Borel. It is immediate to see that reduced boundary of $E$, introduced in Definition \ref{def:norm}, coincides with $\mathbb{K}(b,\mathcal{Q},\mathcal{T})$. In addition to this, since $E$ is a finte perimeter set, \cite[Theorem 4.21]{FSSCArea} implies on the one hand that:
\begin{equation}
    \Theta(\partial E,x):=\lim_{r\to 0}\frac{\lvert\partial E\rvert_\HH(B_r(x))}{r^{2n+1}}=\mathcal{H}^{2n}_{\mathrm{eu}}(B_1(0)\cap V\big(2(\mathcal{Q}-\mathcal{T}J)x_H+b)\big)=\mathfrak{c}_n,
    \label{numeroo41}
\end{equation}
for $\lvert\partial E\rvert$-almost every $x\in\mathbb{K}(b,\mathcal{Q},\mathcal{T})$.
On the other, identity \cite[Theorem 4.21(4.23)]{FSSCArea} shows that:
\begin{equation}
\lvert\partial E\rvert_\HH\llcorner\Omega^\prime=\Theta(\partial E,\cdot)\mathcal{C}^{2n+1}\llcorner \mathbb{K}(b,\mathcal{Q},\mathcal{T}) \cap \Omega^\prime.
\label{numeroo42}
\end{equation}
Summing up, putting together \eqref{numeroo40}, \eqref{numeroo41} and \eqref{numeroo42}, we deduce that:
\begin{equation}
\begin{split}
    \mathcal{C}^{2n+1}\llcorner \mathbb{K}(b,\mathcal{Q},\mathcal{T})\cap\Omega^\prime=&\Theta(\partial E,x)^{-1}\lvert\partial E\rvert_\HH=\frac{1}{\mathfrak{c}_n\lvert\mathcal{T}\rvert}\int_{\Omega\cap \pi_H(A)} \lvert b+2(\mathcal{Q}-\mathcal{T}J)x\rvert dx\\
    =&\frac{1}{\mathfrak{c}_n\lvert\mathcal{T}\rvert}\int_{\Omega}\chi_A(x,f(x)) \lvert b+2(\mathcal{Q}-\mathcal{T}J)x\rvert dx.
    \nonumber
\end{split}
\end{equation}
The standard approximation procedure of positive measurable functions together with Beppo Levi's convergence theorem concludes the proof.
\end{proof}

\subsection{Area on vertical quadric}
\label{vertic}

This subsection is devoted to find representation formulas of the centered Hausdorff measure $\mathcal{C}^{2n+1}$ when restricted to subsets of vertical quadratics. Here below we assume that $F:\R^{2n}\times \R\to\R$ is the quadratic polynomial:
 $$F(w,t):=\langle b,w\rangle+\langle w,\mathcal{Q} w\rangle,$$ 
and we let $\Sigma(F):=\{x\in\mathbb{K}(b,\mathcal{Q},0):2\mathcal{Q}x_H+b=0\}$ be the set of singular points of $\mathbb{K}(b,\mathcal{Q},0)$. As a first step, we prove the representation formula for the centered spherical Hausdorff measure concentrated on vertical quadrics:

\smallskip

The following proposition tells us that for vertical quadrics the singular set is either negligible with respect the surface measure, or the quadric itself is flat. It will be useful in Section \ref{buchi}.

\begin{proposizione}\label{esclusion}
For the quadric $\mathbb{K}(b,\mathcal{Q},0)$ one of the two mutually excluding alternatives holds:
\begin{itemize}
\item[(i)] $b=0$ and $\mathcal{Q}=a\otimes a$ for some $a\in\R^{2n}\setminus\{0\}$,
\item[(ii)] $\mathcal{C}^{2n+1}(\Sigma(F))=0$.
\end{itemize}
\end{proposizione}

\begin{proof}
If $\text{dim}(\text{ker}(\mathcal{Q}))\leq 2n-2$, then $\Sigma(F)$ is contained in an affine subspace of topological dimension at most $2n-1$ and thus, it is immediate to see by an explicit computation that $\mathcal{C}^{2n+1}(\Sigma(F))=0$ as the intrinsic Hausdorff dimension of $\Sigma(F)$ must be smaller than $2n$.

On the other hand, if $\text{dim}(\text{ker}(\mathcal{Q}))= 2n-1$, we have that $Q=a\otimes a$ for some $a\in\R^{2n}$ and the expression for $F$ boils down to $F(x,t)=\langle a,x\rangle^2+\langle b,x\rangle$.
If $b\not\in \mathrm{span}(a)$, the equation $2\langle a,    z_H\rangle a+b=0$ can never be satisfied and thus $\Sigma(F)=\emptyset$. 

At last, if $a=\lambda b$, then $\mathbb{K}(b,\mathcal{Q},0)= V(a)\cup -\lambda a/\lvert a\rvert^2+V(a)$, while a simple computation shows that $\Sigma(F)\subseteq \lambda a/2\lvert a\rvert^2 +V(a)$, which concludes that $\Sigma(F)=\emptyset$. The only left out case is when $b=0$ in which $\Sigma(F)=V(a)$, which proves the claim.
\end{proof}

\begin{proposizione}\label{MANCA}
Let $\Omega$ be an open set in $\HH^n$. For any positive Borel function $h:\HH^n\to\R$ we have:
$$\int h(x)d\mathcal{C}^{2n+1}\llcorner{ \mathbb{K}(b,\mathcal{Q},0)\cap \Omega}(x)=\frac{1}{\mathfrak{c}_n}\int h(x)d\mathcal{H}^{2n}_{eu}\llcorner \mathbb{K}(b,\mathcal{Q},0)\cap\Omega(x),$$
where $\mathfrak{c}_n$ is the constant introduced in Proposition \ref{rapperhor}.
\end{proposizione}

\begin{proof}
Let $E$ be the open set $E:=\{z\in\R^{2n+1}:F(z)<0\}$ and note that $\partial E=\{F=0\}=\mathbb{K}(b,\mathcal{Q},0)$. Thanks to Proposition \ref{Monti}, since $E $ has a Lipschitz boundary, it is a set of intrinsic finite perimeter and:
$$\lvert \partial E\rvert\llcorner{\Omega}(A)=\int_{\partial E\cap \Omega}\chi_A\lvert \mathfrak{n}_H\rvert d\mathcal{H}^{2n}_{eu}\llcorner \mathbb{K}(b,\mathcal{Q},0)=\mathcal{H}^{2n}_{eu}(\mathbb{K}(b,\mathcal{Q},0)\cap \Omega\cap A),$$
for any Borel set $A$ of $\HH^n$, since by Remark \ref{rkA} we have $\lvert\mathfrak{n}_H\rvert=1$. With the same argument we used in the Proposition \ref{rapperhor} we see that $\lvert\partial E\rvert_\HH\llcorner\Omega^\prime=\mathfrak{c}_n\mathcal{C}^{2n+1}\llcorner \mathbb{K}(b,\mathcal{Q},\mathcal{T}) \cap \Omega$ and thus:
$$\mathcal{C}^{2n+1}\llcorner \mathbb{K}(b,\mathcal{Q},\mathcal{T}) \cap \Omega(A)=\mathfrak{c}_n^{-1}\mathcal{H}^{2n}_{eu}\llcorner \mathbb{K}(b,\mathcal{Q},0)\cap \Omega(A),$$
for any Borel set $A$.

The usual approximation of positive measurable functions together with Beppo Levi's convergence theorem concludes the proof.
\end{proof}


\begin{lemma}\label{RAPP1}
If $\mu$ is a vertical $(2n+1)$-uniform cone then:
$$\mu=\mathfrak{c}_n^{-1}\mathcal{H}_{eu}^{2n-1}\llcorner\pi_H(\supp(\mu))\otimes \mathcal{H}^1_{eu}\llcorner \R e_{2n+1}.$$
\end{lemma}

\begin{proof}
Thanks to Proposition \ref{supportoK} and \ref{MANCA}
we have that
$\mu=\mathcal{C}^{2n+1}\llcorner{\supp(\mu)}=\mathfrak{c}_n^{-1}\mathcal{H}^{2n}_{eu}\llcorner\supp(\mu)$.
Furthermore, since by Proposition \ref{spt2} we have that $\supp(\mu)=\pi_H(\supp(\mu))\times \R e_{2n+1}$.
\cite[Proposition 3.2.23]{Federer1996GeometricTheory} finally implies:
$$\mathcal{H}^{2n}_{eu}\llcorner\supp(\mu)=\mathcal{H}_{eu}^{2n-1}\llcorner \pi_H(\supp(\mu))\otimes \mathcal{H}^1_{eu}\llcorner \R e_{2n+1},$$
and this concludes the proof of the proposition.
\end{proof}

An immediate consequence of the above proposition is the following:

\begin{corollario}\label{dege}
Let $a,b$ be two non-parallel vectors in $\R^{2n}$. Then $\mathcal{C}^{2n+1}(V(a)\cap V(b))=0$ and
where $V(\cdot)$ are the planes introduced in Definition \ref{plano}.
\end{corollario}

\begin{proof}
Since the measure $\mathcal{C}^{2n+1}\llcorner V(a)$ is a vertical $(2n+1)$-uniform cone, we have:
\begin{equation}
\mathcal{C}^{2n+1}\llcorner V(a)=\mathfrak{c}_n^{-1}\mathcal{H}_{eu}^{2n-1}\llcorner\pi_H(V(a))\otimes \mathcal{H}^1_{eu}\llcorner \R e_{2n+1}=\mathfrak{c}_n^{-1}\mathcal{H}_{eu}^{2n}\llcorner V(a). 
\label{numeroo50}
\end{equation}
Since $a$ and $b$ are two independent vectors, the set $V(a)\cap V(b)$ is $\mathcal{H}^{2n}_\mathrm{eu}$-null and thus, the claim immediately follows by \eqref{numeroo50}.
\end{proof}

\section{Taylor expansion of area on quadratic \texorpdfstring{$t$}{Lg}-cones}
\label{TYLR}
Before giving a short account on the content of this appendix, let us introduce some notation. \textbf{Throughout this appendix we will always suppose that} $\mathcal{D}\in\mathrm{Sym}(2n)\setminus \{0\}$ \textbf{and let} $f:\R^{2n}\to\R$ be the quadratic polynomial defined as:
$$f(h):=\langle h,\mathcal{D} h\rangle.$$
Furthermore, \textbf{we let} $\lvert\partial \mathbb{K}\rvert$ \textbf{be the} $\HH^n$\textbf{-perimeter measure associated to the epigraph} $E:=\{x\in \HH^n: x_T>f(x_H)\}$ of $f$ in $\HH^n$, which is of finite perimeter since $f$ is a smooth function, see Proposition  \ref{Monti}.
\textbf{Finally, we fix a point} $x\in\R^{2n}\setminus \Sigma(f)=\{h\in\R^{2n}:b+2(\mathcal{Q}-\mathcal{T}J)h\neq0\}$, \textbf{and we let} $\mathcal{X}:=(x,f(x))$.

\medskip

The main goal of this section is to determine an asymptotic expansion of $\lvert\partial \mathbb{K}\rvert(B_r(\mathcal{X}))$ for $r$ small. More precisely, written:
$$\lvert\partial \mathbb{K}\rvert(B_r(\mathcal{X}))=\mathfrak{c}(\mathcal{X})r^{2n+1}+\zeta(\mathcal{X})r^{2n+2}+\mathfrak{e}(\mathcal{X})r^{2n+3}+O(r^{2n+4}),$$
we want to find an expression for the coefficients $\mathfrak{c},\zeta,\mathfrak{e}$ in terms of $x$, $\mathcal{D}$ and $n$.
The coefficient $\mathfrak{c}$ will be quite easy to study and we will show that it is a constant depending only on $n$. On the other hand the coefficients $\zeta$ and $\mathfrak{e}$ will need much more work and they play a fundamental role in the study of the geometric properties of $1$-codimensional uniform measures carried on in Section \ref{HORRI}.

\begin{definizione}
Let $\mathcal{D}$, $\mathcal{X}$ and $f$ be as above. We denote as: $$\mathfrak{n}:=\frac{ (\mathcal{D}+J)x}{\lvert (\mathcal{D}+J)x\rvert},$$
the \emph{horizontal normal} at $\mathcal{X}$ to $\text{gr}(f)$ and we let $c:=2\lvert (\mathcal{D}+J)x\rvert$.
\end{definizione}

The following proposition gives a first characterisation of the shape of the intersection between $B_r(\mathcal{X})$ with $\text{gr}(f)$. In particular we construct a function $G$ at the point $x$ whose sublevel sets are the horizontal projection of $B_r(\mathcal{X})\cap\text{gr}(f)$.

\begin{proposizione}\label{pll}
In the notations above, defined $G(w):=\lvert w\rvert^4+\lvert c\langle \mathfrak{n},w\rangle+\langle w,\mathcal{D}w\rangle \rvert^2$, we have:
\begin{equation}
   \pi_H(B_r(\mathcal{X})\cap\text{gr}(f))=x+\{w\in\R^{2n}:G(w)\leq r^4\}.
   \nonumber
\end{equation}
\end{proposizione}

\begin{proof}
By definition of $\mathcal{X}$ and of the Koranyi norm, we have:
\begin{equation}
\begin{split}
    B_r(\mathcal{X}):
    =&\{z\in\R^{2n+1}:\lvert z_H-x\rvert^4+\lvert z_T-f(x)-2\langle x,Jz_H\rangle \rvert^2\leq r^4\}.
\end{split}
    \nonumber
\end{equation}
Therefore, the intersection of $B_r(\mathcal{X})$ with $\text{gr}(f)$ is:
\begin{equation}
\begin{split}
&B_r(\mathcal{X})\cap\text{gr}(f)=\{(y,f(y))\in\R^{2n+1}:\lvert x-y\rvert^4+\lvert -f(x)+f(y)-2\langle  x,Jy\rangle\rvert^2\leq r^4\}\\
=&\mathcal{X}+\{(w,f(w)+2\langle w, \mathcal{D} x\rangle)\in\R^{2n+1}:\lvert w\rvert^4+\lvert -f(x)+f(x+w)-2\langle  x,Jw\rangle\rvert^2\leq r^4\},
\nonumber
\end{split}
\end{equation}
where in the last line we have performed the change of variable $y=x+w$. By definition of $f$, we have:
\begin{equation}
\begin{split}
    -f(x)+f(x+w)=-\langle x,\mathcal{D}x\rangle+\langle x+w, \mathcal{D}(x+w)\rangle=2\langle x,\mathcal{D} w\rangle+\langle w,\mathcal{D}w\rangle.
\end{split}
    \nonumber
\end{equation}
In particular, this implies that:
    \begin{equation}
\begin{split}
\pi_H(B_r(\mathcal{X})\cap\text{gr}(f))=&x+\{w\in\R^{2n}:\lvert w\rvert^4+\lvert 2\langle x,\mathcal{D} w\rangle+\langle w,\mathcal{D}w\rangle+2\langle  Jx,w\rangle\rvert^2\leq r^4\}\\
=&x+\{w\in\R^{2n}:\lvert w\rvert^4+\lvert c\langle \mathfrak{n},w\rangle+\langle w,\mathcal{D}w\rangle\rvert^2\leq r^4\}.
\label{N:1}
\end{split}
\end{equation}
Identity \eqref{N:1} and the definition of $G$ conclude the proof.
\end{proof}

The following proposition introduces a special set of polar coordinates, which are going to be very useful in the study of the intersection $B_r(\mathcal{X})\cap \text{gr}(f)$ when $r$ is small.

\begin{proposizione}\label{coordinate}
For any $w\in \R^{2n}\setminus x+\text{span}( \mathfrak{n})$ there exists a unique triple $(\vartheta,\rho,v)\in \mathscr{C}:=[-\frac{\pi}{2},\frac{\pi}{2})\times(0,\infty)\times\mathbb{S}^{2n-1}\cap \mathfrak{n}^\perp$ such that:
\begin{equation}
    w=x+\frac{\sin \vartheta}{c} \rho^2 \mathfrak{n}+\cos\vartheta\rho v=:x+\mathcal{P}(\vartheta,\rho,v).
    \label{N:4}
\end{equation}
\end{proposizione}

\begin{proof}
Any $w\in\R^{2n}\setminus x+\text{span}(\mathfrak{n})$ can be uniquely written as $w=x+\lambda n+u$ for some $\lambda\in\R$ and $u\in \mathfrak{n}^\perp\setminus\{0\}$. Defined $\rho:=\sqrt{(\lvert u\rvert^2+\sqrt{\lvert u\rvert^4+4\lambda^2c^2})/2}$, we have that $\rho\neq 0$ and that:
\begin{equation}
(c\lambda/\rho^2)^2+(\lvert u\rvert/\rho)^2=1,
\label{numbero25}
\end{equation}
since $\rho^2$ solves the equation $\zeta^2-\lvert u\rvert^2\zeta-c^2\lambda^2=0$.
Thanks to identity \eqref{numbero25}, there is a unique $\vartheta\in[-\pi/2,\pi/2)$ for which $\sin\vartheta=c\lambda/\rho^2$ and $\cos\vartheta=\lvert u\rvert/\rho$. Eventually, if we let $v:=u/\lvert u\rvert$, thanks to the definition of $\vartheta$, we have:
$$w=x+\lambda n+\lvert u\rvert v=x+\frac{\sin \vartheta}{c} \rho^2 \mathfrak{n}+\cos\vartheta\rho v.$$
This concludes the proof of the proposition.
\end{proof}

\textbf{In order to simplify the notations in the forthcoming propositions, we define:}
\begin{equation}
  \alpha_\mathfrak{n}:=\left\langle \mathfrak{n},\mathcal{D}\mathfrak{n}\right\rangle ,\qquad\beta_\mathfrak{n}(v):=\left\langle v,\mathcal{D}\mathfrak{n}\right\rangle,\qquad\gamma(v):=\left\langle v,\mathcal{D}v\right\rangle \qquad \text{for any }v\in\mathbb{S}^{2n-1}. 
  \label{numbero27}
\end{equation}
In the following proposition we give an explicit expression of $G$ in the new polar coordinates $\mathcal{P}(\vartheta,\rho,v)$ introduced in Proposition \ref{coordinate}.

\begin{proposizione}\label{strutto}
Let us define the function $H:\mathscr{C}\to\R$ as:
\begin{equation}
    H(\vartheta,\rho,v):=G(\mathcal{P}(\vartheta,\rho,v)),
    \label{numerooo3}
\end{equation}
where $G$ was introduced in Proposition \ref{pll} and both $\mathscr{C}$ and $\mathcal{P}(\vartheta,\rho,v)$ were defined in Proposition \ref{coordinate}.
Then, $H$ has the following explicit expression:
$$H(\vartheta,\rho,v):=A(\vartheta,v)\rho^4+\frac{\overline{B}(\vartheta, v)}{c}\rho^5+\frac{\overline{C}(\vartheta, v)}{c^2}\rho^6+\frac{\overline{D}(\vartheta, v)}{c^3}\rho^7+\frac{\overline{E}(\vartheta,\rho)}{c^4}\rho^8,$$
and where as $(\vartheta,v)$ varies in $[-\pi/2,\pi/2)\times\mathbb{S}^{2n-1}\cap \mathfrak{n}^\perp$, we define:
\begin{itemize}
\item[(i)] $A(\vartheta,v):=(\cos^4\vartheta+(\cos^2\vartheta\gamma(v)+\sin\vartheta)^2)$,
\item[(ii)] $\overline{B}(\vartheta,v):=cB(\vartheta,v):=4\sin\vartheta\cos\vartheta\beta_\mathfrak{n}(v)(\cos^2\vartheta\gamma(v)+\sin\vartheta)$,
\item[(iii)] $\overline{C}(\vartheta,v):=c^2C(\vartheta,v):=\sin^2\vartheta(\cos^2\vartheta(2+4\beta_\mathfrak{n}(v)^2+2\gamma(v)\alpha_\mathfrak{n})+2\sin\vartheta\alpha_\mathfrak{n})$,
\item[(iv)] $\overline{D}(\vartheta,v):=c^3D(\vartheta,v):=4\alpha_\mathfrak{n}\beta_\mathfrak{n}(v)\sin\vartheta^3\cos\vartheta$,
\item[(v)] $\overline{E}(\vartheta):=c^4E(\vartheta):=(1+\alpha_\mathfrak{n}^2)\sin^4\vartheta$,
\end{itemize}
\end{proposizione}

\begin{proof}
For any $x+w\in\R^{2n}\setminus x+\text{span}(\mathfrak{n})$, by Proposition \ref{coordinate} we can find a unique $(\vartheta,\rho,v)\in\mathscr{C}$ such that $w=\mathcal{P}(\vartheta,\rho,v)$. First note that the following identities hold:
\begin{equation}
\begin{split}
    \lvert w\rvert^4=&\lvert\mathcal{P}(\vartheta,\rho,v)\rvert^4=\frac{\sin^4\vartheta}{c^4} \rho^8+2\frac{\sin^2\vartheta\cos^2\vartheta}{c^2} \rho^6+\cos^4\vartheta\rho^4,\\
    c\langle \mathfrak{n},w\rangle=&c\langle \mathfrak{n},\mathcal{P}(\vartheta,\rho,v)\rangle=\Big\langle c\mathfrak{n},\frac{\sin \vartheta}{c} \rho^2 \mathfrak{n}+\cos\vartheta\rho v\Big\rangle=\sin\vartheta\rho^2,\\
     \langle w,\mathcal{D}w\rangle=&\langle \mathcal{P}(\vartheta,\rho,v),\mathcal{D}[\mathcal{P}(\vartheta,\rho,v)]\rangle=\frac{\sin^2\vartheta}{c^2} \rho^4\alpha_\mathfrak{n}+2\frac{\sin \vartheta\cos\vartheta}{c} \rho^3\beta_\mathfrak{n}(v)+\cos^2\vartheta\rho^2\gamma(v).
\end{split}
    \label{numbero26}
\end{equation}
Therefore, thanks to the definitions of $G$ and $H$ and the three identities in \eqref{numbero26}, we have:
\begin{equation}
\begin{split}
H(\rho,\vartheta&,v)=G(\mathcal{P}(\vartheta,\rho,v))=\lvert\mathcal{P}(\vartheta,\rho,v)\rvert^4+\lvert c\langle\mathfrak{n},\mathcal{P}(\vartheta,\rho,v)\rangle+\langle \mathcal{P}(\vartheta,\rho,v),\mathcal{D}[\mathcal{P}(\vartheta,\rho,v)]\rangle\rvert^2\\
=&\frac{\sin^4\vartheta}{c^4} \rho^8+2\frac{\sin^2\vartheta}{c^2}\cos^2\vartheta\rho^6+\cos^4\vartheta\rho^4+(\cos^2\vartheta\gamma(v)+\sin\vartheta)^2\rho^4\\
&+4\frac{\sin^2\vartheta\cos^2\vartheta}{c^2}\beta_\mathfrak{n}(v)^2\rho^6
+\frac{\sin^4\vartheta}{c^4}\alpha_\mathfrak{n}^2\rho^8
+4\frac{\sin\vartheta\cos\vartheta}{c}(\cos^2\vartheta\gamma(v)+\sin\vartheta)\beta_\mathfrak{n}(v)\rho^5\\
&+2(\cos^2\vartheta\gamma(v)+\sin\vartheta)\frac{\sin^2\vartheta}{c^2}\alpha_\mathfrak{n}\rho^6
+4\frac{\sin\vartheta^3\cos\vartheta}{c^3}\alpha_\mathfrak{n}\beta_\mathfrak{n}(v)\rho^7.
\nonumber
\end{split}
\end{equation}
The claim follows recollecting the various powers of $\rho$.
\end{proof}

We summarize in the following lemma some algebraic properties of the functions $A,\ldots,E$ introduced in Proposition \ref{strutto}, since they will be very useful in the forthcoming computations.

\begin{lemma}\label{symA}
Consider $(\vartheta,v)\in[-\pi/2,\pi/2)\times\mathbb{S}^{2n}\cap \mathfrak{n}^\perp$ fixed. Then:
\begin{itemize}
\item[(i)] $A(\vartheta,v)=A(\vartheta,-v)$ and $C(\vartheta,v)=C(\vartheta,-v)$,
\item[(ii)] $B(\vartheta,v)=-B(\vartheta,-v)$ and $D(\vartheta,v)=-D(\vartheta,-v)$,
\item[(iii)] $D(\vartheta,v)=-D(-\vartheta,v)$,
\item[(iv)] $E$ does not depend on $v$,
\item[(v)] $A$ is bounded away from $0$, i.e. $\omega:=\min_{(\vartheta,v)\in[-\pi/2,\pi/2)\times\mathbb{S}^{2n}\cap \mathfrak{n}^\perp}A(\vartheta,v)>0$.
\end{itemize}
\end{lemma}

\begin{proof}
The first four points are direct consequence of the definition of $A$, $B$, $C$, $D$, $E$, and of $\alpha_\mathfrak{n},\beta_\mathfrak{n}(v),\gamma(v)$, see \eqref{numbero27}.
We are left to prove the last point. Since $A(\pi/2,v)=A(-\pi/2,v)$, the function $A$ has minimum in $(-\pi/2,\pi/2)\times\mathbb{S}^{2n}\cap \mathfrak{n}^\perp$. Suppose such minimum is $0$ and it is attained at $(\vartheta,v)$.
This would imply that $0=\cos^4\vartheta+(\cos^2\vartheta\gamma(v)+\sin\vartheta)^2$, but this is not possible as it would force $\sin\vartheta=\cos\vartheta=0$.
\end{proof}

The following proposition allows us to determine, up to a certain degree of precision, the shape of the set $\pi_H(B_r(\mathcal{X})\cap\text{gr}(f))$ when $r$ is small:

\begin{proposizione}\label{propexp}
There exists an $0<\mathfrak{r}_1(\mathcal{X})=\mathfrak{r}_1<1$ such that for any $0<r<\mathfrak{r}_1$, if $\rho(r)$ is a solution to the equation:
\begin{equation}
    H(\vartheta,\rho(r),v)=r^4,
    \label{eq152}
\end{equation}
then:
\begin{equation}
    \rho(r)=P_{\vartheta,v}(r)+O(r^4):=\frac{r}{A^\frac{1}{4}}-\frac{Br^2}{4A^\frac{3}{2}}+\left(\frac{7}{32}\frac{B^2}{A^\frac{11}{4}}-\frac{C}{4A^\frac{7}{4}}\right)r^3+O(r^4),
    \label{N:3}
\end{equation}
and the remainder $O(r^4)$ is independent on $v$ and on $\vartheta$.
\end{proposizione}

\begin{proof}
By Proposition \ref{strutto}, equation \eqref{eq152} turns into:
\begin{equation}
    A\rho(r)^4+B\rho(r)^5+C\rho(r)^6+D\rho(r)^7+E\rho(r)^8=r^4,
    \label{numbero28}
\end{equation}
where we dropped the dependence on $v$ and $\vartheta$ from $A$, $B$, $C$, $D$ and $E$ to simplify the notation.

We claim that there are $\mathfrak{r}_1>0$ and $\mathfrak{c}_1>0$, independent on $\vartheta$ and $v$, such that $\rho(r)\leq \mathfrak{c}_1r$ for any $0<r<\mathfrak{r}_1$. Suppose by contradiction there exists a sequence $(r_i,\vartheta_i,v_i)$ such that $G(\rho(r_i),\vartheta_i,v_i)=r_i^4$ for any $i\in\N$, $\{r_i\}_{i\in\N}$ is infinitesimal and $\rho(r_i)> 2r_i/\omega^{1/4}$, where $\omega$ is the constant introduced in Proposition \ref{strutto}(v). Thanks to identity \eqref{numbero28}, we have that for any $i\in\N$ the existence of such a sequence $\{r_i\}_{i\in\N}$ implies that:
\begin{equation}
\begin{split}
    1=&A\left(\frac{\rho(r_i)}{r_i}\right)^4+B\left(\frac{\rho(r_i)}{r_i}\right)^5r_i+C\left(\frac{\rho(r_i)}{r_i}\right)^6r_i^2+D\left(\frac{\rho(r_i)}{r_i}\right)^7r_i^3+E\left(\frac{\rho(r_i)}{r_i}\right)^8r_i^4\\
    >&\underbrace{A\left(\frac{2}{\omega^{1/4}}\right)^4+B\left(\frac{2}{\omega^{1/4}}\right)^5r_i+C\left(\frac{2}{\omega^{1/4}}\right)^6r_i^2+D\left(\frac{2}{\omega^{1/4}}\right)^7r_i^3+E\left(\frac{2}{\omega^{1/4}}\right)^8r_i^4}_{=:RHS_\eqref{numbero29}}
    \label{numbero29}
\end{split}
\end{equation}
Define the constant:
$$M:=(1+2/\omega^{1/4})^8\max_{(\vartheta,\gamma)\in[-\pi/2,\pi/2]\times\mathbb{S}^{2n}\cap \mathfrak{n}^\perp}(\lvert B\rvert+\lvert C\rvert+\lvert D\rvert+\lvert E\rvert)(\vartheta,\gamma),$$
and note that provided $r_i<\min(1,1/M)$, inequality \eqref{numbero29} implies:
\begin{equation}
    1>RHS_\eqref{numbero29}\geq A\left(\frac{2}{\omega^{1/4}}\right)^4-Mr_i>A\left(\frac{2}{\omega^{1/4}}\right)^4-1.
    \label{numerooo1}
\end{equation}
Rearranging \eqref{numerooo1}, we infer that $A/\omega<1/8$ however, this is not possible by the very definition of $\omega$. In particular, this proves that there exists an $\mathfrak{r}_1>0$ such that for any $0<r<\mathfrak{r}_1$ we have $\rho(r)\leq 2r/\omega^{1/4}$.
Thanks to \eqref{numbero28}, we have:
\begin{equation}
    \rho(r)=\frac{r}{A^\frac{1}{4}}+R_1(r),
    \label{eq30}
\end{equation}
where $\lvert R_1(r)\rvert\leq\mathfrak{c}_2 r^2$ for any $0<r<\mathfrak{r}_1$ and for some constant $\mathfrak{c}_2>0$ independent on $\vartheta$ and $v$  since we have established the function $A$ is bounded below  on $[-\pi/2,\pi/2)\times\mathbb{S}^{2n-1}\cap \mathfrak{n}^\perp$ in Lemma \ref{symA}(v).
In order to obtain the other coefficients of the expansion of $\rho$ we proceed in the same way as above: plug the expression for $\rho$ into \eqref{numbero28}, get an equation the remainder must satisfy and solve it, obtaining:
\begin{equation}
\begin{split}
    A(r/A^{1/4}+R_1(r))^4+B(r/A^{1/4}+R_1(r))^5+\underbrace{(C\rho(r)^6+D\rho(r)^7+E\rho(r)^8)}_{\tilde{R}_2(r)}=r^4.
    \end{split}
\end{equation}
The above identity thanks to few algebraic computations boils down to:
\begin{equation}
\begin{split}
  r^4 
    =&r^4+4A^\frac{1}{4}r^3R_1(r)+\frac{Br^5}{A^\frac{5}{4}}+R_2(r)
   \label{numerooo2}
\end{split}
\end{equation}
where $\lvert R_2\rvert\leq \mathfrak{c}_3r^6$ for any $0<r<\mathfrak{r}_1$ and for some constant $\mathfrak{c}_3$ independent on $\vartheta$ and $v$. Solving equation \eqref{numerooo2} with respect to $R_1$, we infer that:
$$R_1(r)=-\frac{Br^2}{4A^\frac{3}{2}}-\frac{R_2(r)}{4A^\frac{1}{4}r^3}.$$
Substituting the above identity in \eqref{eq30}, we have:
\begin{equation}
    \rho(r)=\frac{r}{A^\frac{1}{4}}-\frac{Br^2}{4A^\frac{3}{2}}+R_3(r),
    \label{eq31}
\end{equation}
where $R_3(r):=R_2(r)/4A^\frac{1}{4}r^3$. In particular $\lvert R_3(r)\rvert\leq \mathfrak{c}_3/4\omega^\frac{1}{4}r^3$ for any $0<r<\mathfrak{r}_1$. 
Finally, substituting the newfound expression for $\rho(r)$ given by \eqref{eq31} in \eqref{numbero28}, we get the following expression for $R_3$:
$$R_3(r):=\left(\frac{7}{32}\frac{B^2}{A^\frac{11}{4}}-\frac{C}{4A^\frac{7}{4}}\right)r^3+R_4(r),$$
where $\lvert R_4(r)\rvert\leq \mathfrak{c}_4r^4$ for any $0<r<\mathfrak{r}_1$ and for some constant $\mathfrak{c}_4>0$ independent on $\vartheta$ and $v$.
\end{proof}

\textbf{In the following we will denote with $\mathfrak{r}_2$ the supremum of those positive numbers for which} $P_{\vartheta,v}(r)\geq \omega^{1/4} r/2$.
Proposition \ref{propexp} has the following immediate consequence:
 
\begin{corollario}\label{Co1}
For any $0<r<\min\{1,\mathfrak{r}_2\}$ and any $\delta\in (-\omega^{1/4}/2,\omega^{1/4}/2)$, define the set:
$$\mathcal{B}_{r,\delta}:=\mathcal{P}( \{(\rho,\vartheta,v)\in\mathscr{C}: \rho\leq P_{\vartheta,v}(r)+\delta r^3\}),$$
where $P_{\vartheta,v}$ was defined in \eqref{N:3}.
Then, there exists an $\epsilon_0(\mathcal{X})=\epsilon_0>0$ such that for any $0<\epsilon<\epsilon_0$, there is an $0<\mathfrak{r}_3(\epsilon)=\mathfrak{r}_3$ such that for any $0<r<\mathfrak{r}_3$, we have:
\begin{equation}
x+\mathcal{B}_{r,-\epsilon}\subseteq\pi_H( B_{r}(\mathcal{X})\cap\text{gr}(f))\subseteq x+\mathcal{B}_{r,\epsilon}.
    \nonumber
\end{equation}
\end{corollario}

\begin{proof}
Proposition \ref{pll}, the definition of $\mathcal{P}$, see \eqref{N:4}, and of $H$, see \eqref{numerooo3}, imply:
\begin{equation}
    \pi_H(B_r(\mathcal{X})\cap\text{gr}(f))=x+\mathcal{P}(\{(\rho,\vartheta, v)\in\mathscr{C}:H(\rho,\vartheta, v)\leq r^4\}).
    \label{numbero42}
\end{equation}
The function $\rho\mapsto H(\rho,\vartheta, v)$ is a polynomial of $8$th degree in $\rho$, and thus the equation:
\begin{equation}
    H(\rho,\vartheta, v)=r^4,
    \label{eq172}
\end{equation}
has at most $8$ solutions in $\rho$. Assume $\rho_1<\ldots<\rho_k$, where $k\in\{1,\ldots,8\}$, is the number of \emph{positive} distinct solutions of \eqref{eq172}. If $\rho>\rho_k$ then $H(\rho,\vartheta,v)>r^4$ and on the other hand, since $H(0,\vartheta,v)=0$, if $0\leq\rho<\rho_1$ then $G(\rho,\vartheta, v)<r^4$. This implies that:
\begin{equation}
    \begin{split}
        &\mathcal{P}\left( \left\{(\rho,\vartheta,v)\in\mathscr{C}: \rho\leq \rho_1\right\}\right)\\
        &\qquad\subseteq\mathcal{P}(\{(\rho,\vartheta, v)\in\mathscr{C}:H(\rho,\vartheta, v)\leq r^4\})\\
        &\qquad\qquad\subseteq \mathcal{P}\left( \left\{(\rho,\vartheta,v)\in\mathscr{C}: \rho\leq \rho_k\right\}\right).
        \nonumber
    \end{split}
\end{equation}
Proposition \ref{propexp} concludes the proof since $\rho_1$ and $\rho_k$ coincide up to an error of order $r^4$.
\end{proof}

The following technical lemma will be needed in the computations of Proposition \ref{TEXP}, and it is a Taylor expansion formula for the sub-Riemmanian area element at a non-characteristic point of a horizontal quadric.

\begin{lemma}\label{sviluppodens}
For any $(\rho,\vartheta, v)\in\mathscr{C}$ we have:
\begin{equation}
2\lvert (\mathcal{D}+J)[x+\mathcal{P}(\rho,\vartheta, v)]\rvert=c+\mathcal{A}\rho+\mathcal{B}\rho^2+R_5(\rho),
\nonumber
\end{equation}
where:
\begin{itemize}
\item[(i)] $\mathcal{A}:=\mathcal{A}(\vartheta,v):=2\cos\vartheta(\beta_\mathfrak{n}(v)+\langle Jv,n\rangle)$,
\item[(ii)] $\overline{\mathcal{B}}=c\mathcal{B}:=c\mathcal{B}(\vartheta,v):=2\left(\alpha_\mathfrak{n} \sin\vartheta+\lvert P_\mathfrak{n}[(\mathcal{D}+J)v]\rvert^2\cos^2\vartheta\right)$,
and $P_\mathfrak{n}$ is the orthogonal projection in $\R^{2n}$ on $\mathfrak{n}^\perp$.
\item[(iii)] $\lvert R_5(\rho)\rvert\leq \mathfrak{c}_5\rho^3$ for any $0<\rho<\mathfrak{r}_4$ and for some constant $\mathfrak{c}_5>0$, independent on $\vartheta$ and $v$.
\end{itemize}
\end{lemma}

\begin{proof}
In order to simplify the notation we let $M:=(\mathcal{D}+J)$. First of all we find a more explicit expression for the vector $2M[x+\mathcal{P}(\rho,\vartheta, v)]$:
\begin{equation}
\begin{split}
     2M[x+\mathcal{P}(\rho,\vartheta, v)]
    =2M\left[x+\frac{\sin \vartheta}{c} \rho^2 \mathfrak{n}+\cos\vartheta\rho v\right]=c\mathfrak{n}+\frac{2\sin \vartheta}{c} \rho^2M \mathfrak{n}+2\cos\vartheta\rho  Mv.
\end{split}
\nonumber
\end{equation}
Secondly, we compute the squared norm of the vector $ 2M[x+\mathcal{P}(\rho,\vartheta, v)]$:
\begin{equation}
    \begin{split}
        \lvert  2M[x+\mathcal{P}(\rho,\vartheta, v)]\rvert^2=c^2&+4c\cos\vartheta\rho  \langle Mv,n\rangle+4\sin \vartheta \rho^2\langle M \mathfrak{n},\mathfrak{n}\rangle+4\cos^2\vartheta\rho^2\lvert Mv\rvert^2\\
        &+\frac{8\sin \vartheta\cos\vartheta}{c}\rho^3\langle M\mathfrak{n},Mv\rangle+\frac{4\sin^2 \vartheta}{c^2} \rho^4\lvert M \mathfrak{n}\rvert^2.
    \end{split}
\label{numerooo6}
\end{equation}
Thanks to the definition of $\alpha_\mathfrak{n}$ and $\beta_\mathfrak{n}(v)$, on the one hand we have $\langle Mv,\mathfrak{n}\rangle=\beta_\mathfrak{n}(v)+\langle Jv,\mathfrak{n}\rangle$ and thus:
\begin{equation}
    4c\cos\vartheta\langle Mv,\mathfrak{n}\rangle=2c\mathcal{A}.
    \label{numerooo4}
\end{equation}
Moreover, the fact that $\langle M\mathfrak{n},\mathfrak{n}\rangle=\alpha_\mathfrak{n}$ and that $\lvert Mv\rvert^2-\langle Mv,\mathfrak{n}\rangle^2=\lvert P_\mathfrak{n}(Mv)\rvert^2$ imply:
\begin{equation}
    \begin{split}
  4\sin\vartheta\langle M\mathfrak{n},\mathfrak{n}\rangle+4\cos^2\vartheta\lvert Mv\rvert^2=2c\mathcal{B}+\mathcal{A}^2.
        \label{numerooo5}
    \end{split}
\end{equation}
Thanks to identities \eqref{numerooo6}, \eqref{numerooo4} and \eqref{numerooo5}, we deduce that:
\begin{equation}
\begin{split}
\lvert 2M[x+&\mathcal{P}(\rho,\vartheta, v)]\rvert-c-\mathcal{A}\rho-\mathcal{B}\rho^2=\frac{\lvert 2M[x+\mathcal{P}(\rho,\vartheta, v)]\rvert^2-(c+\mathcal{A}\rho+\mathcal{B}\rho^2)^2}{\lvert 2M[x+\mathcal{P}(\rho,\vartheta, v)]\rvert+c+\mathcal{A}\rho+\mathcal{B}\rho^2}\\
&=\frac{\left(8c^{-1}\sin \vartheta\cos\theta\langle M\mathfrak{n},Mv\rangle-2\mathcal{A}\mathcal{B}\right)\rho^3+\left(4c^{-2}\sin^2 \vartheta\lvert M \mathfrak{n}\rvert^2-\mathcal{B}^2\right)\rho^4}{\lvert 2M[x+\mathcal{P}(\rho,\vartheta, v)]\rvert+c+\mathcal{A}\rho+\mathcal{B}\rho^2}.
\label{numerooo7}
\end{split}
\end{equation}
Summing up, \eqref{numerooo7} shows that when $\rho$ is small we have:
$$\lvert 2M[x+\mathcal{P}(\rho,\vartheta, v)]\rvert-c-\mathcal{A}\rho-\mathcal{B}\rho^2\sim (2c)^{-1}(8c^{-1}\sin \vartheta\cos\theta\langle M\mathfrak{n},Mv\rangle-2\mathcal{A}\mathcal{B})\rho^3,$$
or more precisely, we infer that there exist an $\mathfrak{r}_4>0$ and a constant $\mathfrak{c}_5>0$, that can be chosen both independent on $\vartheta$ and $v$, for which:
$$\big\lvert \lvert 2M[x+\mathcal{P}(\rho,\vartheta, v)]\rvert-c-\mathcal{A}\rho-\mathcal{B}\rho^2\big\rvert\leq \mathfrak{c}_5\rho^3,\qquad \text{for any }0<\rho<\mathfrak{r}_4.$$
This concludes the proof.
\end{proof}

\begin{osservazione}\label{rk22}
The functions $\mathcal{A}(\cdot,\cdot)$ and $\mathcal{B}(\cdot,\cdot)$ defined in the statement of Proposition \ref{sviluppodens} have the following symmetries.
For any $(\vartheta,v)\in[-\pi/2,\pi/2]\times\mathbb{S}^{2n-1}\cap \mathfrak{n}^\perp$, we have that:
\begin{itemize}
\item[(i)] $\mathcal{A}(\vartheta,v)=-\mathcal{A}(\vartheta,-v)$,
\item[(ii)] $\mathcal{B}(\vartheta,v)=\mathcal{B}(\vartheta,-v)$.
\end{itemize}
\end{osservazione}

\begin{proposizione}\label{rapr1}
For any $0<r<\mathrm{dist}(x,\Sigma(f))$, we have:
$$\lvert\partial \mathbb{K}\rvert(B_r(\mathcal{X}))= \int_{\mathbb{S}^{2n-1}\cap\mathfrak{n}^\perp}\int_{-\frac{\pi}{2}}^{\frac{\pi}{2}}\int_{\{H(\rho,\vartheta,v)\leq r^4\}}\Xi(\rho,\vartheta, v)  d\rho d\vartheta d\sigma(v),$$
where:
\begin{itemize}
\item[(i)] $\Xi(\rho,\vartheta, v):=c^{-1}\rho^{2n}\cos^{2n-2}\vartheta(1+\sin^2\vartheta)\cdot 2\lvert(\mathcal D+J)[x+\mathcal{P}(\rho,\vartheta, v)]\rvert$,
\item[(ii)] $\sigma:=\mathcal{H}^{2n-2}_{eu}\llcorner \mathbb{S}^{2n-1}\cap \mathfrak{n}^\perp$.
\end{itemize}
\end{proposizione}

\begin{proof}
Defined $U_r:=\pi_H(B_r(\mathcal{X})\cap \mathrm{gr}(f))$, \cite[Proposition 3.2]{Monti2014IsoperimetricGroup} implies that:
\begin{equation}
    \lvert\partial \mathbb{K}\rvert(B_r(\mathcal{X}))=\int_{U_r} \lvert\nabla f(w)+2Jw\rvert dw=2\int_{U_r} \lvert(\mathcal{D}+J)w\rvert dw.
    \label{numbero40}
\end{equation}
We want to perform in the right-hand side of the above equation, the following change of variables:
\begin{equation}
   w=x+\frac{\sin\vartheta\rho^2}{c}\mathfrak{n}+\cos\vartheta\rho v=x+\mathcal{P}(\rho,\vartheta,v),\text{ where }(\rho,\vartheta, v)\in\mathscr{C}. 
   \label{numbero41}
\end{equation}
If $n>1$, we can parametrize the sphere $\mathbb{S}^{2n-1}\cap \mathfrak{n}^\perp$ with the usual polar coordinates and in order to do so we let $v=v(\psi)$, where $\psi$ varies in $\Psi:= [-\pi,\pi)\times[\pi/2,\pi/2]^{2n-3}$. 
The Jacobian determinant, obtained using the Laplace formula, we omit the computations, of the change of variables \eqref{numbero41} is:
\begin{equation}
\begin{split}
\Big\lvert\mathrm{det}\frac{\partial w(\rho,\vartheta, v)}{\partial(\rho,\vartheta,\psi)}\Big\rvert=&\rho^{2n}c^{-1}(1+\sin^2\vartheta)\cos^{2n-2}\vartheta\Big\lvert \mathrm{det}\Big(\mathfrak{n},v(\psi),\frac{\partial v(\psi)}{\partial \psi}\Big)\Big\rvert.
\nonumber 
\end{split}
\end{equation}
Let $\mathscr{J}(\psi):=\Big\lvert\mathrm{det}\Big(\mathfrak{n},v(\psi),\frac{\partial v(\psi)}{\partial \psi}\Big)\Big\rvert$ and note that by the change of variable formula, \eqref{numbero40} becomes:
\begin{equation}
\begin{split}
&\qquad\qquad\qquad\qquad\qquad\qquad\qquad\lvert\partial \mathbb{K}\rvert(B_r(\mathcal{X}))=\\
=&2\int_{\Psi}\int_{-\frac{\pi}{2}}^{\frac{\pi}{2}}\int \chi_{U_r}({\scriptstyle x+\mathcal{P}(\rho,\vartheta,\psi)})\frac{{\scriptstyle\rho^{2n}(1+\sin^2\vartheta)\cos^{2n-2}\vartheta}}{c}\mathscr{J}(\psi) \lvert(\mathcal{D}+J)({\scriptstyle x+\mathcal{P}(\rho,\vartheta,v)})\rvert d\rho d\vartheta d\psi\\
&\qquad\qquad=\int_{\Psi}\mathscr{J}(\psi)\int_{-\frac{\pi}{2}}^{\frac{\pi}{2}}\int \chi_{U_r}(x+\mathcal{P}(\rho,\vartheta,\psi))\Xi(\rho,\vartheta, v(\psi)) d\rho d\vartheta d\psi\\
&\qquad\qquad=\int_{\mathbb{S}^{2n-1}\cap\mathfrak{n}^\perp}\int_{-\frac{\pi}{2}}^{\frac{\pi}{2}}\int \chi_{U_r}(x+\mathcal{P}(\rho,\vartheta,\psi))\Xi(\rho,\vartheta, v)  d\rho d\vartheta d\sigma(v)\\
&\qquad\qquad=\int_{\mathbb{S}^{2n-1}\cap\mathfrak{n}^\perp}\int_{-\frac{\pi}{2}}^{\frac{\pi}{2}}\int_{\{H(\rho,\vartheta,v)\leq r^4\}}\Xi(\rho,\vartheta, v)  d\rho d\vartheta d\sigma(v),
\nonumber
\end{split}
\end{equation}
where the last equality sign comes from the identity \eqref{numbero42}.
If on the other hand $n=1$, the computation is even simpler since $\mathbb{S}^{2n-1}\cap \mathfrak{n}^\perp=\{\pm J\mathfrak{n}\}$.
\end{proof}

The following two lemmas will allow us to compute some integrals in Propositions \ref{TEXP} and \ref{TEXP2}.

\begin{lemma}\label{conto1}
For any $k\in\N$ and any $\alpha>(k+1)/2$ we have:
\begin{equation}
    \int_{-\infty}^\infty \frac{x^k}{(1+x^2)^\alpha}dx=\begin{cases}
    0 &\text{if }k\text{ is odd,}\\
   \frac{\Gamma\big(\frac{k+1}{2}\big)\Gamma\big(\alpha-\frac{k+1}{2}\big)}{\Gamma(\alpha)} &\text{if }k\text{ is even.}
    \end{cases}
    \nonumber
\end{equation}
\end{lemma}

\begin{proof}
If $k$ is odd, then $\int x^k/(1+x^2)^\alpha=0$. If on the other hand $k$ is even, the change of variable $t=1/(1+x^2)$ implies:
\begin{equation}
    \begin{split}
        \int_{-\infty}^\infty \frac{x^k}{(1+x^2)^\alpha}dx=&\int_0^1(1-t)^{\frac{k+1}{2}-1}t^{(\alpha-\frac{k+1}{2})-1}dt\\
        =&\beta\bigg(\frac{k+1}{2},\alpha-\frac{k+1}{2}\bigg)=\frac{\Gamma\big(\frac{k+1}{2}\big)\Gamma\big(\alpha-\frac{k+1}{2}\big)}{\Gamma(\alpha)},
        \nonumber
    \end{split}
\end{equation}
where $\beta(\cdot,\cdot)$ is the Euler's beta function. The last equality follows from a well known property of $\beta$, see for instance \cite[Theorem 12.41]{Whittaker.1902AAnalysis}.
 \end{proof}

\begin{lemma}\label{conto2}
Suppose $f:\R\to\R$ is a measurable function such that $f(x)/(1+x^2)^\alpha\in L^1(\R)$ and let $\mathfrak{d}(\vartheta):=\cos^{2n-2}\vartheta(\cos^2\vartheta+2\sin^2\vartheta)$. Then the following identity holds:
$$\int_{-\frac{\pi}{2}}^\frac{\pi}{2}\mathfrak{d}(\vartheta)\frac{\cos^{4\alpha-2n-1}\vartheta f\left(\frac{\sin\vartheta}{\cos^2\vartheta}\right)}{A(\vartheta,v)^\alpha}d\vartheta=\int_{-\infty}^\infty\frac{f(x)}{\big(1+\big(x+\gamma(v)\big)^2\big)^\alpha}dx,$$
where $\gamma(v)$ was defined in \eqref{numbero27} and where $A(\vartheta,v)$ was introduced in Proposition \ref{strutto}.
\end{lemma}

\begin{proof}
Thanks to the definition of $A$, we have that:
\begin{equation}
\begin{split}
    &\int_{-\frac{\pi}{2}}^\frac{\pi}{2}\mathfrak{d}(\vartheta)\frac{\cos^{4\alpha-2n-1}(\vartheta) f\left(\frac{\sin\vartheta}{\cos^2\vartheta}\right)}{A(\vartheta,v)^\alpha}d\vartheta=\int_{-\frac{\pi}{2}}^\frac{\pi}{2}\mathfrak{d}(\vartheta)\frac{\cos^{4\alpha-2n-1}\vartheta f\left(\frac{\sin\vartheta}{\cos^2\vartheta}\right)}{\big(\cos^4\vartheta+(\cos^2\vartheta\gamma(v)+\sin\vartheta)^2\big)^\alpha}d\vartheta\\
    =&\int_{-\frac{\pi}{2}}^\frac{\pi}{2}\frac{\cos^2\vartheta+2\sin^2\vartheta}{\cos^3\vartheta}\frac{f\left(\frac{\sin\vartheta}{\cos^2\vartheta}\right)}{\big(1+(\frac{\sin\vartheta}{\cos^2\vartheta}+\gamma(v))^2\big)^\alpha}d\vartheta=\int_{-\infty}^\infty\frac{f(x)}{\big(1+\big(x+\gamma(v)\big)^2\big)^\alpha}dx,
\end{split}
   \nonumber
\end{equation}
where the last equality is obtained with the change of variable $x=\sin\vartheta/\cos^2\vartheta$.
\end{proof}

Proposition \ref{TEXP} is the technical core of this appendix. It gives a first description of the structure of the coefficients $\mathfrak{c}$,  $\zeta$ and $\mathfrak{e}$.

\begin{proposizione}\label{TEXP}
For any $\epsilon>0$ there exists $\mathfrak{r}_5=\mathfrak{r}_5(\epsilon)>0$ such that for any $0<r<\mathfrak{r}_5$ there holds:
\begin{equation}
    \lvert\partial\mathbb{K}\rvert(B_r(\mathcal{X}))=\mathfrak{c}_nr^{2n+1}+\mathfrak{e}(\mathcal{X})r^{2n+3}+\epsilon R_6(r),
    \label{N:5}
\end{equation}
where defined  $\mathbb{S}(\mathfrak{n}):=\mathbb{S}^{2n-1}\cap\mathfrak{n}^\perp$, we have:
\begin{itemize}
\item[(i)]the coefficient $\mathfrak{c}(\mathcal{X})$ is independent on $\mathcal{X}$ and:
$$\mathfrak{c}(\mathcal{X})=\mathfrak{c}_n:=\frac{\sqrt{\pi}\Gamma\left(\frac{2n-1}{4}\right)\sigma(\mathbb{S}(\mathfrak{n}))}{(2n+1)\Gamma\left(\frac{2n+1}{4}\right)},$$
\item[(ii)]the coefficient $\mathfrak{e}(\mathcal{X})$ has the following expression: 
$$\mathfrak{e}(\mathcal{X})=\int_{\mathbb{S}(\mathfrak{n})}\int_{-\frac{\pi}{2}}^{\frac{\pi}{2}}\frac{\mathfrak{d}(\vartheta)\Big(\big(\frac{7}{32}+\frac{n}{16}\big)\frac{\overline{B}^2}{A^2}-\frac{\overline{C}}{4A}-\frac{\mathcal{A}\overline{B}}{4A}+\frac{\overline{\mathcal{B}}}{(2n+3)}\Big)}{c^2 A^{\frac{2n+3}{4}}}d\vartheta d\sigma,$$
where the coefficients $A,\overline{B},\overline{C}$ were introduced in Proposition \ref{strutto} and $\mathcal{A}$ and $\overline{\mathcal{B}}$ in Proposition \ref{sviluppodens},
\item[(iii)] $\lvert R_6(r)\rvert\leq \mathfrak{C}_{11}(n) r^{2n+3}$ for any $0<r<\mathfrak{r}_5$ and for some constant $\mathfrak{C}_{11}(n)=\mathfrak{C}_{11}(n,\mathcal{X})$ depending only on $\mathcal{X}$.
\end{itemize}
\end{proposizione}

\begin{proof}First of all, thanks to the definition of $\mathfrak{d}(\vartheta)$, see Proposition \ref{conto2}, we have that the density $\Xi(\rho,\vartheta, v)$ introduced in Proposition \ref{rapr1}(i), can be rewritten as:
\begin{equation}
    \Xi(\rho,\vartheta, v)=\mathfrak{d}(\vartheta)\cdot c^{-1}\rho^{2n}\cdot2\lvert(\mathcal{D}+J)[x+\mathcal{P}(\rho,\vartheta, v)]\rvert.
    \label{numerooo8}
\end{equation}
Hence, thanks to \eqref{numerooo8}, Lemma \ref{sviluppodens} and Proposition \ref{rapr1}, we deduce that:
    \begin{equation}
    \begin{split}
&\qquad\qquad\lvert\partial\mathbb{K}\rvert(B_r(\mathcal{X}))=\int_{\mathbb{S}^{2n-1}\cap\mathfrak{n}^\perp}\int_{-\frac{\pi}{2}}^{\frac{\pi}{2}}\int_{\{H(\rho,\vartheta,v)\leq r^4\}}\Xi(\rho,\vartheta, v)  d\rho d\vartheta d\sigma(v)\\
=&\int_{\mathbb{S}(\mathfrak{n})}\int_{-\frac{\pi}{2}}^{\frac{\pi}{2}}\int_{\{H(\rho,\vartheta,v)\leq r^4\}} \mathfrak{d}(\vartheta)\left( \rho^{2n}+\frac{\mathcal{A}}{c}\rho^{2n+1}+\frac{\mathcal{B}}{c}\rho^{2n+2}+\frac{R_5(\rho)\rho^{2n}}{c}\right)  d\rho d\vartheta d\sigma(v).
\nonumber
\end{split}
\end{equation}
We now proceed giving estimates of each term in the last line of the above identity. In the proof of Proposition \ref{propexp} we showed that $\{H(\rho,\vartheta,v)\leq r^4\}\subseteq\{\rho\leq \mathfrak{c}_1r\}$ whenever $0<r<\mathfrak{r}_1$ and where we recall that $\mathfrak{c}_1$ is a constant independent on $\vartheta$ and $v$. Therefore, if $r\leq \mathfrak{r}_1$ we have:
\begin{equation}
\begin{split}
     \int_{\{H(\rho,\vartheta,v)\leq r^4\}}R_5(\rho)\rho^{2n} d\rho\leq\int_0^{\mathfrak{c}_1r}R_5(\rho)\rho^{2n} d\rho\leq\mathfrak{c}_5\int_0^{\mathfrak{c}_1r}\rho^{2n+3} d\rho
     \leq \frac{\mathfrak{c}_5 \mathfrak{c}_1^{2n+4} r^{2n+4}}{(2n+4)},
\end{split}
\label{numbero44}
\end{equation}
where the second inequality comes from Lemma \ref{sviluppodens} provided $\mathfrak{c}_1r\leq\mathfrak{r}_4$.
Moreover, by Proposition \ref{Co1} for any $\epsilon>0$ there exists $\mathfrak{r}_3>0$ for which for any $0<r<\mathfrak{r}_3$ we have:
\begin{equation}
    \begin{split}
    \Big\lvert\int_{\{H(\rho,\vartheta,v)\leq r^4\}} &\rho^j d\rho-\int_0^{P_{\vartheta, v}(r)} \rho^j d\rho\Big\rvert  \leq \int_0^{P_{\vartheta,v}(r)+\epsilon r^3} \rho^j d\rho-\int_0^{P_{\vartheta,v}(r)-\epsilon r^3} \rho^j d\rho\\
    =&\frac{(P_{\vartheta,v}(r)+\epsilon r^3)^{j+1}-(P_{\vartheta,v}(r)-\epsilon r^3)^{j+1}}{j+1}\leq \frac{2^{j+1}\epsilon P_{\vartheta,v}(r)^j r^3}{j+1},
        \label{numbero45}
    \end{split}
\end{equation}
where in the last estimate we used the fact that $(1+a)^{j+1}-(1-a)^{j+1}\leq 2^{j+1}a$ provided $0<a<1$, which is true provided $\epsilon$ is chosen sufficiently small: note that such a choice depends only on $\mathcal{X}$ and not on $\vartheta,v$ thanks to the definition of $P_{\vartheta,v}(r)$.
Thanks to the bounds \eqref{numbero44} and \eqref{numbero45}, for any $0<r<\min\{\mathfrak{r}_1,\mathfrak{r}_2,\mathfrak{r}_3,\mathfrak{r}_4/\mathfrak{c}_1\}$ we have:
\begin{equation}
    \begin{split}
        \lvert R_7^{\vartheta,v}(r)\rvert:=
        &\Bigg\lvert\int_{\{H(\rho,\vartheta,v)\leq r^4\}} \Big( \rho^{2n}+\frac{\mathcal{A}}{c}\rho^{2n+1}+\frac{\mathcal{B}}{c}\rho^{2n+2}+\frac{R_5(\rho)\rho^{2n}}{c}\Big)  d\rho\\
        &\qquad\qquad\qquad\qquad\qquad\qquad-\int_0^{P_{\vartheta, v}(r)} \Big( \rho^{2n}+\frac{\mathcal{A}}{c}\rho^{2n+1}+\frac{\mathcal{B}}{c}\rho^{2n+2}\Big)  d\rho\Bigg\rvert\\
        \leq &\epsilon\frac{2^{2n+1} P_{\vartheta,v}(r)^{2n} r^3}{2n+1}+\epsilon\frac{2^{2n+2}\lVert\mathcal{A}\rVert_\infty P_{\vartheta,v}(r)^{2n+1} r^3}{(2n+2)c}\\
        &\qquad\qquad\qquad\qquad\,\,\,+\epsilon\frac{2^{2n+3}\lVert\mathcal{B}\rVert_\infty P_{\vartheta,v}(r)^{2n+2} r^3}{(2n+3)c}+\frac{\mathfrak{c}_5 \mathfrak{c}_1^{2n+4} r^{2n+4}}{(2n+4)},
    \end{split}
\label{numbero50}
\end{equation}
where $\lVert\mathcal{A}\rVert_\infty$ and $\lVert\mathcal{B}\rVert_\infty$ are the supremum norms of the functions $\mathcal{A}$ and $\mathcal{B}$ on $[-\pi/2,\pi/2]\times \mathbb{S}(\mathfrak{n})$. Furthermore, inequality \eqref{numbero50} together with the definition of $P_{\vartheta,v}$, see \eqref{N:3}, implies that we can find  an $\mathfrak{r}_6>0$ depending only on $\epsilon$ and $\mathcal{X}$ such that for any $0<r<\mathfrak{r}_6$ we have $\lvert R_7^{\vartheta,v}\rvert\leq \epsilon\mathfrak{C}_{12}(n) r^{2n+3}$, where the constant $\mathfrak{C}_{12}(n)$ depends only on $\mathcal{X}$.

Moreover, explicitly computing $P_{\vartheta,v}(r)^j$ for $j=2n+1,\ldots,2n+3$ and truncating the expression to the desired order of $r$, we can prove the  
existence of an $\mathfrak{r}_7>0$ and a constant $\mathfrak{C}_{13}(n)>0$, depending only on $\mathcal{X}$, such that:
\begin{equation}
\begin{split}
\lvert\Delta_1^{\vartheta,v}(r)\rvert:=&\Bigg\lvert P_{\vartheta, v}(r)^{2n+1}-\frac{r^{2n+1}}{A^\frac{2n+1}{4}}+\frac{(2n+1)Br^{2n+2}}{4A^{\frac{2n+6}{4}}}\\
&\qquad\qquad\qquad-\frac{2n+1}{A^\frac{2n+3}{4}}\Big((2n+7)\frac{B^2}{32A^2}-\frac{C}{4A}\Big)r^{2n+3}\Bigg\rvert\leq \mathfrak{C}_{13}(n) r^{2n+4},\\
\lvert\Delta_2^{\vartheta,v}(r)\rvert:=&\Bigg\lvert P_{\vartheta, v}(r)^{2n+2}-\frac{r^{2n+2}}{A^\frac{2n+2}{4}}+\frac{(2n+2)Br^{2n+3}}{4A^\frac{2n+7}{4}}\Bigg\rvert\leq \mathfrak{C}_{13}(n) r^{2n+4},\\
\lvert\Delta_3^{\vartheta,v}(r)\rvert:=&\Bigg\lvert P_{\vartheta, v}(r)^{2n+3}-\frac{r^{2n+3}}{A^\frac{2n+3}{4}}\Bigg\rvert\leq \mathfrak{C}_{13}(n) r^{2n+4}.
\end{split}
\label{numbero51}
\end{equation}

Therefore, defined $R_8^{\vartheta,v}(r):=R_7^{\vartheta,v}(r)+\Delta_1^{\vartheta,v}(r)+\Delta_2^{\vartheta,v}(r)+\Delta_3^{\vartheta,v}(r)$, thanks to the bounds in \eqref{numbero50} and \eqref{numbero51}, we infer that, for any $\epsilon>0$ there exists an $\mathfrak{r}_8:=\mathfrak{r}_8(\epsilon)>0$ and a constant $\mathfrak{C}_{14}(n)$, depending only on $\epsilon$ and $\mathcal{X}$,
such that for any $0<r<\mathfrak{r}_8$, we have $\lvert R_8^{\vartheta,v}(r)\rvert\leq \epsilon\mathfrak{C}_{14}(n) r^{2n+3}$ and:
\begin{equation}
    \begin{split}
    \int_{\{H(\rho,\vartheta,v)\leq r^4\}} \bigg( \rho^{2n}+&\frac{\mathcal{A}}{c}\rho^{2n+1}+\frac{\mathcal{B}}{c}\rho^{2n+2}+\frac{R_5(\rho)\rho^{2n}}{c}\bigg)  d\rho\\
    =&\int_0^{P_{\vartheta, v}(r)} \bigg( \rho^{2n}+\frac{\mathcal{A}}{c}\rho^{2n+1}+\frac{\mathcal{B}}{c}\rho^{2n+2}\bigg)  d\rho+R_7^{\vartheta, v}(r)\\
    =& \frac{P_{\vartheta,v}(r)^{2n+1}}{2n+1}+\frac{\mathcal{A}P_{\vartheta,v}(r)^{2n+2}}{(2n+2)c}+\frac{\mathcal{B} P_{\vartheta,v}(r)^{2n+3}}{(2n+3)c}+R_7^{\vartheta,v}(r)\\
    =&\frac{r^{2n+1}}{(2n+1)A^{\frac{2n+1}{4}}}
    +\bigg(\frac{\mathcal{A}}{(2n+2)c}-\frac{B}{4A}\bigg)\frac{r^{2n+2}}{A^{\frac{2n+2}{4}}}\\&\quad+\bigg((2n+7)\frac{\overline{B}^2}{32A^2}-\frac{\overline{C}}{4A}-\frac{\mathcal{A}\overline{B}}{4A}+\frac{\overline{\mathcal{B}}}{(2n+3)}\bigg)\frac{r^{2n+3}}{c^2A^\frac{2n+3}{4}}+R_8^{\vartheta,v}(r),
        \nonumber
    \end{split}
\end{equation}
 where the functions $\overline{B},\ldots,\overline{E}$ and $\mathcal{A}$, $\overline{\mathcal{B}}$ where introduced in the statement of Proposition \ref{strutto} and in Lemma \ref{sviluppodens} respectively. 
Putting together what we have deduced so far, we get:
\begin{equation}
    \begin{split}
&\qquad\lvert \partial\mathbb{K}\rvert(B_r(x))=\frac{r^{2n+1}}{(2n+1)}\int_{\mathbb{S}(\mathfrak{n})}\int_{-\frac{\pi}{2}}^{\frac{\pi}{2}}\frac{\mathfrak{d}(\vartheta)}{A^{\frac{2n+1}{4}}}d\vartheta d\sigma\\
&\qquad\qquad\qquad\qquad+r^{2n+2}\int_{\mathbb{S}(\mathfrak{n})}\int_{-\frac{\pi}{2}}^{\frac{\pi}{2}} \frac{\mathfrak{d}(\vartheta)}{A^{\frac{2n+2}{4}}}\left(\frac{\mathcal{A}}{(2n+2)c}-\frac{B}{4A}\right)d\vartheta d\sigma\\
+r^{2n+3}&\int_{\mathbb{S}(\mathfrak{n})}\int_{-\frac{\pi}{2}}^{\frac{\pi}{2}}\frac{\mathfrak{d}(\vartheta)}{c^2A^{\frac{2n+3}{4}}}\left(\frac{(2n+7)\overline{B}^2}{32A^2}-\frac{\overline{C}}{4A}-\frac{\mathcal{A}\overline{B}}{4A}+\frac{\overline{\mathcal{B}}}{(2n+3)}\right)d\vartheta d\sigma+R_6(r),
        \label{numerooo10}
    \end{split}
\end{equation}
where $R_6(r):=\int_{\mathbb{S}(\mathfrak{n})}\int_{-\frac{\pi}{2}}^{\frac{\pi}{2}}R_8^{\vartheta,v}(r) d\vartheta d\sigma$. Furthermore, we note that for any $0<r<\mathfrak{r}_8$ we have $\lvert R_6(r)\rvert\leq \epsilon \mathfrak{C}_{15} r^{2n+3}$, where $\mathfrak{C}_{15}$ depends only on $\mathcal{X}$. 

In addition to this, since $B$ and $\mathcal{A}$ are odd functions on $\mathbb{S}(\mathfrak{n})$, see Lemma \ref{symA} and Remark \ref{rk22}, we deduce that:
$$\int\limits_{\mathbb{S}(\mathfrak{n})}\frac{\mathcal{A}}{A^\frac{2n+2}{4}} d\sigma=0,\qquad\int\limits_{\mathbb{S}(\mathfrak{n})}\frac{B}{4A^{\frac{2n+6}{4}}}d\sigma=0.$$
Thanks to Fubini's theorem the above identities combined with \eqref{numerooo10}, yield:
\begin{equation}
    \begin{split}
&\qquad\qquad\qquad\lvert \partial\mathbb{K}\rvert(B_r(x))=\frac{r^{2n+1}}{(2n+1)}\int_{\mathbb{S}(\mathfrak{n})}\int_{-\frac{\pi}{2}}^{\frac{\pi}{2}}\frac{\mathfrak{d}(\vartheta)}{A^{\frac{2n+1}{4}}}d\vartheta d\sigma\\
+&r^{2n+3}\int_{\mathbb{S}(\mathfrak{n})}\int_{-\frac{\pi}{2}}^{\frac{\pi}{2}}\frac{\mathfrak{d}(\vartheta)}{A^{\frac{2n+3}{4}}}\left(\frac{(2n+7)\overline{B}^2}{32A^2}-\frac{\overline{C}}{4A}-\frac{\mathcal{A}\overline{B}}{4A}+\frac{\overline{\mathcal{B}}}{(2n+3)}\right)d\vartheta d\sigma+R_3(r),
        \nonumber
    \end{split}
\end{equation}
We are left to prove that $\int_{\mathbb{S}(\mathfrak{n})}\int_{-\frac{\pi}{2}}^{\frac{\pi}{2}}\frac{\mathfrak{d}(\vartheta)}{A^{\frac{2n+1}{4}}}d\vartheta d\sigma$ is a constant depending only on $n$. This is done by explicit computation:
\begin{equation}
\begin{split}
    \int_{-\frac{\pi}{2}}^{\frac{\pi}{2}}\frac{\mathfrak{d}(\vartheta)}{A^{\frac{2n+1}{4}}}d\vartheta=\int_{-\infty}^\infty\frac{dx}{(1+(x+\gamma(v))^2)^\frac{2n+1}{4}}=\frac{\sqrt{\pi}\Gamma\left(\frac{2n-1}{4}\right)}{\Gamma\left(\frac{2n+1}{4}\right)},
\end{split}
    \nonumber
\end{equation}
where we used Lemma \ref{conto2} in the first equality and Lemma \ref{conto1} in the second .
\end{proof}

\begin{osservazione}
The constant $\mathfrak{c}_n$ is the same constant appearing in Propositions \ref{rapperhor} and \ref{MANCA}. Although we could deduce from Proposition \ref{rapperhor} that the leading term in the expansion \eqref{N:5} is constant and its value, we nevertheless decided to carry out the explicit computation of the coefficient $\mathfrak{c}(\mathcal{X})$ in the proof of Proposition \ref{TEXP}.
\end{osservazione}

\begin{definizione}\label{Cn}
Throughout the rest of this appendix we will always denote by $\mathcal{C}_n$ the constant:
$$\mathcal{C}_n:=\frac{\sqrt{\pi}\Gamma\left(\frac{2n+1}{4}\right)}{\frac{2n+3}{4}\Gamma\left(\frac{2n+3}{4}\right)}.$$
\end{definizione}

In the previous propositions we gave a first characterisation of the coefficient of the Taylor expansion of the perimeter of quadratic surfaces. The coefficient relative $r^{2n+1}$ has been proved to be a constant depending only on $n$ and the one of $r^{2n+2}$ has been proved null. In the following proposition we investigate more carefully the structure of the coefficient relative to $r^{2n+3}$:

\begin{proposizione}\label{TEXP2}
In the notation of the previous propositions we have:
\begin{equation}
   \begin{split}
    \frac{c^2\mathfrak{e}(\mathcal{X})}{\mathcal{C}_n\sigma(\mathbb{S}(\mathfrak{n}))}=&\frac{1}{4}\frac{\text{Tr}(\mathcal{D}^2)-2\langle \mathfrak{n},\mathcal{D}^2\mathfrak{n}\rangle+\langle \mathfrak{n},\mathcal{D}\mathfrak{n}\rangle^2}{2n-1}+\frac{n-1}{2n-1}
        -\frac{1}{4}\\
        &\qquad\qquad\qquad\qquad\qquad\qquad+\frac{\langle \mathcal{D}J\mathfrak{n},\mathfrak{n}\rangle}{2n-1}-\frac{1}{8}\frac{\left(\text{Tr}(\mathcal{D})-\langle \mathfrak{n},\mathcal{D}\mathfrak{n}\rangle\right)^2}{2n-1}.
    \nonumber
\end{split} 
\end{equation}
\end{proposizione}

\begin{proof}
Thanks to Proposition \ref{TEXP}(ii), we have the following expression for the term $\mathfrak{e}(\mathcal{X})$:
\begin{equation}
\begin{split}
c^2\mathfrak{e}(\mathcal{X})=&\int_{\mathbb{S}(\mathfrak{n})}\Bigg(\underbrace{\frac{2n+7}{32}\int_{-\frac{\pi}{2}}^{\frac{\pi}{2}}\frac{\mathfrak{d}(\vartheta)\overline{B}^2}{A^\frac{2n+11}{4}}d\vartheta }_\text{(I)}-\underbrace{\int_{-\frac{\pi}{2}}^{\frac{\pi}{2}}\frac{\mathfrak{d}(\vartheta)\overline{C}}{4A^\frac{2n+7}{4}}d\vartheta}_\text{(II)}\\
&\qquad\qquad\qquad\qquad\qquad-\underbrace{\int_{-\frac{\pi}{2}}^{\frac{\pi}{2}}\frac{\mathfrak{d}(\vartheta)\mathcal{A}\overline{B}}{4A^\frac{2n+7}{4}}d\vartheta}_\text{(III)} +\underbrace{\int_{-\frac{\pi}{2}}^{\frac{\pi}{2}} \frac{\mathfrak{d}(\vartheta)\overline{\mathcal{B}}}{(2n+3)A^\frac{2n+3}{4}}d\vartheta}_\text{(IV)}\Bigg) d\sigma.
    \label{bigexp}
\end{split}
\end{equation}
We now study each one of the terms (I)$,\ldots,$(IV) separately. Let us start with the integral (I). Since $\overline{B}^2=\cos^{10}\vartheta16\beta_\mathfrak{n}(v)^2(\frac{\sin\vartheta}{\cos^2\vartheta})^2(\frac{\sin\vartheta}{\cos^2\vartheta}+\gamma(v))^2$, Lemmas \ref{conto1} and \ref{conto2} imply:
\begin{equation}
\begin{split}
&\frac{2n+7}{32}\int_{-\frac{\pi}{2}}^{\frac{\pi}{2}}\frac{\mathfrak{d}(\vartheta)\overline{B}^2}{A^{\frac{2n+11}{4}}}d\vartheta
=\frac{2n+7}{32}\int_{-\infty}^\infty \frac{16\beta_\mathfrak{n}(v)^2x^2(\gamma(v)+x)^2}{(1+(\gamma(v)+x)^2)^{\frac{2n+11}{4}}}d\vartheta\\
=&\frac{(2n+7)\beta_\mathfrak{n}(v)^2}{2}\int_{-\infty}^{\infty}\frac{x^2(x-\gamma(v))^2}{(1+x^2)^{\frac{2n+11}{4}}}dx\\
=&\frac{(2n+7)\beta_\mathfrak{n}(v)^2}{2}\left(\int_{-\infty}^{\infty}\frac{x^4}{(1+x^2)^{\frac{2n+11}{4}}}dx+\gamma(v)^2\int_{-\infty}^{\infty}\frac{x^2}{(1+x^2)^{\frac{2n+11}{4}}}dx\right)\\
=&2\mathcal{C}_n\beta_\mathfrak{n}(v)^2\left(\frac{3}{4}+\frac{(2n+1)\gamma(v)^2}{8}\right),
\label{numbero(I)}
\end{split}
\end{equation}
where the constant $\mathcal{C}_n$ was introduced in Definition \ref{Cn}.

We turn now our attention to (II). Since $\overline{C}=\cos^6\vartheta(\frac{\sin\vartheta}{\cos^2\vartheta})^2\Big((2+4\beta_\mathfrak{n}(v)^2+2\gamma(v)\alpha_\mathfrak{n})+2\alpha_\mathfrak{n}\frac{\sin\vartheta}{\cos^2\vartheta}\Big)$, Lemmas \ref{conto1} and \ref{conto2} imply:
\begin{equation}
    \begin{split}
        \int_{-\frac{\pi}{2}}^{\frac{\pi}{2}}&\frac{\mathfrak{d}(\vartheta)\overline{C}}{4A^\frac{2n+7}{4}}d\vartheta
         =\frac{1}{2}\int_{-\infty}^{\infty}\frac{x^2((1+2\beta_\mathfrak{n}(v)^2+\gamma(v)\alpha_\mathfrak{n})+\alpha_\mathfrak{n}x)}{\left(1+(x+\gamma(v))^2\right)^{\frac{2n+7}{4}}}dx\\
         =&\frac{1}{2}\int_{-\infty}^{\infty}\frac{(x-\gamma(v))^2((1+2\beta_\mathfrak{n}(v)^2)+\alpha_\mathfrak{n}x)}{\left(1+x^2\right)^{\frac{2n+7}{4}}}dx\\
         =&\frac{(1+2\beta_\mathfrak{n}(v)^2-2\alpha_\mathfrak{n}\gamma(v))}{2}\int_{-\infty}^{\infty}\frac{x^2}{\left(1+x^2\right)^{\frac{2n+7}{4}}}dx\\
         &\qquad\qquad\qquad\qquad\qquad\qquad\qquad+\frac{(1+2\beta_\mathfrak{n}(v)^2)\gamma(v)^2}{2}\int_{-\infty}^\infty\frac{dx}{\left(1+x^2\right)^{\frac{2n+7}{4}}}\\
         =&\mathcal{C}_n\frac{(2n+1)(1+2\beta_\mathfrak{n}^2(v))\gamma(v)^2+2+4\beta_\mathfrak{n}(v)^2-4\gamma(v)\alpha_\mathfrak{n}}{8}.
    \end{split}
    \label{numbero(II)}
\end{equation}
Since $\mathcal{A}\overline{B}=8\cos^6\vartheta(\beta_\mathfrak{n}(v)+\langle Jv,\mathfrak{n}\rangle)\beta_\mathfrak{n}(v)\frac{\sin\vartheta}{\cos^2\vartheta}\Big(\gamma(v)+\frac{\sin\vartheta}{\cos^2\vartheta}\Big)$, Lemmas \ref{conto1} and \ref{conto2} imply:
\begin{equation}
    \begin{split}
        &\int_{-\frac{\pi}{2}}^{\frac{\pi}{2}}\frac{\mathfrak{d}(\vartheta)\mathcal{A}\overline{B}}{4A^{\frac{2n+7}{4}}}d\vartheta
        =2(\beta_\mathfrak{n}(v)+\langle Jv,\mathfrak{n}\rangle)\beta_\mathfrak{n}(v)\int_{-\infty}^{\infty}\frac{x(\gamma(v)+x)}{(1+(\gamma(v)+x)^2)^{\frac{2n+7}{4}}}d\vartheta\\
        =&2(\beta_\mathfrak{n}(v)+\langle Jv,\mathfrak{n}\rangle)\beta_\mathfrak{n}(v)\int_{-\infty}^{\infty}\frac{x(x-\gamma(v))}{(1+x^2)^{\frac{2n+7}{4}}}dx
        =\mathcal{C}_n(\beta_\mathfrak{n}(v)+\langle Jv,\mathfrak{n}\rangle)\beta_\mathfrak{n}(v),
    \end{split}
    \label{numbero(III)}
\end{equation}
which concludes the discussion of the integral (III). Finally, we are left with the discussion of (IV). Thanks to the fact that $\overline{\mathcal{B}}=2\cos^2\vartheta\big(\alpha_\mathfrak{n}\frac{\sin\vartheta}{\cos^2\vartheta}+\lvert P_\mathfrak{n}(\mathcal{D}+J)v\rvert^2\big)$ Lemmas \ref{conto1} and \ref{conto2}, imply that:
\begin{equation}
    \begin{split}
    \int_{-\frac{\pi}{2}}^{\frac{\pi}{2}}\frac{\mathfrak{d}(\vartheta)}{A^{\frac{2n+3}{4}}}\frac{\overline{\mathcal{B}}}{(2n+3)}d\vartheta
    =&\frac{2}{2n+3}\int_{-\infty}^\infty \frac{\alpha_\mathfrak{n}x+\lvert P_\mathfrak{n}(\mathcal{D}+J)v\rvert^2}{(1+(x+\gamma(v)))^\frac{2n+3}{4}}dx\\
    =&\mathcal{C}_n\frac{-\alpha_\mathfrak{n} \gamma(v)+\lvert P_\mathfrak{n}(\mathcal{D}+J)v\rvert^2}{2}.
    \end{split}
    \label{numbero(IV)}
\end{equation}
Plugging the identities \eqref{numbero(I)}, \eqref{numbero(II)}, \eqref{numbero(III)}, \eqref{numbero(IV)} into \eqref{bigexp}, we get:
\begin{equation}
    \begin{split}
    \frac{c^2\mathfrak{e}(\mathcal{X})}{\mathcal{C}_n\sigma(\mathbb{S}(\mathfrak{n}))}=&\fint_{\mathbb{S}(\mathfrak{n})} \bigg[\beta_\mathfrak{n}(v)^2\left(\frac{3}{2}+\frac{2n+1}{4}\gamma(v)^2\right)\bigg]d\sigma(v)\\
    &-\fint_{\mathbb{S}(\mathfrak{n})} \bigg[\frac{(2n+1)(1+2\beta_\mathfrak{n}^2(v))\gamma(v)^2+2+4\beta_\mathfrak{n}(v)^2-4\gamma(v)\alpha_\mathfrak{n}}{8}\bigg]d\sigma(v)\\
    &+\fint_{\mathbb{S}(\mathfrak{n})}\left[-(\beta_\mathfrak{n}(v)+\langle Jv,\mathfrak{n}\rangle)\beta_\mathfrak{n}(v) +\frac{-\alpha_\mathfrak{n} \gamma(v)+\lvert P_\mathfrak{n}(\mathcal{D}+J)v\rvert^2}{2}\right] d\sigma(v)\\
    =-\frac{2n+1}{8}&\underbrace{\fint_{\mathbb{S}(\mathfrak{n})}\gamma(v)^2 d\sigma}_\text{(V)}-\frac{1}{4}-\underbrace{\fint_{\mathbb{S}(\mathfrak{n})}\beta_\mathfrak{n}(v)\langle Jv,n\rangle d\sigma}_\text{(VI)}+\underbrace{\fint_{\mathbb{S}(\mathfrak{n})}\frac{\lvert P_\mathfrak{n}(\mathcal{D}+J)v\rvert^2}{2} d\sigma}_\text{(VII)}.
        \label{equizias}
    \end{split}
\end{equation}

Eventually, in order to make the expression for $\mathfrak{e}(\mathcal{X})$ explicit, we need to compute the integrals (V), (VI) and (VII). To do so, we let $\mathcal{E}:=\{m_1,\ldots,m_{2n}\}$ be an orthonormal basis  of $\R^{2n}$ such that $m_1=\mathfrak{n}$ and:
\begin{equation}
    J m_i=\begin{cases}
    m_{n+i} &\text{if }i\in\{1,\ldots,n\},\\
    -m_{i-n} &\text{if }i\in\{n+1,\ldots,2n\}.
    \end{cases}
\end{equation}
With respect to the basis $\mathcal{E}$, the points $v\in\mathbb{S}(\mathfrak{n})$ are written as $v=\sum_{i=2}^{2n} v_i m_i$ where $v_i:=\langle v,m_i\rangle$. This is due to the fact that $v\in \mathfrak{n}^\perp$ by definition of $\mathbb{S}(\mathfrak{n})$. With these notations, the integral (V) becomes:
\begin{equation}
    \begin{split}
        \fint_{\mathbb{S}(\mathfrak{n})}&\gamma(v)^2 d\sigma(v)=\fint_{\mathbb{S}(\mathfrak{n})}\left(\sum_{i,j=2}^{2n}\langle m_i,\mathcal{D}m_j\rangle v_iv_j\right)^2d\sigma(v)\\
        =&\sum_{i,j,k,l=2}^{2n}\langle m_i,\mathcal{D}m_j\rangle\langle m_k,\mathcal{D}m_l\rangle \fint_{\mathbb{S}(\mathfrak{n})}v_iv_jv_kv_ld\sigma(v)\\
        =&\sum_{i=2}^{2n}\langle m_i,\mathcal{D}m_i\rangle^2\fint_{\mathbb{S}(\mathfrak{n})}v_i^4d\sigma(v)+\sum_{\substack{2\leq i,j\leq 2n\\ i\neq j}}\langle m_i,\mathcal{D}m_i\rangle\langle m_j,\mathcal{D}m_j\rangle\fint_{\mathbb{S}(\mathfrak{n})}v_i^2v_j^2d\sigma(v)\\
        &\qquad\qquad\qquad\qquad\qquad\qquad\qquad\qquad\qquad+2\sum_{\substack{2\leq i,j\leq 2n\\ i\neq j}}\langle m_i,\mathcal{D}m_j\rangle^2\fint_{\mathbb{S}(\mathfrak{n})}v_i^2v_j^2d\sigma(v),
        \nonumber
    \end{split}
\end{equation}
where the last equality comes from the fact that integrals of odd functions on $\mathbb{S}(\mathfrak{n})$ are null.
By direct computation or using formulas stated at the beginning of section $2c$ in \cite{Kowalski1986Besicovitch-typeSubmanifolds}, we have that:
\begin{equation}
    \begin{split}
        \fint_{\mathbb{S}(\mathfrak{n})}\gamma(v)^2 d\sigma(v)=\frac{3}{4n^2-1}\sum_{i=2}^{2n}\langle m_i,\mathcal{D}m_i\rangle^2
        +&\frac{1}{4n^2-1}\sum_{i\neq k=2}^{2n}\langle m_i,\mathcal{D}m_i\rangle\langle m_k,\mathcal{D}m_k\rangle\\
        +&\frac{2}{4n^2-1}\sum_{i\neq j=2}^{2n}\langle m_i,\mathcal{D}m_j\rangle^2\\
        =\frac{2}{4n^2-1}\sum_{i,j=2}^{2n}\langle m_i,\mathcal{D}m_j\rangle^2+&\frac{1}{4n^2-1}\left(\sum_{i=2}^{2n}\langle m_i,\mathcal{D}m_i\rangle\right)^2.
        \label{numerooo15}
    \end{split}
    \end{equation}
Since the matrix $\mathcal{D}$ is symmetric, the well-known expression $\text{Tr}(\mathcal{D}^2)=\sum_{i,j=1}^{2n}\langle m_i,\mathcal{D}m_j\rangle^2$, implies:
\begin{equation}
    \sum_{i,j=2}^{2n}\langle m_i,\mathcal{D}m_j\rangle^2=\mathrm{Tr}(\mathcal{D}^2)-2\sum_{i=2}^{2n}\langle m_i,\mathcal{D}\mathfrak{n}\rangle^2-\langle\mathfrak{n},\mathcal{D}\mathfrak{n}\rangle^2.
    \label{numerooo11}
\end{equation}
Furthermore, since $\lvert \mathcal{D}\mathfrak{n}\rvert^2=\langle \mathfrak{n},\mathcal{D}^2\mathfrak{n}\rangle$, we also have:
\begin{equation}
    \langle \mathfrak{n},\mathcal{D}^2\mathfrak{n}\rangle=\sum_{i=1}^{2n} \langle m_i,\mathcal{D}\mathfrak{n}\rangle^2=\sum_{i=2}^{2n} \langle m_i,\mathcal{D}\mathfrak{n}\rangle^2+ \langle \mathfrak{n},\mathcal{D}\mathfrak{n}\rangle^2.
    \label{numerooo12}
\end{equation}
Putting together \eqref{numerooo11} and \eqref{numerooo12}, we infer that:
\begin{equation}
    \sum_{i,j=2}^{2n}\langle m_i,\mathcal{D}m_j\rangle^2=\mathrm{Tr}(\mathcal{D}^2)-2\big(\langle \mathfrak{n},\mathcal{D}^2\mathfrak{n}\rangle-\langle \mathfrak{n},\mathcal{D}\mathfrak{n}\rangle^2\big)-\langle\mathfrak{n},\mathcal{D}\mathfrak{n}\rangle^2=\mathrm{Tr}(\mathcal{D}^2)-2\langle \mathfrak{n},\mathcal{D}^2\mathfrak{n}\rangle+\langle\mathfrak{n},\mathcal{D}\mathfrak{n}\rangle^2.
    \label{numerooo14}
\end{equation}
Plugging \eqref{numerooo14} into \eqref{numerooo15}, we conclude that:
    \begin{equation}
    \begin{split}
        \fint_{\mathbb{S}(\mathfrak{n})}\gamma(v)^2 d\sigma(v)
        =&\frac{2\text{Tr}(\mathcal{D}^2)-4\langle \mathfrak{n},\mathcal{D}^2\mathfrak{n}\rangle+2\langle \mathfrak{n},\mathcal{D}\mathfrak{n}\rangle^2+\left(\text{Tr}(\mathcal{D})-\langle \mathfrak{n},\mathcal{D}\mathfrak{n}\rangle\right)^2}{(2n-1)(2n+1)}.
        \nonumber
    \end{split}
    \end{equation}

The computation of the integral (VI) is much easier. Indeed:
\begin{equation}
\begin{split}
    \fint_{\mathbb{S}(\mathfrak{n})}\langle Jv,\mathfrak{n}\rangle\beta_\mathfrak{n}(v)d\sigma(v)=&\sum_{i,j=2}^{2n}\langle Jm_i,\mathfrak{n}\rangle\langle \mathcal{D}m_j,\mathfrak{n}\rangle\fint_{\mathbb{S}(\mathfrak{n})}v_iv_j d\sigma(v)=-\frac{\langle \mathcal{D}J\mathfrak{n},\mathfrak{n}\rangle}{2n-1},
    \nonumber
\end{split}
\end{equation}
where the last equality comes from the fact that $\langle Jm_i,\mathfrak{n}\rangle\neq 0$ if and only if $i=n+1$, by the choice of the basis $\mathcal{E}$, and that:
\begin{equation}
    \fint_{\mathbb{S}(\mathfrak{n})}v_iv_j d\sigma(v)=\begin{cases}
    0 &\text{if } i\neq j,\\
    \frac{1}{2n-1} &\text{if } i=j.
    \end{cases}
    \label{eq100}
\end{equation}
We are left to study the integral (VII). Since $P_\mathfrak{n}$ is the orthogonal projection on $\mathfrak{n}^\perp$, we have that:
\begin{equation}
\begin{split}
    \lvert P_{\mathfrak{n}}((\mathcal{D}+J)v)\rvert^2=&\sum_{i=2}^{2n} \langle m_i, (\mathcal{D}+J)v\rangle^2=\sum_{i,j,k=2}^{2n}v_j v_k\langle m_i, (\mathcal{D}+J)m_j\rangle\langle m_i, (\mathcal{D}+J)m_k\rangle.
    \nonumber
\end{split}
\end{equation}
Furthermore, thanks to \eqref{eq100}, we deduce that:
\begin{equation}
\begin{split}
    \fint_{\mathbb{S}(\mathfrak{n})}\lvert P_{\mathfrak{n}}((\mathcal{D}+J)v)\rvert^2&d\sigma(v)=\sum_{i,j,k=2}^{2n}\langle m_i, (\mathcal{D}+J)m_j\rangle\langle m_i, (\mathcal{D}+J)m_k\rangle\fint_{\mathbb{S}(\mathfrak{n})}v_j v_kd\sigma(v)\\
    =&\sum_{i,j=2}^{2n}\langle m_i, (\mathcal{D}+J)m_j\rangle^2\fint_{\mathbb{S}(\mathfrak{n})}v_j^2d\sigma(v)=\frac{1}{2n-1}\sum_{i,j=2}^{2n}\langle m_i, (\mathcal{D}+J)m_j\rangle^2.
    \nonumber
    \end{split}
\end{equation}
We wish now to make $\sum_{i,j=2}^{2n}\langle m_i, (\mathcal{D}+J)m_j\rangle^2$ more explicit. In order to do so note that by definition of $\mathcal{E}$, we have:
\begin{equation}
    \langle m_i,Jm_j\rangle=\begin{cases}
    -1 &\text{if }i\in\{1,\ldots,n\}\text{ and }j=i+n,\\
    1 &\text{if }i\in\{n+1,\ldots,2n\}\text{ and }j=i-n,\\
    0 &\text{otherwise}.
    \end{cases}
    \label{eq101}
\end{equation}
The identities in \eqref{eq101} imply that $\sum_{i,j=2}^{2n}\langle m_i,Jm_j\rangle^2=2n-2$ and since $\mathcal{D}$ is symmetric, we also have that:
$$\sum_{i,j=2}^{2n}\langle m_i,\mathcal{D}m_j\rangle\langle m_i,Jm_j\rangle=0.$$
Summing up what we have proved up to this point, we get:
\begin{equation}
\begin{split}
    \fint_{\mathbb{S}(\mathfrak{n})}\lvert P_{\mathfrak{n}}((\mathcal{D}&+J)v)\rvert^2d\sigma(v)=\frac{1}{2n-1}\sum_{i,j=2}^{2n}\langle m_i,\mathcal{D}m_j\rangle^2+2\langle m_i,\mathcal{D}m_j\rangle\langle m_i,Jm_j\rangle+\langle m_i,Jm_j\rangle^2\\
    =&\frac{\sum_{i,j=2}^{2n}\langle m_i,\mathcal{D}m_j\rangle^2+2n-2}{2n-1}=\frac{\text{Tr}(\mathcal{D}^2)-2\langle \mathfrak{n},\mathcal{D}^2\mathfrak{n}\rangle+\langle \mathfrak{n},\mathcal{D}\mathfrak{n}\rangle^2}{2n-1}+\frac{2n-2}{2n-1},
    \nonumber
    \end{split}
\end{equation}
where the first identity of the second line above comes from \eqref{numerooo14}. Finally putting together the expressions of the integrals (V), (VI) and (VII), we get from \eqref{equizias} that:
\begin{equation}
    \begin{split}
&\frac{c^2\mathfrak{e}(\mathcal{X})}{\mathcal{C}_n\sigma(\mathbb{S}(\mathfrak{n}))}=
-\frac{2n+1}{8}\bigg(\frac{2\text{Tr}(\mathcal{D}^2)-4\langle \mathfrak{n},\mathcal{D}^2\mathfrak{n}\rangle+2\langle \mathfrak{n},\mathcal{D}\mathfrak{n}\rangle^2+\left(\text{Tr}(\mathcal{D})-\langle \mathfrak{n},\mathcal{D}\mathfrak{n}\rangle\right)^2}{(2n-1)(2n+1)}\bigg)\\
        &\qquad\qquad-\frac{1}{4}+\frac{\langle \mathcal{D}J\mathfrak{n},\mathfrak{n}\rangle}{2n-1}+\frac{1}{2}\bigg(\frac{\text{Tr}(\mathcal{D}^2)-2\langle \mathfrak{n},\mathcal{D}^2\mathfrak{n}\rangle+\langle \mathfrak{n},\mathcal{D}\mathfrak{n}\rangle^2}{2n-1}+\frac{2n-2}{2n-1}\bigg)\\
        =&\frac{1}{4}\frac{\text{Tr}(\mathcal{D}^2)-2\langle \mathfrak{n},\mathcal{D}^2\mathfrak{n}\rangle+\langle \mathfrak{n},\mathcal{D}\mathfrak{n}\rangle^2}{2n-1}+\frac{n-1}{2n-1}
        -\frac{1}{4}+\frac{\langle \mathcal{D}J\mathfrak{n},\mathfrak{n}\rangle}{2n-1}-\frac{\left(\text{Tr}(\mathcal{D})-\langle \mathfrak{n},\mathcal{D}\mathfrak{n}\rangle\right)^2}{8(2n-1)},
        \nonumber
    \end{split}
\end{equation}
where the last identity is obtained from the previous ones with few algebraic computations.
\end{proof}

\begin{teorema}\label{appendicefinale}
Assume $\mu$ is a $(2n+1)$-uniform measure supported on the quadric $ \mathbb{K}(0,\mathcal{D},-1)$, see \ref{simmi}. Then, for any $h\in \R^{2n}\setminus\{x\in \R^{2n}:(\mathcal{D}+J)x=0\}$, we have:
\begin{equation}
\begin{split}
0=&\frac{\text{Tr}(\mathcal{D}^2)-2\langle \mathfrak{n}(h),\mathcal{D}^2\mathfrak{n}(h)\rangle+\langle \mathfrak{n}(h),\mathcal{D}\mathfrak{n}(h)\rangle^2}{4(2n-1)}+\frac{n-1}{2n-1}
        -\frac{1}{4}\\
        &\qquad\qquad\qquad\qquad\qquad+\frac{\langle \mathcal{D}J\mathfrak{n}(h),\mathfrak{n}(h)\rangle}{2n-1}-\frac{\left(\text{Tr}(\mathcal{D})-\langle \mathfrak{n}(h),\mathcal{D}\mathfrak{n}(h)\rangle\right)^2}{8(2n-1)},
        \label{eq16}
\end{split}
\end{equation}
where $\mathfrak{n}(h):=(\mathcal{D}+J)h/\lvert(\mathcal{D}+J)h\rvert$.
\end{teorema}

\begin{proof} Suppose that $\mathcal{X}:=(h,f(h))\in\supp(\mu)$. Then Proposition \ref{spt1} implies that if $r$ is sufficiently small, then: $$B_r(\mathcal{X})\cap \mathbb{K}(0,\mathcal{D},1)=B_r(\mathcal{X})\cap \supp(\mu).$$ 
Therefore, Propositions \ref{supportoK} and \ref{TEXP} imply that:
\begin{equation}
\begin{split}
      r^{2n+1}=\mu&(B_r(\mathcal{X}))=\mathcal{C}^{2n+1}\llcorner{\supp(\mu)}(B_r(\mathcal{X}))\\
      &\qquad\overset{\eqref{numeroo41}, \eqref{numeroo42}}{=}\frac{\lvert\partial\mathbb{K}\rvert(B_r(\mathcal{X}))}{\mathfrak{c}_n}=r^{2n+1}+\frac{ \mathfrak{e}(\mathcal{X})}{\mathfrak{c}_n}r^{2n+3}+o(r^{2n+3}),
\end{split}
\end{equation}
whenever $r$ is sufficiently small.

If $n>1$ or $\text{dim}(\Sigma(f))=0$, see \eqref{char}, Proposition \ref{spt1} and Proposition \ref{TEXP2} prove the claim since $\pi_H(\supp(\mu))$ coincides with $\R^{2n}$. On the other hand if $n=1$ and $\dim(\Sigma(f))=1$, Proposition \ref{spt1} implies that one of the two connected components of $\R^2\setminus \Sigma(f)$, that are halfspaces whose boundary passes through $0$, is contained in $\pi_H (\supp(\mu))$ and thus equation \eqref{eq16} holds for any $z$ contained in such connected component. However, the fact that $\mathfrak{n}(-h)=-\mathfrak{n}(h)$ and the parity of $\mathfrak{e}(\mathcal{X})$ with respect to change in sign of $\mathfrak{n}$, show that equation \eqref{eq16} holds on $\R^2\setminus \Sigma(f)$.
\end{proof}

\printbibliography

@book{Capogna2007,
    AUTHOR = {Capogna, Luca and Danielli, Donatella and Pauls, Scott D. and
              Tyson, Jeremy T.},
     TITLE = {An introduction to the {H}eisenberg group and the
              sub-{R}iemannian isoperimetric problem},
    SERIES = {Progress in Mathematics},
    VOLUME = {259},
 PUBLISHER = {Birkh\"{a}user Verlag, Basel},
      YEAR = {2007},
     PAGES = {xvi+223},
      ISBN = {978-3-7643-8132-5},
   MRCLASS = {53C17 (22E30 30C65 32T27 32V15 49Q15)},
  MRNUMBER = {2312336},
MRREVIEWER = {Piotr Haj\l asz},
}

@article {preisstiserBesicovitch,
    AUTHOR = {Preiss, David and Ti\v{s}er, Jaroslav},
     TITLE = {On {B}esicovitch's {$\frac12$}-problem},
   JOURNAL = {J. London Math. Soc. (2)},
  FJOURNAL = {Journal of the London Mathematical Society. Second Series},
    VOLUME = {45},
      YEAR = {1992},
    NUMBER = {2},
     PAGES = {279--287},
      ISSN = {0024-6107},
   MRCLASS = {28A75 (28A78)},
  MRNUMBER = {1171555},
MRREVIEWER = {Hermann Haase},
       DOI = {10.1112/jlms/s2-45.2.279},
       URL = {https://doi.org/10.1112/jlms/s2-45.2.279},
}

@article {FSSCArea,
    AUTHOR = {Franchi, Bruno and Serapioni, Raul Paolo and Serra Cassano,
              Francesco},
     TITLE = {Area formula for centered {H}ausdorff measures in metric
              spaces},
   JOURNAL = {Nonlinear Anal.},
  FJOURNAL = {Nonlinear Analysis. Theory, Methods \& Applications. An
              International Multidisciplinary Journal},
    VOLUME = {126},
      YEAR = {2015},
     PAGES = {218--233},
      ISSN = {0362-546X},
   MRCLASS = {28A75 (22E30 49Q20 53C23)},
  MRNUMBER = {3388880},
MRREVIEWER = {Anne Marie Svane},
       DOI = {10.1016/j.na.2015.02.008},
       URL = {https://doi.org/10.1016/j.na.2015.02.008},
}

@article {EdgarCentered,
    AUTHOR = {Edgar, G. A.},
     TITLE = {Centered densities and fractal measures},
   JOURNAL = {New York J. Math.},
  FJOURNAL = {New York Journal of Mathematics},
    VOLUME = {13},
      YEAR = {2007},
     PAGES = {33--87},
   MRCLASS = {28A78 (28A12 28A80)},
  MRNUMBER = {2288081},
MRREVIEWER = {Lars Olsen},
       URL = {http://nyjm.albany.edu:8000/j/2007/13_33.html},
}

@book{Whittaker.1902AAnalysis,
    AUTHOR = {Whittaker, E. T. and Watson, G. N.},
     TITLE = {A course of modern analysis},
    SERIES = {Cambridge Mathematical Library},
      NOTE = {An introduction to the general theory of infinite processes
              and of analytic functions; with an account of the principal
              transcendental functions,
              Reprint of the fourth (1927) edition},
 PUBLISHER = {Cambridge University Press, Cambridge},
      YEAR = {1996},
     PAGES = {vi+608},
      ISBN = {0-521-58807-3},
   MRCLASS = {01A75 (30-01 33-01)},
  MRNUMBER = {1424469},
       DOI = {10.1017/CBO9780511608759},
       URL = {https://doi.org/10.1017/CBO9780511608759},
}

@article{Magnani2017AMeasure,
    AUTHOR = {Magnani, Valentino},
     TITLE = {A new differentiation, shape of the unit ball, and perimeter
              measure},
   JOURNAL = {Indiana Univ. Math. J.},
  FJOURNAL = {Indiana University Mathematics Journal},
    VOLUME = {66},
      YEAR = {2017},
    NUMBER = {1},
     PAGES = {183--204},
      ISSN = {0022-2518},
   MRCLASS = {28A75 (22E25 53C17)},
  MRNUMBER = {3623407},
       DOI = {10.1512/iumj.2017.66.6007},
       URL = {https://doi.org/10.1512/iumj.2017.66.6007},
}

@article{Kowalski1986Besicovitch-typeSubmanifolds,
    AUTHOR = {Kowalski, Old\v{r}ich and Preiss, David},
     TITLE = {Besicovitch-type properties of measures and submanifolds},
   JOURNAL = {J. Reine Angew. Math.},
  FJOURNAL = {Journal f\"{u}r die Reine und Angewandte Mathematik. [Crelle's
              Journal]},
    VOLUME = {379},
      YEAR = {1987},
     PAGES = {115--151},
      ISSN = {0075-4102},
   MRCLASS = {49F20 (28A75 53A07 58C35)},
  MRNUMBER = {903637},
MRREVIEWER = {Michael Gr\"{u}ter},
}

@book{Federer1996GeometricTheory,
    AUTHOR = {Federer, Herbert},
     TITLE = {Geometric measure theory},
    SERIES = {Die Grundlehren der mathematischen Wissenschaften, Band 153},
 PUBLISHER = {Springer-Verlag New York Inc., New York},
      YEAR = {1969},
     PAGES = {xiv+676},
   MRCLASS = {28.80 (26.00)},
  MRNUMBER = {0257325},
MRREVIEWER = {J. E. Brothers},
}

@article{Preiss1987GeometryDensities,
   AUTHOR = {Preiss, David},
     TITLE = {Geometry of measures in {$\mathbb{ R}^n$}: distribution,
              rectifiability, and densities},
   JOURNAL = {Ann. of Math. (2)},
  FJOURNAL = {Annals of Mathematics. Second Series},
    VOLUME = {125},
      YEAR = {1987},
    NUMBER = {3},
     PAGES = {537--643},
      ISSN = {0003-486X},
   MRCLASS = {28A75},
  MRNUMBER = {890162},
MRREVIEWER = {K. J. Falconer},
       DOI = {10.2307/1971410},
       URL = {https://doi.org/10.2307/1971410},
}

@book{Mattila1995GeometrySpaces,
    AUTHOR = {Mattila, Pertti},
     TITLE = {Geometry of sets and measures in {E}uclidean spaces},
    SERIES = {Cambridge Studies in Advanced Mathematics},
    VOLUME = {44},
      NOTE = {Fractals and rectifiability},
 PUBLISHER = {Cambridge University Press, Cambridge},
      YEAR = {1995},
     PAGES = {xii+343},
      ISBN = {0-521-46576-1},
   MRCLASS = {28A75 (49Q20)},
  MRNUMBER = {1333890},
MRREVIEWER = {Harold Parks},
       DOI = {10.1017/CBO9780511623813},
       URL = {https://doi.org/10.1017/CBO9780511623813},
}

@incollection{Monti2014IsoperimetricGroup,
    AUTHOR = {Monti, Roberto},
     TITLE = {Isoperimetric problem and minimal surfaces in the {H}eisenberg
              group},
 BOOKTITLE = {Geometric measure theory and real analysis},
    SERIES = {CRM Series},
    VOLUME = {17},
     PAGES = {57--129},
 PUBLISHER = {Ed. Norm., Pisa},
      YEAR = {2014},
   MRCLASS = {49Q05 (53C17 53C42)},
  MRNUMBER = {3363670},
MRREVIEWER = {Jos\'{e} Miguel Manzano},
}

@book{Simon1983LecturesTheory,
   AUTHOR = {Simon, Leon},
     TITLE = {Lectures on geometric measure theory},
    SERIES = {Proceedings of the Centre for Mathematical Analysis,
              Australian National University},
    VOLUME = {3},
 PUBLISHER = {Australian National University, Centre for Mathematical
              Analysis, Canberra},
      YEAR = {1983},
     PAGES = {vii+272},
      ISBN = {0-86784-429-9},
   MRCLASS = {49-01 (28A75 49F20)},
  MRNUMBER = {756417},
MRREVIEWER = {J. S. Joel},
}

@ARTICLE{merloMM,
       author = {{Merlo}, Andrea},
        title = "{Marstrand-Mattila rectifiability criterion for $1$-codimensional measures in Carnot Groups}",
      journal = {arXiv e-prints},
         year = 2020,
        month = jul,
          eid = {arXiv:2007.03236},
        pages = {arXiv:2007.03236},
archivePrefix = {arXiv},
       eprint = {2007.03236},
 primaryClass = {math.MG},
       adsurl = {https://ui.adsabs.harvard.edu/abs/2020arXiv200703236M}
}

@article{antonelli2020rectifiable,
       author = {{Antonelli}, Gioacchino and {Merlo}, Andrea},
        title = "{On rectifiable measures in Carnot groups: structure theory}",
      journal = {arXiv e-prints},
     keywords = {Mathematics - Metric Geometry, 53C17, 22E25, 28A75, 49Q15, 26A16},
         year = 2020,
        month = sep,
          eid = {arXiv:2009.13941},
       % pages = {arXiv:2009.13941},
archivePrefix = {arXiv},
       eprint = {2009.13941},
 primaryClass = {math.MG},
       adsurl = {https://ui.adsabs.harvard.edu/abs/2020arXiv200913941A},
      adsnote = {Provided by the SAO/NASA Astrophysics Data System}
}

@article {Marstrand61,
    AUTHOR = {Marstrand, J. M.},
     TITLE = {Hausdorff two-dimensional measure in {$3$}-space},
   JOURNAL = {Proc. London Math. Soc. (3)},
  FJOURNAL = {Proceedings of the London Mathematical Society. Third Series},
    VOLUME = {11},
      YEAR = {1961},
     PAGES = {91--108},
      ISSN = {0024-6115},
   MRCLASS = {28.80},
  MRNUMBER = {123670},
MRREVIEWER = {Wendell H. Fleming},
       DOI = {10.1112/plms/s3-11.1.91},
       URL = {https://doi.org/10.1112/plms/s3-11.1.91},
}

@article {mattila75,
    AUTHOR = {Mattila, Pertti},
     TITLE = {Hausdorff {$m$} regular and rectifiable sets in {$n$}-space},
   JOURNAL = {Trans. Amer. Math. Soc.},
  FJOURNAL = {Transactions of the American Mathematical Society},
    VOLUME = {205},
      YEAR = {1975},
     PAGES = {263--274},
      ISSN = {0002-9947},
   MRCLASS = {28A75},
  MRNUMBER = {357741},
MRREVIEWER = {S. J. Taylor},
       DOI = {10.2307/1997203},
       URL = {https://doi.org/10.2307/1997203},
}

@article {Marstrandoriginal,
    AUTHOR = {Marstrand, John M.},
     TITLE = {The {$(\varphi ,\,s)$} regular subsets of {$n$}-space},
   JOURNAL = {Trans. Amer. Math. Soc.},
  FJOURNAL = {Transactions of the American Mathematical Society},
    VOLUME = {113},
      YEAR = {1964},
     PAGES = {369--392},
      ISSN = {0002-9947},
   MRCLASS = {28.80},
  MRNUMBER = {166336},
MRREVIEWER = {F. J. Almgren, Jr.},
       DOI = {10.2307/1994138},
       URL = {https://doi.org/10.2307/1994138},
}

@article{Chousionis2015MarstrandsGroup,
    AUTHOR = {Chousionis, Vasilis and Tyson, Jeremy T.},
     TITLE = {Marstrand's density theorem in the {H}eisenberg group},
   JOURNAL = {Bull. Lond. Math. Soc.},
  FJOURNAL = {Bulletin of the London Mathematical Society},
    VOLUME = {47},
      YEAR = {2015},
    NUMBER = {5},
     PAGES = {771--788},
      ISSN = {0024-6093},
   MRCLASS = {53C17 (28A78 31E05)},
  MRNUMBER = {3403960},
MRREVIEWER = {Tord Sj\"{o}din},
       DOI = {10.1112/blms/bdv056},
       URL = {https://doi.org/10.1112/blms/bdv056},
}

@book{Bogachev2007MeasureTheory,
    AUTHOR = {Bogachev, V. I.},
     TITLE = {Measure theory. {V}ol. {I}, {II}},
 PUBLISHER = {Springer-Verlag, Berlin},
      YEAR = {2007},
     PAGES = {Vol. I: xviii+500 pp., Vol. II: xiv+575},
      ISBN = {978-3-540-34513-8},
   MRCLASS = {28-02 (28Axx 28Cxx 46G12 60G42 60G44)},
  MRNUMBER = {2267655},
MRREVIEWER = {Ren\'{e} L. Schilling},
       DOI = {10.1007/978-3-540-34514-5},
       URL = {https://doi.org/10.1007/978-3-540-34514-5},
}

@article{Mattila2005MeasuresGroups,
   AUTHOR = {Mattila, Pertti},
     TITLE = {Measures with unique tangent measures in metric groups},
   JOURNAL = {Math. Scand.},
  FJOURNAL = {Mathematica Scandinavica},
    VOLUME = {97},
      YEAR = {2005},
    NUMBER = {2},
     PAGES = {298--308},
      ISSN = {0025-5521},
   MRCLASS = {28A75 (22D05 28A12 28C10)},
  MRNUMBER = {2191708},
MRREVIEWER = {Klaas Pieter Hart},
       DOI = {10.7146/math.scand.a-14977},
       URL = {https://doi.org/10.7146/math.scand.a-14977},
}

@article{ChousionisONGROUP,
       author = {{Chousionis}, Vasilis and {Magnani}, Valentino and {Tyson}, Jeremy T.},
        title = "{On uniform measures in the Heisenberg group}",
      journal = {arXiv e-prints},
     keywords = {Mathematics - Metric Geometry},
         year = "2018",
        month = "08",
          eid = {arXiv:1808.01130},
        %pages = {arXiv:1808.01130},
archivePrefix = {arXiv},
       eprint = {1808.01130},
 primaryClass = {math.MG},
       adsurl = {https://ui.adsabs.harvard.edu/abs/2018arXiv180801130C},
      adsnote = {Provided by the SAO/NASA Astrophysics Data System}
}

@article{Serapioni2001RectifiabilityGroup,
  AUTHOR = {Franchi, Bruno and Serapioni, Raul and Serra Cassano,
              Francesco},
     TITLE = {Rectifiability and perimeter in the {H}eisenberg group},
   JOURNAL = {Math. Ann.},
  FJOURNAL = {Mathematische Annalen},
    VOLUME = {321},
      YEAR = {2001},
    NUMBER = {3},
     PAGES = {479--531},
      ISSN = {0025-5831},
   MRCLASS = {49Q15 (22E25 46E35)},
  MRNUMBER = {1871966},
MRREVIEWER = {Piotr Haj\l asz},
       DOI = {10.1007/s002080100228},
       URL = {https://doi.org/10.1007/s002080100228},
}

@article{LORENT2003RECTIFIABILITYDENSITY,
   AUTHOR = {Lorent, Andrew},
     TITLE = {Rectifiability of measures with locally uniform cube density},
   JOURNAL = {Proc. London Math. Soc. (3)},
  FJOURNAL = {Proceedings of the London Mathematical Society. Third Series},
    VOLUME = {86},
      YEAR = {2003},
    NUMBER = {1},
     PAGES = {153--249},
      ISSN = {0024-6115},
   MRCLASS = {28A75 (28A78)},
  MRNUMBER = {1971467},
MRREVIEWER = {Klaas Pieter Hart},
       DOI = {10.1112/S0024611502013710},
       URL = {https://doi.org/10.1112/S0024611502013710},
}

@article{Ambrosio2000RectifiableSpaces,
AUTHOR = {Ambrosio, Luigi and Kirchheim, Bernd},
     TITLE = {Rectifiable sets in metric and {B}anach spaces},
   JOURNAL = {Math. Ann.},
  FJOURNAL = {Mathematische Annalen},
    VOLUME = {318},
      YEAR = {2000},
    NUMBER = {3},
     PAGES = {527--555},
      ISSN = {0025-5831},
   MRCLASS = {28A75 (46G99 46T99 49Q20)},
  MRNUMBER = {1800768},
MRREVIEWER = {Piotr Haj\l asz},
       DOI = {10.1007/s002080000122},
       URL = {https://doi.org/10.1007/s002080000122},
}

@book{DeLellis2008RectifiableMeasures,
    AUTHOR = {De Lellis, Camillo},
     TITLE = {Rectifiable sets, densities and tangent measures},
    SERIES = {Zurich Lectures in Advanced Mathematics},
 PUBLISHER = {European Mathematical Society (EMS), Z\"{u}rich},
      YEAR = {2008},
     PAGES = {vi+127},
      ISBN = {978-3-03719-044-9},
   MRCLASS = {28-02 (28A75 49Q15)},
  MRNUMBER = {2388959},
MRREVIEWER = {Andrew Lorent},
       DOI = {10.4171/044},
       URL = {https://doi.org/10.4171/044},
}

@article {Balogh2003SizeGradient,
    AUTHOR = {Balogh, Zolt\'{a}n M.},
     TITLE = {Size of characteristic sets and functions with prescribed
              gradient},
   JOURNAL = {J. Reine Angew. Math.},
  FJOURNAL = {Journal f\"{u}r die Reine und Angewandte Mathematik. [Crelle's
              Journal]},
    VOLUME = {564},
      YEAR = {2003},
     PAGES = {63--83},
      ISSN = {0075-4102},
   MRCLASS = {43A80 (42B25 53C17)},
  MRNUMBER = {2021034},
MRREVIEWER = {Adam Sikora},
       DOI = {10.1515/crll.2003.094},
       URL = {https://doi.org/10.1515/crll.2003.094},
}

@misc{Kirchheim2002UniformilySpaces,
     AUTHOR = {Kirchheim, Bernd and Preiss, David},
     TITLE = {Uniformly distributed measures in {E}uclidean spaces},
   JOURNAL = {Math. Scand.},
  FJOURNAL = {Mathematica Scandinavica},
    VOLUME = {90},
      YEAR = {2002},
    NUMBER = {1},
     PAGES = {152--160},
      ISSN = {0025-5521},
   MRCLASS = {28A12 (28C15)},
  MRNUMBER = {1887099},
MRREVIEWER = {Carmen D. Vlad},
       DOI = {10.7146/math.scand.a-14367},
       URL = {https://doi.org/10.7146/math.scand.a-14367},
}

@article {MR2105335,
    AUTHOR = {Magnani, Valentino},
     TITLE = {Unrectifiability and rigidity in stratified groups},
   JOURNAL = {Arch. Math. (Basel)},
  FJOURNAL = {Archiv der Mathematik},
    VOLUME = {83},
      YEAR = {2004},
    NUMBER = {6},
     PAGES = {568--576},
      ISSN = {0003-889X},
   MRCLASS = {53C17 (28A75 53C24)},
  MRNUMBER = {2105335},
MRREVIEWER = {Gerald B. Folland},
       DOI = {10.1007/s00013-004-1057-4},
       URL = {https://doi.org/10.1007/s00013-004-1057-4},
}

%
\end{document}